\documentclass[11pt,english]{smfart}

\usepackage[T1]{fontenc}
\usepackage[english,french]{babel}

\usepackage{amssymb,url,xspace,smfthm}
\usepackage{amsmath}
\usepackage{graphicx}
\usepackage{xcolor}

\usepackage[algo2e,ruled, linesnumbered]{algorithm2e}
\usepackage{cite}
\makeatletter
    \def\ps@copyright{\ps@empty
    \def\@oddfoot{\hfil\small\copyright 2021, \SMF}}
\makeatother

\newcommand{\SMF}{Soci\'et\'e Ma\-th\'e\-Ma\-ti\-que de France}
\newcommand{\BibTeX}{{\scshape Bib}\kern-.08em\TeX}
\newcommand{\T}{\S\kern .15em\relax }
\newcommand{\AMS}{$\mathcal{A}$\kern-.1667em\lower.5ex\hbox
        {$\mathcal{M}$}\kern-.125em$\mathcal{S}$}

\newtheorem{proposition}{Proposition}[section]
\newtheorem{assumption}{Assumption}[section]
\newtheorem{theorem}{Theorem}[section]
\newtheorem{lemma}{Lemma}[section]
\newtheorem{remark}{Remark}[section]

\def\RR{\mathbb R}
\def\EE{\mathbb E}
\def\DH{{\mathcal{D}}}
\def\e{\varepsilon}

\def\L{\textbf{\Lambda}}
\def\F{\mathcal{F}}
\def\w{v}

\def\theta{z}

\def\var{{\mathbb Var}}

\def\cov{{\mathbb Cov}}

\def\LL{L^1_2(\DH\times\RR^{d_v})}

\def\LH{L^1_2(\RR^{d_v})}

\def\llambda{\Lambda}
\def\LLBi{L^1_2(\DH\times\RR^{d_v};L^2(\Omega))}
\def\LHBi{L^1_2(\RR^{d_v};L^2(\Omega))}
\def\mb{\mathbf}
\newcommand{\D}{\mathcal{D}}

\def\vs{\lambda_S}

\def\vi{\lambda_I}

\def\vr{\lambda_R}
\def\be{\begin{equation}}
	\def\ee{\end{equation}}
\def\bea{\begin{eqnarray}}
	\def\eea{\end{eqnarray}}
\def\beas{\begin{eqnarray*}}
	\def\eeas{\end{eqnarray*}}

\newcommand{\h}{\tilde f}
\newcommand{\R}{\ensuremath{\mathbb{R}}}
\newcommand{\g}{q}
\newcommand{\VV}{\mathbb Var}
\newcommand{\CC}{\mathbb Cov}

\def\fP{\em \sf}
\def\ind{\hspace{0.25in}}
\def\cir{\makebox[0.5cm]{$\circ$}}
\def\IR{\mathop{\mbox{\rm Sround}}\nolimits}

\renewcommand{\epsilon}{\varepsilon}
\title[Multi-fidelity methods for kinetic equations]{Multi-fidelity methods for uncertainty propagation in kinetic equations}
\author{Giacomo Dimarco}
\address{Department of Mathematics and Computer Science\\ 
University of Ferrara, Via Machiavelli 30, Ferrara, 44121, Italy}
\email{giacomo.dimarco@unife.it}

\author{Liu Liu}
\address{Department of Mathematics\\ 
The Chinese University of Hong Kong, Shatin, N.T., Hong Kong SAR}
\email{lliu@math.cuhk.edu.hk}

\author{Lorenzo Pareschi}
\address{Department of Mathematics and Computer Science\\ 
University of Ferrara, Via Machiavelli 30, Ferrara, 44121, Italy}
\email{lorenzo.pareschi@unife.it}

\author{Xueyu Zhu}
\address{Department of Mathematics\\ 
University of Iowa, Iowa City, IA 52242, USA}
\email{xueyuzhu@uiowa.edu}

\keywords{Kinetic equations, Boltzmann equation, uncertainty quantification, moment methods, hydrodynamical limits, multi-fidelity methods, surrogate models}

\begin{document}
\def\smfbyname{}

\begin{abstract}
The construction of efficient methods for uncertainty quantification in kinetic equations represents a challenge due to the high dimensionality of the models: often the computational costs involved become prohibitive. On the other hand, precisely because of the curse of dimensionality, the construction of simplified models capable of providing approximate solutions at a computationally reduced cost has always represented one of the main research strands in the field of kinetic equations. Approximations based on suitable closures of the moment equations or on simplified collisional models have been studied by many authors. In the context of uncertainty quantification, it is therefore natural to take advantage of such models in a multi-fidelity setting where the original kinetic equation represents the high-fidelity model, and the simplified models define the low-fidelity surrogate models. The scope of this article is to survey some recent results about multi-fidelity methods for kinetic equations that are able to accelerate the solution of the uncertainty quantification process by combining high-fidelity and low-fidelity model evaluations with particular attention to the case of compressible and incompressible hydrodynamic limits. We will focus essentially on two classes of strategies: multi-fidelity control variates methods and bi-fidelity stochastic collocation methods. The various approaches considered are analyzed in light of the different surrogate models used and the different numerical techniques adopted. Given the relevance of the specific choice of the surrogate model, an application-oriented approach has been chosen in the presentation. 
 
\end{abstract}

\begin{altabstract}
La construction de méthodes rapides pour la quantification de l'incertitude dans les équations cinétiques représente un défi en raison de la grande dimensionnalité des modèles qui impliquent souvent des coûts de calcul prohibitifs. D'autre part, précisément à cause de fléau de la dimension, la construction de modèles simplifiés capables de fournir des solutions approchées à un coût de calcul réduit a toujours représenté l'un des principaux axes de recherche dans le domaine des équations cinétiques. Des approximations basées sur des fermetures appropriées des équations des moments ou sur des modèles collisionnels simplifiés ont été étudiées par de nombreux auteurs. Dans le cadre de la quantification de l'incertitude, il est donc naturel utiliser tels modèles dans un cadre multi-fidélité où l'équation cinétique d'origine représente le modèle haute fidélité, et les modèles simplifiés définissent les modèles à basse fidélité. Ce travail passe en revue quelques résultats récents sur les méthodes multi-fidélité pour les équations cinétiques qui accélèrent la résolution du processus de quantification des incertitudes en combinant des évaluations de modèles haute fidélité et basse fidélité avec une attention particulière au cas des limites hydrodynamiques compressibles et incompressibles. Nous nous concentrerons essentiellement sur deux classes de stratégies : les méthodes des variables de contrôle multi-fidélité et les méthodes de collocation stochastique bi-fidélité. Les différentes approches sont analysées à la lumière des différents modèles simplifiés utilisés et des différentes techniques numériques adoptées. Compte tenu de la pertinence du choix spécifique du modèle de substitution simplifié, une approche orientée aux applications a été choisie dans la présentation.
\end{altabstract}
\maketitle

\tableofcontents

\section{Introduction}
The Boltzmann equation is used to model a very large number of different phenomena ranging from rarefied gas flows such as those found in hypersonic aerodynamics, gases in vacuum technologies, or fluids inside microelectromechanical devices^^>\cite{Cer,cercignani}, to the description of social and biological phenomena^^>\cite{PT2,naldi}. For these reasons, the development of efficient and accurate numerical methods for solving kinetic equations and in particular the Boltzmann equation has experienced great commitment in the past to which contributed many researchers working in different fields^^>\cite{bird1970direct,nanbu1980direct,Hu1,Hu2,Gamba1,Gamba2,Russo1,Russo2,CHWW}. We refer to^^>\cite{DPR,DP15,Russo1,pareschi2001introduction,rjasanow} for recent monographs, collections and surveys.

In spite of the vast amount of existing research on the approximation of Boltzmann and related equations, the study of kinetic equations with stochastic terms has been considered only in recent years^^>\cite{PDL,PZ2020,HPY,JPH,RHS,ZJ,LJ-UQ,GJL,Poette1,Poette2}. See in particular the recent collection^^>\cite{JinPareschi} and the survey^^>\cite{parUQ}. We refer also to^^>\cite{LeMK,PIN_book,PDL,NTW,PDL} for related researches in computational fluid dynamics and hyperbolic conservation laws.

On the other hand, these uncertainties arise naturally in many problems where these models are frequently used. In particular, incomplete knowledge of the interaction mechanism between the particles, imprecise measurements of the initial and boundary data, and unknown details of the domain geometry or forcing terms represent problems that are nearly impossible to overcome in a fully deterministic environment. At the same time, modeling the impact of these uncertainties in the solution is critical to providing reliable results for decision and design processes in applications. 

Most of the literature on quantification of uncertainties in kinetic equations is based on the use of Stochastic-Galerkin methods built via generalized Polynomial Chaos (gPC) expansions^^>\cite{HJ,ZJ, LJ-UQ,JPH,RHS,DJL,Frank}. Only recently these problems have been analyzed in the framework of statistical sampling methods based on Monte Carlo (MC) techniques^^>\cite{Giles, DPUQ1, DPUQ2, DPZ, HPY}. 
We also cite a related research direction where Monte Carlo sampling has been used in the physical space while the random space is still approximated through gPC expansions^^>\cite{CPZ,PZ2020,albi2015MPE,Poette3}.

In particular, when dealing with kinetic equations, non-intrusive sampling methods, such as MC sampling, have several advantages over an intrusive gPC approach, as they allow to be used in combination with existing deterministic numerical solvers designed to satisfy certain relevant physical properties^^>\cite{DP15,DPR}. This permits also to use implementations via fast algorithms and parallelization techniques which are essential to reduce the computational complexity^^>\cite{DLNR}. Furthermore, Monte Carlo methods are very effective when the probability distribution of the random inputs is not known analytically or lacks of regularity while the approaches based on stochastic orthogonal polynomials may be impossible to use or may produce low accurate results^^>\cite{Giles,MSS, Xu}.

One of the challenges central to uncertainty quantification for kinetic equations is the simulation cost. Most of the existing algorithms are mainly developed based on the direct resolution of the main reference model, the so-called high-fidelity model. For many complex systems, in particular, the kinetic equations with multidimensional random inputs, an accurate high-fidelity deterministic simulation can be so time-consuming and memory demanding that only a few high-fidelity simulations can be afforded. Many stochastic algorithms require repetitive implementations of the deterministic solver, the overall accurate stochastic simulation can be difficult and even computationally infeasible. However, there usually exist some approximate, less complex low-fidelity models which compared to the high-fidelity models, usually contain simplified physics and/or are simulated on a coarser physical mesh, and consequently, own a cheaper computational cost. Although their
accuracy may not be high, the low-fidelity models are designed in such a way that they can resolve or capture certain important features of the underlying problem and produce reliable and qualitative predictions. Recently, there has been a surging interest in developing efficient uncertainty quantification algorithms by leveraging the strengths of multiple models where costs and fidelity, to be intended as the capacity of correctly describing the problem under consideration, vary. This approach is known in the literature as the multi-fidelity method^^>\cite{peherstorfersurvey,fernandez2016review,park2017remarks}.

In this work, we precisely survey several recent results about multi-fidelity methods for kinetic equations that accelerate the solution of the uncertainty quantification process based on stochastic samples by combining high-fidelity and low-fidelity model evaluations. This will be done through the introduction of multi-fidelity control variates techniques and bi-fidelity stochastic collocation approaches. 
The first class of methods uses the knowledge of the space spanned by one or more inexpensive low-fidelity models to improve the accuracy of the Monte Carlo solution at the high-fidelity level^^>\cite{DPUQ1,DPUQ2,DUQU}. 
The choice of the low fidelity models follows the classical legacy of kinetic equations based on developing a hierarchy of reduced order models through suitable fluid-dynamics and diffusive scalings. Note that this requires the adoption of a suitable asymptotic-preserving solver for the high-fidelity model in order to take full advantage of the multiscale control variate^^>\cite{DP15}. 

However, the choice of the samples for the high fidelity solver remains arbitrary. The second methodology, is based on a simpler bi-fidelity setting where a single low-fidelity model is used to effectively inform the selection of representative points in the parameter space and then employ this information to construct accurate approximations to high-fidelity solutions^^>\cite{ZNX14,zhu2017multi,zhu2017multi2, munipalli2018multifidelity, LZ,GZJ20}. 
%sameDespite the recent advances on multi-fidelity methods,  
%many approaches still require to be used the low- and high-fidelity solutions to live in the same physical space. This could be a limit when one focuses on applications where the low-fidelity and high-fidelity models are derived from different physical models/representations. Alternative multi-fidelity methods 
%have been developed over last several years particularly on bi-fidelity methods^^>\cite{ZNX14,zhu2017multi,zhu2017multi2, munipalli2018multifidelity, LZ,GZJ20}. The methods make use of inexpensive low-fidelity models to effectively inform the selection of representative points in the parameter space and then employ this information to construct accurate approximations to high-fidelity solutions. Unlike the other existing methods, it combines low-fidelity and high-fidelity models via the parameter space. 
As a result, it does not necessarily require that the low-fidelity and high-fidelity to reside in the same physical space.%, e.g.,  the simulations with 3D/2D geometries are used as the high-fidelity/low-fidelity
%solutions^^>\cite{GZJ20}. 

Consistent with the discussion above, the remainder of this survey is divided into two main parts. The first part addresses the construction of multi-fidelity Monte Carlo methods based on single or multiple control variates. The focus of this part will be primarily on the Boltzmann equation of rarefied gas dynamics^^>\cite{DPUQ1,DPUQ2,Cer}. The concepts will then be extended to the case of kinetic equations in the social and economic sciences where we also discuss how to couple the approach with direct simulation Monte Carlo solvers in the physical space^^>\cite{PT2,PTZ}. The second part will cover bi-fidelity approximations based on an appropriate choice of collocation nodes.  Initially, the Boltzmann equation of rarefied gas dynamics will still be discussed.  Subsequently, the bi-fidelity approach will be extended to linear transport kinetic equations in the diffusive limit, both in the classical case of neutron transport^^>\cite{LPZ,LK} and in the context of epidemiology^^>\cite{EpidemicSurvey,BLPZ}. Some final remarks conclude this review.
%In their work, simple two-velocity Goldstein-Taylor  model^^>\cite{Gol,Tay} is used as the low-fidelity model, motivated by the observation that after the even and odd parity formulation of the transport equation, both models share the similar form. Inspired by this success, a bi-fidelity approach is recently introduced to efficiently quantify uncertainty in spatially dependent (high-fidelity multiscale transport based) epidemic models by using simple two-velocity discrete models for low-fidelity 
%evaluations^^>\cite{BLPZ}. 
%An introductory section, where some useful preliminaries for a better understanding of the sequel will be set, anticipates the multi-fidelity methods discussion. 

\section{Fidelity spectrum of kinetic models}
The fidelity of the kinetic models can vary along a broad spectrum between low- and high-fidelity. This sections provide examples of kinetic models across the fidelity spectrum, while defining the benefits and limitations of each model.
%We introduce in this part the basic necessary facts about kinetic equations. 
We first focus on the Boltzmann equation of rarefied gas dynamics (RGD), including in our discussion the related moment equations, their possible closures, and the Bhatnagar-Gross-Krook (BGK) approximation. We next illustrate the quasi-invariant limit for some homogeneous Boltzmann equation of interest in the socio-economic sciences. Finally, we consider kinetic equations of linear transport in the diffusion limit both in the case of classical neutron transport and in recent applications in epidemiology. In Table \ref{tab:sum} we summarized the various low- and high-fidelity models.
%Their application to epidemiology will the subject of a problem discussed in Section \ref{sec:bi} related to the spread of a disease. The models introduced in this part will be the basis for constructing our multi-fidelity approaches in the rest of the survey. The main idea, on which this class of methods relies, consists in combining the model for which the quantification of the uncertainty is required, the so-called high fidelity model, with one or several different surrogate models sharing analogy with the original one but exhibiting reduced computational complexities. We will, in fact, see that the information attained from the low fidelity models permit to reduce the costs and improve the accuracy related to the quantification of the uncertainty for the high fidelity ones.

\begin{table}[htp]
\caption{Fidelity spectrum of different kinetic models}
{\small
\begin{center}
\begin{tabular}{lll}
\hline\hline\\[-.3cm]
{\bf High-fidelity} & {\bf Low-fidelity} & {\bf Application}\\[+.1cm]
\hline\hline\\[-.3cm]
Homogeneous Boltzmann& - Maxwellian steady state & Trend to equilibrium\\ 
 & - Homogeneous BGK & \\[+.1cm] 
\hline\\[-.3cm]
Full Boltzmann & - Euler equations & Rarefied gas dynamics\\ 
 & - Navier-Stokes equations & \\ 
 & - Full BGK & \\[+.1cm] 
\hline\\[-.3cm]
Boltzmann-type models& - Mean-field steady state& Socio-economy\\ 
 & - Mean-field limit& \\[+.1cm] 
\hline\\[-.3cm]
Linear transport& - Diffusion limit & Neutron transport\\ 
 & - Goldstein-Taylor model & \\[+.1cm] 
\hline\\[-.3cm]
Compartmental transport& - Reaction-diffusion limit & Epidemiology\\ 
 & - Two-velocity models & \\[+.1cm] 
\hline
\end{tabular}
\end{center}
}
\label{tab:sum}
\end{table}%

\subsection{The Boltzmann equation of rarefied gas dynamics}
\label{sec:bolt}
We consider kinetic equations of the general form^^>\cite{Cer, DPUQ1}
\be
\partial_t f+ {\w}\cdot \nabla_x f = \frac1{\varepsilon}Q(f,f),
\label{eq:Boltzmann}
\ee
where $f=f({z},x,{\w},t)$, $t\ge 0$, $x\in\DH \subseteq \RR^{d_x}$, ${\w}\in\RR^{d_{\w}}$, $d_x,d_{\w}\ge 1$, and ${z}\in\Omega\subseteq\RR^{d_{z}}$, $d_{z} \geq 1$, is a {random variable}. In \eqref{eq:Boltzmann} the parameter $\varepsilon > 0$ is the Knudsen number and the particular structure of the interaction term $Q(f,f)$ depends on the kinetic model considered.
The most famous example is represented by the nonlinear Boltzmann equation of rarefied gas dynamics 
\be
Q(f,f)=\int_{S^{d_{\w}-1}\times\RR^{d_{\w}}} B({\w},{\w}_*,\omega,{z}) (f'f'_*-f f_*)\,d{\w}_*\,d\omega
\label{eq:Qcoll}
\ee
where $d_v \geq 2$ and 
\be
v'=\frac12(v+v_*)+\frac12(|v-v_*|\omega),\quad v_*'=\frac12(v+v_*)-\frac12(|v-v_*|\omega).
\ee
{We consider the variable hard sphere (VHS) case with
\be
B({\w},{\w}_*,\omega,{z}) = b(z)|\w-\w_*|^\alpha,\qquad -d_v < \alpha \leq 1.\label{VHS}
\ee}
%Given the equation \eqref{eq:Boltzmann}, we establish the link with the macroscopic fluid models. 
The Boltzmann operator $Q(f,f)$ is such that the local
conservation properties are satisfied, i.e. \be\int_{\RR^{d_v}} \phi(v) Q(f,f)\,
dv=:\langle \phi\, Q(f,f)\rangle=0, \label{eq:QC}\ee where
$\phi(v)=\left(1,v,{|v|^2}/{2}\right)^T$ are the collision invariants. In
addition, the operator satisfies the entropy inequality \be
\frac{d}{dt}H(f) = \int_{\RR^{d_v}} Q(f,f)\log f dv
\leq 0,\quad H(f)=\int_{\RR^{d_v}}f\log f\,dv. \label{eq:entropy} \ee
As a consequence, the functions such that
$Q(f,f)=0$ are local Maxwellian equilibrium functions
\be f^{\infty}(\rho,u,T)=\frac{\rho}{(2\pi
	T)^{d_v/2}}\exp\left(\frac{-|u-v|^{2}}{2T}\right), \label{eq:M}\ee
where $\rho$, $u$, $T$ are the density, mean velocity and
temperature of the gas in the $x$-position and at time $t$ defined as
\be (\rho,\rho u,E)^T=\langle \phi\, f \rangle, \qquad
T=\frac1{d_v\rho}(E-\rho|u|^2). \ee 
Integrating now (\ref{eq:Boltzmann}) against the collision
invariants in the velocity space leads to the following set of non closed conservations laws \be
\partial_t \langle \phi\, f\rangle+{\rm div}_x
\langle v\otimes\phi f\rangle=0.\label{eq:macr}\ee

\begin{figure}
\begin{center}
\includegraphics[scale=0.39]{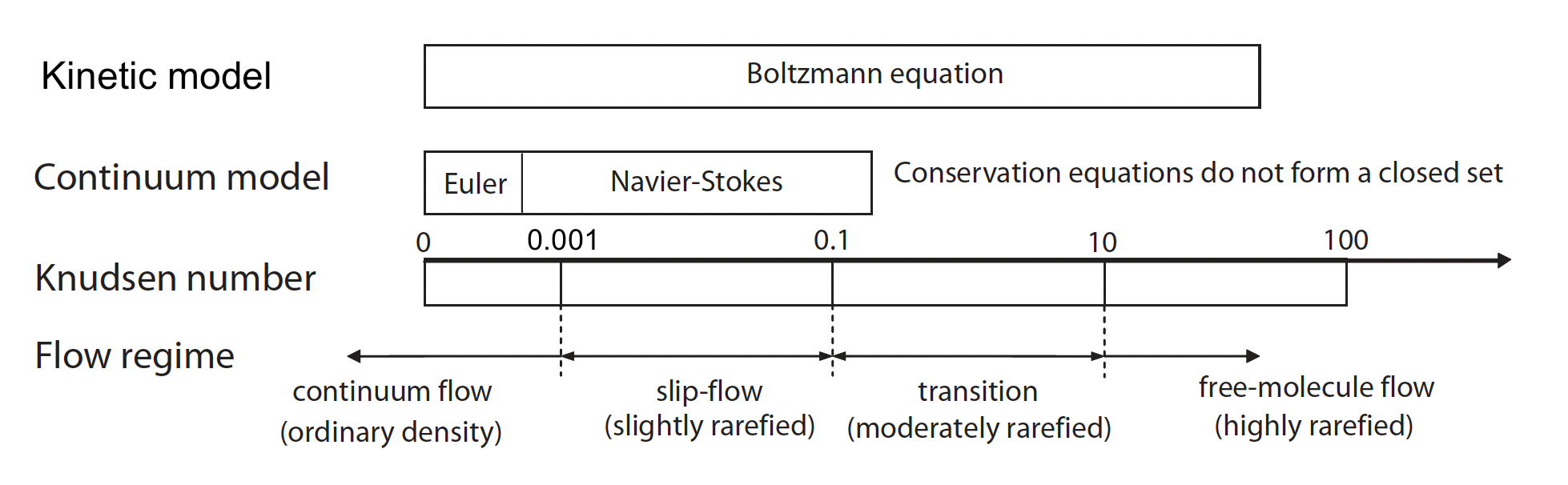}
\end{center}
\caption{The different scales of the Boltzmann equation in rarefied gas dynamics and the corresponding reduced order models}
\label{fig:scales}
\end{figure}

Close to fluid regimes, the mean free path between collisions, and therefore the Knudsen number, is very small (see Figure \ref{fig:scales}). In this situation, passing to the limit $\varepsilon\to 0$ we formally obtain $Q(f,f)=0$ from (\ref{eq:Boltzmann}) and so $f=f^{\infty}$. Thus, at least formally, we recover the closed hyperbolic system of compressible Euler equations
\be
\partial_t U+\nabla_x \cdot \F(U)=0
\label{eq:Euler} \ee with
$U=\langle \phi\, f^{\infty}\rangle = (\rho,\rho u,E)^T$ and
\[
\F(U)=\langle v\otimes\phi\, f^{\infty}\rangle=\left(
\begin{array}{c}
	\rho u  \\
	\rho u \otimes u+pI  \\
	Eu+pu     
\end{array}
\right),\]\[
p=\rho T,\quad T=\frac{1}{3}\left(\frac{2E}{\rho}-|u|^2\right),
\]
where $I$ is the identity matrix. For small but non zero values of the Knudsen number, the evolution
equation for the moments can be derived by the so-called
Chapman-Enskog expansion~\cite{Cer}. This originates the compressible Navier-Stokes equations
as a second order approximation with respect to $\varepsilon$ to the
solution of the Boltzmann equation
\be
\partial_t U+\nabla_x \cdot \F(U)=\varepsilon\,\nabla_x \cdot \D(\nabla_x U)
\label{eq:NavierStokes} \ee with
\begin{eqnarray}
	&\D(\nabla_x U)=
	\left(
	\begin{array}{c}
		0 \\
		\mu\tau(u)\\
		\kappa\nabla_x T+\mu\tau(u)\cdot u     
	\end{array}
	\right),\\
	& 
	\tau(u)=\frac12\left(\nabla_x u +(\nabla_x u)^T -\frac2{3}{\rm div}_x uI\right),
	\label{eq:NSsigma}
\end{eqnarray}
with the viscosity $\mu$ and the thermal conductivity $\kappa$ defined according to the Boltzmann operator^^>\cite{Cer}. The Prandtl number in this setting is defined as the ratio $Pr = 5\mu/(2\kappa)$.

A simplified kinetic model is obtained by replacing \eqref{eq:Qcoll} by a relaxation operator towards the local Maxwellian state. This model is usually referred to as BGK model since its introduction by Bhatnagar, Gross and Krook~\cite{BGK} 
\begin{eqnarray}
	Q(f)(v) = \nu (f^{\infty}[f]-f).
	\label{eq:pbgk}
\end{eqnarray}
In (\ref{eq:pbgk}) the function {$M[f]$} is the local Maxwellian and $\nu$, in general, is proportional to the density and depends on the temperature $\nu(\rho,T)$. For instance, a frequently used law is given by $\nu=C\rho T^{1-\eta}$, where $C>0$ is a constant and $\eta$ is the exponent of the viscosity law of the gas~\cite{Mi00}. Conservation of mass, momentum and energy as well as Boltzmann's H-theorem are readily satisfied and the equilibrium solutions are Maxwellians. Furthermore, the model has the correct fluid dynamic limit, since as $\varepsilon\to 0$ formally the moments $\rho$, $\rho u$, and $E$ satisfy the compressible Euler equations (\ref{eq:Euler}). However, this model exhibits some unphysical features, such as an unrealistic Prandtl number $Pr=1$. The correct Navier-Stokes limit \eqref{eq:NavierStokes} can be recovered using more sophisticated BGK models such as the Ellipsoidal Statistical BGK (ES-BGK) models~\cite{Holway}.

The Boltzmann equation \eqref{eq:Boltzmann} and its low-fidelity counterparts such as the BGK equation \eqref{eq:pbgk} and the compressible Euler \eqref{eq:Euler} and Navier-Stokes \eqref{eq:NavierStokes} systems will be employed as leading examples both for the development of the control variate multi-fidelity methods in Section 3 as well as in the case of the stochastic collocation bi-fidelity approach in Section 4.

\subsection{Kinetic models for the social sciences}
\label{model_socio}
Recently, Boltzmann-type models have also become quite popular for describing binary dynamics between agents, representing individuals interacting in a society (see^^>\cite{PT2} for an introduction to the topic). In presence of uncertainties, the pair of interacting agents is characterized by the pre-interaction states $v,w\in V \subseteq \mathbb R$ and the post-interaction states $v^{\prime}, w^{\prime} \in V$ obtained as follows^^>\cite{DPZ,PTZ}
\begin{equation}
	\label{micro}
	\begin{split} 
		v^{\prime} &= v + \gamma [(p_1(z)-1) v+ q_1(z)w] +D(v,z)\ \eta,\\
		w^{\prime} &= w + \gamma [p_2(z) v+ (q_2(z)-1) w] +D(w,z)\ \eta^*,
	\end{split}
\end{equation}
where $\gamma >0$ is a given constant, $p_i,q_i$,  $i=1,2,$ are suitable interaction functions depending on a random variable $z \in \Omega \subseteq \mathbb R^{d_z}$. Furthermore, $\eta$ and $\eta^*$ are i.i.d. random variables with zero mean and variance ${\sigma}^2$. The function $D(\cdot,z)$ defines the local relevance of the diffusion. 

Introducing the distribution function $f = f(t,w,z)$, its evolution is given in terms of the Boltzmann-type equation \begin{equation}\label{SymBoltzmannModel}
	\begin{split}
		&\frac{d}{dt} \int_V f(t,w,z)\ \varphi(w) dw \\
		&= \frac{1}{2} \left\langle\iint_{V^2} (\varphi(w^{\prime}) +\varphi(v^{\prime})-\varphi(w)-\varphi(v)) f(t,w,z) f(t,v,z)    dwdv \right\rangle_{\eta}
	\end{split}
\end{equation}
being $\varphi:V\rightarrow \mathbb R$ any observable quantity which may be expressed as a function of the microscopic state $w$ of the agents. The symbol $\left\langle \cdot \right\rangle_{\eta}$ denotes the expectation with respect to $\eta$. Taking $\varphi(w)=1$ is easily seen that the number of agents is conserved in time.

In contrast to the classical Boltzmann case \eqref{eq:Boltzmann}, the equilibrium states of the socio-economic Boltzmann models like \eqref{SymBoltzmannModel} are often unknown. A way to achieve some insight into the long time behavior of such systems is to consider the quasi-invariant interaction limit^^>\cite{PT2,pareschi2006self,Vill}. To that aim, given a small parameter $\varepsilon >0$, let us introduce the scaling
\begin{equation}
	\gamma\to\epsilon,\quad \sigma\to\sqrt{\epsilon}\sigma,\quad t\to\epsilon  t, 
	\label{scaling}
\end{equation}
and denote by $f_\epsilon(t,w,z)$ the scaled distribution. 
Thus, small values of $\epsilon$ correspond to the case in which elementary interactions \eqref{micro} produce minimal modification of $v$ and $w$ and, at the same time, the frequency of such interactions increases like $1/\epsilon$. Then, the distribution $ f_\epsilon$ is solution to
\begin{equation}\label{ScaledBol}
	\begin{split}
		&\frac{d}{dt} \int_V \varphi(w){f_\epsilon}(t,w, z)\ dw=\\
		&\frac{1}{2\epsilon}\left \langle \iint_{V^2} (\varphi(w^{\prime})+\varphi(v^{\prime})-\varphi(v)-\varphi(w))   {f_\epsilon}(t, w,z){f_\epsilon}(t,v,z)dwdv \right\rangle_{\eta}. 
	\end{split}
\end{equation}
Now, for small values of $\epsilon$ using a Taylor expansion of the interactions 
it can be shown that in the limit $\epsilon \rightarrow 0$, $ f_\epsilon$ converges, up to subsequences, to a distribution function $\h = \h(t,w,z)$ which is weak solution to the following Fokker-Planck equation 
% equation \eqref{ScaledBol} converges to 
%\[
%\begin{split}
%&\dfrac{d}{dt }\int_V \varphi(w) \tilde{f}(\tau,w,z)dw = \frac{1}{2} \iint_{V^2}  \Big[ (p_1(z) w+p_2(z) v )\varphi^{\prime}(w) + ({q}_1(z) v+{q}_2(z) w)\varphi^\prime(v)\\
%&\qquad+ \frac{1}{2} D^2(z,w) \sigma^2 \varphi^{\prime \prime}(w) + \frac{1}{2}  D^2(z,w) {\sigma}^2\ \varphi^{\prime \prime}(v)\Big]  \tilde{f}(\tau,w,z) \tilde{f}(\tau,v,z) dv\,dw.
%\end{split}
%\]
%Next, integrating back by parts we conclude that the limit density is solution of the following Fokker-Plank equation  
\begin{equation}\label{FP}
	\begin{split}
		\partial_t {\h}(t,w,z)&+\partial_w\left[\left( \int_V P(v,w,z) \h(t,v,z)dv \right)\h(t,w,z)\right]\\
		& = \frac{{\sigma}^2}{2} \partial_w^2(D^2(w,z) {\h}(t, w,z)), 
	\end{split} 
\end{equation}
where 
\[
P(v,w,z) = \dfrac{1}{2} \left[(p_1(z)+q_2(z)-2)w + (p_2(z)+q_1(z))v \right]
\]
provided that suitable boundary conditions are taken for $w\in\partial V$. %The asymptotic analysis of equation \eqref{FP} is considerable simpler compared to the original Boltzmann-type model and exponential convergence toward the a unique equilibrium distribution can be obtained under suitable hypotheses, see^^>\cite{Tos1999}. 
%Therefore, we have obtained a surrogate model with reduced complexity whose large time behavior can be more easily studied. 
%For example, the steady state distribution ${\h}_{\infty}(w,z)$ of \eqref{FP}  is determined by imposing
%\begin{equation}\label{SteadyEQ}
%	\frac{{\sigma}^2}{2} \partial^2_w(D^2(w,z) {\h}_{\infty}(w,z))  =    \partial_w \left[\left(\int_V P(v,w,z) \h_\infty(v,z)dv \right)\h_\infty(w,z)\right]. 
%\end{equation}
%For many models the solution of the differential equation \eqref{SteadyEQ} is known analytically. 

We shortly present now two socio-economic models used in the rest of the survey as examples for developing our multi-fidelity Monte Carlo methods. The first is an opinion formation model where we set $V = [-1,1]$ and the binary interaction rules read^^>\cite{PT2,APTZ17,PTTZ}
\begin{align}\label{InterOpinion}
	\begin{split}
		&v^{\prime}= v + \epsilon  p(|w-v|,z) (w-v)+ D(v)   \eta, \\
		&w^{\prime}= w+\epsilon p(|v-w|,z)( v-w) +D(w) \eta,
	\end{split}
\end{align}
where the function $0\leq p(|v-w|,z) \leq 1$ weights the compromise tendency with respect to the relative opinion $|v-w|$. %In the present case, in order to produce post-interaction opinions in the interval $[-1,1]$,  the random variable $\eta$ should be such that for all $z \in \Omega$ there exist a constant $c>0$ such that
%\[
%D(w,z)|\eta| \le (1-\epsilon p(|v-w|,z))(1-|w|). 
%\]
%In particular for $D(w,z) = 1-w^2$ and $c = 1/2$ and we obtain the following bound $|\eta| \le \frac{1}{2} (1-\epsilon)$. 
In the quasi-invariant limit, the following model is obtained 
\[
\begin{split}
	\partial_t \h(t,w,z) &+ \partial_w \left[\left(\int_V p(|v-w|,z)(v-w)\h(t,v,z)dv  \right) \h(t,w,z) \right] \\
	&= \dfrac{\sigma^2}{2} \partial_w^2 (D^2(w,z) \h(t,w,z)).
\end{split}
\]
If $p(|v-w|,z) = p(z)$ in \eqref{InterOpinion}, $D(w,z)= 1-w^2$, %we easily see that the mean opinion $m(z)$ is conserved in time and, if
and we consider uncertainties in the initial distribution, the steady state distribution of the Fokker-Planck model reads
\begin{equation}\label{eq:steady_max}
	\begin{split}
		{\h}_{\infty}(w,z) =\,\, &C(z)\ (1+w)^{-2+\frac{p(z) m(z)}{2{\sigma} ^2}}\times\\
		&  (1-w)^{-2-\frac{p(z) m(z)}{2 {\sigma}^2}} \ \exp\left\{-\frac{p(z) (1- m(z)w)}{\sigma^2\ (1-w^2)} \right\},
	\end{split}
\end{equation}
where $C(z)$  is a normalization factor such that $\int_V \h_{\infty}(w,z)dw = 1$. Other form of equilibrium distribution can be determined as well^^>\cite{PTTZ}. 

For example the choice $D(w) = \sqrt{1-w^2}$ produce a Beta-type steady state of the form 
\begin{equation}\label{eq:steady_beta}
	\h_{\infty}(w,z)= 2^{1-\frac{2}{{\sigma^2}}       }\frac{1}{\textrm{B}\Big( \frac{1+m(z)}{{\sigma^2}} ,   \frac{1-m(z)}{{\sigma^2}}  \Big)}\ (1+w)^{\frac{1+m(z)}{{\sigma^2}}-1}  (1-w)^{\frac{1-m(z)}{{\sigma^2}}-1},
\end{equation}
where $\textrm{B}(\cdot,\cdot)$ is the Beta function. We refer to^^>\cite{PTTZ} for a detailed discussion.

The second example, known as the Cordier-Pareschi-Toscani (CPT) model^^>\cite{PT2,CPT} for wealth exchanges between agents composing a simple economy. The wealth variable is here allowed to take values on the positive half line therefore $V=\R^+$. The binary interactions with reference to \eqref{micro} are such that $p_1 = q_2 = q(z)$ and $p_2 = q_1 = \lambda(z)$ with $q(z) = 1-\lambda(z)$. We also consider $D(w,z) = w$. The uncertain parameter $\lambda(z) \in (0,1)$ determines the proportion of wealth that a single agents wants to invest, the quantity $1-\lambda(z)$ is the so-called saving propensity. 
Thus, the binary scheme for wealth exchanges reads
\begin{equation}
	\label{InterWealth}
	\begin{split}
		&v^{\prime}= (1-\epsilon\lambda(z))v +\epsilon\lambda(z)w+v \eta\\
		&w^{\prime}= (1-\epsilon \lambda(z)) w +\epsilon\lambda(z) v+w \eta. 
	\end{split}
\end{equation}
The corresponding mean field model reads
\begin{align*}
	\partial_t {\h}(t, w,z) +\lambda(z)\partial_w \left[(m_{\h}(z)-w) {\h}(t,w,z) \right]
	= \frac{{\sigma}^2}{2} \partial_w^2 (w^2 {\h}(t,w,z)),
\end{align*}
If we assume that there is no uncertainty in the initial conditions, following^^>\cite{PT2} the equilibrium state can be computed and reads
\begin{equation}
	\label{eq:steady_invgamma}
	{\h}_{\infty}(w,z)= \frac{(\mu(z)-1)^{\mu(z)}}{\Gamma(w)\ w^{1+\mu(z)}}\exp\left( -\frac{(\mu(z)-1)m_{\h}(z)}{w}\right) ,
\end{equation}
where $\Gamma(\cdot)$ denotes the Gamma function and $\mu(z)=1+{2{\lambda}(z)}/{{\sigma}^2}$. %Notice that in this case the steady state exhibits tails with polynomial decay determined by the uncertain quantity $\mu(z)$. 
We refer to Section \ref{sec:socio} for applications of control variate strategies based on the mean-field equilibrium states \eqref{eq:steady_max} and \eqref{eq:steady_invgamma} or on the full mean-field system \eqref{FP}.

\subsection{Linear transport equations}
\label{sec:diff}
Another class of kinetic models that will be considered in the sequel are the linear transport equation  under diffusive scaling and with random parameters^^>\cite{LK,LPZ,JPH}. Let $f(t,x,v,z)$ be the probability density distribution of particles at time $t>0$, position $x\in\DH\subseteq \RR$, and with $v\in[-1,1]$ the cosine of the angle
between the particle velocity and its position. In particular, we consider in our discussion the case of random scattering coefficients $\sigma(x,z)$ with $z\in\Omega \subseteq \R^{d_z}$ a random vector. The time evolution of the distribution function $f$ is governed by the following linear transport equation under diffusive scaling 
\begin{equation}
	\label{pde_transport1d}
	\epsilon \partial_t  f + v \partial_x f = \frac{\sigma(x,z)}{\varepsilon}
	\left[\frac{1}{2}\int_{-1}^1 f(v')\, dv' -f\right], 
\end{equation}
where $\varepsilon$ is the Knudsen number. In order to shed light on the link between the kinetic equation and its diffusive limit, we first split \eqref{pde_transport1d} into two equations for $v>0$ 
\begin{equation}
	\label{pde_transport1d_sp}
	\begin{split}
		&\varepsilon \partial_t  f(v) + v \partial_x f(v) =
		\frac{\sigma(x,z)}{\varepsilon}\left[\frac{1}{2}\int_{-1}^1f(v')\,
		dv'-f(v)\right], \\[4pt]
		&\varepsilon \partial_t f(-v) - v \partial_x f(-v) =
		\frac{\sigma(x,z)}{\varepsilon}\left[\frac{1}{2}\int_{-1}^1f(v')\,
		dv-f(-v)\right]. 
	\end{split}
\end{equation}
where we omitted for brevity the dependence from $(t,x,z)$ in the kinetic density.
By using now the so-called even-odd decomposition^^>\cite{JPT2}, we introduce then the relative even and odd parities 
\begin{equation}
	\begin{split}
		r(t,x,v,z) &= \frac{1}{2}[f(t,x,v,z) + f(t,x,-v,z)], \\[4pt]
		j(t,x,v,z) &= \frac{1}{2\varepsilon}[f(t,x,v,z) - f(t,x,-v,z)].
	\end{split}
\end{equation}
Thanks to the above reformulation, the system \eqref{pde_transport1d_sp} can then be recast in the following form: 
\begin{equation}
	\label{pde_transport1d_rj}
	\left\{
	\begin{split}
		&\partial_t r + v \partial_x j = \frac{\sigma(x,z)}{\varepsilon^2}(\rho-r), \\[4pt]
		&\partial_t j + \frac{v}{\varepsilon^2} \partial_x r =
		-\frac{\sigma(x,z)}{\varepsilon^2}j, 
	\end{split}
	\right.
\end{equation}
where 
\begin{equation}
	\label{r_integral}
	\rho(t,x)=\int_0^1 rdv.
\end{equation}
Now the formal passage to the limit is easily obtained. In fact,
as $\varepsilon\to 0$, the system \eqref{pde_transport1d_rj} yields
\[r=\rho, \qquad j=-\frac{v}{\sigma(x,z)}\partial_x \rho.\]
Substituting back this limit it into the same system (\ref{pde_transport1d_rj}) and integrating over $v$, one gets the limiting diffusion equation^^>\cite{LK, BSS} 
\begin{equation}
	\label{transport1d_diff}
	\left\{
	\begin{split}
		&j=-\frac{v}{\sigma(x,z)}\partial_x \rho, \\[4pt]
		&\partial_t \rho =\partial_x \left[\frac{1}{3\sigma(x,z)} \partial_x\rho\right].
	\end{split}
	\right.
\end{equation} 
\label{sec:pre}
The neutron transport equation \eqref{pde_transport1d} and its low-fidelity diffusive limit \eqref{transport1d_diff} are employed in Section \ref{sec:diff} to design efficient bi-fidelity approximations, together with a different low-fidelity models given by the Goldstein-Taylor model^^>\cite{PT2,LPZ}.

\subsection{Kinetic models in epidemiology}
\label{sec:epidemic}
We consider the development of hyperbolic transport models for the propagation in space of an epidemic phenomenon described by a classical compartmental dynamics^^>\cite{BBDP,BDP,EpidemicSurvey}. The model is based on a kinetic description with discrete velocities of the spatial movement and interactions of a population of susceptible, infected and recovered individuals. The model is constructed in such a way that a diffusive behavior of the spread can be recovered in a suitable limit similar to the one described in Section \ref{sec:diff}. In the considered framework, the distributions of individuals at position $x\in\DH\subseteq \RR$ at time $t$ moving with velocity $v \in [-1,1]$ are then denoted by 
\[
f_S=f_S(x,v,t,{z}),\quad f_I=f_I(x,v,t,{z}),\quad f_R=f_R(x,v,t,{z}).
\]
These quantities represent, respectively, the fraction of susceptible individuals $S$ (who may be infected by the disease), the fraction of infected  individuals $I$ (who may transmit the disease) and removed individuals $R$ (healed or died due to the disease). We assume to have a population without an a-priori immunity and we neglect the vital dynamics represented by births and deaths due to the time scale considered. The densities for $S$, $I$ and $R$ are given by 
\begin{eqnarray}
	S(x,t,{z})&=&\int_{-1}^{1}  f_S(x,v,t,{z})\,dv\nonumber\\
	I(x,t,{z})&=&\int_{-1}^{1}  f_I(x,v,t,{z})\,dv \label{def.densitiesSIR}\\
	R(x,t,{z})&=&\int_{-1}^{1}  f_R(x,v,t,{z})\,dv.\nonumber
\end{eqnarray}
In this setting, the distribution functions satisfy the epidemic transport equations
\begin{eqnarray}
	\nonumber
	\frac{\partial f_S}{\partial t} + \frac{\partial (v_S f_S)}{\partial x} &=& -F(f_S, I) +\frac1{\tau_S}\left(\frac{S}{2}-f_S\right)\\
	\label{eq:kineticc}
	\frac{\partial f_I}{\partial t} + \frac{\partial (v_I f_I)}{\partial x} &=&  F(f_S, I)-\gamma f_I+\frac1{\tau_I}\left(\frac{I}{2}-f_I\right)\\
	\nonumber
	\frac{\partial f_R}{\partial t} + \frac{\partial (v_R f_R)}{\partial x} &=& \gamma f_I+\frac1{\tau_R}\left(\frac{R}{2}-f_R\right),
\end{eqnarray}
where 
$v_S=\vs(x) v$, $v_I=\vi(x) v$, $v_R=\vr(x) v$ with $\vs,\vi,\vr \geq 0$ take into account the heterogeneities of geographical areas, thus are chosen dependent on the spatial location. The quantity $\gamma=\gamma(x,{z})$ is the recovery rate of infected, and the transmission of infection is governed by an incidence function $F(\cdot,I)$ modeling the transmission of the disease^^>\cite{HWH00} 
\begin{equation}
	F(g,I)=\beta \frac{g I^p}{1+\kappa I},
	\label{eq:incf}
\end{equation}
with the classic bi-linear case corresponding to $p = 1$, $\kappa=0$. Finally, $\tau_S,\tau_I,\tau_R$ represent the relaxation frequencies playing the role of the Knudsen number $\varepsilon$ introduced in Section \ref{sec:diff} when discussing the linear transport equation in the diffusive scaling. In this model, the reproduction number for the system \eqref{eq:kineticc} is given by 
\begin{equation}
	R_0(t,{z})=\frac{\int_{\DH} F(S,I)\,dx}{\int_{\DH} \gamma(x,{z}) I(x,t,{z})\,dx} \geq 1.
	\label{eq:R0}
\end{equation}
This quantity describes the space averaged instantaneous variation of the number of infected individuals at time $t>0$. Now, with the scope of deriving a possible surrogate/low-fidelity model approaching the high fidelity model described by system \eqref{eq:kineticc}, we introduce the flux functions
\begin{equation}
	\begin{split}
	& J_S={\lambda_S} \int_{-1}^{1}  v f_S(x,v,t,{z})\,dv,\\ & J_I={\lambda_I}\int_{-1}^{1}  v f_I(x,v,t,{z})\,dv,\\ & J_R={\lambda_R}\int_{-1}^{1}  v f_R(x,v,t,{z})\,dv.
\end{split}
	\label{eq.fluxes}
\end{equation}
We integrate then the system \eqref{eq:kineticc} over $v$ and obtain the following equations for the macroscopic densities
\begin{eqnarray}
	\nonumber
	\frac{\partial S}{\partial t} + \frac{\partial J_S}{\partial x} &=& -F(S, I)\\
	\label{eq:density}
	\frac{\partial I}{\partial t} + \frac{\partial J_I}{\partial x} &=& F(S, I) -\gamma I\\
	\nonumber
	\frac{\partial R}{\partial t} + \frac{\partial J_R}{\partial x} &=& \gamma I\,.
\end{eqnarray}
The above system is not closed since the flux functions depend upon the distribution functions values. However, by introducing now the so-called diffusion coefficients that characterize the diffusive transport mechanism of susceptible, infected and removed
\begin{equation}
	D_S=\frac13\lambda_S^2\tau_S,\quad D_I=\frac13\lambda_I^2\tau_I,\quad D_R=\frac13\lambda_R^2\tau_R,
	\label{eq:diffcf}
\end{equation}
one can pass to the limit $\tau_J\to 0, \ J=S,I,R$ while keeping the diffusive coefficients finite. This permits to get the following diffusion system characterizing a possible low fidelity model for the spread of a disease
\begin{eqnarray}
	\nonumber
	\frac{\partial S}{\partial t} &=&  -F(S, I)+\frac{\partial}{\partial x} \left({D_S}\frac{\partial S}{\partial x}\right)\\
	\label{eq:diff}
	\frac{\partial I}{\partial t} &=&  F(S, I)-\gamma I+\frac{\partial}{\partial x} \left({D_I}\frac{\partial I}{\partial x}\right)\\
	\nonumber
	\frac{\partial R}{\partial t} &=&  \gamma I+\frac{\partial}{\partial x}\left({D_R}\frac{\partial R}{\partial x}\right).
\end{eqnarray}
In Section \ref{sec:epidem} we will discuss strategies for quantifying uncertainty for the above model while also introducing an alternative low-fidelity model based on a simpler two-velocity dynamic^^>\cite{Bert,BLPZ}.

\section{Multi-fidelity control variate methods}
%\label{sec:MSCV}
This part introduces the multi-fidelity control variate methods for the Boltzmann equation of gas dynamic introduced in Section \ref{sec:bolt} and then discusses its extension to the kinetic models in the social sciences presented in Section \ref{model_socio}.
%^^>\cite{DPUQ1, DPUQ2, DUQU, PTZ}. 
%The surrogate/low fidelity models considered are the the BGK and the compressible Euler equations. The numerical methods employed to solve the deterministic Boltzmann and/or the Fokker-Planck equations are either deterministic^^>\cite{DPUQ1,DPUQ2,DUQU} (Sections \ref{sec:MSCV}-\ref{sec:multil})  either based on Direct Simulation Monte Carlo (DSMC) approaches^^>\cite{PTZ} (Section \ref{sec:socio}).  

\subsection{Some notations and definitions}
First we introduce some notations and assumptions that will be used in the sequel of the section. If $z\in \Omega$ is distributed as $p(z)$ we denote the expected value of any \emph{quantity of interest} expressed as a functional $q[f](z,x, {\w}, t)$ by
\begin{equation}
\mathbb{E}[\g[f]] (x, {\w}, t)=\int_{\Omega} \g[f](z,x, {\w}, t) p(z)dz,
\end{equation}
and its variance by
\begin{equation}
\VV(\g[f]) (x, {\w}, t)=\int_{\Omega} (\g[f](z,x, {\w}, t)-\mathbb{E}[\g[f]](x,\w,t))^2 p(z)dz. 
\end{equation}
%\be
%\EE[f](x, {\w}, t) = \int_{\Omega} f(z,x, {\w}, t)p(z)\,dz,
%\ee
%and its variance by
%\be
%\var(f)(x, {\w}, t) = \int_{\Omega} (f(z,x, {\w}, t)-\EE[f](x, {\w}, t))^2 p(z)\,dz.
%\ee
As quantity of interest, in addition to $q[f]=f$, another natural choice is represented by $q[f]=\langle \phi\,f\rangle$, the moments of the distribution function. In the latter case the dependence on $\w$ is dropped in the previous definitions.

We also introduce the following $L^p$-norm with polynomial weight^^>\cite{PP, RSS}
\be
\| f(z,\cdot,t)\|^p_{L^p_s(\DH\times\RR^{d_v})} = \int_{\DH\times\RR^{d_v}} |f(z,x,v,t)|^p (1+|v|)^s\,dv\,dx.
\ee
Next, for a random variable $Z$ taking values in ${L^p_s(\DH\times\RR^{d_v})}$, we define 
\be
\|Z\|_{{L^p_s(\DH\times\RR^{d_v};L^2(\Omega))}}=\|\EE[Z^2]^{1/2}\|_{{L^p_s(\DH\times\RR^{d_v})}}.
\label{eq:norm1}
\ee
The above norm, in general, differs from a more classical norm used in the context of uncertainty quantification for hyperbolic conservation laws. This latter often reads   
\be
\|Z\|_{L^2(\Omega;{L^p_s(\DH\times\RR^{d_v}))}}=\EE\left[\|Z\|^2_{L^p_s(\DH\times\RR^{d_v})}\right]^{1/2},
\label{eq:norm2}
\ee
which has been used for example in^^>\cite{MSS}. Note, in particular that by Jensen inequality^^>\cite{Rudin}, we have
\be
\|Z\|_{{L^p_s(\DH\times\RR^{d_v};L^2(\Omega))}} \leq \|Z\|_{L^2(\Omega;{L^p_s(\DH\times\RR^{d_v}))}}.
\ee
Let also observe that for $p=2$ the two norms \eqref{eq:norm1} and \eqref{eq:norm2} coincide. In the sequel, we will  consider in our analysis the norm \eqref{eq:norm1} with $p=1$. The extension of the results shown about the control variate approach to norm \eqref{eq:norm2} for $p=1$ typically requires $Z$ to be compactly supported (see^^>\cite{MSS, DPUQ1,MS2} for further details).

We assume that the model (\ref{eq:Boltzmann}) is solved in the phase-space by means of deterministic methods and that the following estimate is satisfied (see^^>\cite{RSS, DP15, Son, DPUQ1,MP})
\be
\|f(\cdot,t^n)-f_{\Delta x,\Delta {\w}}^n\|_{{\LL}} \leq C \left(\Delta x^{q_1}+\Delta {\w}^{q_2} \right),
\label{eq:det}
\ee
with $C$ a positive constant which depends on time and on the initial data, and $f_{\Delta x,\Delta {\w}}^n$ the computed approximation of the deterministic solution $f(x,v,t)$ at time $t^n$ on the mesh $\Delta x$ and $\Delta v$. The positive integers $q_1$ and $q_2$ characterize the accuracy of the discretizations in the phase-space and, for simplicity, we ignored the errors due to the time discretization and to the truncation of the velocity domain in the deterministic solver. We refer to^^>\cite{DP15} for details about the numerical discretization of a kinetic equation of the form \eqref{eq:Boltzmann}. In the sequel, if not otherwise stated, in the numerical examples we will make use of the fast spectral method^^>\cite{MP,FMP} combined with finite volumes WENO solvers in space^^>\cite{JS}. The time discretization is performed by suitable asymptotic-preserving techniques^^>\cite{DP15}. 

\subsection{Monte Carlo sampling method.}
\label{sec:sMC}
Before discussing the multi-fidelity approach, we first recall the standard Monte Carlo method for the estimation of the uncertainties when computing the solution of a kinetic equation of the type (\ref{eq:Boltzmann}) with random initial data $f(\theta,x,{\w},0)=f_0(\theta,x,{\w})$. To avoid unnecessary difficulties, we will mainly restrict to the case of a one-dimensional random input $d_{\theta}=1$ distributed as $p(\theta)$, the extension to multidimensional setting being straightforward. 

The simplest Monte Carlo (MC) sampling method for UQ is described in Algorithm \ref{al:MC}. 

\begin{algorithm2e}[htbp]
\label{al:MC}
	\begin{enumerate}
		\item {\bf Sampling}: Sample $M$ independent identically distributed (i.i.d.) initial data $f_0^k$, $k=1,\ldots,M$ from
		the random initial data $f_0$ and approximate these over the grid.
		\item {\bf Solving}: For each realization $f_0^k$, the underlying kinetic equation (\ref{eq:Boltzmann}) is solved numerically by a deterministic solver. We denote the solutions at time $t^n$ by $f^{k,n}_{\Delta x,\Delta {\w}}$, $k=1,\ldots,M$, where  $\Delta x$ and $\Delta {\w}$ characterizes the discretizations  in $x$ and ${\w}$. 
		\item {\bf Estimating}: Estimate the expected value of the random solution field with the sample mean of the approximate solution
		\be
		E_M[f^{n}_{\Delta x,\Delta {\w}}]=\frac1{M} \sum_{k=1}^M f^{k,n}_{\Delta x,\Delta {\w}}.
		\label{mcest}
		\ee
	\end{enumerate}
	\caption{Monte Carlo sampling method}
\end{algorithm2e}
Similarly to the case of the expectation, higher order statistical moments can be computed as well. The above algorithm present several advantages:
\smallskip
\begin{itemize}
	\item[i)] Straightforward to implement in any existing deterministic or stochastic solver for the particular kinetic equation.
	\smallskip
	\item[ii)]  It operates in a \textit{post-processing} setting, the only data interaction between different samples is in step $3$, when ensemble averages are computed.
	\smallskip
	\item[iii)] It is non-intrusive and easily parallelizable.
\end{itemize} 
\smallskip  
Concerning the error analysis, starting from the fundamental estimate^^>\cite{Caflisch, Lo77}
	\be
	\EE\left[(\EE[f]-E_M[f])^2\right] \leq C\, \var(f)M^{-1},
	\ee
	one can obtain the typical error bound which is summarized by the following proposition (see^^>\cite{MSS, MS2, DPUQ1} for more details).
	\begin{proposition}
		Consider a deterministic scheme satisfying \eqref{eq:det} for equation \eqref{eq:Boltzmann} with random sufficiently regular initial data $f(\theta,x,{\w},0)=f_0(\theta,x,{\w})$. Then, the Monte Carlo estimate \eqref{mcest} satisfies the error bound 
		\be
		\|\EE[f](\cdot,t^n)-E_M[f^{n}_{\Delta x,\Delta {\w}}]\|_{{\LLBi}} \leq C \left(\sigma_f M^{-1/2} + \Delta x^{q_1} + \Delta {\w}^{q_2}\right),
		\label{eq:MCest}
		\ee
		where $\sigma_f = \|\var(f)^{1/2}\|_{L^1_2(\mathcal{D}\times\RR^{d_v})}$, $\var(f)=\EE[(\EE[f]-f)^2]$ and
		the constant $C=C(T, f_0)>0$ depends on the final time $T$ and the initial data $f_0$.
	\end{proposition}
%	\begin{proof}
%		The above bound follows from
%		\begin{eqnarray*}
%			\nonumber
%			&\|\EE[f](\cdot,t^n)-E_M[f^{n}_{\Delta x,\Delta {\w}}]\|_{{\LLBi}} \\  
%			\nonumber
%			&&\hskip -5cm \leq \|\EE[f](\cdot,t^n)-E_M[f](\cdot,t^n)\|_{{\LLBi}}\\[-.25cm] 
%			\\[-.25cm]
%			\nonumber
%			&&\hskip -5cm+\|E_M[f](\cdot,t^n)-E_M[f_{\Delta x,\Delta {\w}}^n]\|_{{\LLBi}}\\ 
%			\nonumber
%			&&\hskip -5cm\leq C_1 M^{-1/2}\|\var(f)^{1/2}\|_{L^1_2(\mathcal{D}\times\RR^{d_v})}+C_2\left(\Delta x^{q_1} + \Delta {\w}^{q_2}\right)\\ 
%			&&
%			\hskip -5cm\leq C \left(\sigma_f M^{-1/2} + \Delta x^{q_1} + \Delta {\w}^{q_2}\right),
%		\end{eqnarray*}
%		where the first term in the inequality is a Monte Carlo error bound and the second term is essentially a discretization error.
%\end{proof}
Once an error estimate is given, it is possible to equilibrate the discretization and the sampling errors in the a-priori estimate taking $M={\mathcal O}(\Delta x^{-2q_1})$ and $\Delta x={\mathcal O}(\Delta {\w}^{q_2/q_1})$. This means that in order to have comparable errors the number of samples should be extremely large, especially when dealing with high order deterministic discretizations. As a consequence, the Monte Carlo approach may result very expensive in practical applications. 
\begin{remark}
The previous Monte Carlo method can be improved using standard variance reduction techniques such as stratified sampling and importance sampling methods. We refer to^^>\cite{Caflisch,refId0,peherstorfersurvey} for more details.
\end{remark}
%In order to address the slow convergence of Monte Carlo methods, we discuss in the next section the development of multi-fidelity Monte Carlo methods.

\subsection{Multi-Scale Control Variate (MSCV) method}
\label{sec:MSCV}
We survey here the MSCV approach recently introduced in^^>\cite{DPUQ1}. The main idea of the method is to reduce the variance of standard Monte Carlo estimators using as control variate different low-fidelity models at the various scales introduced by the Knudsen number. In the following we will refer to this type of approach that uses a high-fidelity model and a single low-fidelity model as the \emph{bi-fidelity} case.

\subsubsection{The space homogeneous case}
\label{sec:MSCV_hom}
For sake of clarity, we first illustrate the method when applied to the solution of a kinetic equation of the type (\ref{eq:Boltzmann}) with deterministic interaction operator $Q(f,f)$ and random initial data $f({z},x,{\w},0)=f_0({z},x,{\w})$ in a homogeneous setting
\be
\frac{\partial f}{\partial t} = Q(f,f),
\label{eq:BH}
\ee
where $f=f({z},\w,t)$. Let observe that without loss of generality, we have fixed here $\varepsilon=1$, since in the space homogeneous case the only temporal scale is the collisional one. Under suitable assumptions, {one can show that} $f({z},{\w},t)$ exponentially decays^^>\cite{Vill, Trist, ToVil, Tos1999} to the unique steady state $f^{\infty}({z},{\w})$ such that $Q(f^\infty,f^\infty)=0$ which satisfies
\be \langle \phi f_0 \rangle=\langle \phi f^\infty \rangle=(\rho,\rho u, E)^T.
\label{eq:mom}
\ee
%We refer to^^>\cite{Vill, Trist, ToVil, Tos1999} for details about the convergence to equilibrium in the homogeneous setting. 
%We also emphasize that, besides the case of the Boltzmann equation for rarefied gases, for many space homogeneous kinetic models, the equilibrium state can be computed directly from the initial data thanks to the conservation properties \eqref{eq:mom}. This makes the approach described next particularly general.

The error of the standard Monte Carlo estimator reads now
\be
\|\EE[f] - E_M[f]\|_{{\LLBi}} = {\sigma}_f M^{-1/2}.
\ee
Thus, a first variance reduction strategy is obtained by splitting the expected value of the solution as
\be
\begin{split}
	{\mathbb E}[f]({\w},t)&=\int_{\Omega} f({z},{\w},t)p({z})d{z}\\
	&=\int_{\Omega} f^\infty({z},{\w})p({z})d{z}+
	\int_{\Omega} (f({z},{\w},t)-f^\infty({z},{\w}))p({z})d{z}\\
	&=\EE[f^\infty](v)+\EE[f-f^\infty](v,t),
\end{split}
\label{eq:mme}
\ee
and exploiting the fact that $\EE[f^\infty]$ can be evaluated with arbitrary accuracy at a negligible cost (for example using a very fine grid of samples) since it does not depend on the solution computed at each time step. Now, if we use \eqref{eq:mme} and estimate 
\be
\EE[f]\approx \EE[f^\infty]+E_M[f-f^\infty]
\label{eq:28}
\ee 
we obtain an error of the type
\[ 
\| {\EE}[f](\cdot,t)-\EE[f^\infty](\cdot)-E_M[f-f^\infty](\cdot,t)\|_{{\LHBi}} = \sigma_{f-f^\infty} M^{-1/2}.
\] 
Since the non equilibrium part $f-f^\infty$ goes to zero in time exponentially fast, then also its variance goes to zero, which means that for long times estimate \eqref{eq:28} becomes exact and depends only on the accuracy of the evaluation of $\EE[f^\infty]$. 

The above argument can be generalized by considering a time dependent approximation of the solution $\tilde{f}({z},{\w},t)$, whose evaluation is significantly cheaper than computing $f({z},{\w},t)$, such that 
$\langle(\phi\tilde{f})\rangle=\langle\phi f\rangle$ for some moments and that
$\tilde{f}({z},{\w},t)\to f^{\infty}({z},{\w})$ as $t\to\infty$.  
For example, one can consider to use the BGK approximation \eqref{eq:pbgk}
which in this simple case can be exactly solved giving the following expression
\be
\tilde{f}({z},{\w},t) = e^{-\nu t} f_0({z},{\w})+(1-e^{-\nu t}) f^{\infty}({z},{\w}).
\label{eq:BGKexa}
\ee
Then we have
\be
\EE[\tilde{f}](v,t)=e^{-\nu t} {\bf f}_0({\w})+(1-e^{-\nu t}) {\bf f}^{\infty}({\w}),
\label{eq:BGKexas}
\ee
where ${\bf f}_0={\mathbb E}[f_0(\cdot,{\w})]$ and ${\bf f}^{\infty}={\mathbb E}[f^{\infty}](\w)$ or accurate approximations of the same quantities.

%We can now reasonably assume that the expected value of the control variate $\EE[\tilde{f}]({\w},t)$ is computed with arbitrary accuracy at a negligible cost. In fact the exact solution is a convex combination of the initial data and the equilibrium part. We denote this value by 
%\be
%\tilde{\bf f}(v,t)=e^{-\nu t} {\bf f}_0({\w})+(1-e^{-\nu t}) {\bf f}^{\infty}({\w}),
%\label{eq:BGKexas}
%\ee
%where ${\bf f}_0={\mathbb E}[f_0(\cdot,{\w})]$ and ${\bf f}^{\infty}={\mathbb E}[f^{\infty}](\w)$ or accurate approximations of the same quantities. 
Using the same estimator \eqref{eq:28}
\be
\EE[f]\approx \EE[\tilde f]+E_M[f-\tilde f]
\label{eq:29}
\ee 
gives the following estimate for the error
\[ 
\| {\EE}[f](\cdot,t)-\EE[\tilde f](\cdot)-E_M[f-\tilde f](\cdot,t)\|_{{\LHBi}} = \sigma_{f-\tilde f} M^{-1/2},
\] 
where even in this case, $\sigma_{f-\tilde f} \to 0$ as $t\to \infty$.
%In the rest of this section, starting to the above considerations, we generalize the described approach to arbitrary decompositions of the distribution function.

Let now observe that a general variance reduction technique can be obtained by introducing the following control variate estimator^^>\cite{DPUQ1}
\be
\tilde{E}^{\lambda}_M[f](v,t)=\frac1{M} \sum_{k=1}^M f^{k}(v,t) - \lambda\left(\frac1{M} \sum_{k=1}^M \tilde{f}^{k}(v,t)-\tilde{\bf f}({\w},t)\right).
\label{eq:nest2}
\ee
In particular, for $\lambda=0$ we recover the standard MC estimator $\tilde{E}^{0}_M[f]=E_M[f]$, whereas for $\lambda=1$ we have the  estimator $\tilde{E}^{1}_M[f]=\EE[\tilde f]+E_M[f-\tilde f]$ corresponding to \eqref{eq:29}. The following result holds true
\begin{lemma}
	The control variate estimator \eqref{eq:nest2} is unbiased and consistent for any $\lambda\in\RR$. %In particular, for $\lambda=0$ we obtain $\tilde E^{0}_M[f]=E_M[f]$ the standard MC estimator and for $\lambda=1$ we get
%	\be
%	\tilde E^{1}_M[f](v,t) = {\bf \tilde f}({\w},t) + \frac1{M} \sum_{k=1}^M (f^{k}(v,t) - \tilde f^{k}(v,t)) = {\bf \tilde f}({\w},t)+E_M[f-\tilde f](v,t),
%	\ee
%	the micro-macro estimator based on \eqref{eq:29}.
	\label{le:1}
\end{lemma}
\begin{proof}
	The expected value of the control variate estimator $\tilde E^{\lambda}_M[f]$ in \eqref{eq:nest2} yields the unbiasedness for any choice of $\lambda\in\RR$ 
	\[
	\EE[\tilde E^{\lambda}_M[f]]=\frac1{M} \sum_{k=1}^M \EE[f^{k}]-\lambda\left(\frac1{M} \sum_{k=1}^M \EE[\tilde f^{k}]-\EE[{\bf \tilde f}]\right)=\EE[f],
	\]
	since $\EE[f^{k}]=\EE[f]$ and $\EE[\tilde f^{,k}]=\EE[\tilde f]$ for $k=1,\ldots,M$.
	Moreover, since $f^k$ and $\tilde f^{k}$ are i.i.d. random variables
	\[
	\displaystyle\lim_{M\to \infty} \frac1{M} \sum_{k=1}^M f^{k} - \lambda\left(\frac1{M} \sum_{k=1}^M \tilde f^{k}-\EE[\tilde f]\right) \overset{p}{=}  \EE[f],
	\]
	from the consistency of the standard MC estimator and where the last identity has to be understood in probability sense^^>\cite{HH,Lo77}. %The last part of Lemma follows directly from \eqref{eq:nest2}.
\end{proof}

If we now consider a new the random variable depending on %the parameter 
$\lambda$
\[
f^{\lambda}(z,v,t)=f(v,z,t)-\lambda(\tilde f(z,v,t)-\tilde{\bf f}(v,t)),
\]
we have that the expectation for this new random variable is such that ${\mathbb E}[f^\lambda]={\mathbb E}[f]$, $E_M[f^\lambda]=\tilde E_M^{\lambda}[f]$, i.e. it shares the same expectation of the distribution function $f$ in terms of the random variable. Moreover, we can quantify its variance as 
\be
\var(f^\lambda)=\var(f)+\lambda^2 \var(\tilde f)-2\lambda\cov(f,\tilde f)
\label{eq:var1}
\ee
and we can prove the following result^^>\cite{DPUQ1}
\begin{theorem}
	The quantity   
	\be
	\lambda^* = \frac{\cov(f,\tilde f)}{\var(\tilde f)}
	\label{eq:lambdas}
	\ee
	minimizes the variance of $f^\lambda$ at the point $(v,t)$ and gives
	\be
	\var(f^{\lambda^*}) = (1-\rho_{f,\tilde f}^2)\var(f), 
	\label{eq:var2}
	\ee
	where $\rho_{f,\tilde f} \in [-1,1]$ is the correlation coefficient between $f$ and $\tilde f$. In addition, we have  
	\be
	\lim_{t\to\infty} \lambda^*(v,t) =1,\qquad \lim_{t\to\infty} \var(f^{\lambda^*})(v,t)=0\qquad \forall\, v \in \RR^{d_v}.
	\label{eq:lambdasa}
	\ee
	\label{th:1}
\end{theorem}
\begin{proof}
	Equation \eqref{eq:lambdas} is readily found by direct differentiation of \eqref{eq:var1} with respect to $\lambda$ and then observing that $\lambda^*$ is the unique stationary point. The fact that $\lambda^*$ is a minimum follows from the positivity of the second derivative $2\var(f^\infty)>0$. Then, by substitution in \eqref{eq:var1} of the optimal value $\lambda^*$ one finds \eqref{eq:var2} where
	\[
	\rho_{f,f^\infty} = \frac{\cov(f,f^\infty)}{\sqrt{\var(f)\var(f^\infty)}}.
	\] 
	In addition, since as $t\to\infty$ we have $f\to f^\infty$, asymptotically $\lambda^*\to 1$ and $\var(f^{\lambda^*})\to 0$ independently of $v$.
\end{proof}
The above result permits in combination with a deterministic solver satisfying \eqref{eq:det} to get the following error estimate^^>\cite{DPUQ1,MSS}
\bea
\|\EE[f](\cdot,t)-\tilde{E}^{\lambda^*}_M[f]\|_{{\LHBi}} \leq  C\left\{
\sigma_{f^{\lambda^*}}
M^{-1/2}+\Delta v^{q_1}\right\} 
\label{eq:errHMMC}
\eea
where $\sigma_{f^{\lambda^*}}=\|(1-\rho_{f,\tilde f}^2)^{1/2}\var(f)^{1/2}\|_{\LH}$, and $C>0$ depends on the final time and on the initial data. Let observe that in the above estimate, we ignored the statistical errors due to the approximation of the control variate expectation and the error in the estimate of $\lambda^*$. 
Note finally that, since $\rho^2_{f,\tilde{f}}\to 1$ as $t\to\infty$ the statistical error will vanish for large times. 

\begin{remark}
	In practice, $\cov(f,\tilde f)$ appearing in $\lambda^*$ is not known and has to be estimated. Starting from the $M$ samples we have the following unbiased estimators
	\bea
	\label{eq:varm}
	{\var}_M(\tilde f) &=& \frac1{M-1}\sum_{k=1}^M (\tilde f^{k}-E_M[\tilde f])^2,\\
	\label{eq:covm}
	{\cov}_M(f,\tilde f) &=& \frac1{M-1}\sum_{k=1}^M (f^{k}-E_M[f])(\tilde f^{k}-E_M[\tilde f]),
	\eea
	which allow to estimate
	\be
	{\lambda}_M^*= \frac{{\cov}_M(f,\tilde f)}{{\var}_M(\tilde f)}.
	\ee 
	It can be verified easily that ${\lambda}_M^*\to 1$ as $f\to f^\infty$.  
\end{remark}
The resulting multi-scale control variate method based on a low-fidelity model is summarized in Algorithm \ref{al:mscv2}.

\begin{algorithm2e}[!htbp]
	%\label{alg:3}
	\begin{enumerate}
		%\item {\bf Initialize the control variate}: From the random initial data $f_0$ compute $\tilde{f}_{\Delta {\w}}^{0}$ on the mesh $\Delta {\w}$ and denote by $\tilde{\bf f}_{\Delta {\w}}^{0}$ and $\tilde{\bf f}_{\Delta {\w}}^{\infty}$ accurate estimates of $\EE[\tilde{f}_{\Delta {\w}}^{0}]$ and $\EE[\tilde{f}_{\Delta {\w}}^{\infty}]$. 
		\item {\bf Sampling}: Sample $M$ i.i.d. initial data $f_0^k$, $k=1,\ldots,M$ from
		the random initial data $f_0$ and approximate these over the grid.  
		\item {\bf Solving:} For each realization $f_0^k$, $k=1,\ldots,M$
		\begin{enumerate}
			\item Compute the control variate $\tilde{f}_{\Delta {\w}}^{k,n}$, $k=1,\ldots,M$ at time $t^n$ using \eqref{eq:BGKexa} and denote by $\tilde{\bf f}_{\Delta {\w}}^{n}$ an accurate estimate of $\EE[\tilde{f}_{\Delta {\w}}^{n}]$ obtained from $\tilde{\bf f}_{\Delta {\w}}^{0}$ and $\tilde{\bf f}_{\Delta {\w}}^{\infty}$ using \eqref{eq:BGKexas}. 
			
			\item Solve numerically
			the underlying kinetic equation (\ref{eq:BH})  by the corresponding deterministic solvers. We denote the solution at time $t^n$ by $f^{k,n}_{\Delta {\w}}$, $k=1,\ldots,M$. 
		\end{enumerate} 
		\item {\bf Estimating}: 
		\begin{enumerate}
			\item
			Estimate the optimal value of $\lambda^*$ as
			\[
			{\lambda}_M^{*,n}= \frac{\sum_{k=1}^M (f_{\Delta \w}^{k,n}-E_M[f_{\Delta \w}^n])(\tilde f_{\Delta \w}^{k,n}-E_M[\tilde f_{\Delta \w}^n])}{\sum_{k=1}^M (\tilde f_{\Delta \w}^{k,n}-E_M[\tilde f_{\Delta \w}^n])^2}.
			\]
			\item
			Compute the expectation of the random solution with the control variate estimator
			\be
			\tilde{E}^{\lambda^*}_M[f^n_{\Delta v}]=\frac1{M} \sum_{k=1}^M f^{k,n}_{\Delta {\w}} - \lambda_M^{*,n}\left(\frac1{M} \sum_{k=1}^M \tilde{f}_{\Delta {\w}}^{k,n}-\tilde{\bf f}_{\Delta {\w}}^{n}\right).
			\label{mcest2}
			\ee
			%\be
			%E_{M,M_E}[f^{n}_{\Delta {\w}}]=\frac1{M_E} \sum_{k=1}^{M_E} f_{\Delta {\w}}^{\infty,k}+\frac1{M} \sum_{k=1}^M g^{k,n}_{\Delta {\w}}.
			%\label{mcest2}
			%\ee
		\end{enumerate}
	\end{enumerate}
	\caption{Bi-fidelity MSCV method}
	\label{al:mscv2}
\end{algorithm2e}

In Figure \ref{Figure1}, we show some results of the MSCV method in the space homogeneous case in comparison with standard MC approaches. {In this test, we solve the homogeneous Boltzmann equation with Maxwellian kernel, i.e. equation \eqref{VHS} with $\alpha=0$, for $v\in [-V,V]^2$ through the fast spectral method^^>\cite{FastSpectral} using $V=16$ and $N_v=64$ modes in each direction}. The initial condition is a two bumps problem with uncertainty
\be
f_0(z,v)=\frac{\rho_0}{2\pi} \left(\exp\left(-\frac{|v-(2+sz)|^2}{\sigma}\right)+\exp\left({-\frac{|v+(1+sz)|^2}{\sigma}}\right)\right)
\label{eq:twobumps}
\ee  
with $s=0.2$, $\rho_0=0.125$, $\sigma=0.5$ and $z$ uniform in $[0,1]$. The figure shows the expectation of the distribution function together with the $L_2$ error of the expected value of the solution. 

\begin{figure}[htb]
	\begin{center}
		\includegraphics[width=0.41\textwidth]{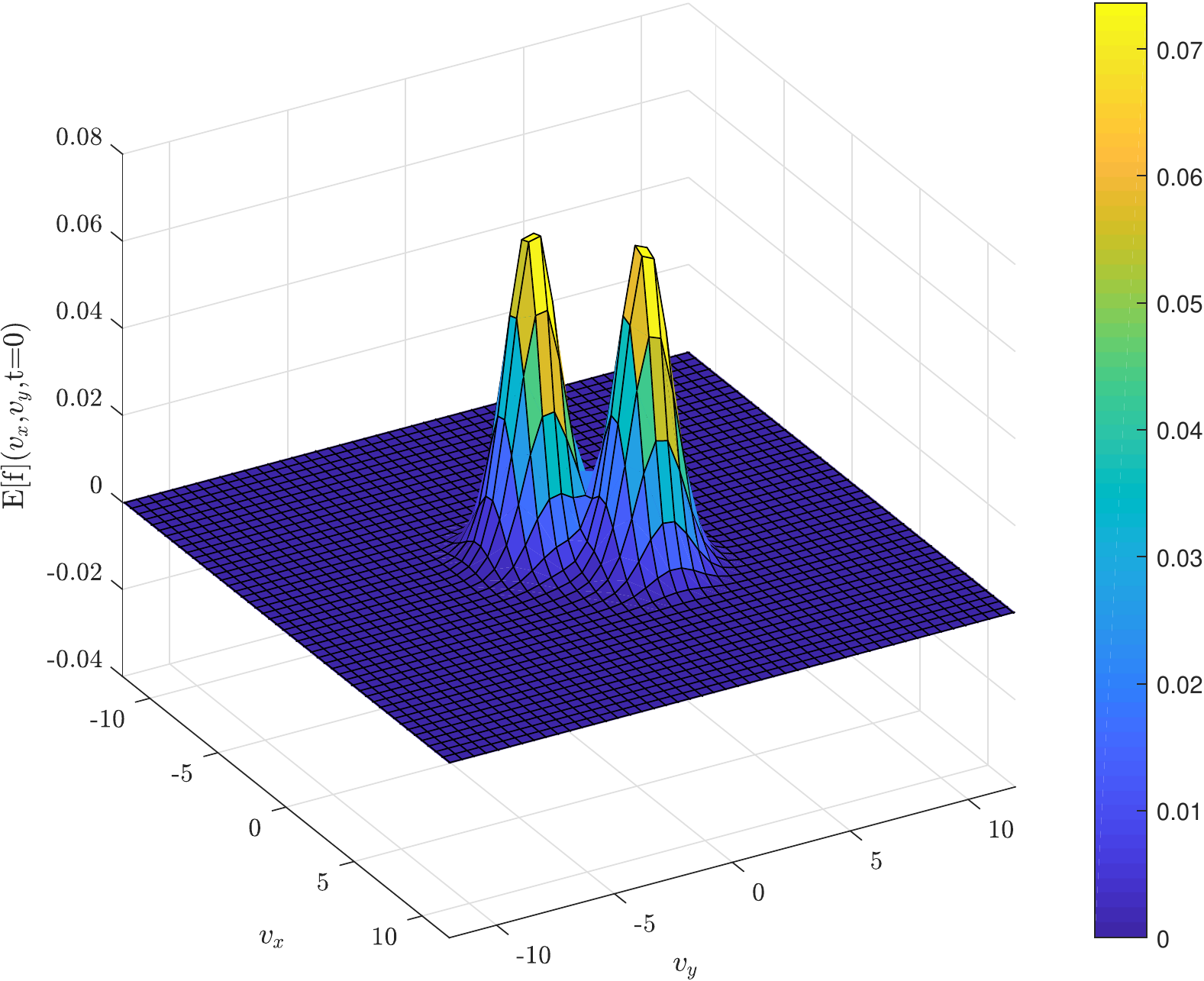}\hspace{0.5cm}%
		\includegraphics[width=0.41\textwidth]{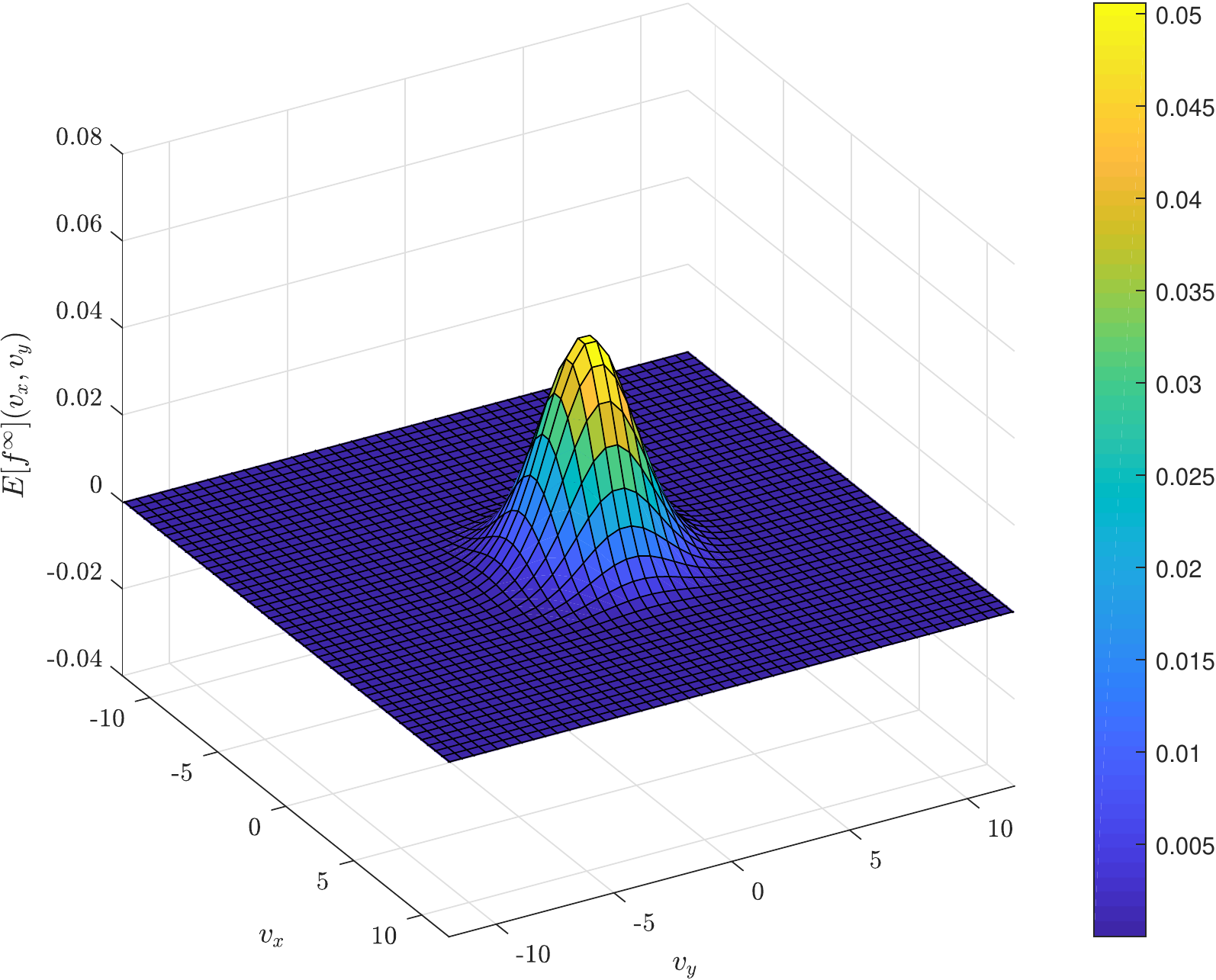}\\ 

			\includegraphics[width=.41\textwidth]{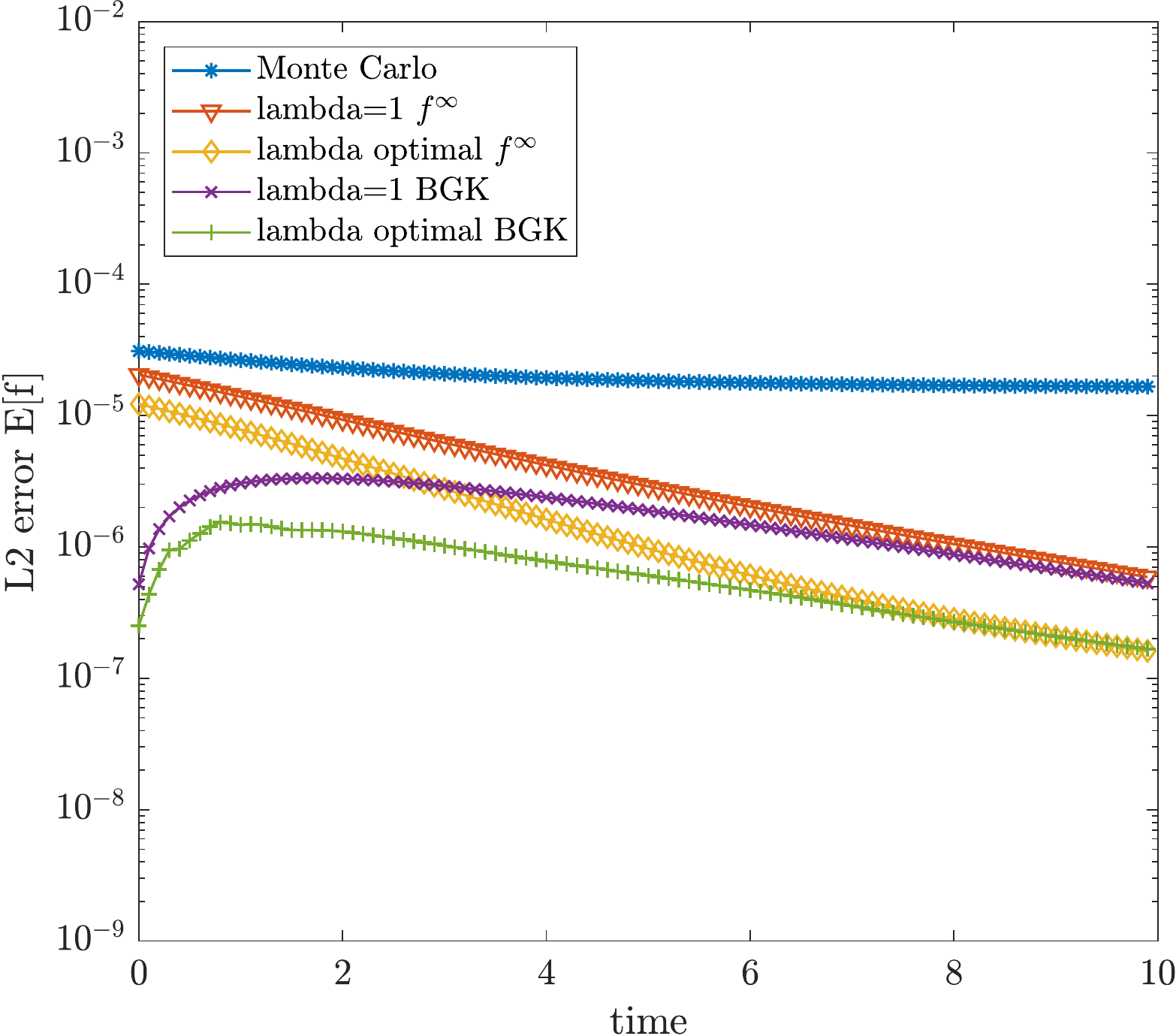}\hspace{0.5cm}
		\includegraphics[width=.41\textwidth]{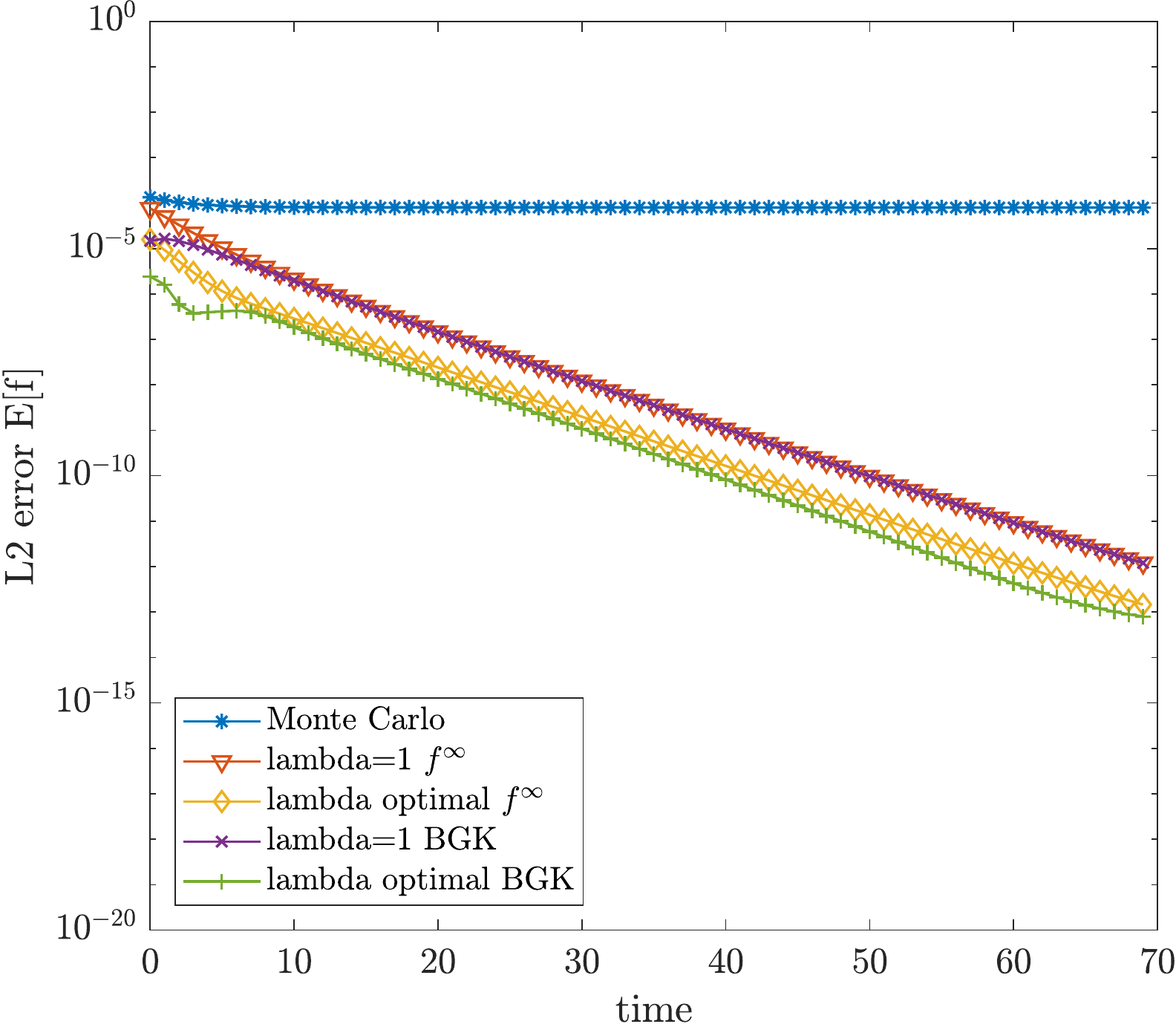}
		\caption{{MSCV method - space homogeneous case.} Expected value of the distribution function $\EE[f]$ (top) and $L_2$ norm of the error for $\EE[f]$ (bottom) using the Monte Carlo and the MSCV methods with various control variates strategies with $M=100$ samples. }
		\label{Figure1}
	\end{center}
\end{figure}

The gain in accuracy obtained with the MSCV methods is of several orders of magnitudes larger with respect to standard Monte Carlo.

\subsubsection{The space non homogeneous case}
\label{sec:MSCVnh}
We focus now on the full space non homogeneous problem \eqref{eq:Boltzmann}. 
%Assuming the classical setting of the Boltzmann equation \eqref{eq:Boltzmann}-\eqref{eq:Boltzmann}, the integration of \eqref{eq:Boltzmann} with respect to the collision invariants $\phi({\w})=1,{\w},|{\w}|^2/2$ gives the moments equations
%\begin{eqnarray}
%	\label{eq:smallscale}
%	\partial_t U +\nabla_x \cdot
%	{\mathcal F}(U)+\nabla_x \cdot \int_{\RR^{d_{\w}}} {\w}\, \phi (f-f^\infty)\,d{\w}&=&0,
%\end{eqnarray} 
Generalizing the space homogeneous method based on the local equilibrium $f^\infty$ as control variate, we consider here the Euler closure.
%, namely we assume $f=f^\infty$ in \eqref{eq:smallscale}. 
If we denote by $U_F=(\rho_F,u_F,T_F)^T$ the solution of the fluid model \eqref{eq:Euler}, 
for the same initial data, the corresponding equilibrium state $f_F^\infty$ can be used as low fidelity model. In this case, the control variate estimate based on $M$ i.i.d. samples reads 
\be
\begin{split}
& E^{\lambda}_M[f](x,v,t)=\\ &\frac1{M} \sum_{k=1}^M f^{k}(x,v,t) - \lambda\left(\frac1{M} \sum_{k=1}^M f_F^{\infty,k}(x,v,t)-{\bf f}_F^{\infty}(x,{\w},t)\right),
\end{split}
\label{eq:nesth1}
\ee
where ${\bf f}_F^{\infty}(x,{\w},t)$ is an accurate approximation of ${\mathbb E}[f_F^{\infty}(\cdot,x,{\w},t)]$. Consistency and unbiasedness of \eqref{eq:nesth1} for any $\lambda\in\RR$ follows again from Lemma \ref{le:1}. The fundamental difference is that now the variance of
\[
f^{\lambda}(z,x,v,t)=f(z,x,v,t)-\lambda({f_F}^{\infty}(z,x,{\w},t)-\EE[{f}_F^{\infty}](\cdot,x,{\w},t))
\]
will not vanish asymptotically in time since $f^\infty\neq f_F^\infty$, unless the kinetic equation is close to the fluid regime, namely for small values of the Knudsen number. 
We can state the following
\begin{theorem}
	If $\var(f_F^\infty)\neq 0$ the quantity   
	\be
	\lambda^* = \frac{\cov(f,f_F^\infty)}{\var(f_F^\infty)}
	\label{eq:lambdash}
	\ee
	minimizes the variance of $f^\lambda$ at the point $(x,v,t)$ and gives
	\be
	\var(f^{\lambda^*}) = (1-\rho_{f,f_F^\infty}^2)\var(f), 
	\label{eq:var2h}
	\ee
	where $\rho_{f,f_F^\infty} \in [-1,1]$ is the correlation coefficient between $f$ and $f_F^\infty$. In addition, we have  
	\be
	\lim_{\varepsilon\to 0} \lambda^*(x,v,t) =1,\qquad \lim_{\varepsilon\to 0} \var(f^{\lambda^*})(x,v,t)=0\qquad \forall\, (x,v) \in \mathcal{D}\times\RR^{d_v}.
	\label{eq:alambda}
	\ee
	\label{pr:2}
\end{theorem}
The proof follows the same lines of Theorem \ref{th:1}. Note that, since as $\varepsilon\to 0$
	we formally have $Q(f,f)=0$ which implies $f=f^\infty$ and $f_F^\infty = f^\infty$, from \eqref{eq:lambdash} and \eqref{eq:var2h} we obtain \eqref{eq:alambda}.

Similarly to the homogeneous case, the generalization to an improved control variate based on a suitable approximation of the kinetic solution by a low fidelity model can be done with the aid of a more accurate fluid approximation, like the compressible Navier-Stokes system, or a simplified kinetic model. In the latter case, we can solve a BGK model \eqref{eq:pbgk}
%\be
%\frac{\partial}{\partial t} \tilde{f} + {\w} \cdot \nabla_x \tilde{f} = \frac{\nu}{\varepsilon} (\tilde{f}^\infty-\tilde{f}),
%\label{eq:BGK}
%\ee 
for the same initial data and use its solution as control variate.
More precisely, given $M$ i.i.d. samples of the solution $f^k(x,v,t)$ and of the control variate $\tilde f^k(x,v,t)$ we define the new estimator
\be
\tilde{E}^{\lambda}_M[f](x,v,t)=\frac1{M} \sum_{k=1}^M f^{k}(x,v,t) - \lambda\left(\frac1{M} \sum_{k=1}^M \tilde{f}^{k}(x,v,t)-\tilde{\bf f}(x,{\w},t)\right),
\label{eq:nesth2}
\ee 
where $\tilde{\bf f}(x,{\w},t)$ is an accurate approximation of ${\mathbb E}[\tilde{f}(\cdot,x,{\w},t)]$. As for the case of the compressible Euler equations in the space non homogeneous case the variance of
\[
f^{\lambda}(z,x,v,t)=f(z,x,v,t)-\lambda(\tilde{f}(z,x,{\w},t)-\tilde {\bf f}(z,x,{\w},t))
\]
will not vanish asymptotically in time since $f^\infty\neq \tilde{f}$, unless the two solutions are very close together, such as in the fluid regime. Thus, while the first part of Theorem \ref{th:1} is still valid, the optimal value  
\be
\lambda^* = \frac{\cov(f,\tilde{f})}{\var(\tilde{f})}
\label{eq:lambdash2}
\ee
and the variance
\be
\var(f^{\lambda*})=(1-\rho^2_{f,\tilde f})\var(f)
\label{eq:varl2}
\ee
now satisfy a different condition relating the low fidelity and the high fidelity models. In fact, one can prove  that    
\be
\lim_{\varepsilon\to 0} \lambda^*(x,v,t) =1,\quad \lim_{\varepsilon\to 0} \var(f^{\lambda^*})(x,v,t)=0\quad \forall\, (x,v) \in \RR^{d_x}\times\RR^{d_v}.
\label{eq:alambda2}
\ee
In fact, since as $\varepsilon\to 0$ from \eqref{eq:Boltzmann}
we formally have $Q(f,f)=0$ which implies $f=f^\infty$ and $\tilde{f} = f^\infty$, from \eqref{eq:lambdash2} and \eqref{eq:varl2} we obtain \eqref{eq:alambda2}.

Let notice now that even if simulating the control variate system is cheaper than the full model, however its computational cost is no more negligible and thus we cannot ignore it. In this respect, we assume then that the control variate model is computed over a fine grid of $M_E \gg M$ samples. At the same time, we replace the exact computation of the expectation of the low fidelity BGK model with the approximation
\[
\tilde{\bf f}(x,{\w},t)=E_{M_E}[\tilde f](x,v,t),
\]  
in the estimator \eqref{eq:nesth2} which we will denote now by $\tilde{E}^{\lambda}_{M,M_E}[f]$. This replacement has an impact on the optimal value of $\lambda$. In fact, using the independence of $E_M[\cdot]$ and $E_{M_E}[\cdot]$ by the central limit theorem^^>\cite{Lo77,HH} we have \[\var(E_M[f])=M^{-1}\var(f), \ \var(E_{M_E}[\tilde f])=M_E^{-1}\var(\tilde f).\] Thus, minimizing the variance now leads to the optimal value
\be
\tilde\lambda^* = \frac{M_E}{M+M_E}\lambda^*,
\ee
with $\lambda^*$ given by \eqref{eq:lambdash2}. As can be easily seen from the above formula, this correction may be relevant only in the cases when $M$ and $M_E$ do not differ too much. In many practical cases and in the simulations shown after, however, $M_E \gg M$ so that $\frac{M_E}{M+M_E}\approx 1$ and we consequently assume $\tilde \lambda^* \approx \lambda^*$. 

Using the optimal value \eqref{eq:lambdash2} for the control variate variable and a deterministic solver which satisfies \eqref{eq:det}, we obtain the error estimate^^>\cite{DPUQ1, MSS}
\bea
\nonumber
&& \|E[f](\cdot,t)-\tilde{E}^{\lambda^*}_{M,M_E}[f]\|_{\LLBi}\\[-.2cm]
\label{eq:errHMMC2}
\\
\nonumber
&& \hskip 4cm \leq {C}\left\{\sigma_{f^{\lambda^*}} M^{-1/2}+\tau_{f^{\lambda^*}} M_E^{-1/2}+\Delta x^{q_1}+\Delta v^{q_2}\right\} 
\eea
with $\sigma_{f^{\lambda^*}}=\|(1-\rho_{f,\tilde f}^2)^{1/2}\var(f)^{1/2}\|_{\LL}$, $\tau_{f^{\lambda^*}}=\|\rho_{f,\tilde f} \var(f)^{1/2}\|_{\LL}$ and with a constant $C>0$ depending on the final time and on the initial data. %Let also observe that since $\rho^2_{f,\tilde{f}}\to 1$ as $\varepsilon\to 0$, therefore in the fluid limit one recovers the statistical error of the fine scale control variate model. 

\begin{remark}
	For space non homogeneous simulations, one is typically interested in the evolution of the moments $\langle \phi\, f\rangle$ instead that on the evolution of the shape of the distribution function. More in general, one can compute the optimal value of $\lambda$ with respect to any quantity of interest $q[f]$ as 
	\be
	\lambda_{q}^* = \frac{\cov(q[f],q[\tilde f]))}{\var(q[\tilde f]))}.
	\label{eq:momcv}
	\ee
In particular, in the case $q[f]=\langle \phi\, f\rangle$, we have $\lambda^*_{q}=\lambda^*_{q}(x,t)$. Note that, using \eqref{eq:momcv} all estimates in this section for $\EE[f]$ translates into estimates for $\EE[q[f]]$. This is particularly relevant when compressible Euler or Navier-Stokes equations are used as the low-fidelity model.
\end{remark}

\begin{figure}[tbp]
	\begin{center}
		\includegraphics[width=0.41\textwidth]{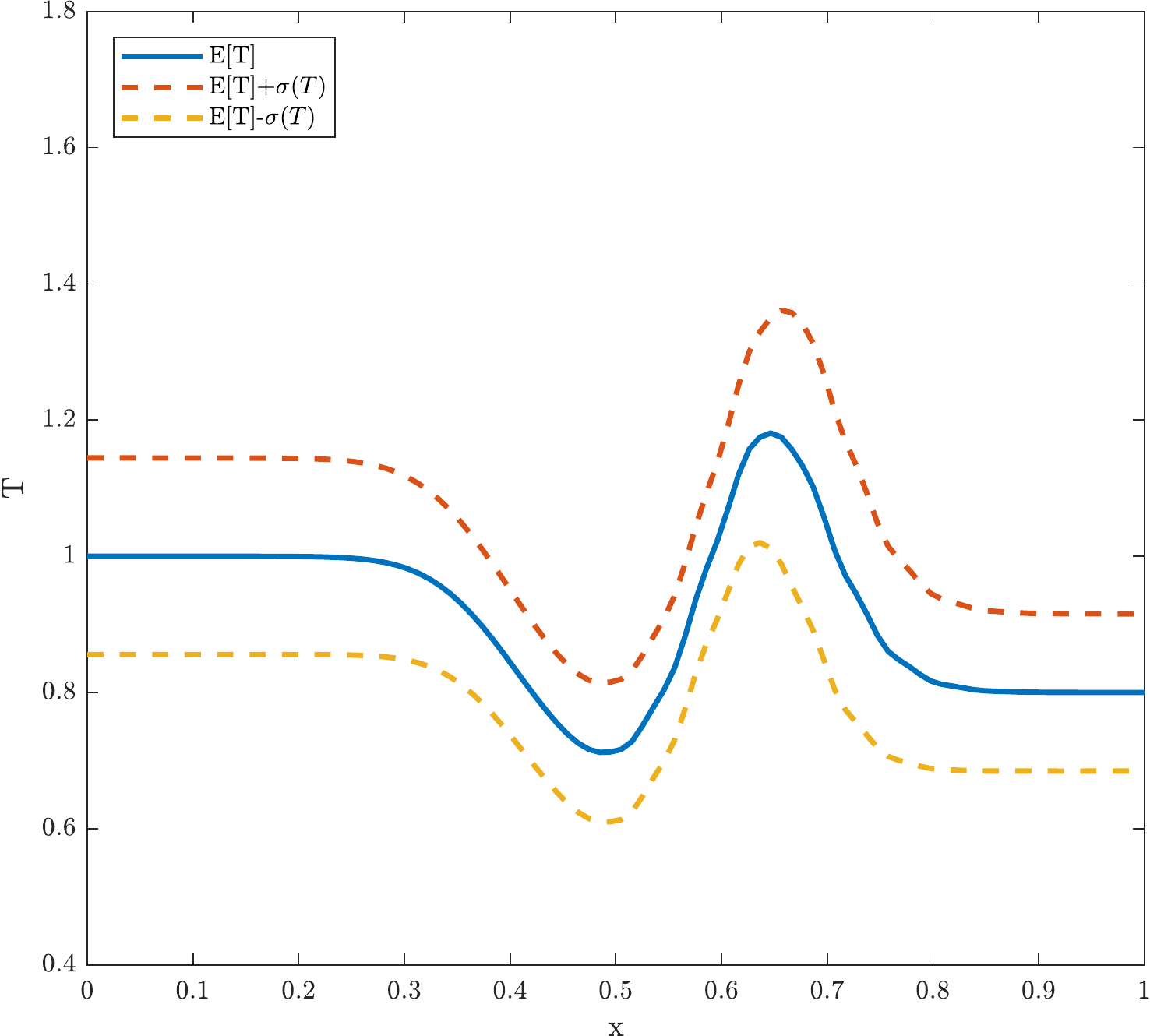}\hskip .5cm
		\includegraphics[width=0.41\textwidth]{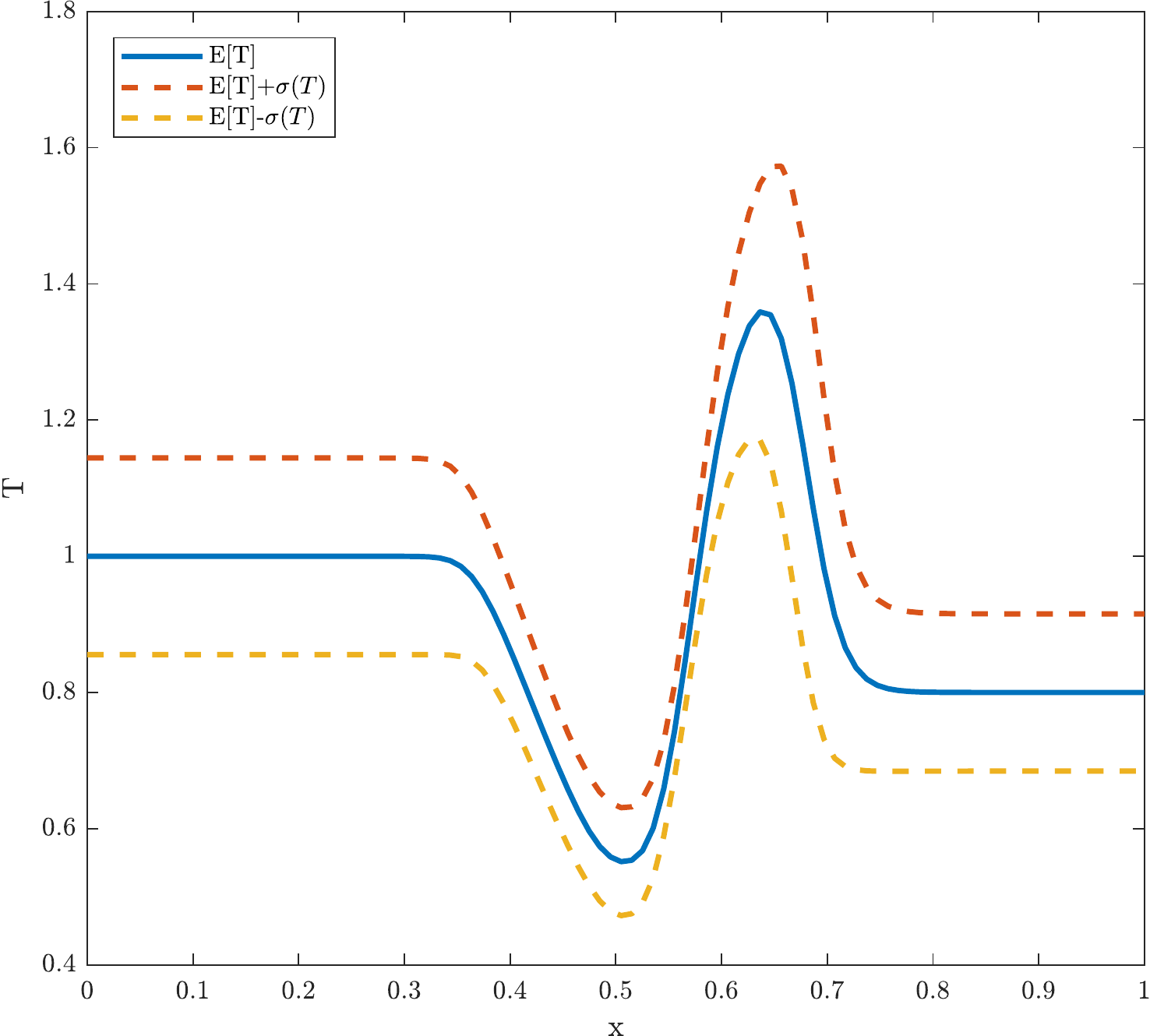}\\
		\includegraphics[width=.41\textwidth]{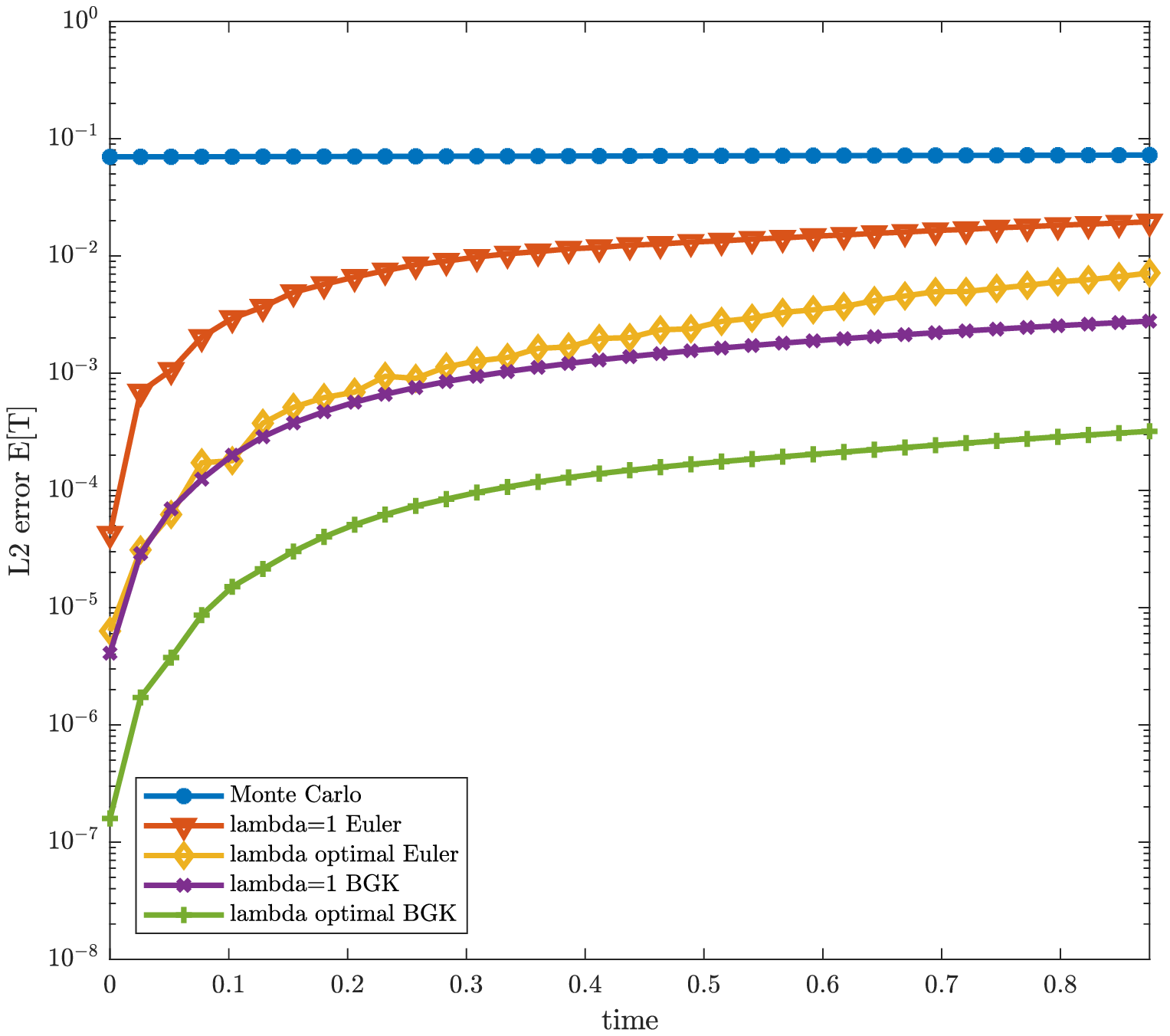}\hskip .5cm
		\includegraphics[width=.41\textwidth]{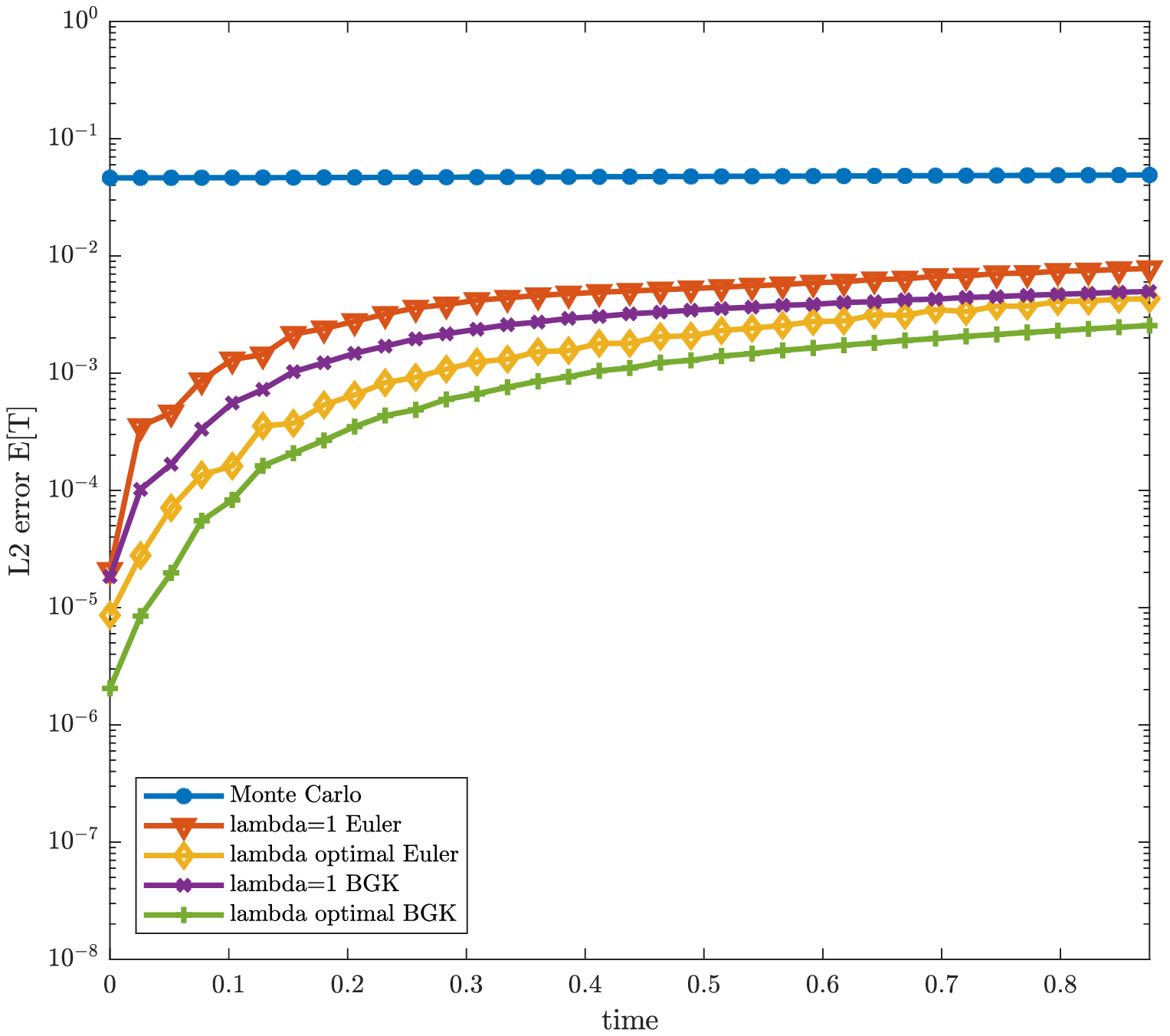}\\
		\caption{MSCV method - Sod test with uncertainty in the initial data and different low fidelity strategies.
		Expectation for the temperature $\EE[T]$ and confidence bands at final time (top) and $L_2$ norm of the error for $\EE[T]$ (bottom) with $M=10$ and $M_E=10^4$. Left: $\varepsilon=10^{-2}$. Right: $\varepsilon=10^{-3}$.}
		\label{Figure2}
	\end{center}
\end{figure}

In Figure \ref{Figure2}, we show the results of the MSCV method where the BGK or the compressible Euler equations are used as low fidelity models. {The Boltzmann equation is solved for $v\in [-V,V]^2$, $x\in [0,L]$, and $\alpha=0$ in \eqref{VHS}. The fast spectral method^^>\cite{MP} for $V=8$ with $N_v=32$ modes in each direction is used combined with a fifth order WENO method^^>\cite{Shu_2020} in space for $L=1$ with $N_x=100$ grid points.}. The initial conditions are given by Sod problem with uncertain temperature
\bea
\nonumber
&\rho_0(x)=1, \ \ T_0(z,x)=1+sz \qquad &\textnormal{if} \ \ 0<x<L/2 \\[-.25cm]
\\[-.25cm]
\label{eq:sodtest}
\nonumber
&\rho_0(x)=0.125,\ T_0(z,x)=0.8+sz  \qquad &\textnormal{if} \ \ L/2<x<L
\eea
with $s=0.25$, $z$ uniformly distributed in $[0,1]$ and equilibrium initial distribution given by
\[
f_0(z,x,\w)=\frac{\rho_0(x)}{2\pi } \exp\left({-\frac{|v|^2}{2T_0(z,x)}}\right).
\] 
We report the results for two different regimes, namely $ \varepsilon=10^{-2}$ and $ \varepsilon=10^{-3}$.  
The images show on the top the expectation of the temperature at the final time together with the confidence bands $\EE[T]-\sigma_T, \EE[T]+\sigma_T$ with $\sigma_T$ the standard deviation.  
On the bottom, we report the various errors for the expected value of the temperature as a function of time. The number of samples used to compute the expected value of the solution is $M=10$ while the number of samples used to compute the control variate is $M_E=10^4$. The optimal values of $\lambda^*(x,t)$ have been computed with respect to the temperature. 
%Other optimal values for different macroscopic or microscopic quantities can be computed as well as post-processing. 

In Figure \ref{Figure2NS} we show the results of the same MSCV approach where the Navier-Stokes equations \eqref{eq:NavierStokes} are used as low fidelity model. The setting is similar to the one shown in Figure \ref{Figure2} where now, however, uncertainty is also present in the initial density. We then have
\bea
\nonumber
&\rho_0(x)=1+sz, \ \ T_0(z,x)=1+sz \qquad &\textnormal{if} \ \ 0<x<L/2 \\[-.25cm]
\\[-.25cm]
\label{eq:sodtest2}
\nonumber
&\rho_0(x)=0.125(1+sz),\ T_0(z,x)=0.8+sz  \qquad &\textnormal{if} \ \ L/2<x<1,
\eea
the other discretization parameters remaining unchanged. In this case, the optimal values of $\lambda^*(x,t)$ have been computed both with respect to the density (left images) as well as with respect to the temperature (right images). We stress that the computation of the optimal values is done through an offline procedure. This employs, irrespectively of the quantity of interest, the same results of the deterministic simulations for the low and high fidelity models. These results are then combined to improve the estimation of the expectation of respectively the density and the temperature. For all situations reported the Navier-Stokes control variate approach improves the result of the compressible Euler case, especially when far from the thermodynamical equilibrium.

\begin{figure}[!htbp]
	\begin{center}
		\includegraphics[width=0.41\textwidth]{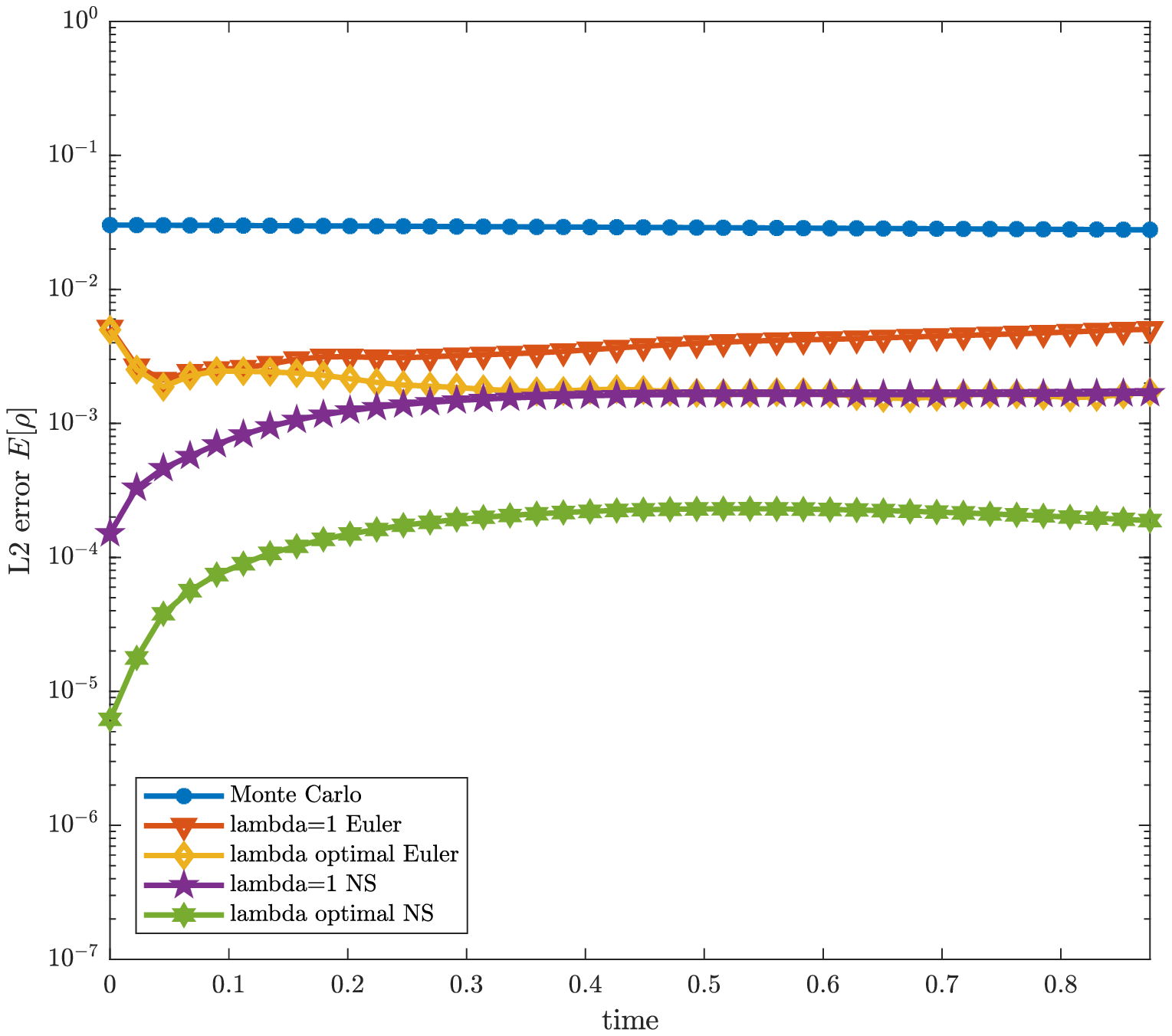}\hskip .5cm
		\includegraphics[width=0.41\textwidth]{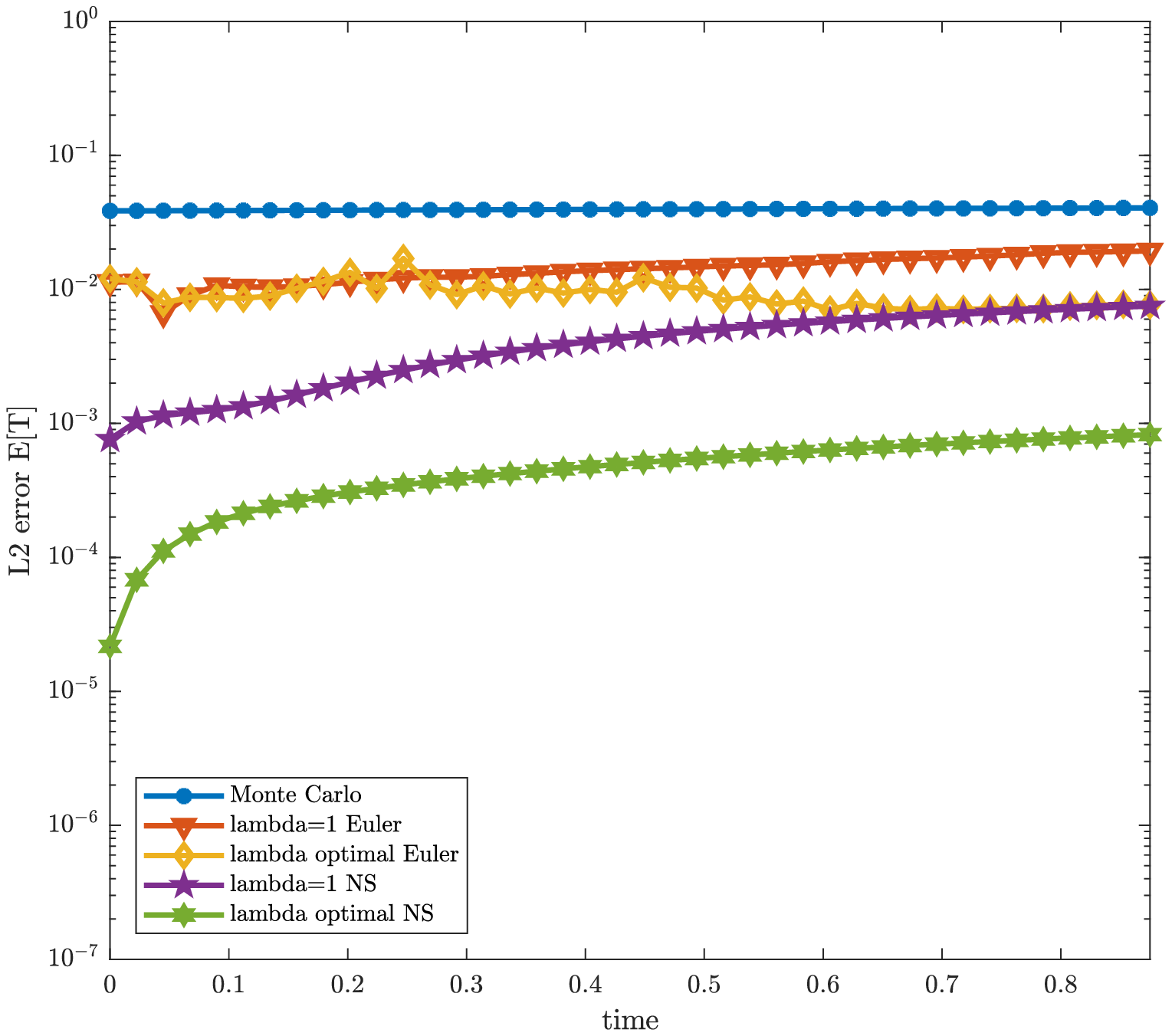}\\
		\includegraphics[width=.41\textwidth]{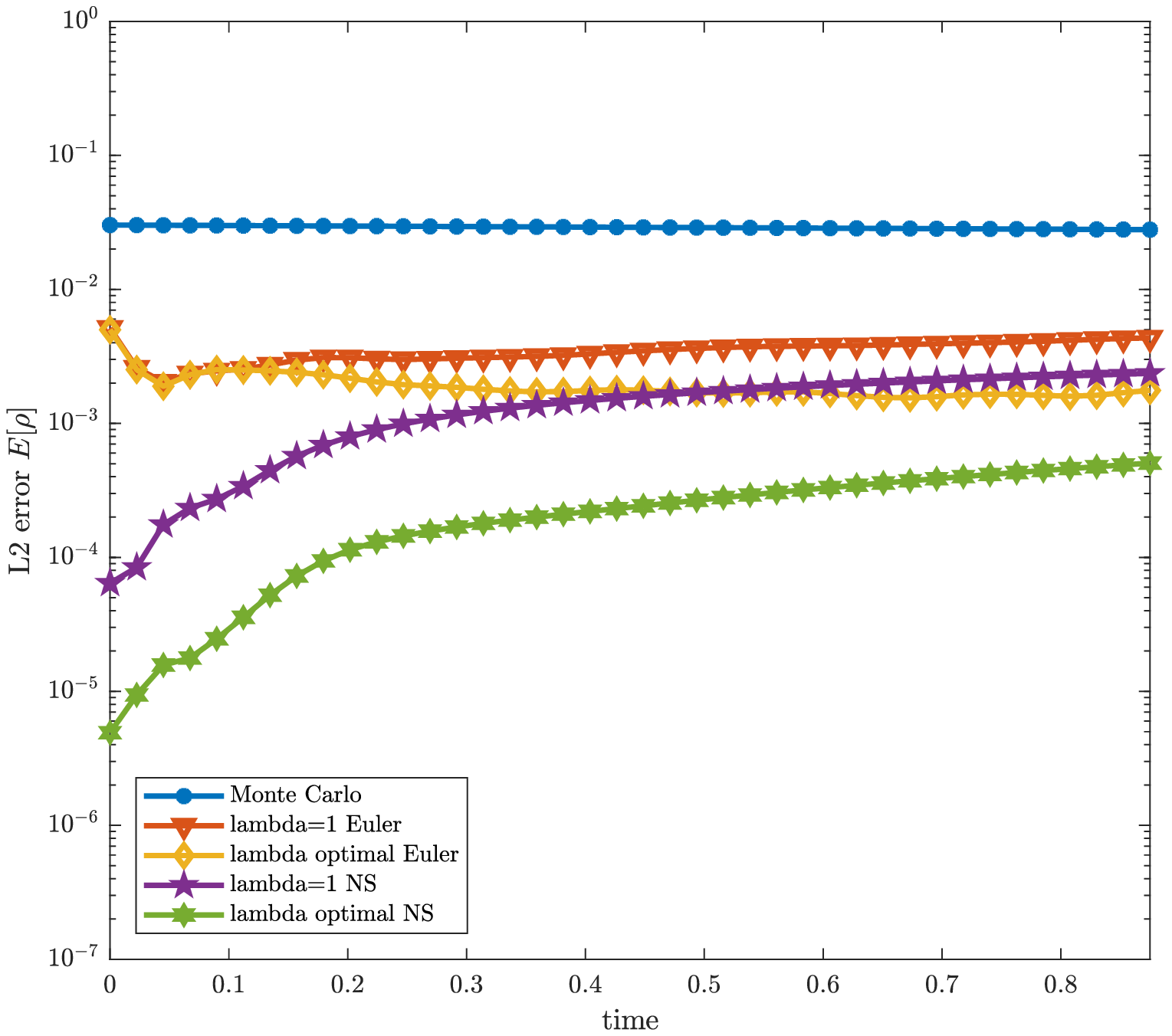}\hskip .5cm
		\includegraphics[width=.41\textwidth]{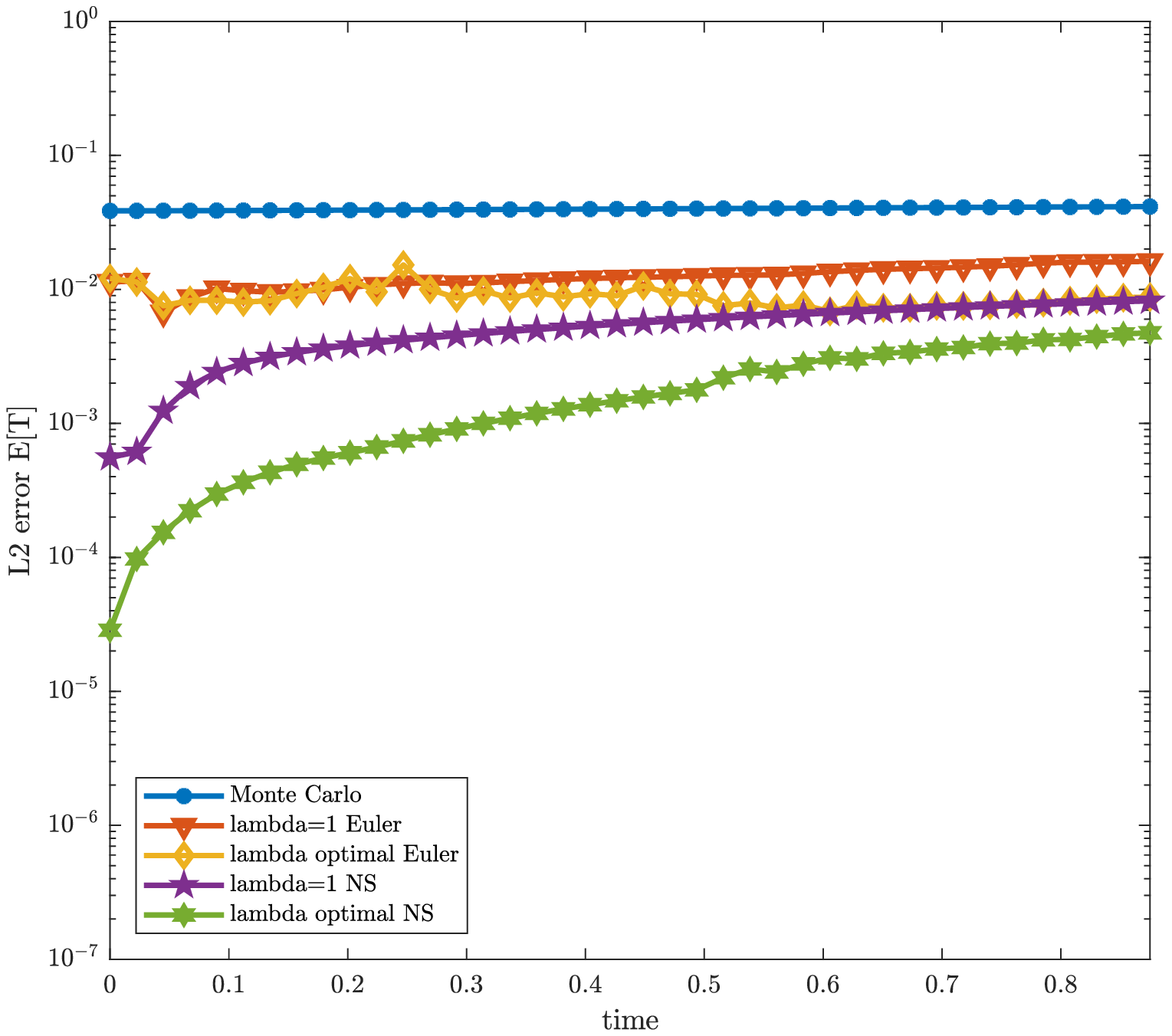}\\
		\includegraphics[width=.41\textwidth]{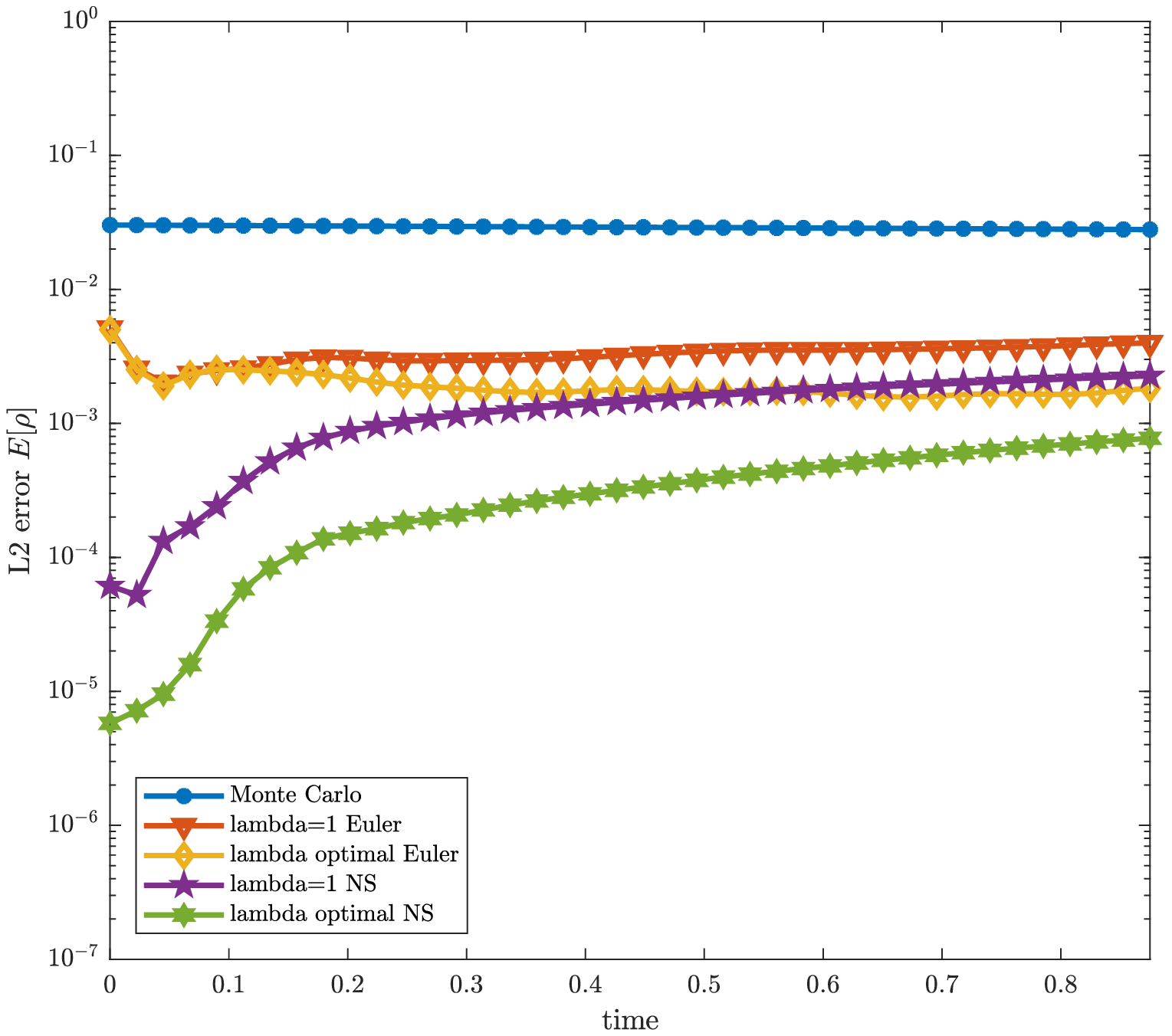}\hskip .5cm
		\includegraphics[width=.41\textwidth]{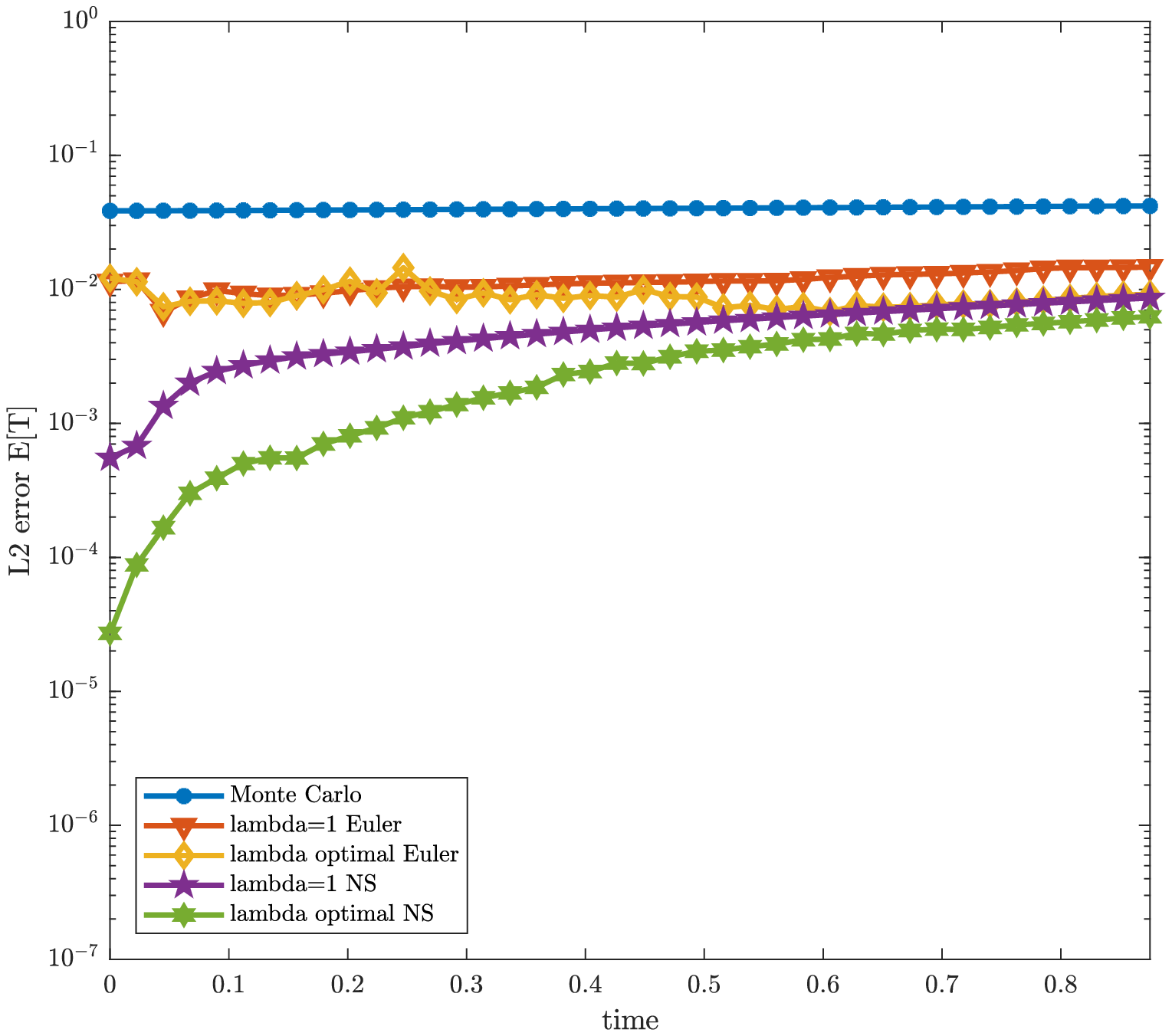}\\
		\caption{MSCV method - Sod test with uncertainty in the initial data. Comparison between the compressible Euler and the Navier-Stokes control variate models. $L_2$ norm of the error for $\EE[\rho]$ (left) and $\EE[T]$ (right) with $M=10$. Top: $\varepsilon=10^{-2}$. Middle: $\varepsilon=10^{-3}$. Bottom: $\varepsilon=5 \times 10^{-4}$.}
		\label{Figure2NS}
	\end{center}
\end{figure}
%\newpage
\subsection{Multi-fidelity MSCV methods}
\subsubsection{The space homogeneous case.}
In this section we extend the control variate strategy to the case in which several low-fidelity models are employed to accelerate the statistical convergence. We start again from the space homogeneous equation \eqref{eq:BH} for sake of clarity. Let us consider ${f}_{1}(z,v,t),\ldots,{f}_{L}(z,v,t)$ approximations of $f(z,v,t)$ furnished by a set of low fidelity models whose properties will be discussed later. We can then define a new random variable as
\be
f^{\lambda_1,\ldots,\lambda_L}(z,v,t)=f(z,v,t)-\sum_{h=1}^L \lambda_h ({f}_{h}(z,v,t)-\EE[{f}_{h}](v,t)).
\label{eq:mcv}
\ee  
Clearly \eqref{eq:mcv} is such that $\EE[f^{\lambda_1,\ldots,\lambda_L}]=\EE[f]$ while the variance of this new variable is given by
\begin{eqnarray}
\var(f^{\lambda_1,\ldots,\lambda_L})&=&\var(f)+\sum_{h=1}^L \lambda_h^2 \var({f}_{h})+\nonumber\\
&&+2\sum_{h=1}^L\lambda_h\left(\sum_{k=1 \atop k\neq h}^L \lambda_k
\cov({f}_{h},{f}_{k}) - \cov(f,{f}_{h})\right).
\end{eqnarray}
The above quantity can be rewritten in a more compact form by introducing the following notations
\[
{\llambda}=(\lambda_1,\ldots,\lambda_L)^T,\quad
{b}=(\cov(f,{f}_{1}),\ldots,\cov(f,{f}_{L}))^T
\]
giving then
\be
\var(f^{\llambda})=\var(f)+\llambda^T  C \llambda - 2\llambda^T b
\label{eq:varmcv}
\ee
where $ C=( c_{ij})$, ${ c_{ij}}=\cov({f}_{i},{f}_{j})$ is the symmetric $L\times L$ covariance matrix. Following the same path of the bi-fidelity case, one can now try to find the set $\Lambda^*$ minimizing the variance of the new variable $f^\Lambda$. This is obtained thanks to the following Theorem (see^^>\cite{DPUQ2} for a proof):
\begin{theorem}
	Assuming the covariance matrix is not singular, the vector
	\be
	\llambda^* = {C}^{-1} b,
	\label{eq:lambdamcv}
	\ee
	minimizes the variance of $f^\llambda$ at the point $(v,t)$ and gives
	\be
	\var(f^{\llambda^*})= \left(1-\frac{b^T(C^{-1})^T b}{\var(f)}\right)\var(f).
	\label{eq:varmvcs}
	\ee
	\label{th:2c}
\end{theorem}
%\begin{proof}
%	To compute the minimizing values $\lambda^*_h$, $h=1,\ldots,L$ the first order optimality conditions are found by equating to zero the partial derivatives with respect to $\lambda_h$
%	\be
%	\frac{\partial\, \var(f^{\llambda})}{\partial \lambda_h} = 0, \quad h=1,\ldots,L. 
%	\label{eq:min}
%	\ee
%	This corresponds to solve the following linear system
%	\be
%	\cov(f,{f}_{h}) = \sum_{k=1}^L \lambda_k \cov({f}_{h},{f}_{k}),\quad h=1,\ldots, L,
%	\label{eq:sys1}
%	\ee
%	or in vector form
%	\be
%	b = {C} \llambda.
%	\ee
%	Therefore, assuming the covariance matrix is not singular, we obtain the solution \eqref{eq:lambdamcv}.
%	It is easily shown, via the second order optimality conditions that $\llambda^*$ is
%	indeed the variance-minimizing choice of $\llambda$. By direct substitution in \eqref{eq:varmcv} we obtain \eqref{eq:varmvcs}.
%\end{proof}

Having the above result in mind, one can then introduce the control variate estimator based on \eqref{eq:mcv} which takes the form
\be
E^{\llambda}_M[f](v,t) = E_M[f](v,t)-\sum_{h=1}^L \lambda_h \left(E_{M}[f_h](v,t)-{\bf f}_{h}(v,t)\right),
\label{eq:mscvg}
\ee
where ${\bf f}_{h}(v,t)$ is an accurate approximation of $\EE[f_h](v,t)$. In the above formula, we also assumed to have at disposal $M$ i.i.d. samples from the solution $f(z,v,t)$ and from the control variate functions $f_h(z,v,t)$ for $h=1,\ldots,L$. 
Moreover, in practice, as done for the bi-fidelity case, to estimate the value of the vector $\llambda^*$  we can use directly the Monte Carlo samples as in the bi-fidelity case. 
The resulting multi-fidelity MSCV method is summarized in Algorithm \ref{alg:2}.

\begin{algorithm2e}[!htbp]
	\begin{enumerate}
		\item {\bf Sampling}: 
		Sample $M$ i.i.d. initial data from
		the random initial data $f_0$ and approximate these over the grid $\Delta v$. Denote these samples by $f_{\Delta v}^{k,0}$, $k=1,\ldots,M$.
		\item {\bf Solving}: 
		\begin{enumerate}
			\item 
			For each control variate and for each realization of the random
			input data $f_{\Delta v}^{k,0}$, $k=1,\ldots,M$, the resulting control variate model is solved with mesh width $\Delta v$. We denote the resulting ensemble of deterministic solutions for $h=1,\ldots,L$ at time $t^n$ by
			\[
			f_{h,\Delta v}^{k,n},\quad k=1,\ldots,M.
			\]
			\item For each realization $f_{\Delta v}^{k,0}$, $k=1,\ldots,M$ 
			the underlying kinetic equation (\ref{eq:BH}) is solved  with mesh width $\Delta v$. We denote the solution at time $t^n$ by $f^{k,n}_{\Delta {\w}}$, $k=1,\ldots,M$. 
		\end{enumerate}
		\item {\bf Estimating}: 
		\begin{enumerate}
			\item
			Estimate the optimal vector of values $\llambda^*$ solving
			\be
			C^n_M \llambda^{*,n} = b^n_M,
			\label{eq:ill}
			\ee
			where $(C^n_M)_{ij} = \cov_M(f_{i,\Delta v}^{n},f_{j,\Delta v}^{n})$ and $(b^n_M)_i=\cov_M(f_{\Delta v}^{n},f_{i,\Delta v}^{n})$.
			\item
			Compute the expectation of the random solution with the control variate estimator
			\be
			E^{\llambda^*}_M[f^n_{\Delta v}] = \frac1{M} \sum_{k=1}^M f^{k,n}_{\Delta {\w}}-\sum_{h=1}^L \lambda^{*,n}_{h} \left(\frac1{M} \sum_{k=1}^M {f}_{h,\Delta {\w}}^{k,n}-{\bf f}_{h,\Delta v}^n\right).
			\label{mcest2b}
			\ee
		\end{enumerate}
	\end{enumerate}
	\caption{Multi-fidelity MSCV method}
	\label{alg:2}
\end{algorithm2e}

An interesting result is obtained introducing the vector $F=(F_1,\ldots,F_L)^T$, such that $F_h=f_h-\EE[f_h]$. This gives $\EE[F_h]=0$, $h=1,\ldots,L$ and equation \eqref{eq:mcv} reads 
\be
f^{\lambda_1,\ldots,\lambda_L}(z,v,t)=f(z,v,t)-\sum_{h=1}^L \lambda_h F_h(z,v,t).
\label{eq:mcv2}
\ee 
This permits to conclude that the variance of $f^{\llambda^*}$ is reduced to zero if $f$ is in the span of the set of functions  $F_1,\ldots,F_L$. Using now Gram--Schmidt orthogonalization %, we may assume that the $L$ components of the control variate vector $F$ are orthogonal in the $L^2$ inner product
%\[
%\langle f,g \rangle = \int_{\Omega} f(z) g(z) p(z) dz.
%\] 
and observing that
\[
\langle F_h,F_k \rangle=\cov(f_h,f_k)=\cov(F_h,F_k), \quad h,k=1,\ldots, L,
\]
we can also construct the vector $G=(G_1,\ldots,G_L)^T$, with orthogonal components, $\langle G_h,G_k \rangle = 0$ for $h\neq k$, as follows^^>\cite{GV}
\be
g_h = f_h - \sum_{j=1}^{h-1} \frac{\cov(g_j,f_h)}{\var(g_j)} g_j,\qquad h=1,\ldots,L,
\label{eq:gs}
\ee
and define $G_h=g_h-\EE[g_h]$, such that $\EE[G_h]=0$, $h=1,\ldots,L$.
Then, we may try to minimize the variance of the random variable 
\be
f^{\gamma_1,\ldots,\gamma_L}(z,v,t)=f(z,v,t)-\sum_{h=1}^L \gamma_h G_h(z,v,t),
\label{eq:mcv3}
\ee
which now using the orthogonality property reads 
\[
\var(f^{\gamma_1,\ldots,\gamma_L})=\var(f)+\sum_{h=1}^L \gamma_h^2 \var(g_h)-2\sum_{h=1}^L \gamma_h\cov(f,g_h).
\] 
%Denoting with $\Gamma=(\gamma_1,\ldots,\gamma_L)^T$, $D$ the diagonal matrix with elements $d_h=\var(g_h)$ and with $e$ the vector with components $e_h=\cov(f,g_h)$ we get
%\be
%\var(f^\Gamma)=\var(f)+\Gamma^T D \Gamma - 2\Gamma^T e.
%\ee
%By the same arguments as in Theorem \ref{th:2}, if the matrix $D$ is not singular, we have that the vector $\Gamma^*=D^{-1}e$ minimizes the above variance.
Given the above arguments, one can prove the following result^^>\cite{DPUQ2}:
\begin{theorem}
	If the control variate vector $G=(G_1,\ldots,G_L)^T$ in \eqref{eq:mcv3} has orthogonal components, $\langle G_h,G_k \rangle = 0$ for $h\neq k$, then if $\langle G_h,G_h \rangle \neq 0$ the vector $\Gamma^*$ with components
	\be
	\gamma^*_h = \frac{\cov(f,g_h)}{\var(g_h)},\qquad h=1,\ldots,L,
	\label{eq:gammah}
	\ee
	minimizes the variance of $f^\Gamma$ at the point $(v,t)$ and gives
	\be
	\var(f^{\Gamma_*})=\left(1-\sum_{h=1}^L \rho^2_{f,g_h}\right)\var(f)
	\label{eq:var3}
	\ee
	where $\rho_{f,g_h}\in [-1,1]$ is the correlation coefficient between $f$ and $g_h$. 
\end{theorem}
Finally, estimating the orthogonal set of control variates using $M$ samples by
\be
E^{\Gamma}_M[f](v,t) = E_M[f](v,t)-\sum_{h=1}^L \gamma_h \left(E_{M}[g_h](v,t)-{\bf g}_h(v,t)\right),
\label{eq:mscvg2}
\ee
where ${\bf g}_h(v,t)=\EE[g_h](v,t)$ or its accurate approximation, 
in combination with a deterministic solver satisfying \eqref{eq:det}, one obtains the following result^^>\cite{DPUQ1,MSS}:
\begin{proposition}
	%Consider a deterministic scheme which satisfies \eqref{eq:det} in the velocity space for the solution of the homogeneous kinetic equation \eqref{eq:BH} with deterministic interaction operator $Q(f,f)$ 
	Let consider \eqref{eq:BH} with random and sufficiently regular initial data $f({z},{\w},0)=f_0({z},{\w})$. The multi-fidelity MSCV method \eqref{eq:mscvg2} with the optimal values \eqref{eq:gammah} satisfies the error bound 
	\bea
	\|\EE[f](\cdot,t^n)-{E}^{\Gamma^*}_M[f^n_{\Delta v}]\|_{{\LHBi}} \leq  C\left\{\sigma_{f^{\Gamma^*}} M^{-1/2}+\Delta v^{q_2}\right\} 
	\label{eq:errHMMC2b}
	\eea
	where $\sigma_{f^{\Gamma^*}}=\left\|\left(1-\sum_{h=1}^L \rho^2_{f,g_h}\right)^{1/2}\var(f)^{1/2}\right\|_{\LH}$,
	and $C>0$ depends on the final time and on the initial data. 
\end{proposition}

\subsubsection{A leading example: two control variates}
To better exemplify the multi-fidelity approach, here we give the details of the method in the case $L=2$, where $f_1 = f_0$, the initial data, and $f_2 = f^\infty$, the stationary state. In this case we know that $f$ is in the span of the control variates at $t=0$ and as $t\to\infty$. A straightforward computation shows that the optimal values $\lambda^*_1$ and $\lambda_2^*$ are given by
\begin{eqnarray}
	\nonumber
	\lambda_1^* &=& \frac{\var(f^\infty)\cov(f,f_0)-\cov(f_0,f^\infty)\cov(f,f^\infty)}{\Delta},\\[-.2cm]
	\label{eq:l12}
	\\[-.2cm]
	\nonumber
	\lambda_2^* &=& \frac{\var(f_0)\cov(f,f^\infty)-\cov(f_0,f^\infty)\cov(f,f_0)}{\Delta},
\end{eqnarray}
where $\Delta = \var(f_0)\var(f^\infty)-\cov(f_0,f^\infty)^2\neq 0$.
Using $M$ samples for both control variates, the optimal estimator reads
\begin{eqnarray}
E^{\lambda^*_1,\lambda^*_2}_M[f](v,t) &=&E_M[f](v,t)-\lambda^*_1 \left(E_M[f_0](v)-{\bf f}_0(v)\right)+\nonumber
\\[-.2cm]
\label{eq:nest2b}
	\\[-.3cm]
	\nonumber
&&-\lambda^*_2 \left(E_M[f^\infty](v)-{\bf f}^\infty(v)\right).
\end{eqnarray}
Now, at $t=0$ since $f(z,v,0)=f_0(z,v)$ we clearly have $\lambda_1^*=1$ and $\lambda_2^*=0$ so that the estimator \eqref{eq:nest2b} is exact \[E^{1,0}_M[f](v,0)={\bf f}_0(v).\] Moreover, 
by the same arguments as in Theorem \ref{th:1}, for large times since $f(z,v,t)\to f^\infty(z,v)$ from \eqref{eq:l12} we get 
\[
\lim_{t\to\infty} \lambda_1^* = 0, \qquad \lim_{t\to\infty} \lambda_2^* = 1,
\]
and thus, the variance of the estimator vanishes asymptotically in time
\[
\lim_{t\to\infty} E^{\lambda^*_1,\lambda^*_2}_M[f](v,t) = E^{0,1}_M[f](v)={\bf f}^\infty(v).
\]
We want now to emphasize the relation between this last example and the bi-fidelity case based on the BGK model discussed in \eqref{eq:nest2}. This latter can be rewritten in the form \eqref{eq:nest2b}
\[
E^{\lambda^*}_M[f](v,t) =E_M[f](v,t)-\tilde\lambda^*_1 \left(E_M[f_0](v)-{\bf f}_0(v)\right)-\tilde\lambda^*_2 \left(E_M[f^\infty](v)-{\bf f}^\infty(v)\right)
\]
where
\[
\tilde\lambda_1^* = e^{-t} \lambda^*,\qquad \tilde\lambda_2^* = (1-e^{-t}) \lambda^*.
\]
Therefore, the single control variate based on the BGK model can be understood as a suboptimal solution to the minimization problem for the control variates $f_0$ and $f^\infty$. 
In particular, if the solution $f$ has the form \eqref{eq:BGKexa}, namely the full model is the BGK model, then it is in the span generated by $f_0$ and $f^\infty$ and we obtain $\lambda_1^*=\tilde\lambda_1^*$ and  $\lambda_2^*=\tilde\lambda_2^*$. 

We now compare the case of the single control variate given by the BGK model discussed in \ref{sec:MSCV} with the case of the two control variate approach discussed here. The number of samples used to compute the expected solution for the Boltzmann equation is $M=100$ while the expected values of the control variates can be evaluated offline as in the BGK case. The initial condition is the two bumps problem with uncertainty \eqref{eq:twobumps} and {the same discretization parameters as in Figure \ref{Figure1} have been used.}
%\be
%f_0(z,v)=\frac{\rho_0}{2\pi} \left(\exp\left(-\frac{|v-(2+sz)|^2}{\sigma}\right)+\exp\left({-\frac{|v+(1+sz)|^2}%{\sigma}}\right)\right)
%\ee  
%with $s=0.2$, $\rho_0=0.125$, $\sigma=0.5$ and $z$ uniform in $[0,1]$. 
In Figure \ref{Figure3}, we report the $L_2$ error with respect to the random variable in the computation of the expected value for the distribution function $\EE[f](v,t)$ for the different methods. We observe that the multi-fidelity method with two control variates permits to gain one order of accuracy with respect to the standard bi-fidelity approach.

\begin{figure}
	\begin{center}
		\includegraphics[width=.6\textwidth]{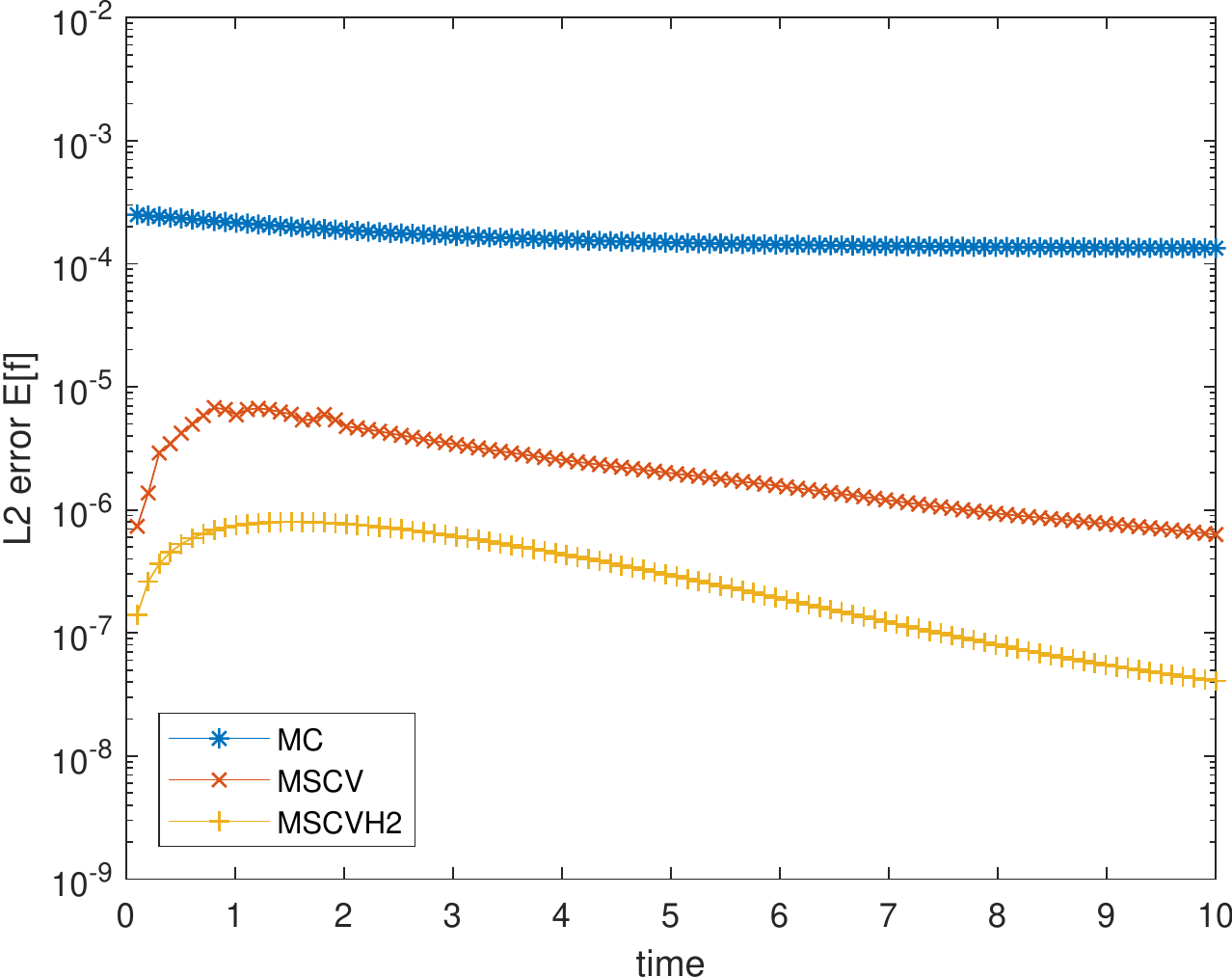}
		\caption{{Multi-fidelity MSCV method - space homogeneous case}. $L_2$ norm of the error in time for the expectation of the distribution function for the MC method, the bi-fidelity MSCV method based on the BGK solution and the multi-fidelity MSCV2 method based on the two control variates $f_0$ and $f^\infty$.}
		\label{Figure3}
	\end{center}
\end{figure}

\subsubsection{The space non homogeneous case.}
For non homogeneous problems, as in the bi-fidelity case summarized in Section \ref{sec:MSCVnh}, we cannot assume to know offline the expectation of the control variate. Therefore, the control variates have a non-negligible computational cost. Each control variate, in fact, acts at a certain scale and requires the numerical solution of a suitable time dependent model in the phase space. The multi-fidelity MSCV estimator \eqref{eq:mscvg}, based on the multi fidelity control variates $f_1(z,x,v,t)$, $\ldots$, $f_L(z,x,v,t)$ for the solution $f(z,x,v,t)$ to the space non homogeneous problem \eqref{eq:Boltzmann}, reads 
\begin{eqnarray}
	E^{\Lambda}_{M,M_E}[f](x,v,t) &=& E_M[f](x,v,t)+\nonumber\\[-.1cm]
	\label{eq:mscvgh}
	\\[-.4cm]
	\nonumber
	&&-\sum_{h=1}^L \lambda_h \left(E_{M}[f_h](x,v,t)-E_{M_{E}}[f_h](x,v,t)\right),
\end{eqnarray}
where $M_E \gg M$ samples have been used to estimate the expectations of the control variates $\EE[{f}_{h}(z,v,t)]$. As in the bi-fidelity case, minimization of the variance of \eqref{eq:mscvgh},
leads to the optimal values
\be
\tilde\Lambda^* = \frac{M_E}{M+M_E}\Lambda^*,
\label{eq:lambdamcv2}
\ee
with $\Lambda^*$ given by \eqref{eq:lambdamcv}. In the sequel we assume 
$M_E \gg M$ so that $\tilde\Lambda^*\approx \Lambda^*$.

We can apply the Gram--Schmidt orthogonalization \eqref{eq:gs} and estimate 
\begin{eqnarray}
	E^{\Gamma}_{M,M_E}[f](v,t) &=& E_M[f](x,v,t)+\nonumber
	\\[-.1cm]
	\label{eq:mscvg2ih}
	\\[-.4cm]
	\nonumber
&&-\sum_{h=1}^L \gamma_h \left(E_{M}[g_h](x,v,t)-E_{M_E}[g_h](x,v,t)\right),
\end{eqnarray}
with the optimal vector of values $\Gamma^*$ defined by \eqref{eq:gammah}.
For the estimator \eqref{eq:mscvg2ih} we have the following generalization of the error estimate \eqref{eq:errHMMC2b}.
\begin{proposition}
	Consider a deterministic scheme which satisfies \eqref{eq:det} for problem \eqref{eq:Boltzmann} with random and sufficiently regular initial data $f(z,x,{\w},0)=f_0(z,x,{\w})$.
	Then, the multi-fidelity MSCV estimate defined in \eqref{eq:mscvg2ih} with the optimal values given by \eqref{eq:gammah} satisfies the error bound 
	\bea
	\nonumber
	&&\|\EE[f](\cdot,t^n)-{E}^{\Gamma^*}_{M,M_E}[f^n_{\Delta x,\Delta v}]\|_{\LLBi} \\[-.2cm]
	\label{eq:errHMMC2bb}
	\\
	\nonumber
	&& \hskip 2cm
	\leq  C\left\{\sigma_{f^{\Gamma^*}} M^{-1/2}+\tau_{f^{\Gamma^*}} M_{E}^{-1/2}+\Delta {x}^p+\Delta v^q\right\} 
	\eea
	with 
	\begin{eqnarray*}
	\sigma^2_{f^{\Gamma^*}}&=&\left\|\left(1-\sum_{h=1}^L \rho^2_{f,g_h}\right)\var(f)\right\|_{\LL},\\ 
	\tau^2_{f^{\Gamma^*}}&=&\left\|\sum_{h=1}^L \rho^2_{f,g_h}\var(f)\right\|_{\LL},
	\end{eqnarray*}
	and $C>0$ depends on the final time and on the initial data. 
\end{proposition}

\newpage
\subsection{Hierarchical multi-fidelity MSCV methods}\label{sec:multil}
\subsubsection{The space homogeneous case}
We formulate in this part a recursive construction of the multiple control variate estimator \eqref{eq:mscvg} based on the use of several low-fidelity models. 

To this aim, let us assume that the control variates $f_1,\ldots,f_L$ represent kinetic models with an increasing level of fidelity. Under this assumption the control variate $f_1$ represents the less accurate model whereas the control variate $f_L$ is the closer model to the high fidelity model $f$.

To start with, we estimate $\EE[f]$ with $M_L$ samples using $f_L$ as control variate
\[
\EE[f]\approx E_{M_L}[f]-{\hat \lambda}_L\left(E_{M_L}[f_L]-\EE[f_L]\right).
\]
Next, to estimate of $\EE[f_L]$ we use $M_{L-1}\gg M_L$ samples and consider $f_{L-1}$ as control variate 
\[
\EE[f_L]\approx E_{M_{L-1}}[f_L]-{\hat \lambda}_{L-1}\left(E_{M_{L-1}}[f_{L-1}]-\EE[f_{L-1}]\right).
\]
Similarly, in a recursive way we can construct estimators for the remaining expectations of the control variates $\EE[f_{L-2}],\EE[f_{L-3}],\ldots,\EE[f_2]$ using respectively $M_{L-3}\ll M_{L-4}\ll\ldots \ll M_1$ samples
until 
\[
\EE[f_2]\approx E_{M_{1}}[f_2]-{\hat \lambda}_{1}\left(E_{M_{1}}[f_{1}]-\EE[f_{1}]\right),
\]
and we stop with the final estimate
\[
\EE[f_1]\approx E_{M_0}[f_1],
\]
with $M_0 \gg M_1$. By combining the estimators of each stage together we obtain the hierarchical MSCV estimator
\begin{eqnarray}
	\nonumber
	E_L^{{r,\hat \llambda}}[f] &=& E_{M_L}[f] -{\hat \lambda}_L\left(E_{M_L}[f_L]-E_{M_{L-1}}[f_L]\right.\\
	\nonumber
	&+&{\hat \lambda}_{L-1}\left(E_{M_{L-1}}[f_{L-1}]-E_{M_{L-2}}[f_{L-1}]\right.\\[-.3cm]
	\label{eq:lambdar}
	\\[-.2cm]
	\nonumber
	&\ldots&\\
	\nonumber
	&+&{\hat \lambda}_{1}\left.\left.\left(E_{M_{1}}[f_{1}]-E_{M_{0}}[f_{1}]\right) \ldots\right)\right).
\end{eqnarray}
Now, if we compute the optimal values ${\hat \lambda}_h^*$ independently for each level by ignoring the errors due to the approximations of the various expectations, if $\var(f_{h})\neq 0$, we obtain
\be
{\hat \lambda}_h^* = \frac{\cov(f_{h+1},f_{h})}{\var(f_{h})},\quad h=1,\ldots,L
\label{eq:lambdasr}
\ee
where we used the notation $f_{L+1}=f$. We refer to this set of values as quasi-optimal since they are obtained in the hypothesis in which the expectations are computed without any approximations. Note also that, since the control variates $f_{h+1}$ and $f_h$ are known on the same set of samples $M_h$ the values ${\hat \lambda}^*_h$ can be estimated using \eqref{eq:varm}-\eqref{eq:covm}. 

We now analyze the estimator \eqref{eq:lambdar} in more details. We can recast it in the form 
\be
\begin{split}
E_L^{{\hat \llambda}}[f] &=  E_{M_{L}}[f_{L+1}]-\sum_{h=1}^{L}\lambda_h(E_{M_h}[f_h]-E_{M_{h-1}}[f_h])\\
&=\lambda_1 E_{M_0}[f_1] + \sum_{h=1}^{L} (\lambda_{h+1} E_{M_h}[f_{h+1}]-\lambda_{h} E_{M_{h}}[f_{h}]),
\end{split}
\label{eq:mscvgr}
\ee
where we defined 
\be
\lambda_h=\prod_{j=h}^L \hat\lambda_j,\quad h=1,\ldots,L,\qquad \lambda_{L+1}=1.
\label{eq:lambdah}
\ee
Since by the central limit theorem^^>\cite{Lo77,HH} we have $\var(E_M[f])=M^{-1}\var(f)$, using the independence of the estimators $E_{M_h}[\cdot]$, $h=0,\ldots,L$, the total variance of the estimator \eqref{eq:mscvgr} is 
\[
\begin{split}
\var(E_L^{{\hat \llambda}}[f])&=\lambda_1^2 M_0^{-1} \var(f_1)\\
&+ \sum_{h=1}^{L} M_h^{-1}\left\{\lambda_{h+1}^2\var(f_{h+1})+\lambda_{h}^2\var(f_{h})-2\lambda_{h+1}\lambda_{h}\cov(f_{h+1},f_h)\right\}.
\end{split}
\]
Now, the first order optimality conditions 
\[
\frac{\partial\,\var(E_L^{{\hat \llambda}}[f])}{\partial \lambda_h} = 0,\qquad h=1,\ldots,L
\]
lead to the tridiagonal system for $h=1,\ldots,L$
%\[
%M_{h-1}^{-1}\left\{\lambda_h\var(f_h)-\lambda_{h-1} \cov(f_h,f_{h-1})\right\}+
%M_{h}^{-1}\left\{\lambda_h \var(f_h)-\lambda_{h+1}\cov(f_{h+1},f_{h})\right\}=0,
%\]
%or equivalently 
\be
\begin{split}
\lambda_h \var(f_h) &- \lambda_{h-1} \frac{M_{h}}{M_{h-1}+M_{h}} \cov(f_{h},f_{h-1})\\
& - \lambda_{h+1}\frac{M_{h-1}}{M_{h-1}+M_{h}} \cov(f_{h+1},f_{h}) =0
\label{eq:sys2}
\end{split}
\ee
where we assumed $\lambda_{0}=0$ and $\lambda_{L+1}=1$. 
The system \eqref{eq:sys2} can be rewritten as
%\[
%\lambda_h \var(f_h)  -
%\lambda_{h+1}\cov(f_{h+1},f_{h}) =\frac{M_{h}\left(\lambda_{h-1}  \cov(f_{h},f_{h-1}) + \lambda_{h+1} \cov(f_{h+1},f_{h})\right)}{M_{h-1}+M_{h}}
%\]
%which, reverting to the original control variate variables becomes 
\[
\hat\lambda_h \var(f_h)  -
\cov(f_{h+1},f_{h}) =\frac{M_{h}\left(\hat\lambda_{h}\hat\lambda_{h-1}  \cov(f_{h},f_{h-1}) +  \cov(f_{h+1},f_{h})\right)}{M_{h-1}+M_{h}}.
\]
The above expression permits to conclude that the quasi-optimal values computed in \eqref{eq:lambdasr} solves the above system up to $O({M_{h}}/({M_{h-1}+M_{h}}))$ error and thus they are a sufficiently good approximation in the case in which $M_{h-1}\ll M_{h}$. The following Theorem holds true^^>\cite{DPUQ2}
\begin{theorem}
	The vector $\Lambda^*=(\lambda_1^*,\ldots,\lambda_L^*)^T$ solution of \eqref{eq:sys2} minimizes the variance of the estimator \eqref{eq:mscvgr}. In particular the vector $\hat \Lambda^*$ \eqref{eq:lambdasr} of quasi-optimal solutions is such that 
	\be
	\prod_{j=h}^L \hat\lambda^*_j=\lambda^*_h + O\left(\bar{\mu}_h\right),\quad h=1,\ldots,L
	\label{th:ll}
	\ee
	where $\bar{\mu}_h = \displaystyle\max_{h \leq k \leq L}\left\{\displaystyle\frac{M_{k}}{M_{k-1}+M_{k}}\right\}$.
\end{theorem}
%\begin{proof}
%	We can rewrite \eqref{eq:sys2} in the form $C \Lambda = b$ where $C=\hat C + {\mathcal M} C_0$ with
%	\[
%	{\hat C} = \left(
%	\begin{array}{cccc}
%		c_{11} &  - c_{12}  &  &0\\
%		0 & \ddots   &\ddots   &\\
%		& \ddots   & \ddots  & -c_{L-1 L}\\
%		0  &   &0 & c_{LL}  
%	\end{array}
%	\right),\qquad C_0=\left(
%	\begin{array}{cccc}
%		0 &  c_{12}  &  &0\\
%		-c_{21} & \ddots   &\ddots   &\\
%		& \ddots   & \ddots  & c_{L-1 L}\\
%		0  &   &-c_{L L-1} & 0  
%	\end{array}
%	\right),
%	\]
%	$c_{ij}=\cov(f_{i},f_{j})$, ${\mathcal M}={\rm diag}\{\mu_1,\ldots,\mu_L\}$, $\mu_h={M_{h}}/({M_{h-1}+M_{h}})$ and $b = (I-{\mathcal M}){\hat b}$, ${\hat b}=\left(0,\ldots,0,\cov(f,f_L)\right)^T$. By construction the vector $\Lambda^*$, such that $(\hat C + {\mathcal M} C_0)\Lambda^* = (I-{\mathcal M}){\hat b}$, minimizes the variance of \eqref{eq:mscvgr}.
%	
%	Let us define the vector $\tilde\Lambda^*$ of elements $\prod_{j=h}^L \hat\lambda^*_j$, $h=1,\ldots,L$ where $\hat\lambda^*_j$ are given by \eqref{eq:lambdasr}. We have $\hat C \tilde \Lambda^* = \hat b$ 
%	so that 
%	\[
%	\hat{C}(\Lambda^*-{\tilde \Lambda^*})=-{\mathcal M}(\hat b + C_0 \Lambda^*)
%	\]
%	therefore if $\var(f_h)\neq 0$, $h=1,\ldots,L$ we can write
%	\[
%	{\tilde \Lambda^*}=\Lambda^*+{\hat C}^{-1}{\mathcal M}(\hat b + C_0 \Lambda^*).
%	\]
%	Now, since ${\hat C}^{-1}$ is upper triangular and ${\mathcal M}$ diagonal we have  
%	\eqref{th:ll}. 
%\end{proof}

We summarized in Algorithm \ref{alg:3} the details of the hierarchical MSCV method applied to the space homogeneous problem \eqref{eq:BH} in combination with a deterministic solver.

\begin{algorithm2e}[!htbp]
	\begin{enumerate}
		\item {\bf Sampling}: 
		For each control variate $f_h$, we draw a number
		$M_h$ of i.i.d. samples from the random initial data $f_0$ and approximate these over the mesh $\Delta {\w}$. 
		Denote these control variate dependent number of samples for $h=1,\ldots,L$ by
		$$f_{h,\Delta v}^{k,0},\quad k=1,\ldots,M_h$$ and set $f_{\Delta v}^{k,0}=f_{L,\Delta v}^{k,0},\quad k=1,\ldots,M_L$.
		\item {\bf Solving}: 
		\begin{enumerate}
			\item 
			For each control variate and for each realization of the random
			input data $f_{h,\Delta v}^{k,0}$, $k=1,\ldots,M_h$, the resulting model is solved with mesh widths $\Delta v$. We denote the resulting ensemble of deterministic solutions for $h=1,\ldots,L$ at time $t^n$ by
			\[
			f_{h,\Delta v}^{k,n},\quad k=1,\ldots,M_h.
			\]
			\item For each realization $f_{\Delta v}^{k,0}$, $k=1,\ldots,M_L$ 
			the underlying kinetic equation (\ref{eq:Boltzmann}) is solved with mesh widths $\Delta v$. We denote the solution at time $t^n$ by $f^{k,n}_{\Delta {\w}}$, $k=1,\ldots,M_L$. 
		\end{enumerate}
		\item {\bf Estimating}: 
		\begin{enumerate}
			\item
			Estimate the quasi-optimal vector of values $\hat\llambda^*$ as
			\be
			\hat\lambda^{*,n}_{h}=\frac{\cov_{M_h}(f_{h+1,\Delta v}^{n},f_{h,\Delta v}^{n})}{\var_{M_h}(f_{h,\Delta v}^{n})},\quad h=1,\ldots,L
			\label{eq:hlambdam}
			\ee
			where we used the notation $f_{L+1,\Delta v}^{k,n}=f_{\Delta v}^{k,n}$, $k=1,\ldots,M_L$.
			\item
			Compute the expectation of the random solution with the control variate estimator
			\be
			E^{\hat\llambda^*}_L[f^n_{\Delta v}] = \frac1{M_L} \sum_{k=1}^{M_L} f^{k,n}_{\Delta {\w}}-\sum_{h=1}^L \lambda^{*,n}_{h} \left(\frac1{M_h} \sum_{k=1}^{M_h} {f}_{h,\Delta {\w}}^{k,n}-{\bf f}_{h,\Delta v}^n\right),
			\label{mcestr2}
			\ee
			where
			\[
			{\bf f}_{h,\Delta v}^n=\frac1{M_{h-1}} \sum_{k=1}^{M_{h-1}} {f}_{h,\Delta {\w}}^{k,n},\qquad \lambda^{*,n}_h=\prod_{j=h}^L \hat\lambda^{*,n}_j,\quad h=1,\ldots,L. 
			\]
		\end{enumerate}
	\end{enumerate}
	\caption{Hierarchical MSCV method}
	\label{alg:3}
\end{algorithm2e}
Regarding the error bound that we obtain using \eqref{eq:mscvgr} with the values given by \eqref{eq:lambdasr} le us observe that if, at each stage, we denote
\[
{E}^{\hat\lambda_h}_{M_h}[f_h] = E_{M_h}[f_h]-{\hat\lambda_h}\left(E_{M_h}[f_{h-1}]-E_{M_{h-1}}[f^n_{h-1}]\right)
\]
then by the error bound \eqref{eq:errHMMC2} we have
\[
\|\EE[f_h](\cdot,t)-{E}^{\hat\lambda_h^*}_{M_h}[f_{h}](\cdot,t)\|_{{\LHBi}} \leq  C_h\left\{\sigma_h
M_h^{-1/2}+\tau_h M^{-1/2}_{h-1}\right\} 
\]
where $C_h>0$ is a suitable constant and we defined
\be
\sigma_h=\left\|\left(1-\rho^2_{f_h,f_{h-1}}\right)^{1/2}\var(f_h)^{1/2}\right\|_{\LH},\ee
\be \tau_h=\|\rho_{f_h,f_{h-1}}\var(f_h)^{1/2}\|_{\LH}.
\ee
Thus one can prove the following result 
\begin{proposition}
	Consider a deterministic scheme which satisfies \eqref{eq:det} for the solution of \eqref{eq:BH} with random sufficiently regular initial data $f({z},{\w},0)=f_0({z},{\w})$.	
	Then, the hierarchical MSCV estimate defined in \eqref{eq:lambdar} satisfies the error bound 
	\bea
	\nonumber
	&&\|\EE[f](\cdot,t^n)-{E}^{\hat\Lambda_h^*}_{L}[f^n_{\Delta v}]\|_{{\LHBi}} \\[-.2cm]
	\label{eq:errREC2}
	\\[-.2cm]
	\nonumber
	&& \hskip 4cm \leq C \left(\sum_{h=1}^L \xi_h \sigma_h M_h^{-1/2}+\xi_0 M_0^{-1/2}+\Delta v^{q_2} \right)
	\eea
	where $\xi_h=\prod_{j=h+1}^L \tau_j$, and
	$C>0$ depends on the final time and on the initial data. 
\end{proposition}

%Using the multi-level estimator, in combination with a deterministic solver satisfying \eqref{eq:det}, we can write
%\bea
%\nonumber
%&& \|\EE[f](\cdot,t^n)-{E}^{\hat\Lambda^*}_{L}[f^n_{\Delta v}]\|_{{\LHBi}} \\
%\nonumber
%&& \hskip 4cm \leq  
%\|\EE[f](\cdot,t^n)-{E}^{\hat\Lambda^*}_{L}[f](\cdot,t^n)\|_{{\LHBi}}\\
%\nonumber
%&& \hskip 4cm +
%\|{E}^{\hat\Lambda^*}_{L}[f](\cdot,t^n)-{E}^{\hat\Lambda^*}_{L}[f^n_{\Delta v}]\|_{{\LHBi}}.
%\eea 
%The second term is bounded as usual by the discretization error of the scheme, whereas, ignoring the statistical errors in estimating the quasi-optimal vector of values $\hat\Lambda^*$, the first term can be estimated recursively as
%\bea
%\nonumber
%&&\|\EE[f](\cdot,t^n)-{E}^{\hat\Lambda^*}_{L}[f](\cdot,t^n)\|_{{\LHBi}}\\
%\nonumber
%&& \hskip 2cm \leq C_L 
%\left\{\sigma_L M_L^{-1/2}\right.\\
%\nonumber
%&& \hskip 2cm +
%|\|\hat\lambda_L^*(\EE[f_{L-1}](\cdot,t^n)-{E}^{\hat\Lambda^*}_{L-1}[f_{L-1}](\cdot,t^n))\|_{{\LHBi}}\bigg\}
%\\
%\label{eq:errREC}
%&& \hskip 2cm \leq C_{L}
%\left\{\sigma_L M_L^{-1/2}+\tau_L C_{L-1}\left\{\sigma_{L-1} M_{L-1}^{-1/2}\right.\right.\\
%\nonumber
%&& \hskip 2cm +  \|\hat\lambda_{L-1}^*(\EE[f_{L-2}](\cdot,t^n)-{E}^{\hat\Lambda^*}_{L-2}[f_{L-2}](\cdot,t^n))\|_{{\LHBi}}\bigg\}
%\\
%\nonumber
%&&\hskip 2cm \ldots\\
%\nonumber
%&&\hskip 2cm \leq C \left(\sum_{h=1}^L \xi_h \sigma_h M_h^{-1/2}+\xi_0 M_0^{-1/2}\right)
%\eea
%where we defined
%\be
%\xi_h = \prod_{j=h+1}^L \tau_j.
%\label{eq:muh}
%\ee

\subsubsection{The space non homogeneous case.}
In a space non homogeneous setting the hierarchical multi-fidelity MSCV estimator \eqref{eq:mscvgr}, based on the control variates $f_1(z,x,v,t)$, $\ldots$, $f_L(z,x,v,t)$ with increasing level of fidelity for the solution $f(z,x,v,t)$ to \eqref{eq:Boltzmann}, is
\be
\begin{split}
E^{\Lambda}_L[f](x,v,t) &= E_{M_L}[f](x,v,t)\\
&-\sum_{h=1}^L \lambda_h \left(E_{M_h}[f_h](x,v,t)-E_{M_{h-1}}[f_h](x,v,t)\right),
\label{eq:mscvgrh}
\end{split}
\ee
with $M_{h-1} \gg M_h$ and where now the optimal values of $\Lambda^* = (\lambda_1^*,\ldots,\lambda_L^*)^T$ are obtained from the quasi-optimal solution \eqref{eq:lambdasr} using \eqref{eq:lambdah} or by the correction introduced by the solution of the tridiagonal system \eqref{eq:sys2} if relevant. In this case, the extension of algorithms \ref{alg:3} and estimate \eqref{eq:errREC2} to the non homogeneous case follows straightforwardly simply replacing $f^{n,k}_{\Delta v}$ and $f^{n,k}_{h,\Delta v}$ with $f^{n,k}_{\Delta x,\Delta v}$ and $f^{n,k}_{h,\Delta x,\Delta v}$, and is omitted for brevity. 

%\subsubsection{A leading example: Euler and BGK control variates}

For sake of clarity and due to its importance in practical applications, we describe the details of the hierarchical  method in the case $L=2$. We consider $f_1(z,x,v,t)$ as the equilibrium state $f_F^\infty(z,x,v,t)$ associated to the system of Euler equations \eqref{eq:Euler}
%\be
%\partial_t U_F +\partial_x {\mathcal F}(U_F) = 0, 
%\label{eq:Euler2}
%\ee
with $U_F=(\rho_F,u_F,T_F)^T$, and corresponding to the limit case $\varepsilon\to 0$ in \eqref{eq:Boltzmann}. As a second control variate we consider $f_2(z,x,v,t)$ as  the solution of the BGK model \eqref{eq:pbgk}.
%\be
%\frac{\partial}{\partial t}{f_2} + {\w} \cdot \nabla_x {f_2} = \frac{\nu}{\varepsilon} (f_2^\infty-{f_2}).
%\label{eq:BGK2}
%\ee
Both models are solved for the same initial data $f_0(z,x,v)$. Now, the Euler equations are used as control variate to improve the computation of the expectation in the BGK model, that in turn is used as control variate to improve the computation of the expectation in the full Boltzmann model. 

\begin{figure}
	\begin{center}
		\includegraphics[width=.43\textwidth]{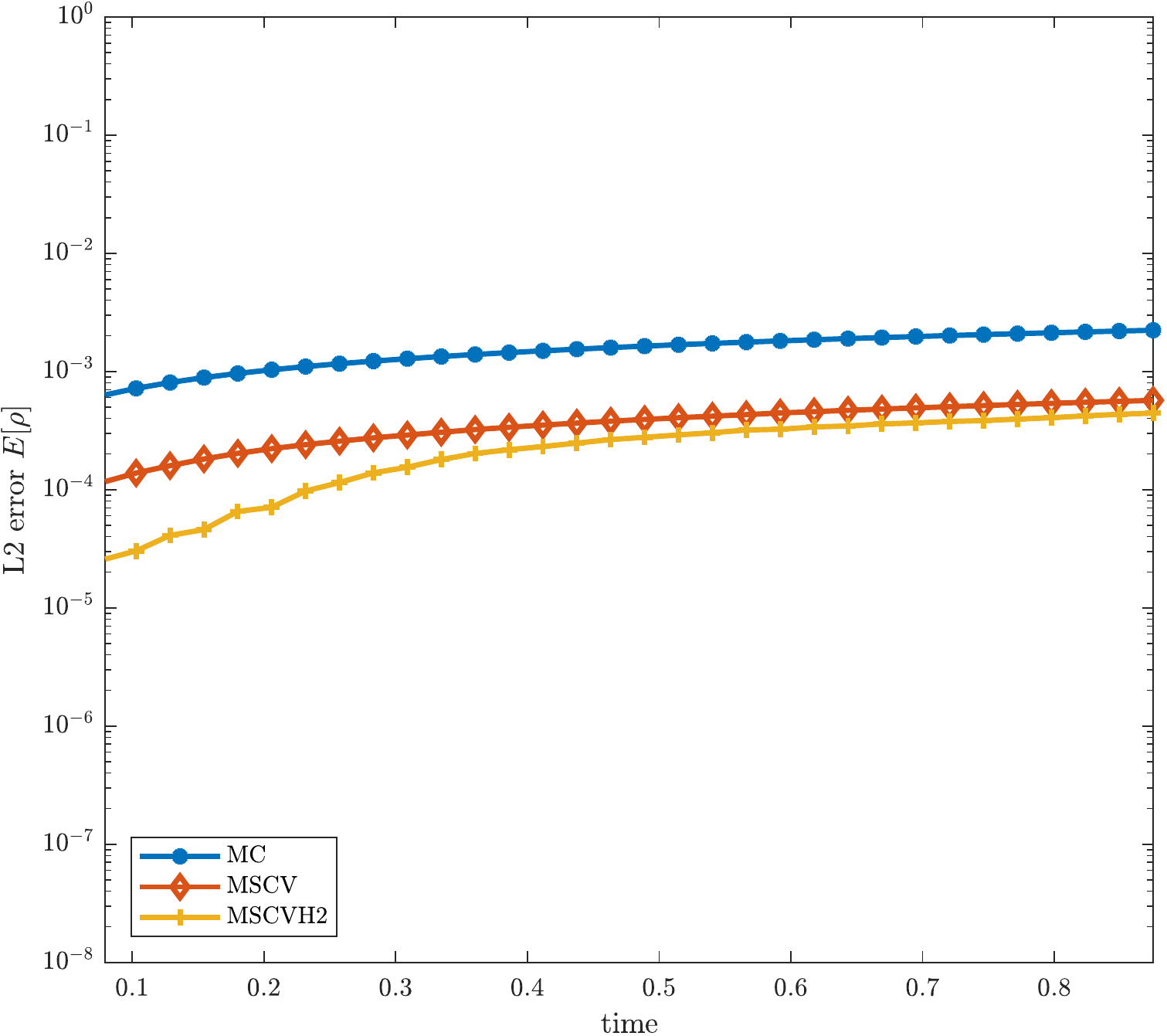}
		\includegraphics[width=.43\textwidth]{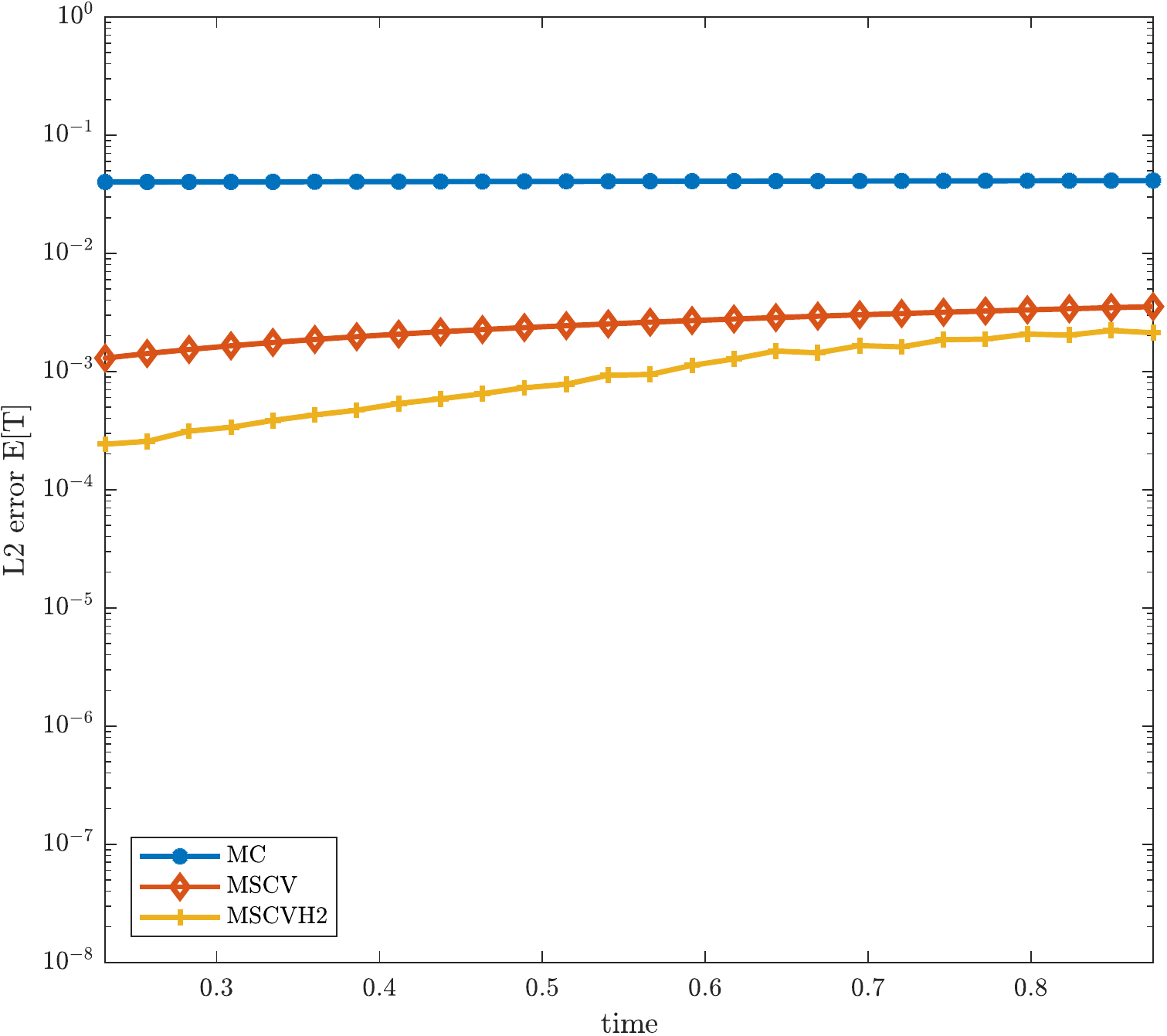}\\
		\includegraphics[width=.43\textwidth]{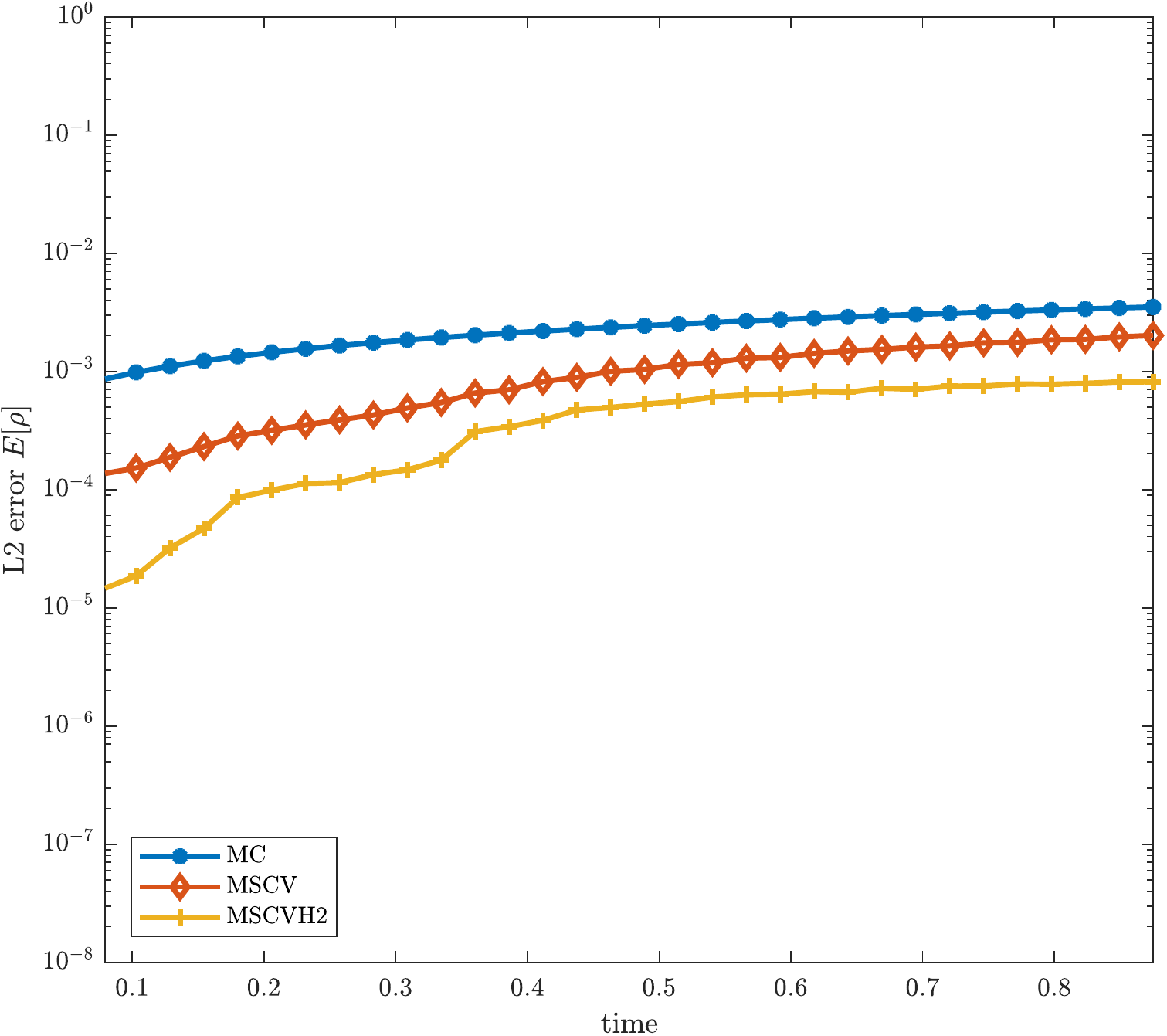}
		\includegraphics[width=.43\textwidth]{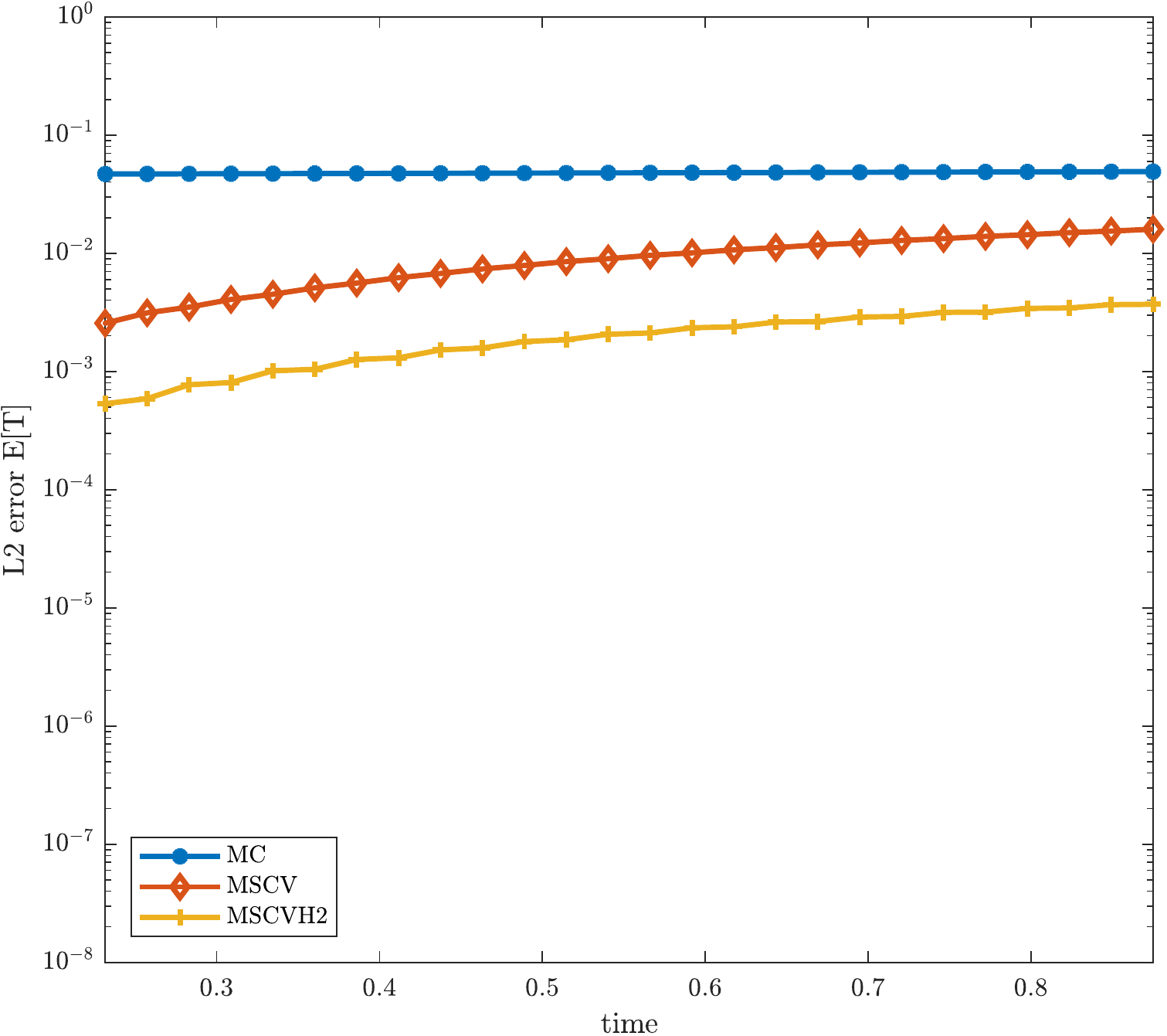}\\
		\includegraphics[width=.43\textwidth]{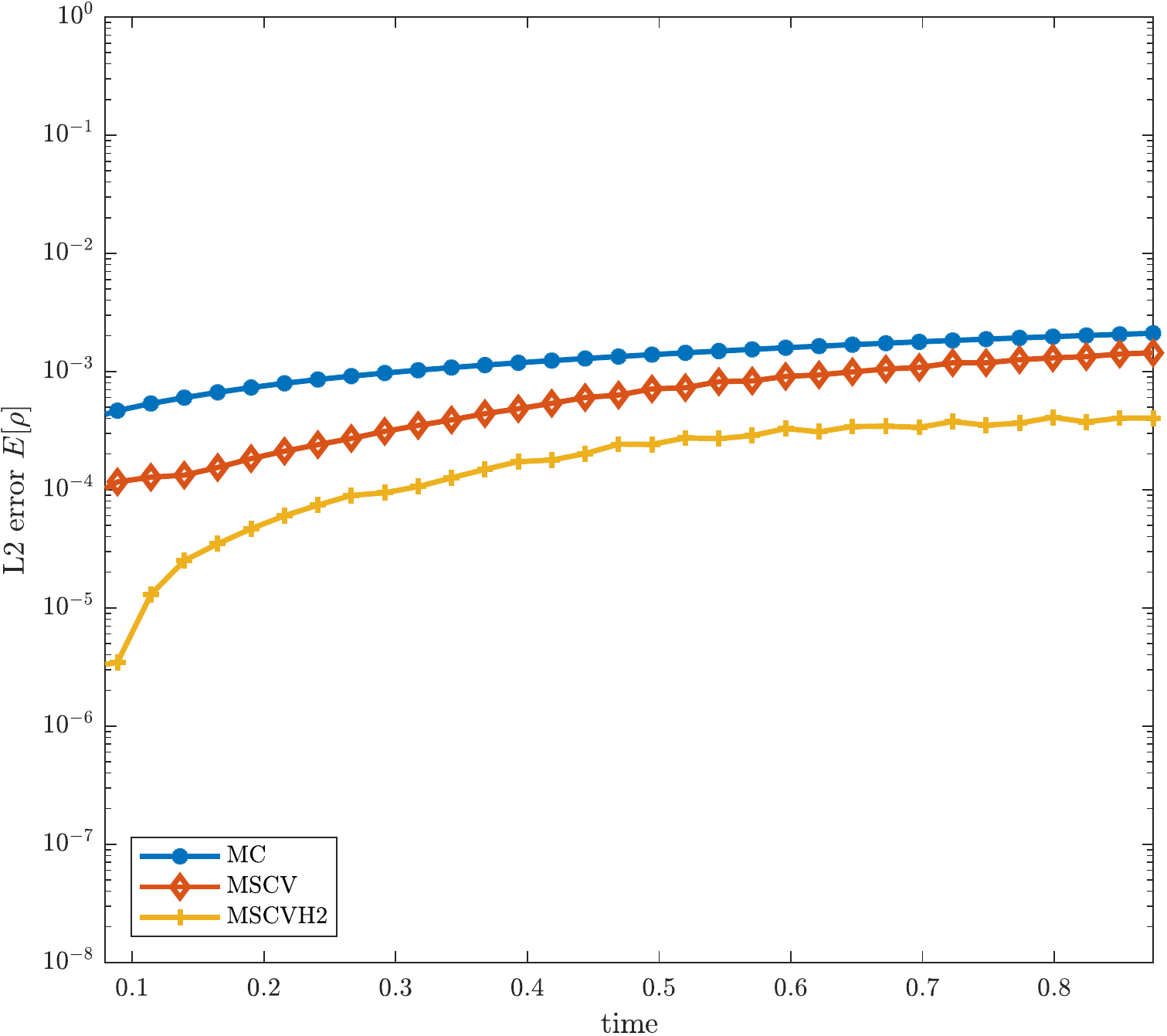}
		\includegraphics[width=.43\textwidth]{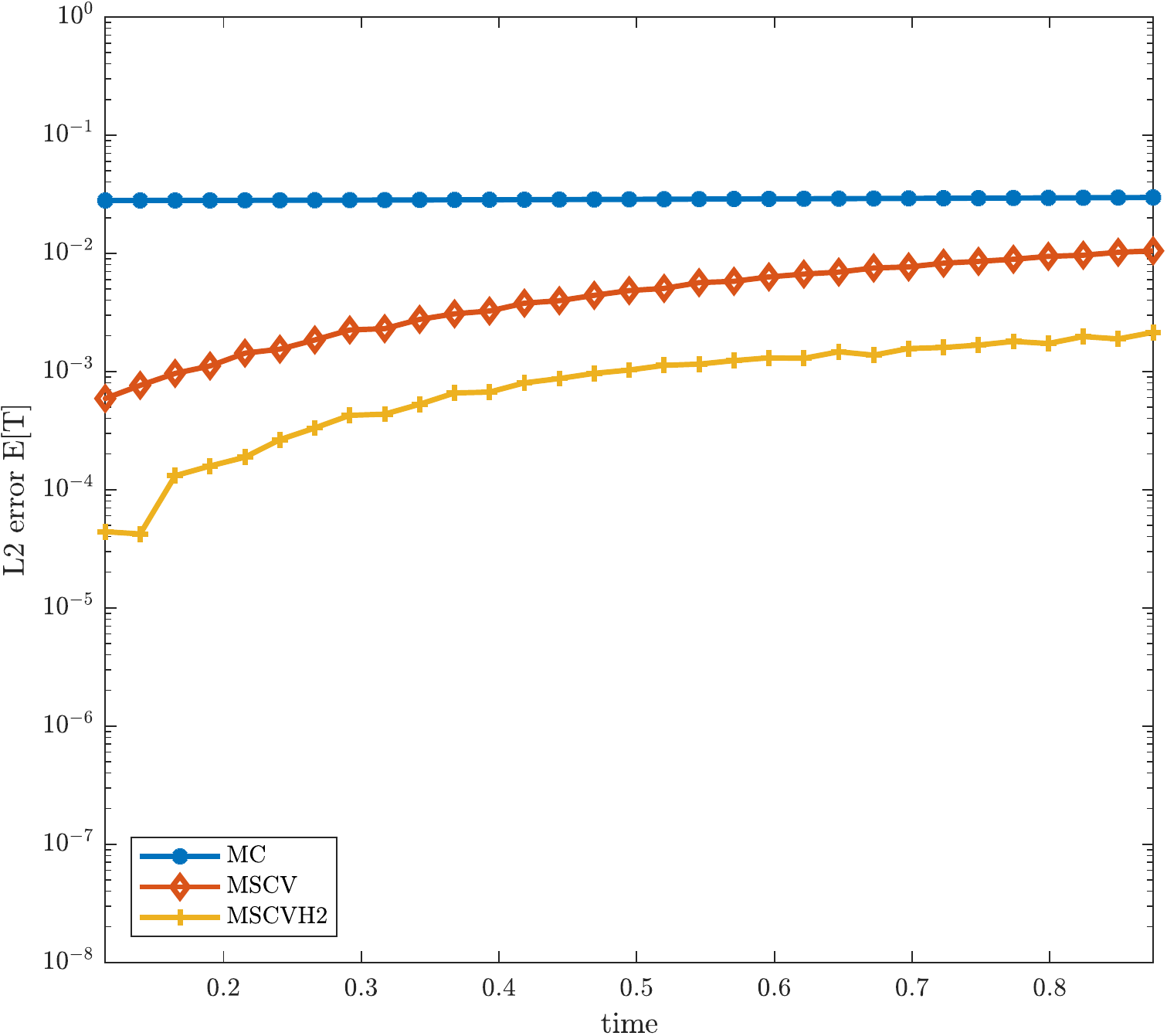}
		\caption{Hierarchical MSCV method - Sod test with uncertainty in the initial data. $L_2$ norm of the error for the MC method, the MSCV method and the MSCVH2 method for $\EE[\rho]$ (left) and $\EE[T]$ (right). The number of samples is $M=10$, for the BGK model $M_{E_1}=100$ and for the Euler system $M_{E_2}=10^5$. From top to bottom: $\varepsilon=10^{-2}, 10^{-3}, 2 \times \ 10^{-4}$. }\label{Figure4}
	\end{center}
\end{figure}

The hierarchical two-level estimator reads
\be
\begin{split}
	E_2^{{\hat\lambda_1,\hat\lambda_2}}[f] &= E_{M_2}[f]\\ 
	&-{\hat \lambda}_2\left(E_{M_2}[f_2]-E_{M_{1}}[f_2]
	+{\hat \lambda}_{1}\left(E_{M_{1}}[f_{1}]-E_{M_{0}}[f_{1}]\right)\right),
\end{split}
\ee
where $M_0 \gg M_1 \gg M_2$. If we define $\lambda_2=\hat\lambda_2$ and $\lambda_1=\hat\lambda_1\hat\lambda_2$  their optimal values are computed as solutions of system \eqref{eq:sys2} for $L=2$
\bea
\nonumber
\lambda_1 \var(f_1)  -
\lambda_{2}(1-\mu_1) \cov(f_{2},f_{1}) &=&0\\
\nonumber
\lambda_2 \var(f_2) - \lambda_{1} \mu_2 \cov(f_{2},f_{1}) &=&
(1-\mu_2) \cov(f,f_{2}).
\eea  
with $\mu_h = M_h/(M_{h-1}+M_h)$, which gives
\beas
\lambda_1^* &=& \frac{(1-\mu_1)(1-\mu_2)\cov(f_{2},f_{1}) \cov(f,f_{2})}{\displaystyle\var(f_1)\var(f_2)-(1-\mu_1)\mu_2 \cov(f_{2},f_{1})^2}\\
\lambda_2^* &=& \frac{ (1-\mu_2) \var(f_{1}) \cov(f,f_{2})}{\var(f_1)\var(f_2)-(1-\mu_1)\mu_2\cov(f_{2},f_{1})^2}.  
\eeas
The quasi-optimal values are instead obtained assuming $\mu_1,\mu_2\approx 0$ and are characterized by 
\be
\hat\lambda_1^* = \frac{\cov(f_2,f_1)}{\var(f_1)},\qquad \hat\lambda_2^* = \frac{\cov(f,f_2)}{\var(f_2)}. 
\ee
%In the fluid limit $\varepsilon\to 0$ we have $f,f_1,f_2\to f^\infty$ so that
%\[
%\lim_{\varepsilon\to 0} \lambda_1^* = \frac{(1-\mu_1)(1-\mu_2)}{1-(1-\mu_1)\mu_2}=\frac{{M_{0}}}{M_0+M_{1}+M_{2}},\] \[\lim_{\varepsilon\to 0} \lambda_2^* = \frac{(1-\mu_2)}{1-(1-\mu_1)\mu_2}= \frac{M_0+M_1}{M_0+M_{1}+M_{2}},
%\]
%then
%\[
%\lim_{\varepsilon\to 0} E_2^{{\lambda_1^*,\lambda_2^*}}[f] = \frac{M_0 E_{M_0}[f^\infty]+M_1 E_{M_1}[f^\infty]+M_2E_{M_2}[f^\infty]}{M_0+M_{1}+M_{2}}.
%\]
%On the contrary, the quasi-optimal values are such that
%\[
%\lim_{\varepsilon\to 0} \hat\lambda_1^* = 1,\qquad \lim_{\varepsilon\to 0} \hat\lambda_2^* = 1
%\]
%and therefore
%\[
%\lim_{\varepsilon\to 0} E_2^{{\hat\lambda_1^*,\hat\lambda_2^*}}[f] = E_{M_0}[f^\infty]
%\]
%which corresponds to the equilibrium solution over the finest grid of samples.

In Figure \ref{Figure4} we compare the results of the MSCV approach with two hierarchical levels (MSCHVH2) against the standard MC and the bi-fidelity MSCV method based on the BGK model discussed in Section \ref{sec:MSCVnh}. The initial conditions are given by Sod problem with uncertainty \eqref{eq:sodtest}.
%\bea
%&\rho_0(x)=1, \ \ T_0(z,x)=1+sz \qquad &\textnormal{if} \ \ 0<x<L/2 \\
%&\rho_0(x)=0.125,\ T_0(z,x)=0.8+sz  \qquad &\textnormal{if} \ \ L/2<x<1
%\eea
%with $s=0.25$, $z$ uniform in $[0,1]$ and equilibrium initial distribution
%\[
%f_0(z,x,\w)=\frac{\rho_0(x)}{2\pi } \exp\left({-\frac{|v|^2}{2T_0(z,x)}}\right),
%\] 
%and 
{The Boltzmann equation has been solved with the same discretization parameters as in Figure \ref{Figure2}.} For this situation, we perform three different computations corresponding to $ \varepsilon=10^{-2}$, $ \varepsilon=10^{-3}$ and $ \varepsilon=2 \times 10^{-4}$.
The $L_2$ norms of the errors for the standard MC method, the MSCV approach and the MSCVH2 method for the expected value of the temperature and the density as a function of time are shown in the figure. The number of samples used for the BGK model is $M_{E_1}=100$ while for the compressible Euler system is $M_{E_2}=10^5$. In all regimes the gain of the hierarchical two-level approach is remarkable and improves for smaller values of the Knudsen number.

\subsection{Mean-field control variate methods}
\label{sec:socio} 
In this section we discuss two important issues that have remained open since the previous presentation of the MSCV method. Namely, how to couple the MSCV strategy with particle-based solvers in physical space, and how to extend the range of applicability of the methods to other kinetic equations where an equilibrium state is not known and thus the classical closure of rarefied gas dynamics cannot be adopted.

As a prototype example to develop our arguments, we will consider the Boltzmann model for socio-economic interactions with uncertainty introduced in Section \ref{model_socio}. 

%method to solve the uncertain Boltzmann equation \eqref{ScaledBol}. In this method, both the physical variables as well as the uncertain parameters are solved by Monte Carlo approximations. 
%This represents the starting point for the successive construction of our mean-field control variate strategy. 
%The Monte Carlo methods for kinetic equations naturally employ the microscopic dynamics to satisfy the physical constraints and are much less sensitive to the curse of dimensionality^^>\cite{Caflisch} with respect to deterministic methods.

\subsubsection{The DSMC method for kinetic equations.}
Let us first recall the classical Direct Simulation Monte Carlo (DSMC) method for the solution of the Boltzmann equation^^>\cite{bird1970direct,nanbu1980direct} in case without uncertainty. We are in particular interested in the evolution of the density $f=f(w,t)$ solution of \eqref{ScaledBol} with initial condition $f(0,w)= f_0(w)$. The DSMC method in the form originally proposed by Nanbu^^>\cite{nanbu1980direct} is reported in Algorithm \ref{al:dsmc1}. We refer to^^>\cite{pareschi2001introduction,rjasanow,PT2} for an introduction to Monte Carlo methods for kinetic equations. 

\begin{algorithm2e}[!htbp]
\begin{enumerate}
\item Let us consider a time interval $[0,T]$, and let us discretize it in $n_t$ intervals of size $\Delta t$. 
\item Compute the initial sample particles $ \{w_i^0, i=1,\ldots,N\} $,\\ 
      by sampling them from the initial density $f_0(w)$ 
\item   \begin{tabbing}
\= {\fP for} \= $n=0$  {\fP to} $n_t-1$ \\
          \>          \> given $\{w_i^n,i=1,\ldots,N\}$\\
          \>       \>   \ind \= \cir \= set $N_c = \IR(N\Delta t/2)$ (stochastic rounding) \\
          \>       \>        \> \cir \> select $N_c$ pairs $(i,j)$ uniformly among all possible pairs, \\
          \>       \>        \>      \> - \= perform the collision between $i$ and $j$, and compute \\
          \>       \>        \>      \>      \> $w_i'$ and $w_j'$ according to the  collision law \eqref{micro}\\
          \>       \>        \>      \> - set $w_i^{n+1} = w_i'$, $w_j^{n+1}=w_j'$ \\
          \>       \>        \> \cir \> set $w_i^{n+1}=w_i^n$ for all the particles not selected\\
         \> {\fP end for}
\end{tabbing}
\end{enumerate}
\caption{DSMC method}
\label{al:dsmc1}
\end{algorithm2e}

%Here, by $\IR(x)$ we denote the stochastic rounding of a
%positive real number $x$
%\[
%   \IR(x) = \left\{\begin{array}{lll}
%                     {\intg{x}} + 1 & \mbox{with probability} & x-{\intg{x}} \\
%                     {\intg{x}}     & \mbox{with probability} & 1-x+{\intg{x}}
%                   \end{array}
%            \right.
%\]
%where $\intg{x}$ denotes the integer part of $x$.
The kinetic distribution as well as its moments are then recovered from the empirical density function
\begin{equation}
f_N(t,w)=\frac1{N}\sum_{i=1}^N \delta(w-w_i(t)),
\label{eq:emp}
\end{equation}
where $\delta(\cdot)$ is the the Dirac delta and $\{w_i(t), i = 1, \dots,N\}$ are the samples of particles at time $t\ge 0$. 
For any test function $\varphi$, if we now denote by 
\[
\langle\varphi, f\rangle(t) = \int_V \varphi(w)f(t,w)dw, 
\]
we have
\be
\langle\varphi,f_N\rangle (t) = \dfrac{1}{N} \sum_{i = 1}^N \varphi(w_i(t)). 
\label{eq:momp}
\ee
Hence, by assuming that $\int_V f(t,w)dw= 1$ we have that $\langle\varphi,f\rangle=\mathbb E_V[\varphi]$, where $\mathbb E_V[\cdot]$ is the expectation of the observable quantity $\varphi$ with respect to the density $f$ to be distinguished from $\mathbb E[\cdot]$ used to denote the expectation in the random space of uncertainties. 
Thanks to the central limit theorem we have^^>\cite{Caflisch}
\begin{lemma}\label{lem:1}
The root mean square error is such that for each $t\ge 0$
\begin{equation}
\label{eq:th_e1}
\mathbb E_{V}\left[\left( \langle\varphi,f\rangle - \langle\varphi,f_N\rangle \right)^2\right]^{1/2}= \dfrac{\sigma_\varphi}{N^{1/2}}, 
\end{equation}
where $\sigma^2_\varphi = \VV_V[\varphi]$ with
\begin{equation}
\label{eq:sigma_varphi}
\VV_V[\varphi](t) = \int_V( \varphi(w)- \langle\varphi,f\rangle(t))^2 f(t,w)dw.
\end{equation}
\end{lemma}
While the moments of the distribution can be easily computed via \eqref{eq:momp} if one is interested in the shape of the distribution function one must do an appropriate reconstruction. For example, one can operate as follows: set up a uniform grid in $V\subseteq \mathbb R$ where each cell has width $\Delta w>0$ and subsequently define a smoothing function $S_{\Delta w}\ge 0$ such that 
\[
\Delta w \int_{V} S_{\Delta w}(w)dw = 1.
\]
Then, the approximation of the empirical density \eqref{eq:emp} is obtained by
\begin{equation}
\label{eq:f_RN}
f_{N,\Delta w}(t,w) = \dfrac{1}{N} \sum_{i = 1}^N S_{\Delta w}(w-w_i(t)). 
\end{equation}
In the simplest case, $S_{\Delta w}(w)=\chi(|w|\leq \Delta w/2)/\Delta w$, where $\chi(\cdot)$ is the indicator function, \eqref{eq:f_RN} corresponds to the standard histogram reconstruction. Then, the numerical error of the reconstructed DSMC solution \eqref{eq:f_RN}, can be estimated 
from   
\[
	\|g\|_{L^p(V,L^2(V))}=\|\mathbb E_{V}\left[g^2\right]^{1/2}\|_{L^p(V)},
\]
as^^>\cite{PTZ}
\begin{theorem} The error introduced by the reconstruction function \eqref{eq:f_RN} satisfies
	\begin{equation}
		\left\| f(t,\cdot)-f_{N,\Delta w}(t,\cdot)\right\|_{L^p(V,L^2(V))} \leq \frac{\|\sigma_S\|_{L^p(V)}}{N^{1/2}} + C_f (\Delta w)^q,
	\end{equation}
	where $C_f$ depends on the $q$ derivative in velocity of $f$ and $\sigma_{S}^2$ is given by
	\begin{equation}
	\sigma^2_S(w,t)=\VV_V[S_{\Delta w}(w-\cdot)](t).
	\label{eq:sigmas}
	\end{equation} 
	\label{th:1b}
\end{theorem}

\subsubsection{The combined MC-DSMC method for uncertainty quantification.}\label{sect:QoI}
%In order to analyze the uncertainty of our model \eqref{SymBoltzmannModel}, let notice that one may not be interested 
%%in the expected value and variance of the distribution function $f$ but rather 
%in any quantity of interest (QoI) which can be computed from the distribution function. For that purpose we introduce the operator $\g[f](t,w,z)$ which is, in the simplest setting the identity, i.e. $\g[f] = f$. More generally, the QoI defined by $\g[f]$ is a functional of $f$, like, for example, the moments of the kinetic solution: mass, momentum, energy, etc.. . 
%
%Let $p(z)$ be the probability distribution function of the uncertainty $z \in \Omega$, hence we define as in Section 2 the expected value of the operator $\g[f](t,w,z)$ by
%\begin{equation}
%\mathbb{E}[\g[f]] (t,w)=\int_{\Omega} \g[f](t,w,z) p(z)dz,
%\end{equation}
%whose variance is defined as
%\begin{equation}
%\VV[\g[f]] (t,w)=\int_{\Omega} (\g[f](t,w,z)-\mathbb{E}[\g[f]](t,w))^2 p(z)dz. 
%\end{equation}

%In socio-economic applications, together with the moments of the distribution $f$ that are linked with observable quantities, we can define other operators characterizing the main features of the emerging distribution. A leading example is given by the tail distribution
%\begin{align*}
%1-F(t,w,z):= \int_w^{\infty} f(t,v,z)dv.
%\end{align*}
%Other examples of common interest are the Lorenz curve and the Gini coefficient^^>\cite{PT2,PTZ,during2018kinetic}. These quantities are frequently used measures in order to study the inequality of the wealth distribution. 

%We now give some details about the MC-DSMC method. 
Let assume $f(t,w,z)$, $w \in V$, solution of a PDE with uncertainties only in the initial distribution $f_0(w,z)$, $z \in \Omega \subseteq \mathbb R^{d_z}$. The MC sampling method for the uncertainty quantification has been formulated in Section \ref{sec:sMC}. The only difference with respect to \eqref{mcest} is the way in which the deterministic kinetic equation is solved since here a DSMC method is used. 

The empirical kinetic distribution in presence of uncertainty is given by 
\[
f_N(t,w,z) = \dfrac{1}{N} \sum_{i=1}^N \delta(w-w_i(t,z)), 
\]
being $\{w_i(t,z), i = 1,\dots,N\}$ the samples of the particles at time $t\ge 0$ such that $w_i\in L^2(\Omega)$. 
%For example, if $q[f]=\langle\varphi,f\rangle$ we have
%\[
%q[f_N](z,t)=\langle\varphi,f_N\rangle(z,t) = \dfrac{1}{N} \sum_{i=1}^N \varphi(w_i(z,t)).
%\]

While the algorithm is similar to the deterministic case, the error analysis is different and the following result holds true^^>\cite{PTZ}
\begin{lemma}
The root mean square error of the MC-DSMC method satisfies 
\[
\mathbb E\left[\mathbb E_{V}[ \left( \mathbb E[\langle\varphi,f\rangle]- E_M[\langle\varphi,f_N\rangle] \right)^2 ]\right]^{1/2}\le \frac{\nu_{\langle\varphi,f\rangle}} {M^{1/2}} + \frac{\sigma_{\varphi,M}} {N^{1/2}}, 
\]
where $\nu^2_{\langle\varphi,f\rangle} = \VV[\langle\varphi,f\rangle]$
and $\sigma^2_{\varphi,M} = E_M[\sigma^2_{\varphi}]$ with $\sigma_\varphi^2=\VV_V[\varphi]$.
\label{lem:2}
\end{lemma}
%\begin{proof}
%The above estimate follows from 
%\[
%\begin{split}
%\mathbb E\left[\mathbb E_{V}[ \left( \mathbb E[\langle\varphi,f\rangle]- E_M[\langle\varphi,f_N\rangle] \right)^2 ]\right]^{1/2}   \le & \,\mathbb E\Big[(\mathbb E[\langle\varphi,f\rangle]  - E_M[\langle\varphi,f\rangle])^2 \Big]^{1/2}  \\
% & + \mathbb E_{V}\Big[ \left( E_M[\langle\varphi,f\rangle]  - E_M[\langle\varphi,f_N\rangle] \right)^2\Big]^{1/2}  \\
%  \leq   & \,\dfrac{\nu_{\langle\phi,f\rangle}}{M^{1/2}}+ \dfrac{\sigma_{\varphi,M}}{N^{1/2}}, 
%\end{split}
%\]
%where we used the fact that by Cauchy-Schwartz and Lemma \ref{lem:1} we have
%\[
%\begin{split}
%\mathbb E_{V}\Big[ \left( E_M[\langle\varphi,f\rangle]  - E_M[\langle\varphi,f_N\rangle] \right)^2\Big]& \leq  \frac1{M}\sum_{k=1}^M \mathbb E_V[(\langle\varphi,f\rangle(z_k)-\langle\varphi,f_N\rangle(z_k))^2]\\
%& = \frac1{N}\left(\frac1{M}\sum_{k=1}^M \sigma^2_{\varphi}(z_k)\right).
%\end{split}
%\]
%\end{proof}

Let us now consider the reconstructed distribution with uncertainty
\begin{equation}
\label{eq:f_RNz}
f_{N,\Delta w}(t,w,z) = \dfrac{1}{N} \sum_{i = 1}^N S_{\Delta w}(w-w_i(t,z)), 
\end{equation}
and let us focus on the accuracy of the expectation of the solution $\mathbb E[f]$. One can give, 
using 
\[
\|g \|_{L^p(V,L^2(\Omega,L^2(V)))} = \| \mathbb E[\mathbb E_V[g^2]]^{1/2} \|_{L^p(V)},
\] 
the following estimate^^>\cite{PTZ}
\begin{theorem}
	The error introduced by the reconstruction function \eqref{eq:f_RNz} in the MC-DSMC method satisfies
	\begin{equation}
		\label{eq:estim_2}
		\begin{split}
			\left\| \mathbb E[f](t,\cdot)\right.&\left.-E_M[f_{N,\Delta w}](t,\cdot)\right\|_{L^p(V,L^2(\Omega,L^2(V)))}\\
			&\leq \frac{\|\nu_{\langle S,f\rangle}\|_{L^p(V)}}{M^{1/2}}+\frac{\|\sigma_{S,M}\|_{L^p(V)}}{N^{1/2}} + C_{\mathbb E[f]} (\Delta w)^q
		\end{split}
	\end{equation}
	where $\nu_{\langle S,f\rangle}^2$ is defined as
	\begin{equation}
		\nu^2_{(S,f)} = \VV[(S_{\Delta w}(w-\cdot),f)].
		\label{eq:nus}
	\end{equation} and $\sigma^2_{S,M} = E_M[\sigma^2_{S}]$ with  $\sigma_S^2$ defined in \eqref{eq:sigmas}.
	\label{th:2}
\end{theorem}

\subsubsection{Mean Field Control Variate DSMC methods}\label{MFCV}
In order to improve the accuracy of standard MC sampling methods, we introduce a class of mean field control variate methods playing the role of the low-fidelity model. The key idea is to take advantage of the reduced cost of the mean field model which approximates the asymptotic behavior of the original Boltzmann model. More precisely we consider two different control variates strategies obtained by the mean field approximation: the direct numerical solution of the mean field model and its corresponding steady state. In the sequel most of the analysis is reported for a general quantity of interest $q[f]$. 

%The strong form of the kinetic model \eqref{ScaledBol} is given by
%\begin{equation}
%\label{StrongBoltzmann}
%\partial_t f_\epsilon(t,w,z) = Q_\epsilon(f_\epsilon,f_\epsilon)(t,w,z),
%\end{equation}
%complemented by the initial distribution $f_\epsilon(0,w,z)=f_0(w,z)$. As shown in Section \ref{model_socio}, the mean field approximation of  the integro-differential equation \eqref{StrongBoltzmann} is a nonlocal  Fokker-Planck equation. 
Let us denote by $\h= \h(t, w, z)$ the solution of the mean field model \eqref{FP}
complemented by the same initial distribution $\h(0, w, z)= f_0(w,z)$ of the high fidelity model.  
As discussed in Section \ref{model_socio} for small values of the scaling parameter $\epsilon$ we have 
$$
\lim\limits_{\epsilon\to 0} f_\epsilon(t,w,z) = \h(t,w,z),
$$
being $f_\epsilon$ solution of the high-fidelity Boltzmann model \eqref{ScaledBol}. Therefore, also the equilibrium distribution is such that
$$
\lim\limits_{t\to\infty} \lim\limits_{\epsilon\to 0} f_\epsilon(t,w,z) = \h^{\infty}(w,z). 
$$
%As we observed, the steady state $\h_{\infty}(w,z)$ of the Fokker-Planck model are analytically computable in several cases, whereas the steady state of the Boltzmann model \eqref{StrongBoltzmann} is unknown. 
In this setting, the parameter dependent control variate method with $\lambda\in\R$ can be formulated introducing the quantity 
 \begin{align}
  \g^{\lambda}[f_\epsilon]= \g[f_\epsilon]- \lambda (\g[\h]-\mathbb{E}[\g[\h]]).\label{CV}
\end{align}     
%It is straightforward to observe that the expected value satisfies
%$$
%\mathbb{E}[\g^{\lambda}[f_\epsilon]]= \mathbb{E}[\g[\h]].
%$$     
By the same arguments as in Section \ref{sec:MSCV} we can state also the following
\begin{theorem}
The optimal value $\lambda^*$ which minimizes the variance of \eqref{CV} is given by
\begin{align}
\lambda^*:= \frac{\CC[\g[f_\epsilon],\g[\h]]}{\VV[\g[{\h}]]},
\label{eq:olambda}
\end{align}
where $\CC[\cdot,\cdot]$ denotes the covariance.
The corresponding variance of $q^{\lambda^*}[{f_\epsilon}]$ is then 
\begin{equation}\label{VarClev}
\VV[q^{\lambda^*}[f_\epsilon]]= \left(1-\rho^2_{\g[f_\epsilon], \g[\h]}\right)\ \VV[\g[f_\epsilon]],
\end{equation}
where
$$
\rho_{\g[f_\epsilon], \g[\h]}:=\frac{\CC[\g[f_\epsilon],\g[\h]]}{\sqrt{\VV[\g[f_\epsilon]]\  \VV[\g[\h]]}} \in (-1,1), 
$$
is the correlation coefficient between $\g[f_\epsilon]$ and $\g[\h]$. 
In particular, we have 
$$
\lim\limits_{\epsilon \to 0}\frac{\CC[\g[f_\epsilon],\g[\h]]}{\VV[\g[\h]]}= 1,\quad \lim\limits_{\epsilon \to 0}\VV[\g^{\lambda^*}[f_\epsilon]] = 0.
$$
\label{th:lambda}
\end{theorem}
%\begin{proof}
%We aim to choose the coefficient $\lambda$ such that the variance $\VV[\g^{\lambda}[f_\epsilon]]$ of the control variate formulation is minimized. We have
%\begin{align}\label{VarRed}
%\VV[\g^{\lambda}[f_\epsilon]]=\VV[\g[f_\epsilon]]-2\ \lambda\ \CC[\g[f_\epsilon], \g[\h]]+ \lambda^2\ \VV[\g[{\h}]].
%\end{align}
%Then we can differentiate \eqref{VarRed} with respect to $\lambda$ 
%to obtain that \eqref{eq:olambda} minimizes the variance. The second part of the theorem follows immediately since $\lim\limits_{\epsilon \to 0} f_\epsilon = \h$ holds in the quasi-invariant regime for $f_\epsilon$ solution to \eqref{StrongBoltzmann}.
%\end{proof}
%Equation \eqref{VarClev} reveals that the mean field control variate approach may lead to a strong variance reduction provided that the correlation coefficient is close to one. 
Of course, the control variate formulation in \eqref{CV} can be modified using the steady state $\h_{\infty}(w,z)$ of the mean-field model \eqref{FP} 
 \begin{equation}
  \g^{\lambda}[f_\epsilon]= \g[f_\epsilon]- \lambda (\g[\h_{\infty}]-\mathbb{E}[\g[\h_{\infty}]]).\label{CVSteady}
\end{equation}  
Then, for the mean field control variate steady state \eqref{CVSteady}, a similar results holds in the large time limit, as the ones shown in Theorem \ref{th:lambda}. 
%Thus
%$$
%\lim\limits_{t \to \infty}\lim\limits_{\epsilon \to 0} \frac{\CC[\g[{f}],\g[{\h}_{\infty}]]}{\VV[\g[{\h}_{\infty}]]}= 1,\quad \lim\limits_{t \to \infty}\lim\limits_{\epsilon \to 0}\VV[\g^{\lambda^*}[{f}]] = 0,
%$$
%where now
%\[
%\lambda^*:= \frac{\CC[\g[f_\epsilon],\g[\h^\infty]]}{\VV[\g[{\h^\infty}]]}.
%\] 
%Let us also notice that, in practice, it is only possible to compute the optimal $\lambda^*$ numerically as already discussed in the previous sections. Furthermore, it remains important to be able to compute $\mathbb{E}[\g[{\h}]]$ or $\mathbb{E}[\g[{\h^\infty}]]$ exactly or with very small error in order to keep advantage of the control variate approach.

%\subsubsection{A Mean-Field control variate estimator}
Let us now give the details of the Mean Field Control Variate (MFCV) algorithm. To that aim, we recall that using $M$ realizations of our random variable $z$ to define the Mean Field Control Variate (MFCV) estimator, we have
\begin{equation*}
 \mathbb{E}[\g^{\lambda^*}[f_\epsilon]]\approx E_M[\g^{\lambda^*}[f_\epsilon]]= E_M[ \g[f_\epsilon]]- \frac{C_M[\g[f_\epsilon],\g[\h]]}{V_M[\g[\h]]} (E_M[\g[\h]]-\mathbb E[\g[\h]]),
\end{equation*}
where $\mathbb E[\g[\h]]$ denotes the exact value of the expectation of the quantity of interest or its numerical approximation with negligible error. Furthermore, we have the following notations
\begin{align*}
& E_M[ \g[f_\epsilon]]:= \frac{1}{M}\sum_{k=1}^M \g[f_\epsilon^k],\qquad E_M[ \g[\h]]:= \frac{1}{M}\sum_{k=1}^M \g[\h^k] \\
&V_M[\g[\h]]:= \frac{1}{M-1}\sum\limits_{k=1}^M (\g[\h^k]- E_M[\g[\h]])^2,\\
& C_M[\g[f_\epsilon], \g[\h]]:= \frac{1}{M-1}\sum\limits_{k=1}^M (\g[f^k_{\epsilon}]- E_M[\g[f_\epsilon]])\ (\g[\h^k]- E_M[\g[\h]]), 
\end{align*}
being $f_\epsilon^k$ and $\h^k$ the solutions of the Boltzmann-type and the mean-field models, respectively, relative to $k$th realization of the random variable $z$. 
The corresponding MFCV algorithm based on a DSMC method is reported in Algorithm \ref{al:mfcv}.

\begin{algorithm2e}[!htbp]
\label{al:mfcv}
%The main steps of the MFCV method for uncertainty in the initial data can be summarized as follows:
\begin{enumerate}
 \item \textbf{Sampling:}  Sample $M$ independent identically distributed (i.i.d.) samples of the initial distribution ${f}^{k,0} = f_0(w,z^k)$, $k=1,...,M$ from the random initial data ${f}_0(w,z)$.
 \item \textbf{Solving:} For each realization ${f}^{0,k}$, $k=1,...,M$
 \begin{enumerate}
 \item Compute the control variate $\h^{k,n}$, $k=1,...,M$ at time $t^n$ solving with a suitable deterministic method the mean field model \eqref{FP} (or using the steady state $\h_{\infty}$ ) and compute $\mathbb E[\g[\h^n]]$ (or $\mathbb E[\g[\h_\infty]]$) with negligible error. 
 \item Solve the kinetic equation \eqref{ScaledBol} by a MC solver with sample size $N$. We denote the solution at time $t^n$ by ${{f}}^{k,n}_{\epsilon,N},\ k=1,...,M$. 
 \end{enumerate}
 \item \textbf{Estimating:}
  \begin{enumerate}
 \item Estimate the optimal value of $\lambda^*$ at time $t^n$ by
 $$
 \lambda^{*,n}_M =\frac{ C_M[\g[f^{n}_{\epsilon,N}],\g[\h^{n}]]}{V_M[\g[\h^{n}]]},
 $$
 which in the mean field steady state control variate case becomes
  $$
 \lambda^{*,n}_M =\frac{ C_M[\g[f^{n}_{\epsilon,N}],\g[\h_{\infty}]]}{V_M[\g[\h_{\infty}]]}.
 $$
 \item Compute the the expectation of any quantity of interest $\g[f_{\epsilon,N}]$ of the random solution field with the mean-field control estimator
 $$
 E^{\lambda_*}_M[g[f_{\epsilon,N}]]= E_M[ \g[f^{n}_{\epsilon,N}]]-  \lambda^{*,n}_M (E_M[\g[\h^{n}]]-\mathbb E[\g[\h^n]]).
 $$
 \end{enumerate}
\end{enumerate}
\caption{Bi-fidelity MFCV-DSMC method}
\end{algorithm2e}

 Concerning the evaluation of moments, by ignoring the error term due to the approximation of $\lambda_*$, we have the following^^>\cite{PTZ}
 \begin{lemma}
The root mean square error of the MFCV-DSMC method satisfies 
\be\begin{split}
&\mathbb E\left[\mathbb E_{V}\left[ \left( \mathbb E[\langle\varphi,f_\epsilon\rangle]- E^{\lambda_*}_M[\langle\varphi,f_{\epsilon,N}\rangle] \right)^2 \right]\right]^{1/2}\\
&\hskip3cm\le \left(1-\rho^2_{\langle\varphi,f_\epsilon\rangle, \langle\varphi,\h\rangle}\right)^{1/2}\frac{\nu_{\langle\phi,f_\epsilon\rangle}} {M^{1/2}} + \frac{\sigma_{\varphi,M}} {N^{1/2}}, 
\end{split}
\ee
where $\nu^2_{\langle\varphi,f_\epsilon\rangle} = \VV[\langle\varphi,f_\epsilon\rangle]$
and $\sigma^2_{\varphi,M} = E_M[\sigma^2_{\varphi}]$ with $\sigma_\varphi^2=\VV_V[\varphi]$.
\label{lem:3}
\end{lemma}
%\begin{proof}
%Let us observe that
%\[
%\mathbb E[\langle\varphi,f_\epsilon\rangle]=\mathbb E[\langle\varphi,f^{\lambda_*}_\epsilon\rangle],\qquad E^{\lambda_*}_M[\langle\varphi,f_{\epsilon,_N}\rangle]=E_M[\langle\varphi,f^{\lambda_*}_{\epsilon,N}\rangle]
%\]
%and then, using Theorem \ref{th:lambda}, we get 
%\[
%\nu^2_{\langle\varphi,f_\epsilon^{\lambda_*}\rangle} = \VV[\langle\varphi,f_\epsilon^{\lambda_*}\rangle] = \left(1-\rho^2_{\langle\varphi,f_\epsilon\rangle, \langle\varphi,\h\rangle}\right)\ \VV[\langle\varphi,f_\epsilon\rangle].
%\]
%The conclusion follows from Lemma \ref{lem:2} together with the identity
%\[
%\langle\varphi,f_\epsilon^{\lambda_*}\rangle-\langle\varphi,f_{\epsilon,N}^{\lambda_*}\rangle=\langle\varphi,f_\epsilon\rangle-\langle\varphi,f_{\epsilon,N}\rangle.
%\]
%\end{proof}
In the case of the reconstruction function \eqref{eq:f_RNz}, we have
\[
\begin{split}
\left\| \mathbb E[f_\epsilon](t,\cdot)-E^{\lambda_*}_M[f_{\epsilon,N,\Delta w}](t,\cdot)\right\|_{L^p(V,L^2(\Omega,L^2(V)))} &  \\
&\hskip -3.3cm \leq \left\| \mathbb E[f_\epsilon](t,\cdot)-\mathbb E[f_{\epsilon,\Delta w}](t,\cdot)\right\|_{L^p(V)}\\
&\hskip -3.3cm +\left\| \mathbb E[f_{\epsilon,\Delta w}](t,\cdot)-E^{\lambda_*}_M[f_{\epsilon,N,\Delta w}](t,\cdot)\right\|_{L^p(V,L^2(\Omega,L^2(V)))},
\end{split}
\]
where the first term is bounded as in Theorem \ref{th:2} and the second term can be bounded using Lemma \ref{lem:3} with $\phi(\cdot)=S_{\Delta w}(w-\cdot)$. Thus we have the following result.
 \begin{theorem}
The error introduced by the reconstruction function \eqref{eq:f_RNz} in the MFCV-DSMC method satisfies
\begin{equation}
\label{eq:estim_3}
\begin{split}
&\left\| \mathbb E[f_\epsilon](t,\cdot)-E^{\lambda_*}_M[f_{\epsilon,N,\Delta w}](t,\cdot)\right\|_{L^p(V,L^2(\Omega,L^2(V)))}\\
 &\leq 
 \frac{\left\|\left(1-\rho^2_{\langle S,f_\epsilon\rangle, \langle S,\h\rangle}\right)^{1/2}\nu_{\langle S,f_\epsilon\rangle}\right\|_{L^p(V)}} {M^{1/2}} + \frac{\|\sigma_{S,M}\|_{L^p(V)}} {N^{1/2}} + C_{\mathbb E[f]} (\Delta w)^q
\end{split}
\end{equation}
where and $\nu_{\langle S,f_\epsilon\rangle}^2$ is defined as in \eqref{eq:nus} and $\sigma^2_{S,M}$ is defined in Theorem \ref{th:2}.
\label{th:4}
\end{theorem}
As a consequence when the solution of the high-fidelity model is close to the solution of the control variate the statistical error due to the uncertainty vanishes. This justifies the use of a large number of samples in the velocity space in agreement with the reconstruction used in order to balance the last two error terms in \eqref{eq:estim_3}. 

\subsubsection{Application to socio-economic sciences}
We present two numerical examples concerning kinetic models in socio-economic sciences. We will denote by MFCV-S the case where the steady state $\h_\infty(w,z)$ of the mean-field model is analytically known and is used off line as control variate and by MFCV the case where the Fokker-Planck equation \eqref{FP} is numerically solved and used as a time dependent control variate. The mean field model is solved by the second order structure preserving method for nonlocal Fokker-Planck equations developed in^^>\cite{PZ1}. The solutions are averaged over $50$ runs to reduce statistical fluctuations. %The reconstruction of the distribution function is performed using standard first order histogram approximation with $N_Z=100$ grid points. 
In order to compute with negligible error $\mathbb E[\h_\infty]$ we adopt a stochastic collocation approach with $20$ collocation nodes.

We first consider the kinetic model for opinion formation  \eqref{InterOpinion} with uncertainties present on the initial distribution or on the interaction strength. We assume that $z\sim \mathcal{U}([0, 1])$ and the initial distribution $f_0(w,z)$ is given by
\begin{equation}
f_0(w,z) = 
\begin{cases}
1 & w \in \left[\dfrac14 (z-2) , \dfrac14 (z+2)\right] \\
0 & \textrm{otherwise}. 
\end{cases}
\label{eq:t1}
\end{equation} 
We also assume 
\begin{equation}
\label{eqn:A}
p(|v-w|,z) = 1,\qquad D(v)= \sqrt{1-v^2},
\tag{A}
\end{equation}
and thus we obtain a steady state of the Fokker-Planck equation of the form $\h_\infty(w,z)$ given by \eqref{eq:steady_beta}.  
The Boltzmann equation is solved with $N=2\times 10^4$ particles. The number of grid points of the mean field model is set to $N_{MF}=20$ and the number of samples we have chosen is $M_{MF}=10^4$ which correspond to a computational cost comparable to $M=10$ for the Boltzmann solver.
In Figure \ref{ErrorSamples} (left) we report the $L^2$ error of the expected density obtained by the standard MC and MFCV-S methods for increasing number of samples at time $t = 5$. We obtain an improvement in accuracy for the MFCV-S between one and two orders of magnitude using the same number of samples. 

\begin{figure}
\begin{center}
\includegraphics[scale = 0.26]{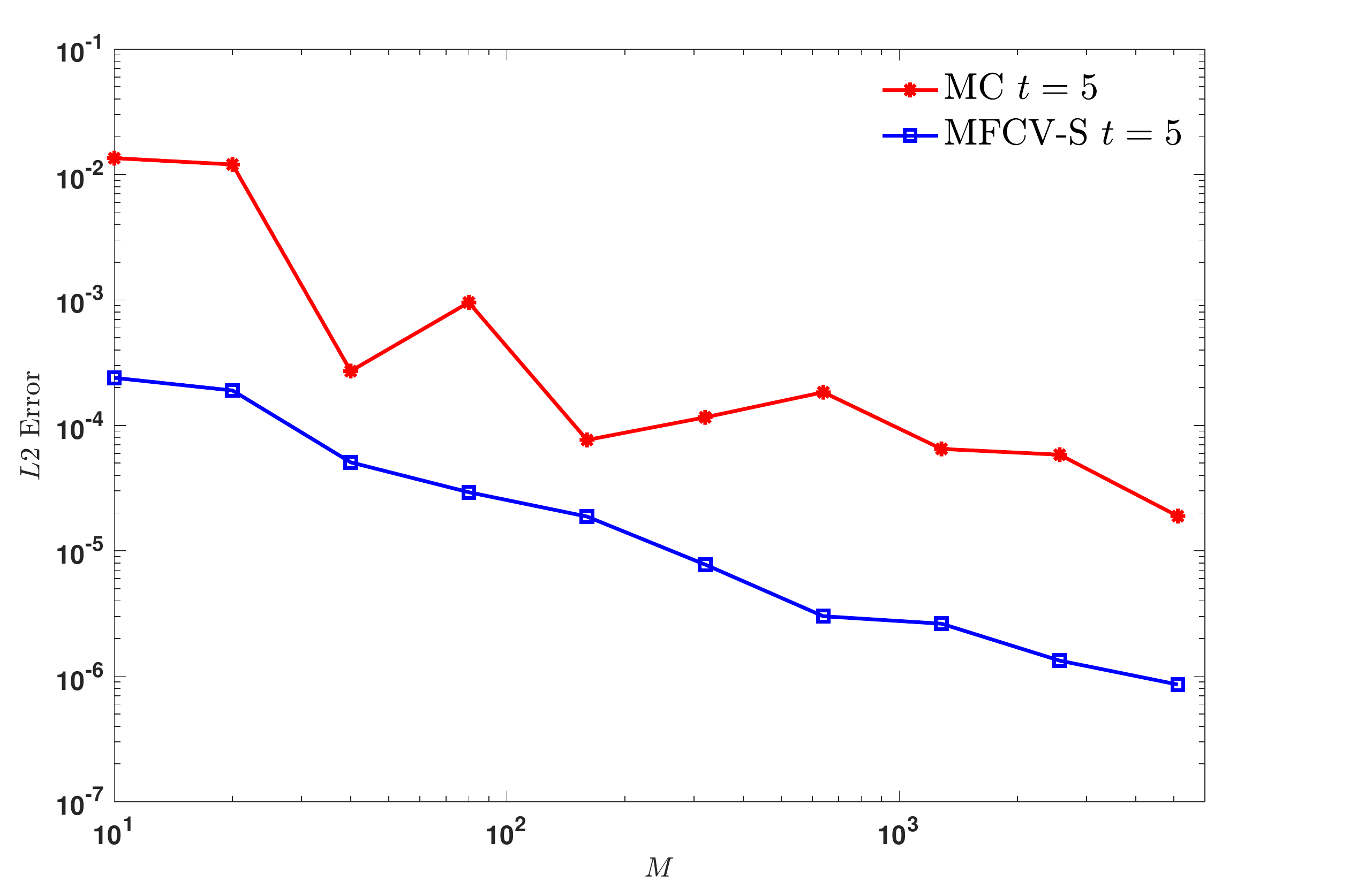}\hskip -.5cm
\includegraphics[scale = 0.26]{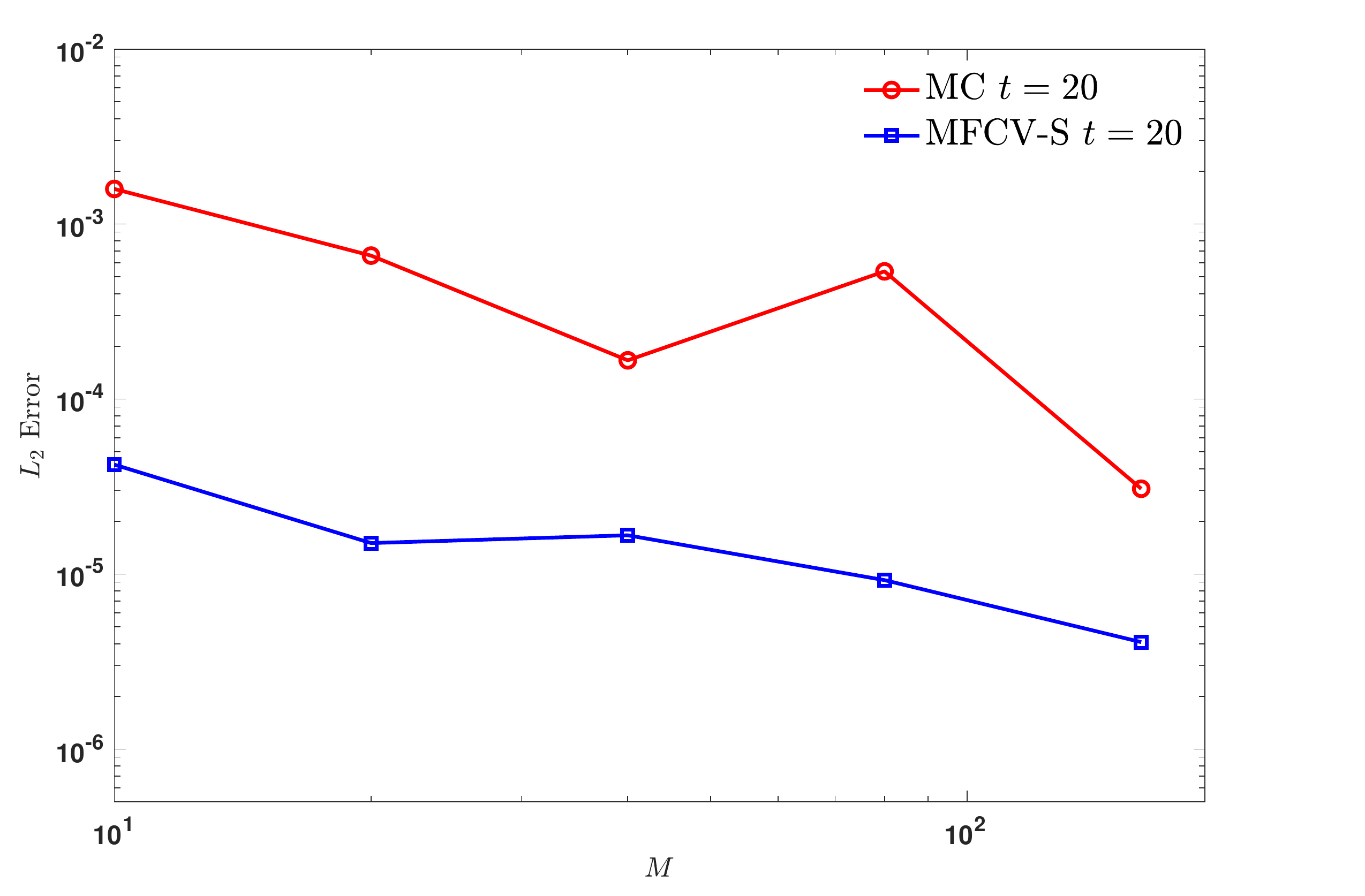}
\caption{{MFCV method for opinion dynamics.} Error of the MFCV-S estimate and classical MC method for increasing number of samples $M$. We considered $N=2\times 10^4$ in the DSMC solver. (Left) Solution at $t=5$ of case \eqref{eqn:A} with uncertain initial data; (Right) Solution at $t=20$ of case \eqref{eqn:B} with uncertain interaction parameters.}\label{ErrorSamples}
\end{center}
\end{figure}
As a further case for opinion dynamics we consider the kinetic model \eqref{InterOpinion} with 
\begin{equation}
\label{eqn:B}
p(|w-v|,z)= \dfrac{3}{4}+\dfrac{z}{4},\qquad z\sim\mathcal{U}([-1,1]),\qquad D(w) = {1-w^2},
\tag{B}
\end{equation}
so that the resulting steady state of the Fokker-Planck model is the Maxwellian-like distribution \eqref{eq:steady_max}. The initial data in this case is \eqref{eq:t1} in the deterministic setting $z=0$. In Figure \ref{ErrorSamples} (right), we report the $L^2$ error of expected probability distribution function computed by the MC and MFCV-S method at the final time for different number of samples. As in case \eqref{eqn:A} we obtain an improvement between one and two orders of accuracy for the MFCV-S method compared to the classical MC method.

%\begin{figure}
%\begin{center}
%\includegraphics[width=0.5 \textwidth]{figs/NumericCV/Test1Comparison}\hfill
%\includegraphics[width=0.5 \textwidth]{figs/NumericCV/Test2Comparison}
%\caption{ \textbf{Test 1}. The $L^2$ error of $\mathbb E[f]$ computed by the MFCV and MFCV-S method for different number of samples $M$ at $t=0.1$ and $t=5$.  
%The left hand side corresponds to the setting of case \eqref{eqn:A} and the right to case \eqref{eqn:B}. }\label{NumTest1}
%\end{center}
%\end{figure}
%
%In Figure \ref{NumTest1} we compare the MFCV method with the MFCV-S for different numbers of samples $M$. The QoI is $\mathbb E[f_\epsilon]$. Especially, we report the performance of the MFCV-S method at different times. We expect that the performances of the MFCV-S method result improved after the Boltzmann solution is close to a steady profile.  Indeed,  for short times $t=0.1$ the error should be larger than at later times e.g. $t=5$. This behavior can be clearly observed for large number of samples. 

%\begin{figure}
%\begin{center}
%%\includegraphics[scale = 0.28]{figs/FinanceDiffusionUQSteady/CVTime}
%\includegraphics[scale = 0.35]{figs/FinanceDiffusionUQSteady/CVSample}
%\caption{\textbf{Test 2}. $L_2$ error for the expected distribution $\mathbb E[f_\epsilon]$ for problem \eqref{eqn:B2} computed by the  MC and MFCV-S method for different number of samples $M$ at $t=30$. We considered $N=2\times 10^{4}$ in the DSMC solver. } 
%\label{SamplesTest4}
%\end{center}
%\end{figure}

%\subsubsection*{Test 2: Wealth model with uncertainty}

We study now two test cases related to the wealth exchange CPT model defined by \eqref{InterWealth}. First, we consider uncertainty in the initial condition and secondly in the saving propensity. The computational domain is the interval $[0,10]$. Let us first consider $z \in \mathcal U([0,1])$ and the initial  distribution $f_0(w,z)$ defined by 
\begin{equation}
f_0(w,z) = 
\begin{cases}
\dfrac{1}{2} & w \in \left[\dfrac{z}{5},2+\dfrac{z}{5} \right] \\
0 & \textrm{otherwise}.
\end{cases}
\end{equation}  
Furthermore, we consider 
\begin{equation}
\lambda(z) = 1,\qquad D(w) = w,
\tag{C}
\label{eqn:A2}
\end{equation}
so that the large time behavior of the Fokker-Planck model $\h_\infty(w,z)$ is given by \eqref{eq:steady_invgamma} with $m(z) = 1+ \dfrac{z}{5}$.
As a second case, we consider uncertainty in the interaction 
\begin{equation}
\lambda(z) = \frac{1}{2}+ \frac{z}{4},\qquad z\sim\mathcal{U}([-1,1]),\qquad D(w)=w.
\tag{D}
\label{eqn:B2}
\end{equation}
The initial condition is uniformly distributed on $[0,2]$, so that the large time behavior of the Fokker-Planck model is given by \eqref{eq:steady_invgamma} with $m_{\h} \equiv 1$.
The DSMC solver for the Boltzmann model uses $N=5\times 10^4$ particles. For the mean field scheme instead we consider $N_{MF}=100$ and we choose $M_{MF}= 5\times 10^3$ which gives a comparable cost of the full Boltzmann solver for $M=10$.
\begin{figure}
\begin{center}
\includegraphics[width=0.51 \textwidth]{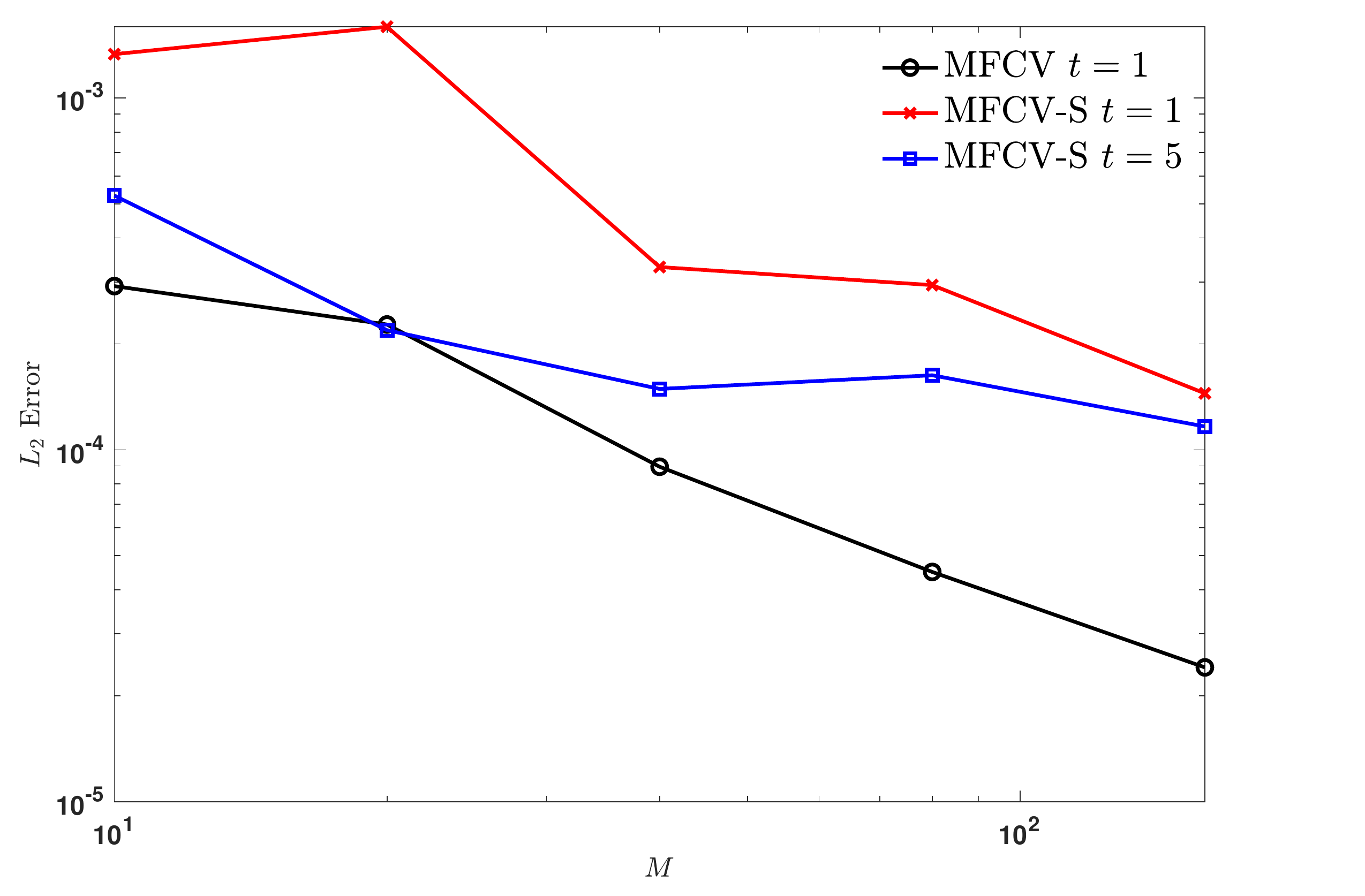}\hskip -.5cm
\includegraphics[width=0.51 \textwidth]{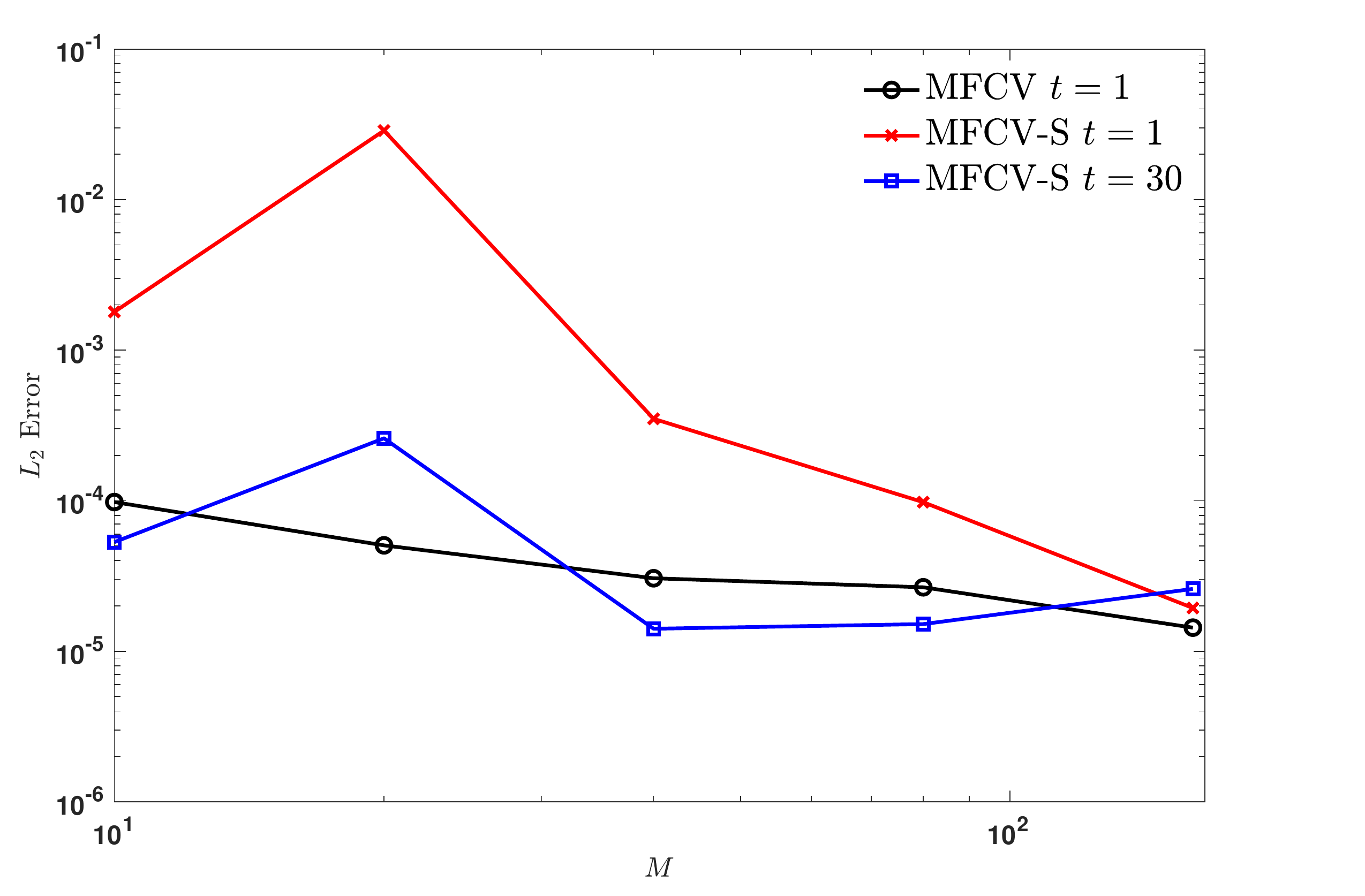}
\caption{{MFCV method for a wealth exchange model}. The $L^2$ error of $\mathbb E[f]$ computed by the MFCV-S and the MFCV methods for different number of samples $M$ at $t=1$ and $t=5$. 
The left hand side corresponds to the setting of case \eqref{eqn:A2} and the right to case \eqref{eqn:B2}.}\label{NumTest4}
\end{center}
\end{figure}
In Figure \ref{NumTest4}, we compare the  $L^2$ error of $\mathbb E[f_\epsilon]$ computed by the MFCV or MFCV-S method at fixed times but for different number of samples $M$. We obtain that the $L^2$ error of the expected probability distribution function of the MFCV-S method is considerably smaller at larger times. In comparison to the MFCV-S method,  MFCV method is able to be more accurate even at early times.

\section{Bi-fidelity stochastic collocation methods}\label{sec:bi}
In the examples of multi-fidelity models that were discussed in the previous section, high-fidelity samples are not selected based on any criteria other than the fact that they are small in number compared to low-fidelity samples.
Here, following^^>\cite{LZ,GJL,LPZ,BLPZ} we explore a different direction and we focus on stochastic collocation methods based on bi-fidelity algorithms. The main idea is that within the bi-fidelity approximation, one employs the cheap low-fidelity model to explore the random parameter space and to select the most important parameter points in this space. After that, by applying exactly the same approximation rule learnt from the low-fidelity model, one solves the high-fidelity model. 
Following the above idea, we first review the BFSC methods in its generality. After we focus on its application to  efficient uncertainty quantification  for the Boltzmann equation of Section \ref{sec:bolt}, the linear transport equation of Section \ref{sec:diff} and finally the epidemic transport models of Section \ref{sec:epidemic}. 

\subsection{A Bi-fidelity stochastic collocation (BFSC) algorithm}
To set the stage for the discussion, we first introduce basic notions in the following. We let $u(x,t;z)$ be the solution of a complex
system subject to uncertainty, where $x\in \DH$ and $t$ are the  spatial and temporal variables. $z\in \Omega \subset {\RR}^{d_z}$ is a
$d_z$-dimensional random variable. Here $\Omega$ is the support of $z$,
where the probability distribution $p(z)$ is defined. For the sake of simplicity,  we denote  $u(x,t;z)$ by  $u(z)$ when from the context the dependence on the other variables will be clear. Let us assume now the high-fidelity solutions $u^H(z)$ and low-fidelity solutions $u^L(z)$ are available. Let also 
$N$ be the number of affordable low-fidelity simulation runs, which, in principle, is  very large due to the reduced complexity of the model. On the other hand, $M$ denotes the number of high-fidelity simulation runs that can be afforded and it is typically very small, i.e. the setting is such that $N\gg M$. {Note that in this section, to simplify notations, we will use $N$ to denote the number of samples used by the low-fidelity model instead of $M_E$ as in Section \ref{sec:MSCV}}. Let finally $\gamma_k=\{z_1, \cdots, z_k\}$, $k\geq 1$ be a set of sample points in $\Omega$. 

We denote by $u^L(\gamma_k) = [u^L(z_1), \cdots, u^L(z_k)]$ the low-fidelity snapshot matrix corresponding to the solution of the low fidelity model for the sample point $z_k$. To this matrix we can associate a corresponding low-fidelity approximation space, i.e. the space spanned by the set of sample points $\gamma_k$, 
$$ U^L(\gamma_k) = \text{span}\{u^L(\gamma_k)\} = \text{span}\{ u^L(z_1), \cdots, u^L(z_k) \}. $$  
Similarly, the high-fidelity snapshot matrix, i.e. the matrix obtained from the sample set $\gamma_k$, and the corresponding high-fidelity approximation space, i.e. the space spanned by the solutions computed at nodes $z_k$, are defined as follows:
$$ u^H(\gamma_k) = [u^H(z_1), \cdots, u^H(z_k)], \qquad U^H(\gamma_k)=\text{span}\{u^H(\gamma_k)\}. $$

The main idea of the BFSC method is to construct an inexpensive surrogate $u^B(z; \gamma_M)$ of the high-fidelity solution in the following non-intrusive  manner
\begin{equation}
%\begin{array}{ll}
\label{bi-fi approx}
%&u^H(z) \approx u^B(z)  = \sum_{i=1}^nc_i(z)u^H(z_i), 
u^H(z) \approx u^B(z; \gamma_M)  = \sum_{k=1}^M c_k(z)u^H(z_{i_k})
%\label{bi-fi mean}
%&\mathbb{E}[u^H(z)] \approx \mu^B = \sum_{i=1}^n w_i u^H(z_i),
%\end{array}
\end{equation}
where $M$ is expected to be the number of high-fidelity samples and where correspondingly $z_{i_k}\in \gamma_M$ is a subset of size $M$ of the sample space of size $M_E$. In other words, we approximated  the solution of the high fidelity model in the space spanned by $u^H(z_{i_k}), \ k=1,..,M$. When constructing such algorithm one seeks for $M$ to be as small as possible, since large $M$ means more high-fidelity simulations and consequently prohibitive computational efforts. Thus, the central idea of the BFSC algorithm is to use cheap low-fidelity models to learn the coefficients $c_i(z)$  in \eqref{bi-fi approx}, and then apply the same approximation rule to a limited number, but selected, of high-fidelity samples to construct the bi-fidelity approximations of high-fidelity samples.   

The BFSC algorithm for approximating the high-fidelity solution consists  of offline and online stages. In the offline stage, we employ the cheap low-fidelity model to explore the parameter space to find the most important parameter points, i.e. a small number of samples permitting to give a suitable approximate solution. During the online stage, we learn the approximation rule from the low-fidelity model for any given $z$, and apply it to construct the bi-fidelity approximation. Thus, to construct the BFSC approximation in 
\eqref{bi-fi approx}, one essentially needs to answer the following two questions:
\smallskip
\begin{itemize}
\item How to choose a good collocation nodal set that results in a good approximation of  the high fidelity solution in the random space?
\smallskip
\item How to recover the coefficients $\{c_i\}$ in an efficient manner? In other words, what is the efficient reconstruction algorithms that can be realized without resorting to an intensive use of the high-fidelity solver?
\end{itemize} 
\smallskip
These two questions, if properly addressed, constitute the key to the efficiency and accuracy of the algorithm. We outline the key ideas in Algorithm \ref{BiFi-pod}.

\begin{algorithm2e}[htbp]
%\begin{alg}[Bi-fidelity method]~
%\label{alg}
\label{BiFi-pod}
\begin{enumerate}
	\item 
{\bf Offline stage}
\begin{enumerate}
	\item 
Select a sample set $\Gamma_N = \{z_1, z_2, \hdots, z_N\}\subset \Omega $.
\item 
Run the low-fidelity model $u_l(z_j)$ for each $z_j \in \Gamma_N$.
\item 
Select $M$ ``important" points from $\Gamma_N$ and denote it by $\gamma_M=\{z_{i_1}, \cdots z_{i_M} \} \subset\Gamma_N$. Construct the low-fidelity approximation space $U^L(\gamma_M)$.

\end{enumerate}

\item 
{\bf Online}
\begin{enumerate}
	\item 
	Run high-fidelity simulations at each sample point of the selected sample set $\gamma_M$. Construct the high-fidelity approximation space $U^H(\gamma_M)$. 
\item 
For any given $z$, run the low-fidelity model to get the corresponding low-fidelity solution $u^L(z)$ and  compute the low-fidelity coefficients by projection:
$$u^L(z) \approx \mathcal{P}_{U^L(\gamma_M)}u^L = \sum_{k=1}^M c_k(z)u^L(z_{i_k}).$$
\item
Construct the bi-fidelity approximation by applying the sample approximation rule learned from the low-fidelity model:
$$u^B(z)  = \sum_{k=1}^M c_k(z)u^H(z_{i_k}).$$
\end{enumerate}
\end{enumerate}
\caption{BFSC method}
\end{algorithm2e}
%\end{alg}
We address in the following the two main questions permitting to construct an efficient and accurate BFSC method: the proper point selection in the random space and {the efficient construction of the bi-fidelity approximation}.

\subsubsection{Point selection.}
To select the subset $\gamma_M$, we search the parameter space by the greedy algorithm proposed in^^>\cite{NGX14, ZNX14}: we gradually select the important point set from a candidate set $\Gamma_N\in \Omega$.
The method works as follows.
\smallskip
\begin{itemize}
	\item Start with a trivial subspace $\gamma_0=\varnothing$, and assume that the first $k-1$ important points $\gamma_{k-1}=\{ z_{i_1}, \cdots, z_{i_{k-1}}\} \subset \Gamma_N$ have been selected.
	\smallskip
	\item  We choose the next point  $z_{i_k} \in \Gamma$ as the point that maximizes the 
	distance between its corresponding low-fidelity solution and the approximation space $U^L(\gamma_{k-1})$, spanned by the low-fidelity solutions on the existing point set $\gamma_{k-1}$, i.e., 
	\begin{equation}\label{Point} z_{i_k} = \arg\max_{z\in\Gamma}\text{dist}(u^L(z), U^L(\gamma_{k-1})), 
		\qquad \gamma_k = \gamma_{k-1} \cup z_{i_k}, \end{equation}
	where $\text{dist}(v,W)$ is the distance function between $v\in u^L(\Gamma)$ 
	and subspace $W\subset U^L(\Gamma_N)$. 
\end{itemize}
\smallskip
Thus, the greedy procedure essentially serves the purpose of searching the linear independent basis set in the parameterized low-fidelity solution space until all $M$ important points are selected.
Let us remark that the whole algorithm allows an efficient implementation by standard linear algebra operations, if $\text{dist}(v,W)$ is chosen as squared Euclidean distance. Let $\mb{G}$ be the Gramian matrix of the low-fidelity solution
$u^L(\Gamma_N)$, i.e.,
\begin{equation} \label{Gramian}
\mathbf{G} =   \left\langle u^L(z_i),
u^L(z_j)\right\rangle^L, \quad {1\leq i,j\leq N}.
\end{equation}
where $\left\langle\,\cdot\,\right\rangle^L$ be an inner product space corresponding to the high-fidelity solution. Similar defintion holds for $\left\langle\,\cdot\,\right\rangle^H$.
We then apply the pivoted Cholesky decomposition to the matrix $\mathbf{G}$,
\begin{equation} \label{Cholesky}
\mb{G} = \mb{P}^T\mb{LL}^T\mb{P},
\end{equation}
where $\mb{L}$ is lower-triangular and $\mb{P}$ is a permutation matrix
due to pivoting.
This will produce an ordered permutation vector $P=(i_1,\dots,i_N)$,
from which we choose the first $M$ points to define
$\gamma_M=\{z_{i_1},\dots,z_{i_M}\}$. 
This procedure guaranteed that can $u^L(\gamma)$ can form a linearly independent collection^^>\cite{NGX14}. More details and properties of the algorithm can be found in
\cite{NGX14, ZNX14}. We remark that other point selections strategies can be used in this setting^^>\cite{perry2019allocation}. 

\subsubsection{Bi-fidelity approximation.} 
Once the important point set $\gamma_M$ is selected, we can construct the low- and high-fidelity approximation space, $U^L(\gamma_M)$  and $U^H(\gamma_M)$, respectively. For any 
given new sample point $z\in \Omega$, we project the corresponding low-fidelity solution $u^L(z)$ onto the low-fidelity approximation space
$U^L(\gamma_M)$: $$u^L(z) \approx \mathcal P_{U^L(\gamma_M)}[u^L(z)] = \sum_{k=1}^M c_k(z) u^L(z_{i_k}), $$ 
where $\mathcal P_{V}$  is the projection operator onto a Hilbert space $V$ 
and the corresponding projection coefficients $\{ c_k\}$ are computed by the following projection:
\begin{equation}\label{Gc}
 {\bf G}^L {\bf c} = {\bf f}^L , \qquad {\bf f}^L  = (f^L _k)_{1\leq k\leq M}, \qquad f^L _k = \langle u^L(z), u^L(z_{i_k})\rangle^L, 
 \end{equation}
where ${\bf G}^L$ is the Gramian matrix of $u^L(\gamma_M)$, defined by 
\begin{equation}\label{GM}
 ({\bf G}^L)_{ij} = \langle u^L(z_{i_k}), u^L(z_{j_k}) \rangle^L, \qquad 1 \leq k\leq M, \quad i_k, j_k \in \gamma_M. 
\end{equation}

These low-fidelity coefficients $\{c_k\}$ serve as the surrogate of the corresponding high-fidelity coefficients of $u^H(z)$. Therefore,
the sought bi-fidelity approximation of $u^H(z)$ can be constructed via \eqref{bi-fi approx}. We emphasize that if the low-fidelity model can mimic the variations of the high-fidelity model in the parameter space,  the low-fidelity coefficients can be  a good approximation of the corresponding high-fidelity coefficients for a given sample $z$. The whole procedure of the bi-fidelity algorithm for a given $z$ is summarized in {Algorithm ~\ref{BiFi-pod}}.
It is worth noting that since the number of low-fidelity basis is typically small, the cost of computing the low-fidelity projection coefficients by solving the linear system \eqref{Gc} is negligible. Therefore, the dominant cost of the online step is one low-fidelity simulation run. If the low-fidelity solver is much cheaper than the high-fidelity solver, the speedup during the online stage can be significant as we will demonstrate with some examples in the rest of the survey.

\subsubsection{An empirical error bound estimation.}
An error bounds analysis for the bi-fidelity method was derived in^^>\cite{NGX14,zhu2017multi} in a rather general setting. Nonetheless, the results presented are largely theoretical. For practical applications of the BFSC approach, it is important to answer the following two questions: \smallskip
\begin{itemize}
\item Given a low-fidelity model, is this model good enough to build a reasonably accurate BFSC approximation for practical purposes? 
\smallskip
\item How to efficiently estimate a reasonably good error bound on the entire parameter space?
\end{itemize}
\smallskip
In other words, an inexpensive but effective quality indicator (EQI) of the approximation is of practical importance. In realistic applications, \emph{a priori} assessment of the model quality and prediction errors is very important. In this direction, a previous study^^>\cite{GZJ20} introduced a novel empirical error bound estimation approach with ease of implementation to evaluate the performance of the bi-fidelity surrogates a priori. We will briefly describe the methodology here. In^^>\cite{GZJ20}, two important quantities have been identified which are useful to build the approximation quality indicator for the BFSC approximation:
\smallskip
\begin{itemize}
\item  $R_s$, model similarity which characterizes the similarity between the low/high-fidelity models:
\begin{equation}
R_{s}(z) = \frac{d^H(u^H(z),U^H(\gamma_{M}))}{\|u^H(z)\|}/\frac{d^L(u^L(z),U^L(\gamma_{M}))}{\|u^L(z)\|}, 
\end{equation}
If $R_{s} (z)\approx 1$, the low-fidelity model is informative for the point selection.
\smallskip
\item $R_{e}$, the ratio between the projection error and the distance error,
\begin{equation}
    R_{e} (z) = 
  \frac{\| \mathcal{P}_{U^H({\gamma_M})}u^H(z) -  u^B(z; \gamma_M)\|}{d^H(u^H(z),U^H(\gamma_{M}))}.
\end{equation}
when  $R_{e}(z)$ is too large, it indicates the  greedy algorithm is less effective for point selection of high-fidelity samples. This suggests that %balance between the in-plane error and the relative distance and 
one should stop collecting new high fidelity samples since the improvement in the result is likely marginal. 
\end{itemize}
\smallskip
Empirically, one can than fix thresholds for $R_{s} \approx 1$ and $R_{e} < 10$ in such a way that the BFSC approximation can usually deliver results which are reasonably accurate, i.e. better than the low-fidelity solutions. 

Using $R_s$ and the observation in^^>\cite[Theorem 1]{GZJ20}, for any given new point $z_{*}$, one has
\begin{equation}
\begin{split}
\label{eq:eb_a}
\frac{\|u^H(z_*) - u^B(z_*)\|}{\|u^H(z_*)\|}
&\leq \frac{d^L(u^L(z_*),U^L(\gamma_M))}{\|u^L(z_*)\|} R_s(z_*)\\
&\times\Big(1+\frac{\| \mathcal{P}_{U^H({\gamma_M})}u^H(z_*) -  u^B(z_*)\|}{ d^H(u^H(z_*),U^H(\gamma_M))}\Big).
\end{split}
\end{equation}

To remove the dependency of the new HF sample  $u^H(z_*)$ on the above right-hand side, one uses $z_{M+1} \in \gamma_{M+1}$ as the testing points served as an error surrogate for the BFSC approximation in the entire parameter space. If the LF and HF models are similar (i.e., $R_s\approx 1$), one can choose some proper constants $c_1$ and $c_2$, such that for the first $M+1$ pre-selected important points $\gamma_{M+1}$, 
\begin{equation}
\label{eq:eb1}
\begin{split}
    \frac{\|u^H(z_*) - u^B(z_*)\|}{\|u^H(z_*)\|}
&\le
\frac{d^L(u^L(z_*),U^L(\gamma_M))}{\|u^L(z_*)\|}\\
&\times \Big[c_1+c_2\frac{\| P_{U^H({\gamma_M})}u^H(z_{M+1}) -  u^B(z_{M+1})\|}{ d^H(u^H(z_{M+1}),U^H(\gamma_M))}\Big]\\
& = \frac{d^L(u^L(z_*),U^L(\gamma_M))}{\|u^L(z_*)\|}
(c_1+c_2R_{e}(z_{M+1})).
\end{split}
\end{equation}
Notably,  in the right-hand side of the inequality, the first term only depends on the inexpensive low-fidelity data and the second term needs the high-fidelity solution at $z_{M+1}$ -- the $(M+1)$-th point of the pre-selected important points $\gamma_{M+1}$ without the need of the high-fidelity sample $u^H(z)$ for a new given $z$. Essentially, $z_{M+1}$ can be regarded as a test point to serve as an error/quality indicator of the BFSC approximation in the entire parameter space. This can be advantageous if the high-fidelity solver is very expensive.

We acknowledge that the above error bound estimation, though not  rigorous, is a useful quantity to access the quality of the BFSC approximation in practical applications. In the following numerical experiments, our empirical results suggest that this error bound estimation is effective if the constants $c_1$ and $c_2$ are set to be 1^^>\cite{GZJ20,LPZ}.
%-------------------------------------------------------------------------
\subsection{A BFSC method for the Boltzmann equation}
In this section, we consider the application of the bi-fidelity algorithm to the study of the the Boltzmann equation under the hydrodynamic scaling and with multi-dimensional random parameters presented in Section \ref{sec:bolt} and already analyzed in Section 3 using MSCV method. 
%^^>\cite{LZ}. We recall for sake of clarity the model \eqref{eq:Boltzmann} and its hydrodynamic limit:
%\begin{align}
%\label{Boltz}
%\left\{
%\begin{array}{l}
%\displaystyle \partial_t f + v\cdot\nabla_x f  =
%\frac{1}{\varepsilon} Q(f, f),   \\[4pt]
%\displaystyle  f(0,x,v,z)= f_{I}(x,v,z), \qquad
%(x,v)\in\Omega\times\mathbb R^{d_v},  z\in \Omega. 
%\end{array}\right.
%\end{align}
%where $z\in \Omega$ is a $d_z$-dimensional random parameter in the system. 
%With the multi-dimensional random parameter, it is computationally challenging to fully sweep the parameter space by solving the Boltzmann equation repeatedly. To mitigate this computational cost, we consider to choose the compressible Euler equations of gas dynamics given by 
%\begin{equation}\label{Euler}
%\partial_t \begin{pmatrix} \rho \\ \rho u \\ E \end{pmatrix} + 
%\nabla_x \cdot \begin{pmatrix} \rho u \\ \rho u\otimes u + p\,\text{I} \\ (E+p) u \end{pmatrix} = 0, 
%\end{equation}
%to be our low-fidelity model. 
%The above Euler system \eqref{Euler} is as already stated in Section 2 a first-order 
%proximation with respect to $\varepsilon$ to the Boltzmann equation, which can mimic the variations of macroscopic quantities of the Boltzmann equation in the fluid regime. For each fixed $z$, solving the deterministic Euler equation is much easier and more efficient in terms of memory and computational time than solving the deterministic Boltzmann equation. 

To examine the performance of the bi-fidelity method, numerical errors are computed in the following way: we choose a fixed set of points $\{ \hat z_k \}_{k=1}^M \subset \Omega$ that is independent of the point sets $\Gamma$, then evaluate the following error between the bi-fidelity and high fidelity solutions at a final time $t$: 
 \begin{equation}\label{Error} \|u^H(t) - u^B(t)\|_{L^2(D \times \Omega)} 
 \approx\frac{1}{n}\sum_{k=1}^n \|u^H(\hat z_k, t) - u^B(\hat z_k, t)\|_{L^2(D)}, 
 \end{equation}
 where $\|\cdot\|_{L^2(D)}$ is the $L^2$ norm in the physical domain $D=\Omega\times\mathbb R^2$.
In the following two examples, the spatial domain is chosen to be  $[0,1]$ with $N_x$ grid points. 
The velocity domain is chosen as $[-L_v, L_v]^2$ with $L_v=8.4$. We denote $N_v^l$ and $N_v^h$ 
as the number of grid points used in velocity discretization of the low-fidelity and high-fidelity solver respectively. The $d_z$-dimensional random variable ${z}$ 
is assumed to follow the uniform distribution on $[-1,1]^{d_z}$. 
Let also the size of training set $\Gamma$ be $N=1000$. We examine the error of bi-fidelity approximation with respect to the number of high-fidelity runs by computing the norm defined in (\ref{Error}).

First, we consider a benchmark Sod shock tube test. Assume the random cross section 
$$ b(z^b) = 1 + 0.5 \sum_{k=1}^{{d_1}+1}\frac{{z_k^b}}{2k}, $$
and the uncertain initial distribution given by
$$ f^0(x,{v},z) = \frac{\rho^0}{2\pi T^0}\, e^{-\frac{|{v}-u^0|^2}{2T^0}}, $$
where the initial data for $\rho^0$, $u^0$ and $T^0$ is
\begin{align*}
\begin{cases}
&\displaystyle\rho_l=1, \quad u_l=(0,0),  \quad T_l(z^T)=1+0.4\sum_{k=1}^{d_1} \frac{{z_k^T}}{2k}, \quad x\leq 0.5, \\[4pt]
&\displaystyle\rho_r=\frac{1}{8}, \quad u_r=(0,0),  \quad T_r(z^T)=\frac{1}{8}(1 + 0.4\sum_{k=1}^{d_1}\frac{{z_k^T}}{2k}),  \quad x>0.5. 
\end{cases}
\end{align*}
Here $z^b$ and $z^T$ represent the random variables in the collision kernel and initial temperature, with {$\{z_i^b\}_{i=1}^{d_1+1}$ and $\{z_i^T\}_{i=1}^{d_1+1}$} following the uniform distribution on $[-1,1]$. 
We then set $d_1=7$, and the total dimension $d_z$ of the random space to be $15$. We solve the Boltzmann (high-fidelity) equation by using the asymptotic-preserving scheme developed in^^>\cite{FJ}, with the second-order MUSCL scheme^^>\cite{LeVeque} for the spatial discretization, and fast spectral method^^>\cite{FastSpectral} for the collision operator implementation. Here we set $\Delta x=0.01$, $N_v^h=24$, and the final time $t=0.15$. We employ the Euler equation as the low-fidelity model, and solve it with the same spatial and temporal resolution as that for the Boltzmann, while using $N_v^l=12$. In this test, the fluid regime, i.e. $\e=10^{-4}$ is chosen. The Figure \ref{Fig6-conv} suggests a fast convergence of $L^2$ errors between the high-fidelity and bi-fidelity solutions. With 10 high-fidelity runs, the bi-fidelity approximation reaches an accuracy level of $\mathcal{O}(10^{-3})$, which is remarkable for a 15-dimensional random variable problem. 
To further illustrate the performance of our bi-fidelity method, we compare the high-fidelity, low-fidelity and the bi-fidelity solutions for a randomly selected sample point $z$ by using $M=10$. One observes that the high-fidelity and bi-fidelity solutions match really well, whereas the low-fidelity solutions are quite off at some spatial points. 

Regarding the moments quantities in the random space, Figure \ref{Fig6-mv} shows that the mean and standard deviation of bi-fidelity approximations for $\rho$, $u$ and $T$ match well with the high-fidelity solutions with $10$ high-fidelity runs. 
These results indicate that even though the Euler model may not be accurate in the physical space, it still can capture well the solution behavior of the Boltzmann equation in the random space. Besides, a significant speedup and memory savings are quite noticeable in this case. 
\begin{figure}[!ht]
\includegraphics[width=0.51\linewidth]{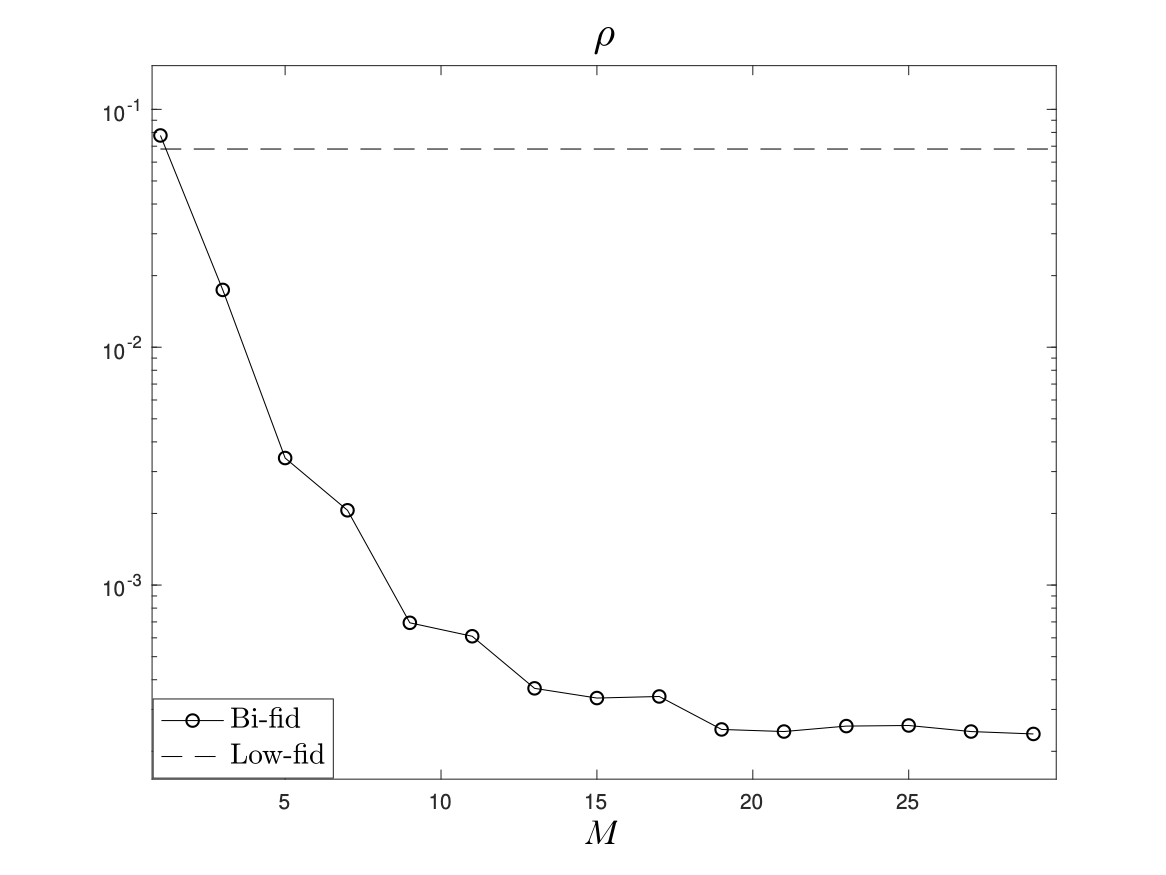}\hskip -.5cm
\includegraphics[width=0.51\linewidth]{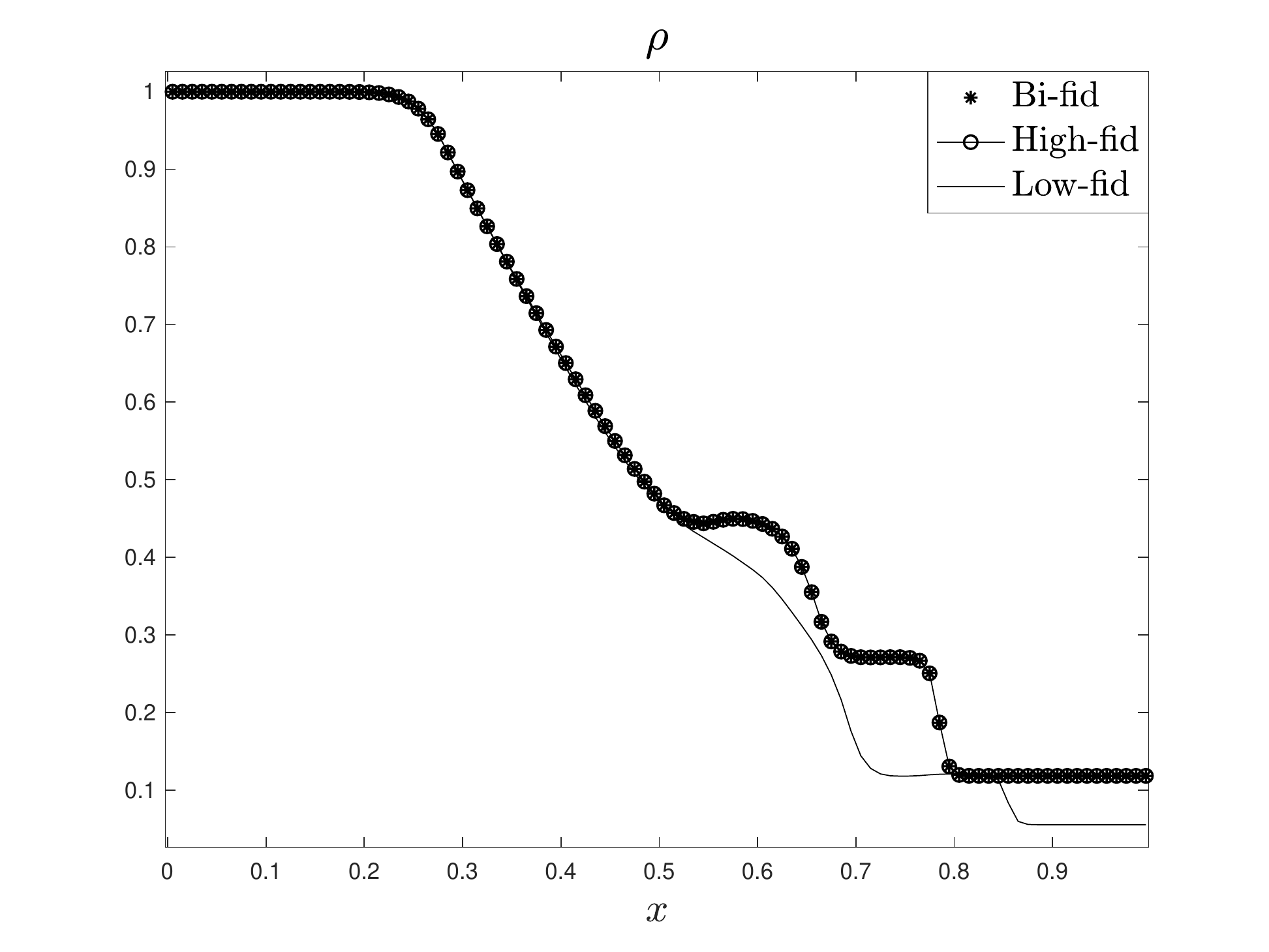}
\includegraphics[width=0.51\linewidth]{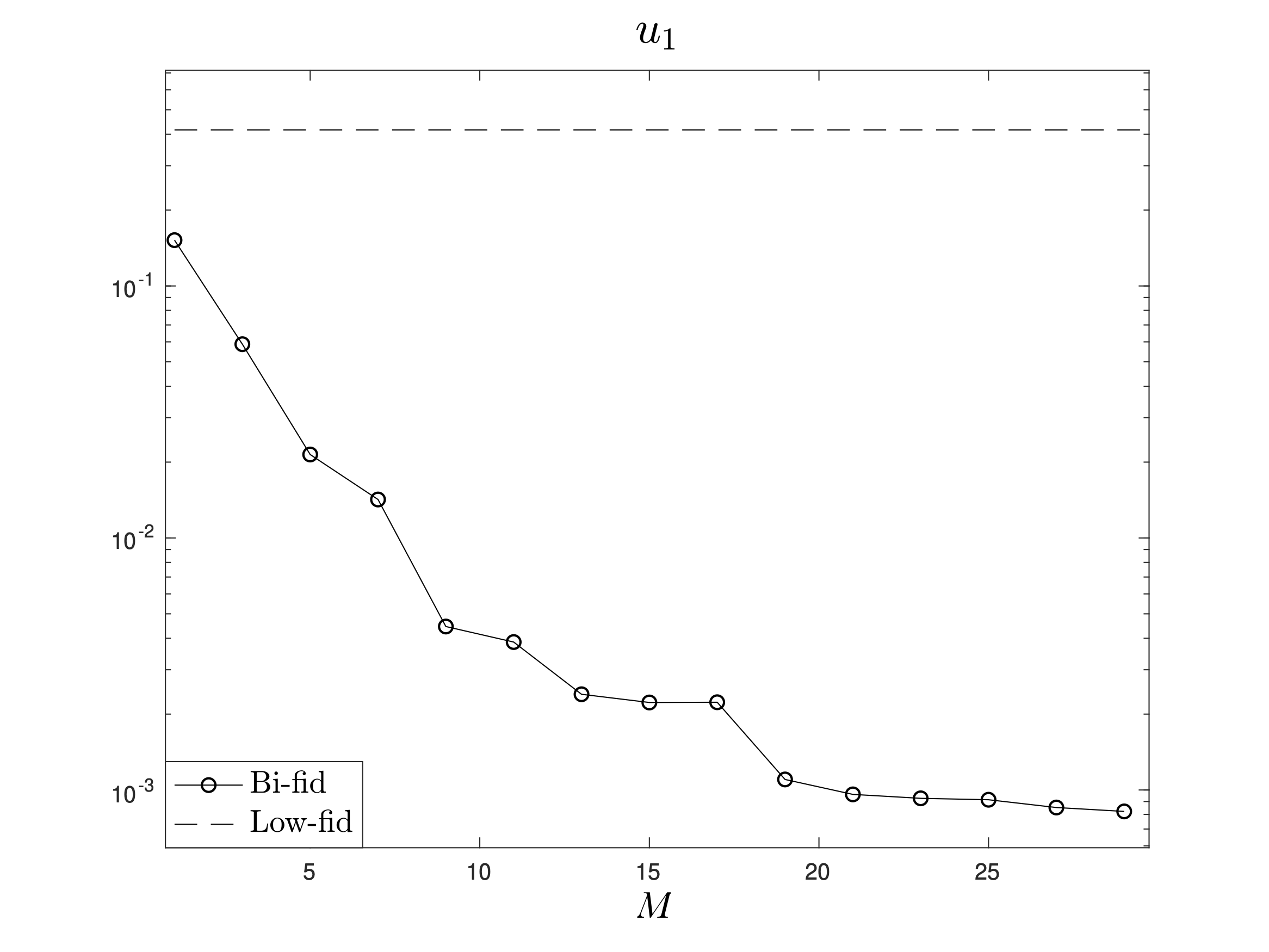}\hskip -.5cm
\includegraphics[width=0.51\linewidth]{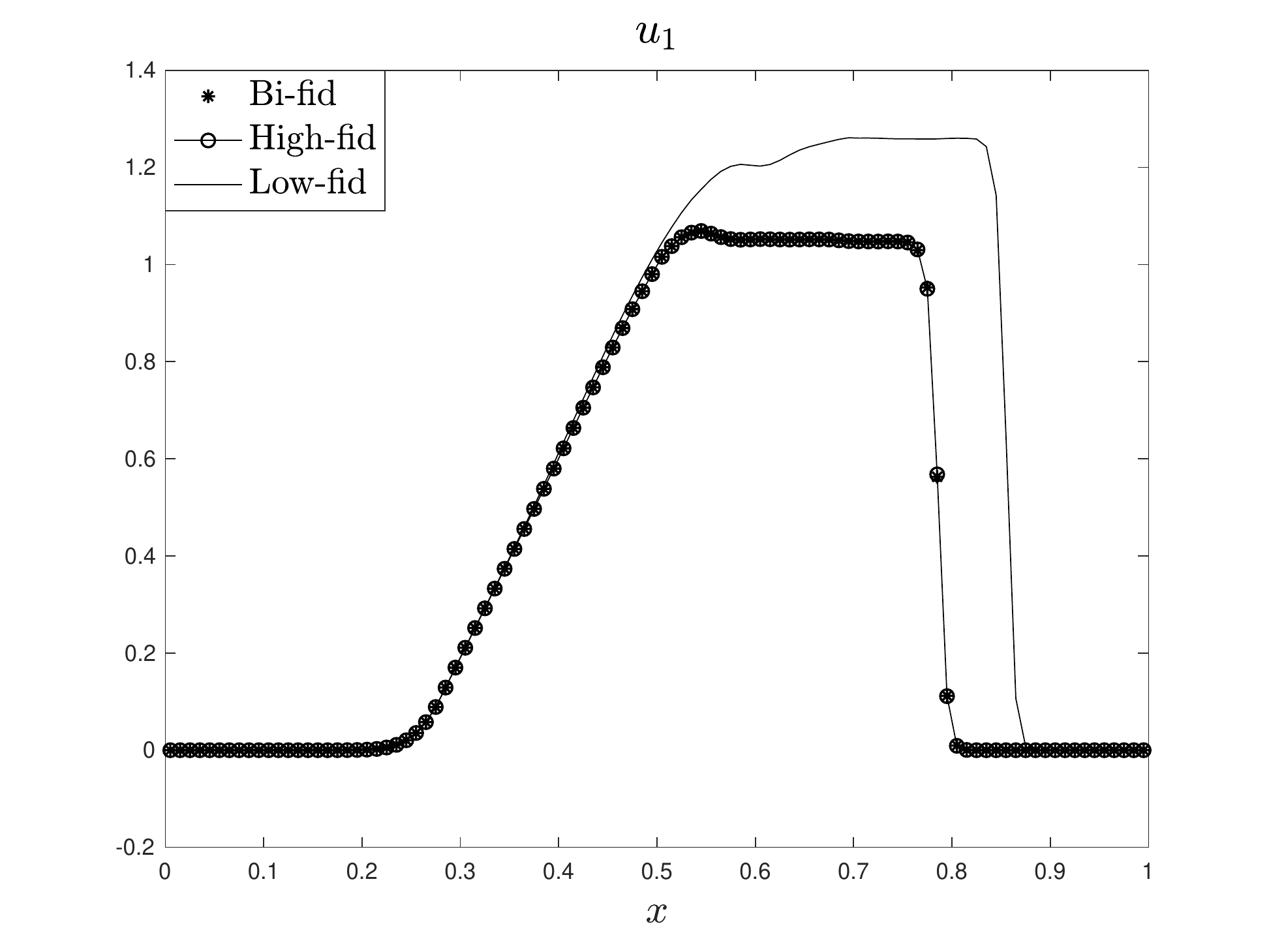}
\includegraphics[width=0.51\linewidth]{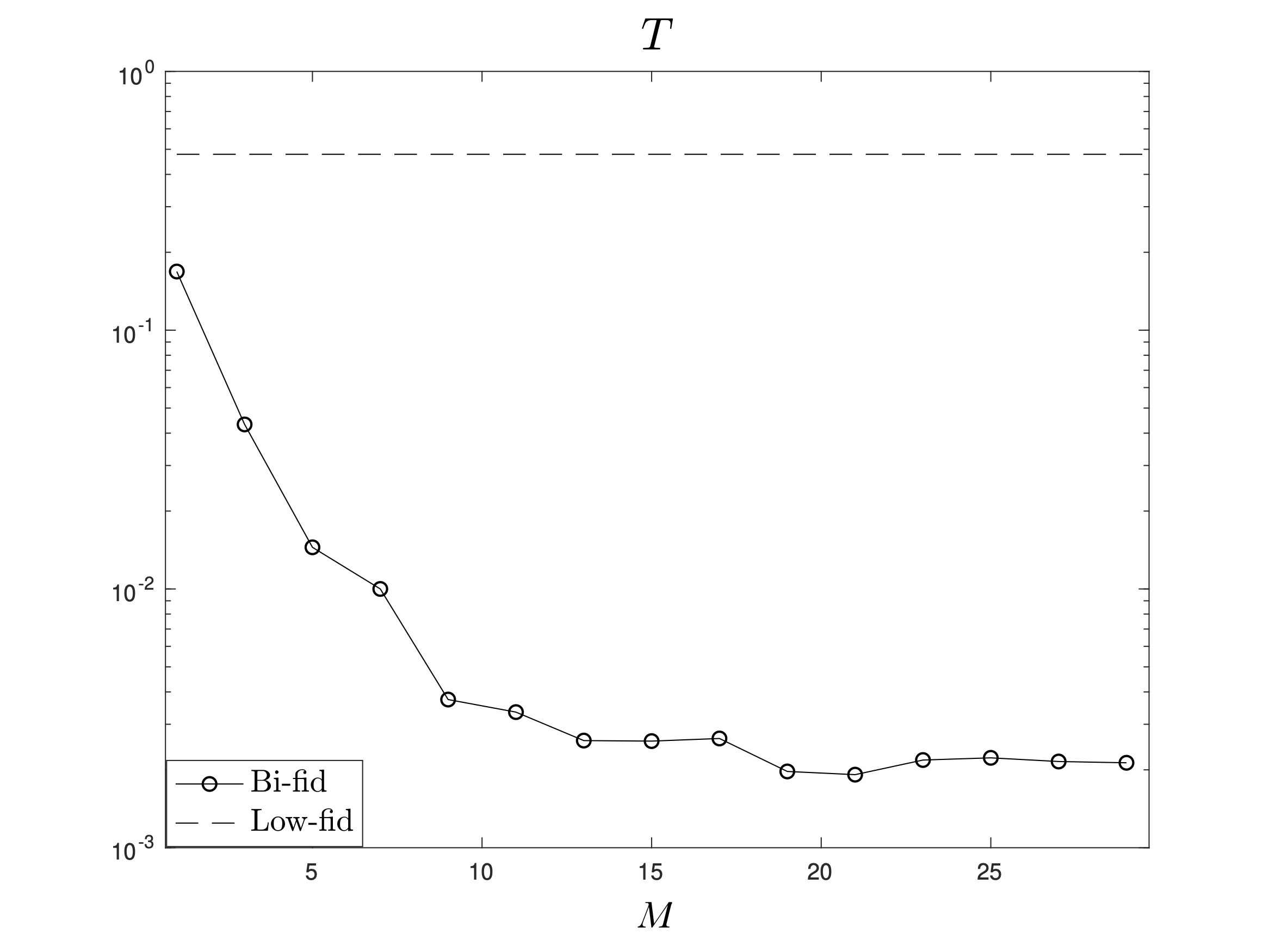}\hskip -.5cm
\includegraphics[width=0.51\linewidth]{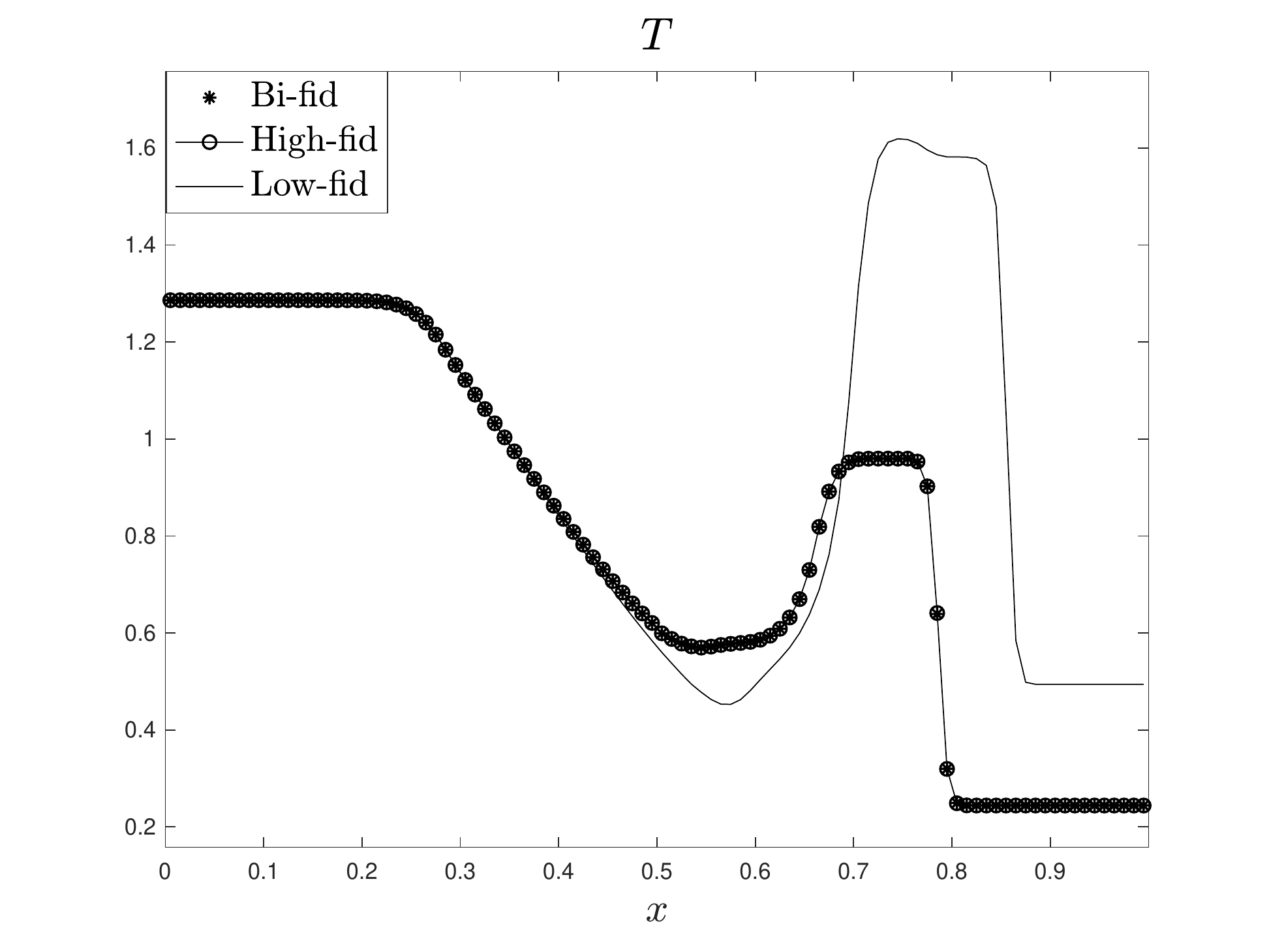}
\caption{{BFSC method for the Boltzmann equation.} Results of the Sod shock tube test. (Left) The mean $L^2$ errors between high-fidelity and 
low-fidelity or bi-fidelity solutions with respect to the number of high-fidelity runs; (Right) Comparison of the low-fidelity solution ($N_v^l=12$), high-fidelity solutions ($N_v^l=24$), and the corresponding bi-fidelity approximations $M=10$ for a fixed $z$. }
\label{Fig6-conv}
\end{figure}

\begin{figure}[!ht]
\centering
\includegraphics[width=0.51\linewidth]{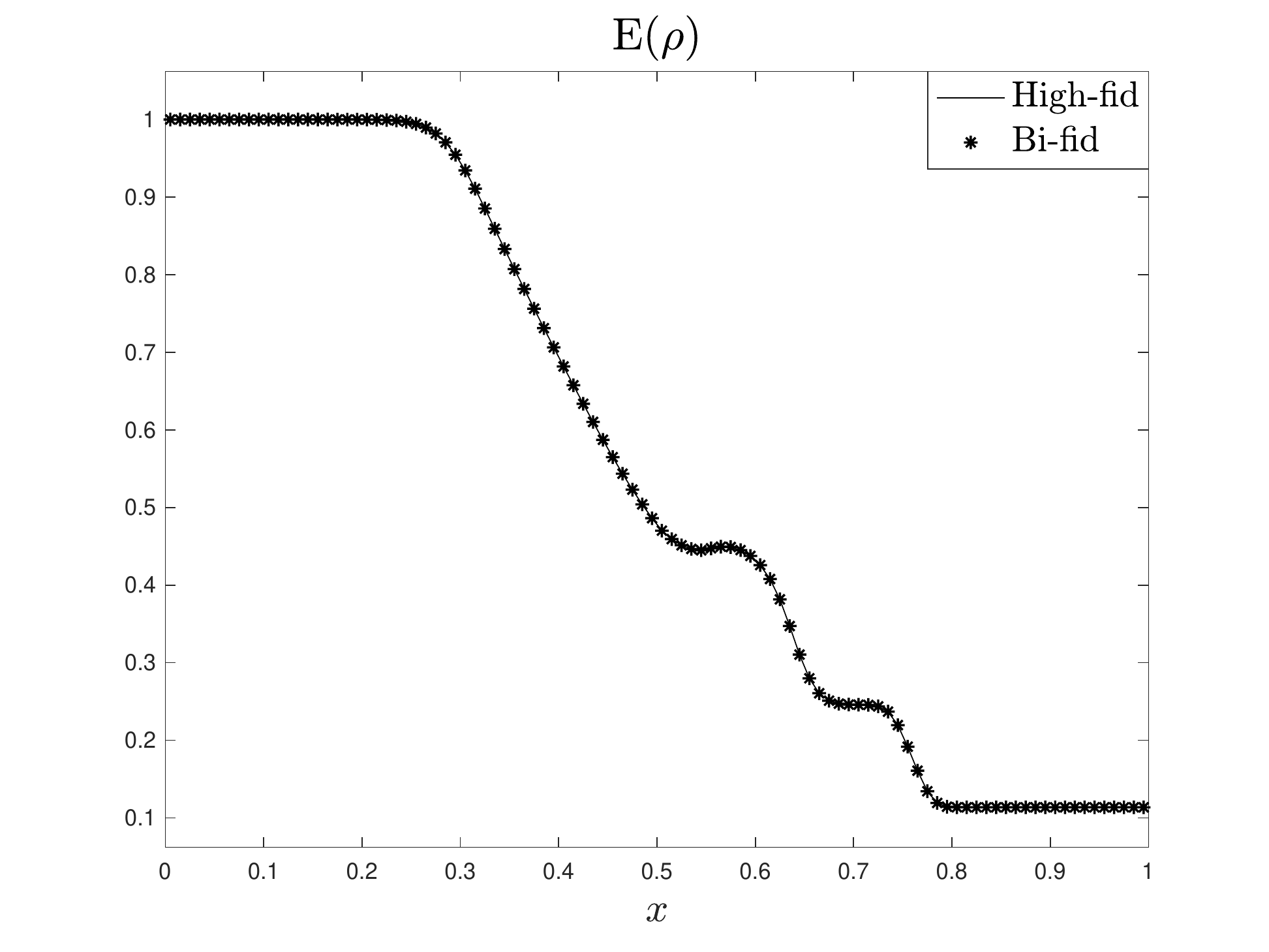}\hskip -.5cm  
\includegraphics[width=0.51\linewidth]{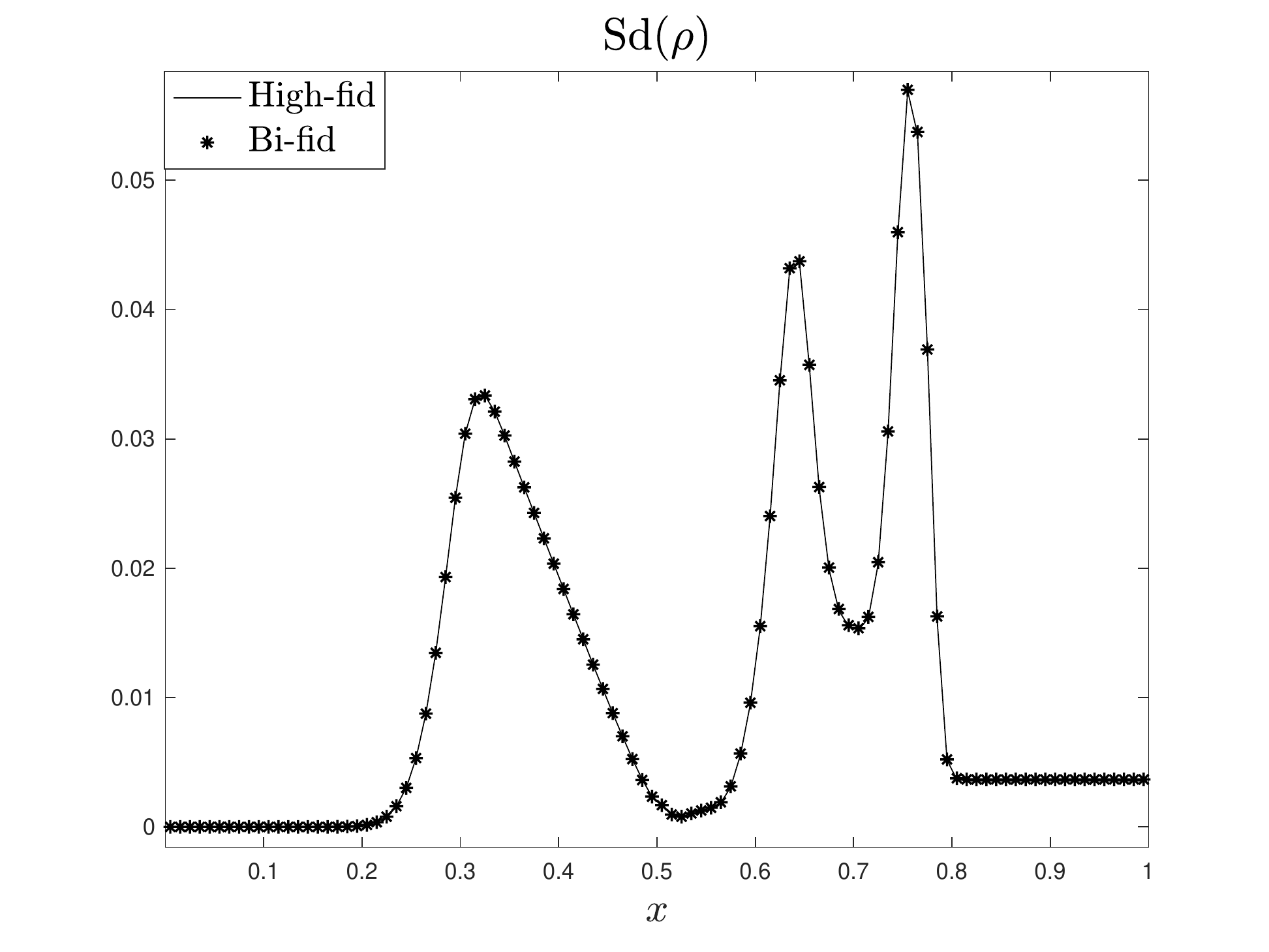}
\includegraphics[width=0.51\linewidth]{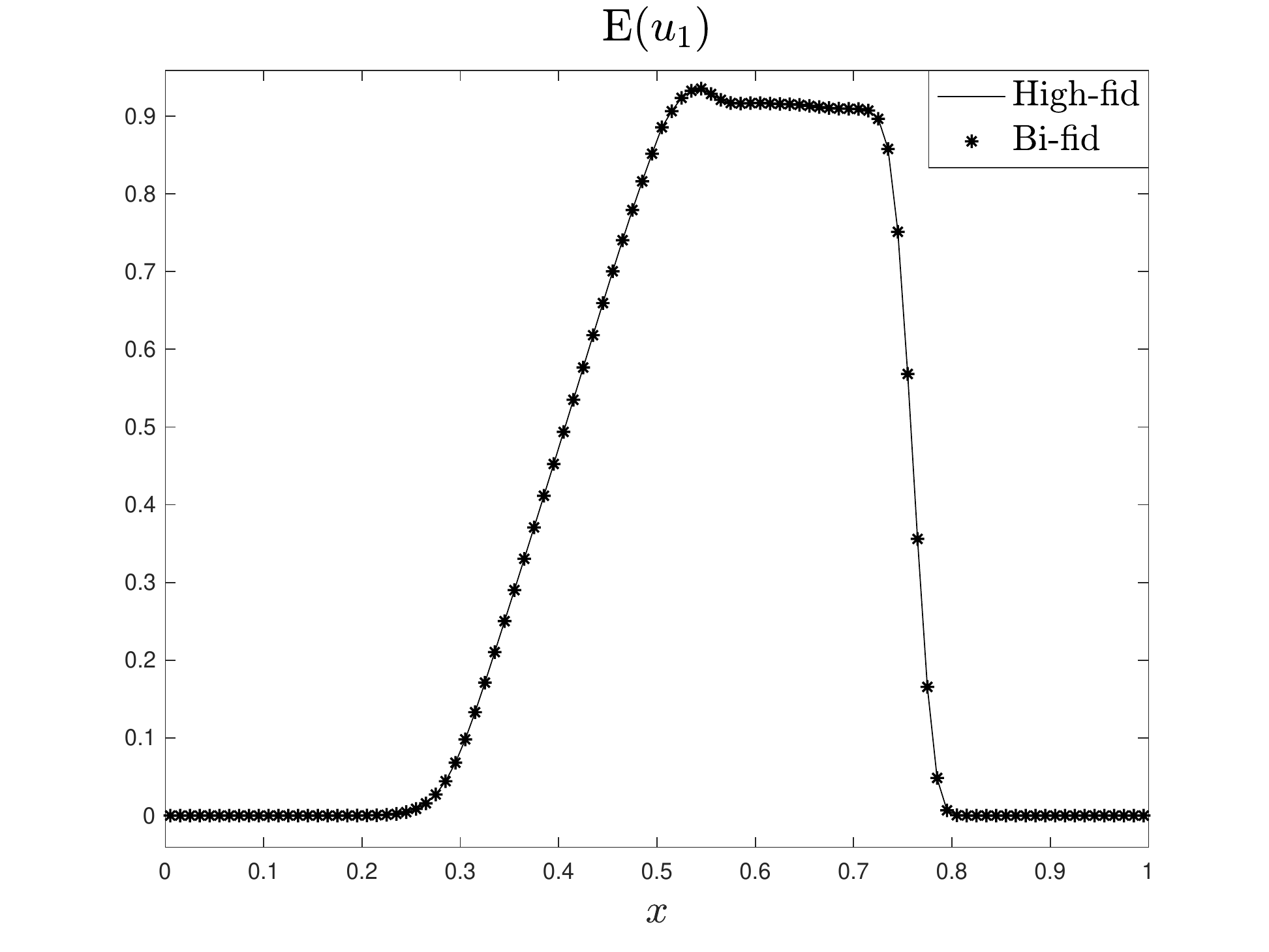}\hskip -.5cm
\includegraphics[width=0.51\linewidth]{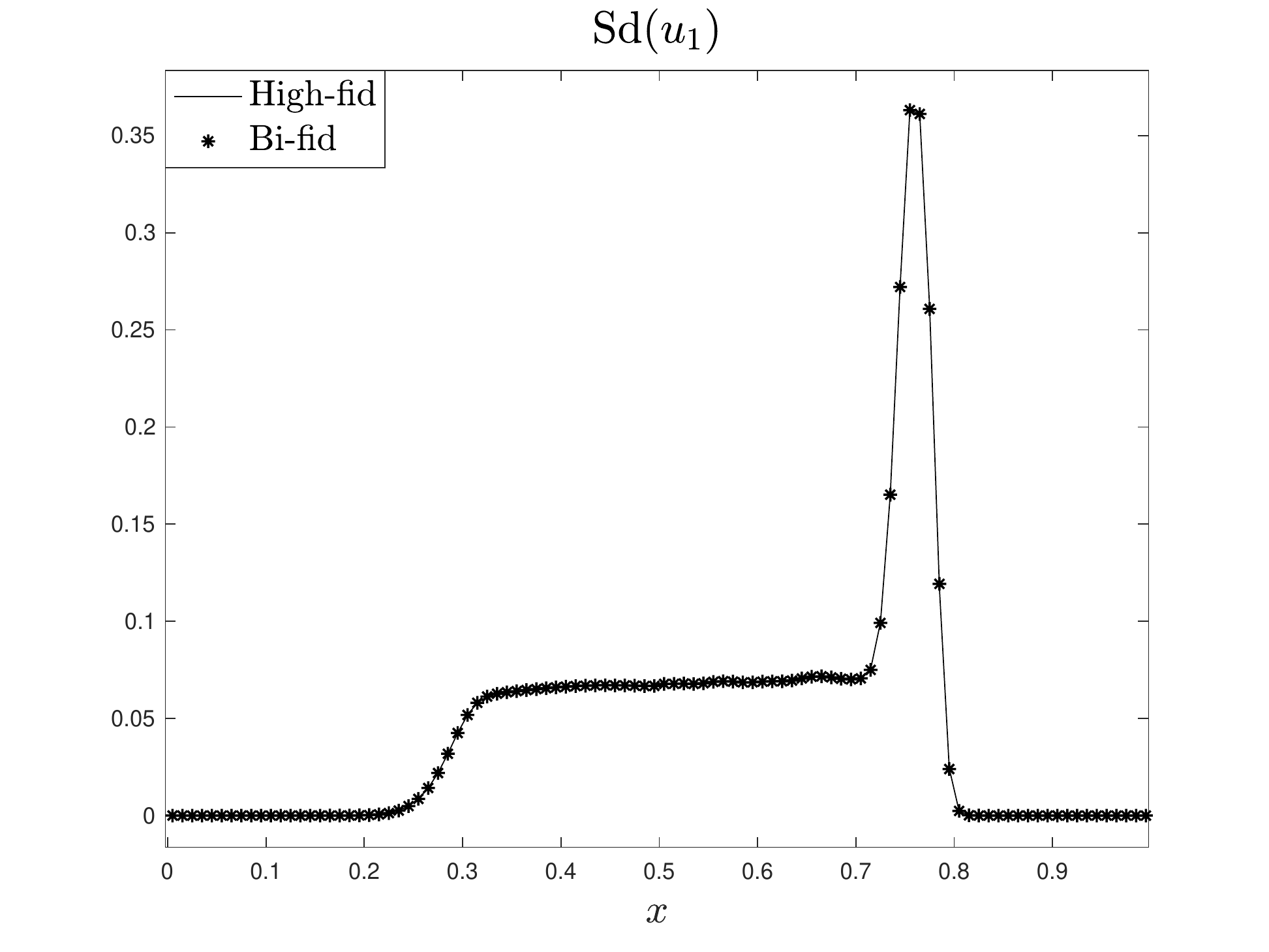}
\includegraphics[width=0.51\linewidth]{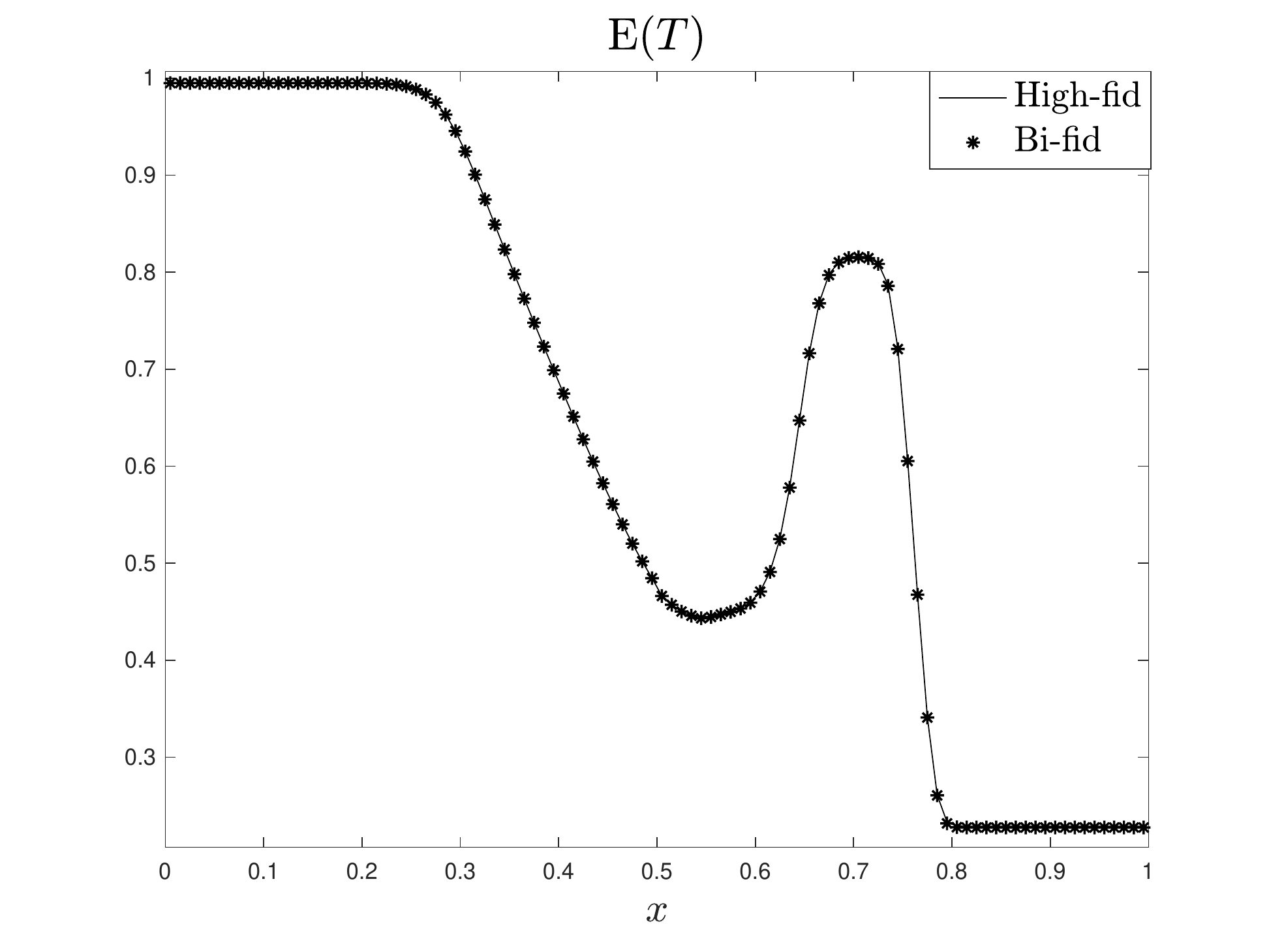}\hskip -.5cm
\includegraphics[width=0.51\linewidth]{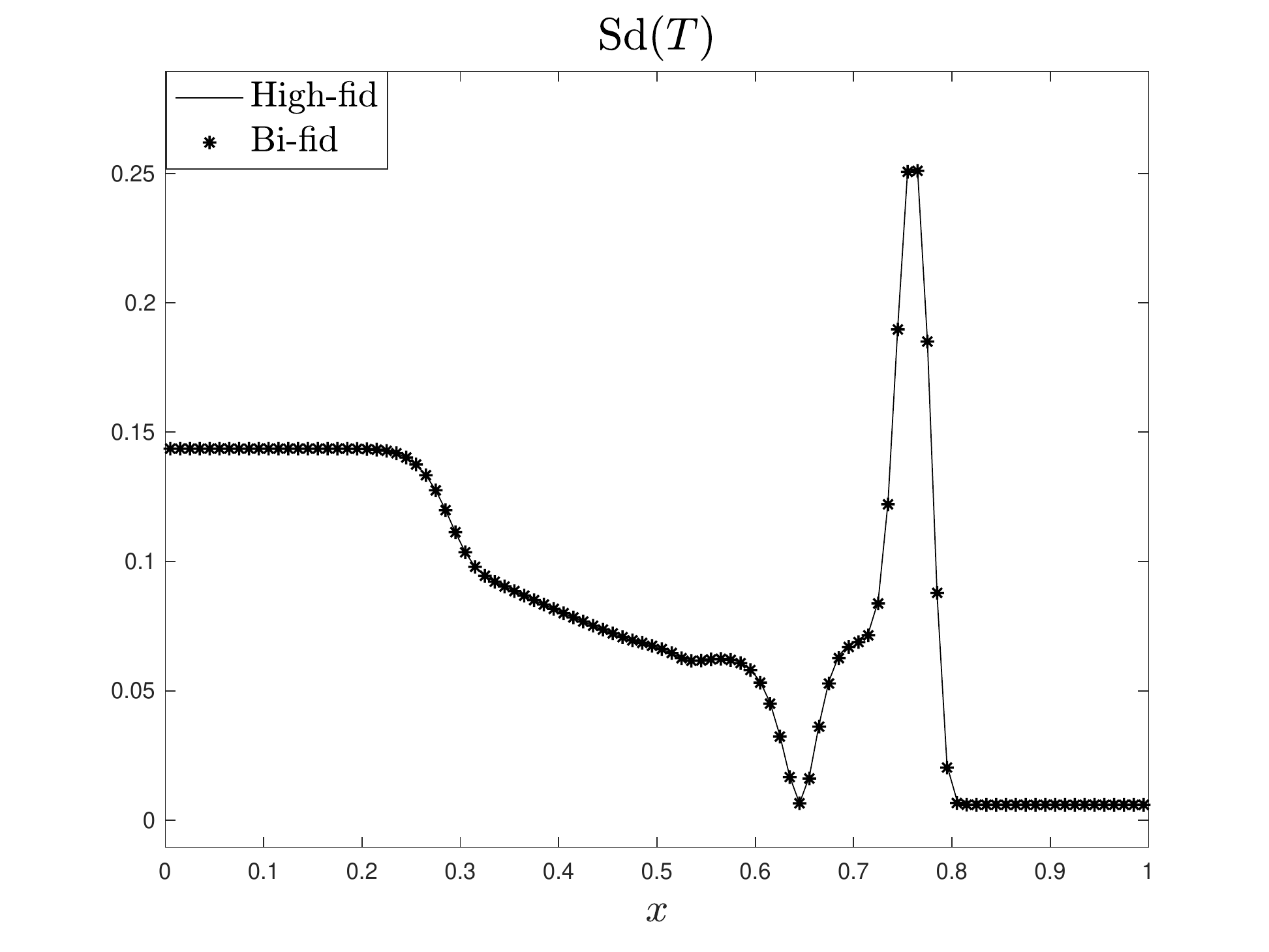}
\caption{{BFSC method for the Boltzmann equation.} Results of the Sod shock tube test. The mean and standard deviation of $\rho$, $u_1$, $T$ of high-fidelity and bi-fidelity solutions with $M=10$. } 
\label{Fig6-mv}
\end{figure}

We now consider a mixed regime problem with the Knudsen number 
$\e$ varying in space, see Figure~\ref{figmixe}, defined by
\begin{equation}
\label{mixe}
\e(x) = 10^{-3} + \frac{1}{2}\left[\tanh \left(1-\frac{11}{2}(x-0.5) \right) + \tanh \left(1+\frac{11}{2}(x-0.5) \right)\right].
\end{equation}
Let initial data mimic the Karhunen-Loeve expansion of the random field: 
\begin{align}
\label{IC}
\left\{
\begin{array}{l}
\displaystyle \rho_0(x, {z}^{\rho})= \frac{1}{3}\left( 2 + \sin(2\pi x) + 0.2\sum_{k=1}^{d_1} \sin[2\pi(k+1)x]\, \frac{z_k^{\rho}}{2 k}\right), \\[10pt]
\displaystyle {u}_0 = (0.2, 0), \\[10pt]
\displaystyle T_0(x, {z}^T) = \frac{1}{4}\left( 3 + \cos(2\pi x) + 0.2\sum_{k=1}^{d_1} \cos[2\pi(k+1)x]\, \frac{z_k^T}{2 k}\right), \\[10pt]
\displaystyle f_0(x, {v}, {z}) = \frac{\rho_0}{4\pi T_0}\left( \exp(-\frac{|{v} - {u}_0|^2}{2 T_0}) + \exp(-\frac{|{v} + {u}_0|^2}{2 T_0})\right), 
\end{array}\right.
\end{align}
where $d_1=7$. The collisional cross section is also assumed random, 
\begin{equation}
\label{R-K} 
b(z)=1+0.5z_1^b. 
\end{equation}

\begin{figure}[!ht]
\centering
\includegraphics[width=0.6\linewidth]{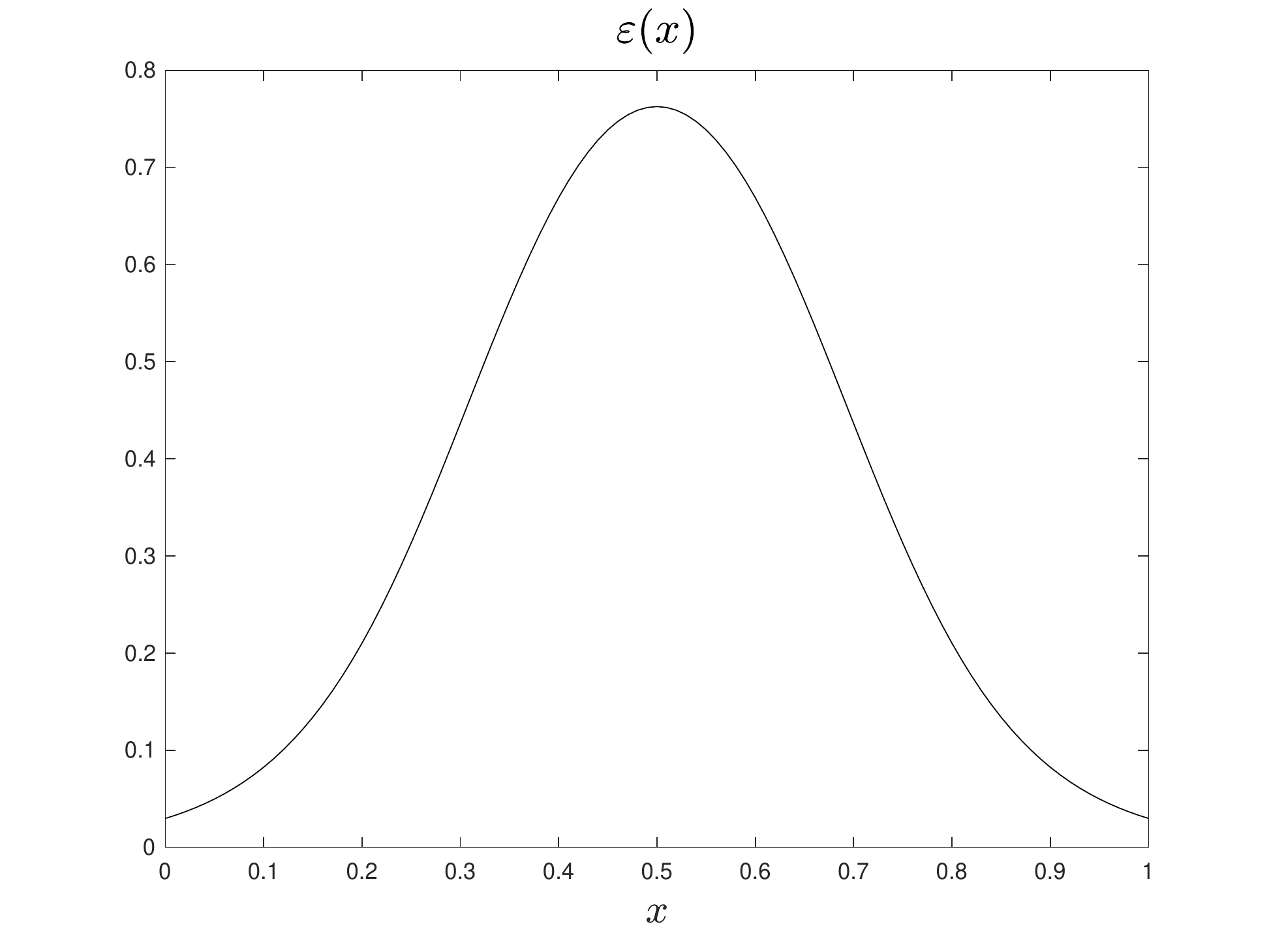}
\caption{{BFSC method for the Boltzmann equation.}  The distribution of $\e(x)$ in \eqref{mixe} for the mixed regime problem.}
\label{figmixe}
\end{figure}

We let $N_v^h=24$ and $N_v^l=8$ in the high-fidelity and low-fidelity solvers. From Figure \ref{Fig5-conv}, we observe a fast convergence of $L^2$ errors between the high-fidelity and bi-fidelity solutions, which saturate quickly when $M$ reaches around $25$. The dotted lines which represent the errors between the high-fidelity and low-fidelity solutions are much larger than that between the high-fidelity and bi-fidelity solutions.
This indicates that even though the low-fidelity solutions alone are relatively inaccurate in the spatial domain, it seems to still behave similarly in the random space. As a result the bi-fidelity approximation based on a small number of high-fidelity runs can reach a reasonable level of accuracy. 

The right column of Figure \ref{Fig5-conv} presents the high-fidelity, low-fidelity and bi-fidelity solutions at a randomly chosen sample point $z$. One can see that the high-fidelity and bi-fidelity solutions match really well, whereas the low-fidelity solutions are not accurate. With $N=1000$ low-fidelity runs of the Euler model, by only $25$ implementation runs of the {AP} solver for the Boltzmann equation, the bi-fidelity approximation is able to capture the Boltzmann solution behaviors well in the random space, up to an accuracy of $10^{-3}$; on the other hand, using the low-fidelity model (Euler equation) alone can not achieve this result, especially under the multiple scaling where $\e$ ranges from $10^{-3}$ to 1. This observation highlights the advantages of the bi-fidelity method.
Figure \ref{Fig5-mv} shows the mean and standard deviation of $\rho$, $u$, $T$ by using $25$ high-fidelity runs. The high-fidelity and bi-fidelity solutions match well, indicating that the bi-fidelity solutions have captured well the characteristics of the macroscopic quantities in the random space. 
Once we construct the bi-fidelity model, the online computational cost can be significantly reduced. In this example, the high-fidelity model (Boltzmann) with $N_v^h=16$ takes about $50$ times computation time of the low-fidelity model on a sequential machine. Since the dominant cost of the bi-fidelity reconstruction lies in the corresponding low-fidelity run, a significant amount of computational cost is reduced in this case. 

\begin{figure}[!ht]
\centering
\includegraphics[scale=0.08]{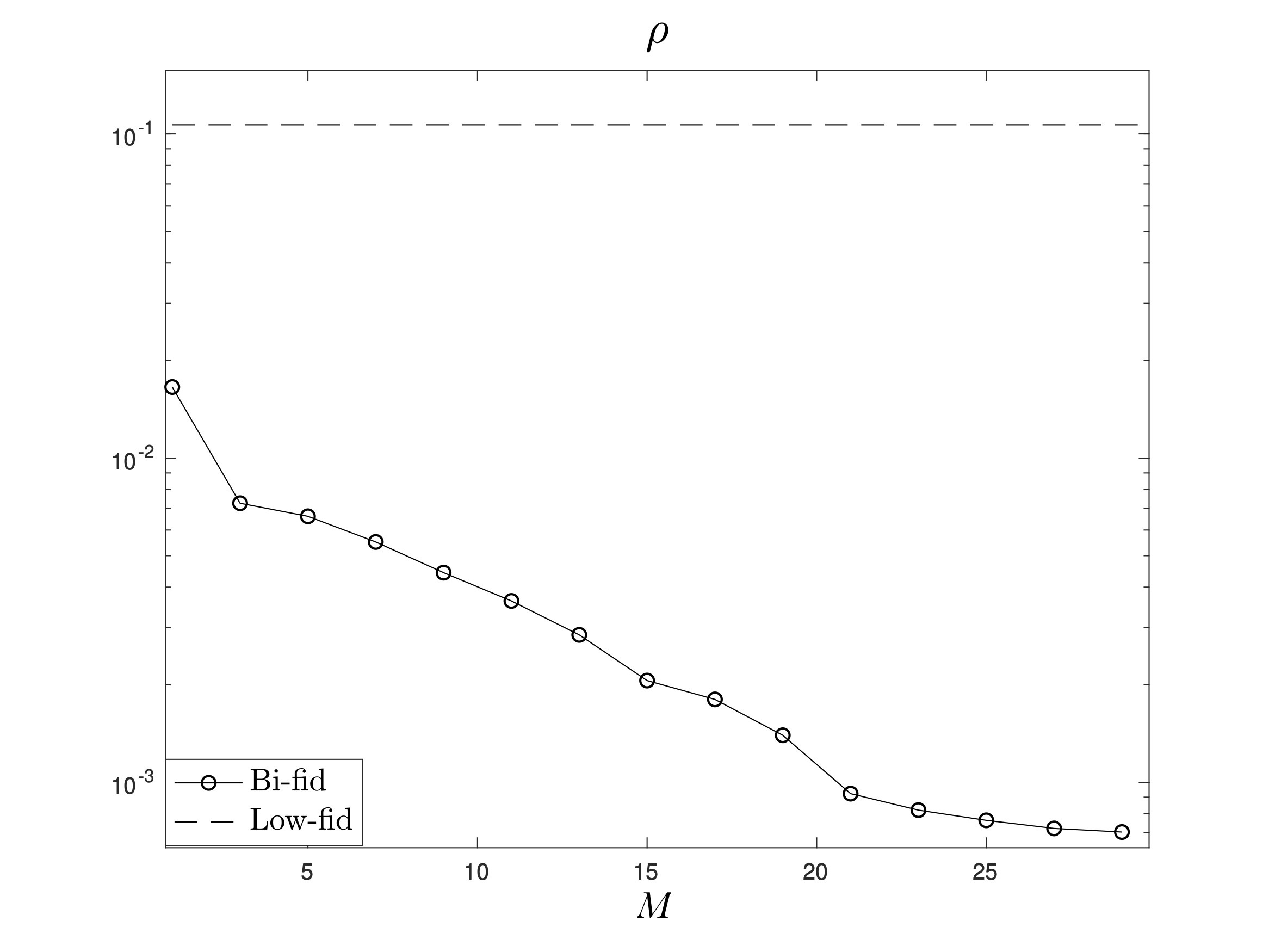}\hskip -.5cm
\includegraphics[scale=0.33]{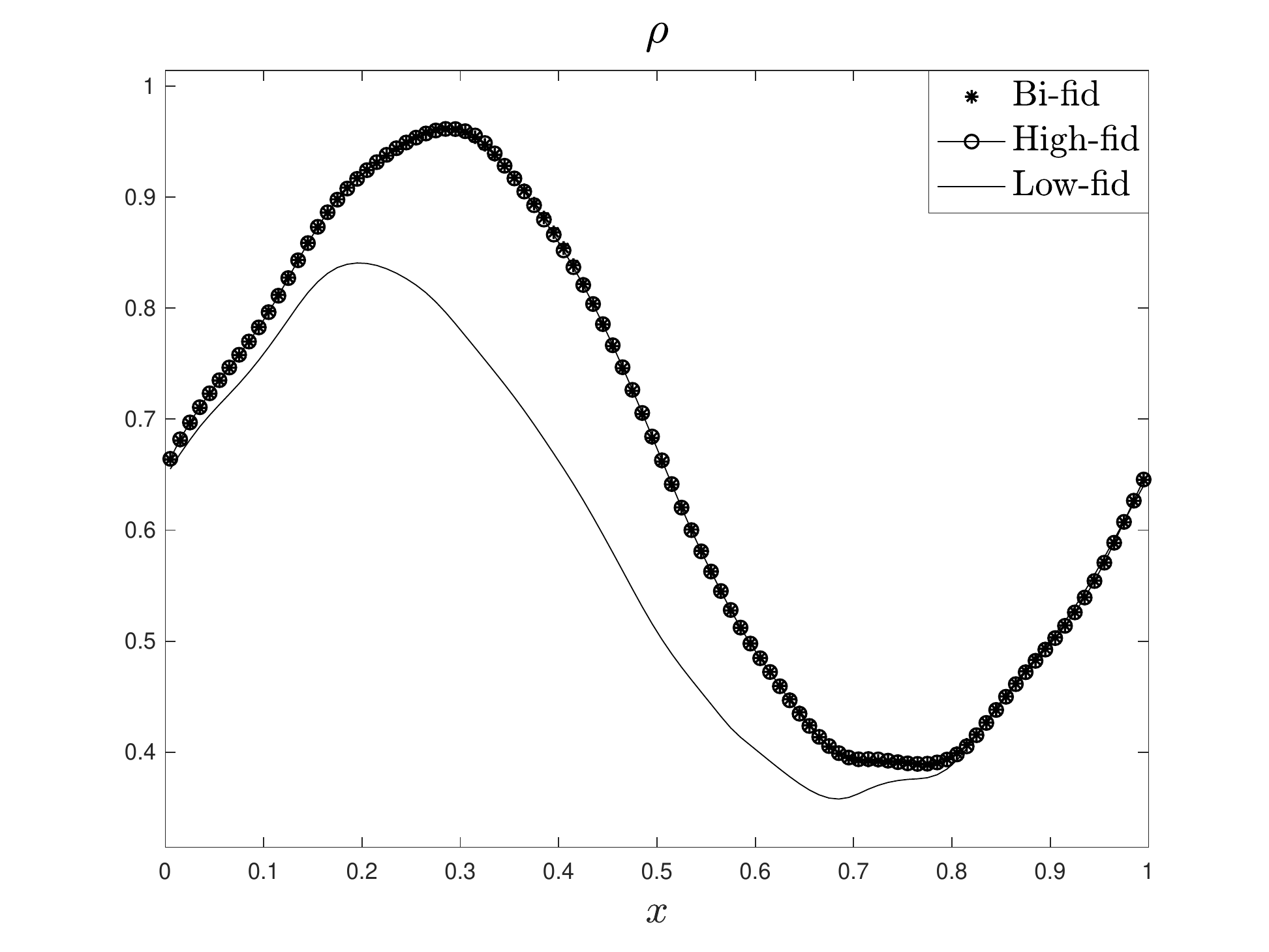}
\includegraphics[scale=0.08]{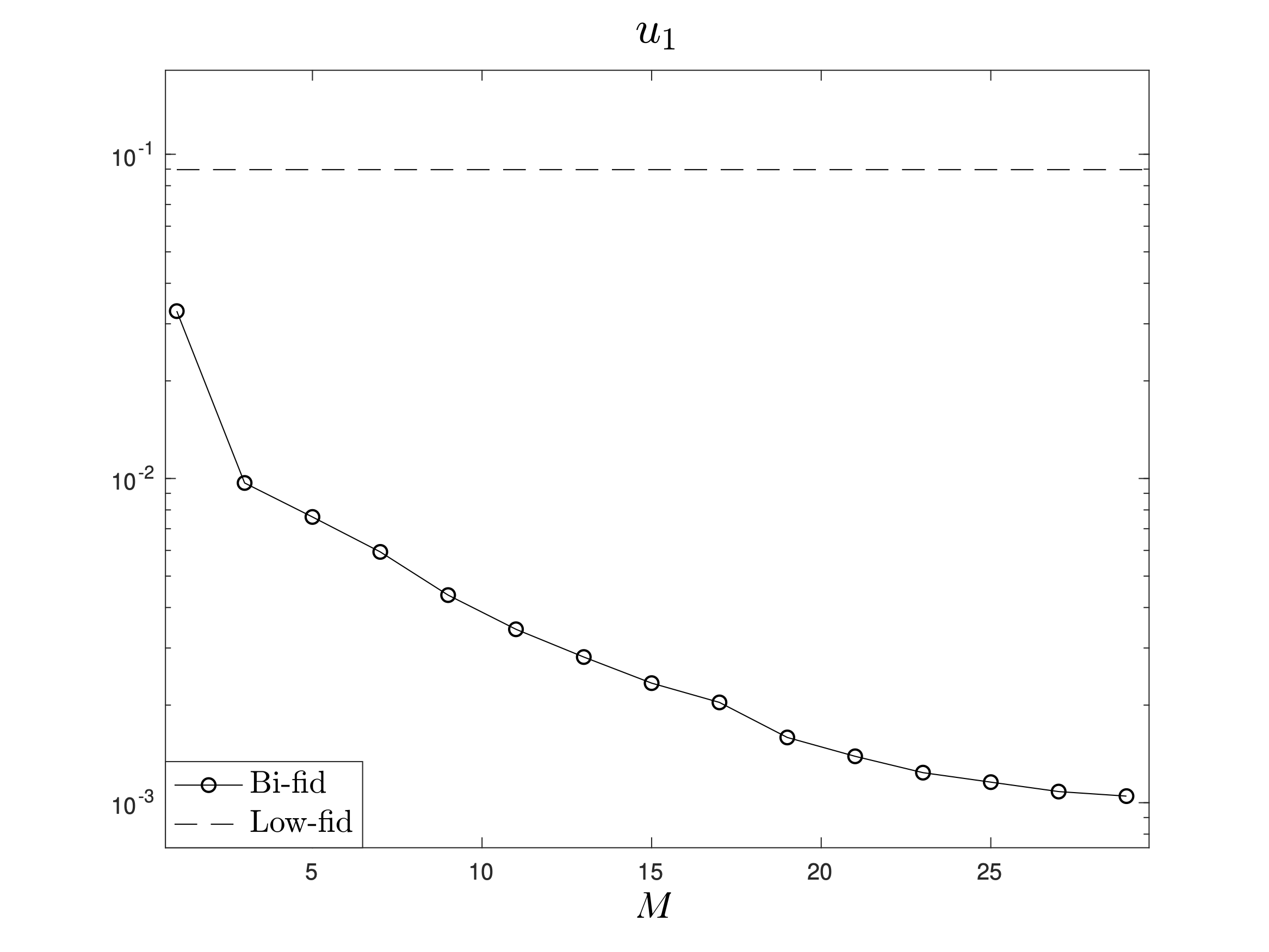}\hskip -.5cm
\includegraphics[scale=0.33]{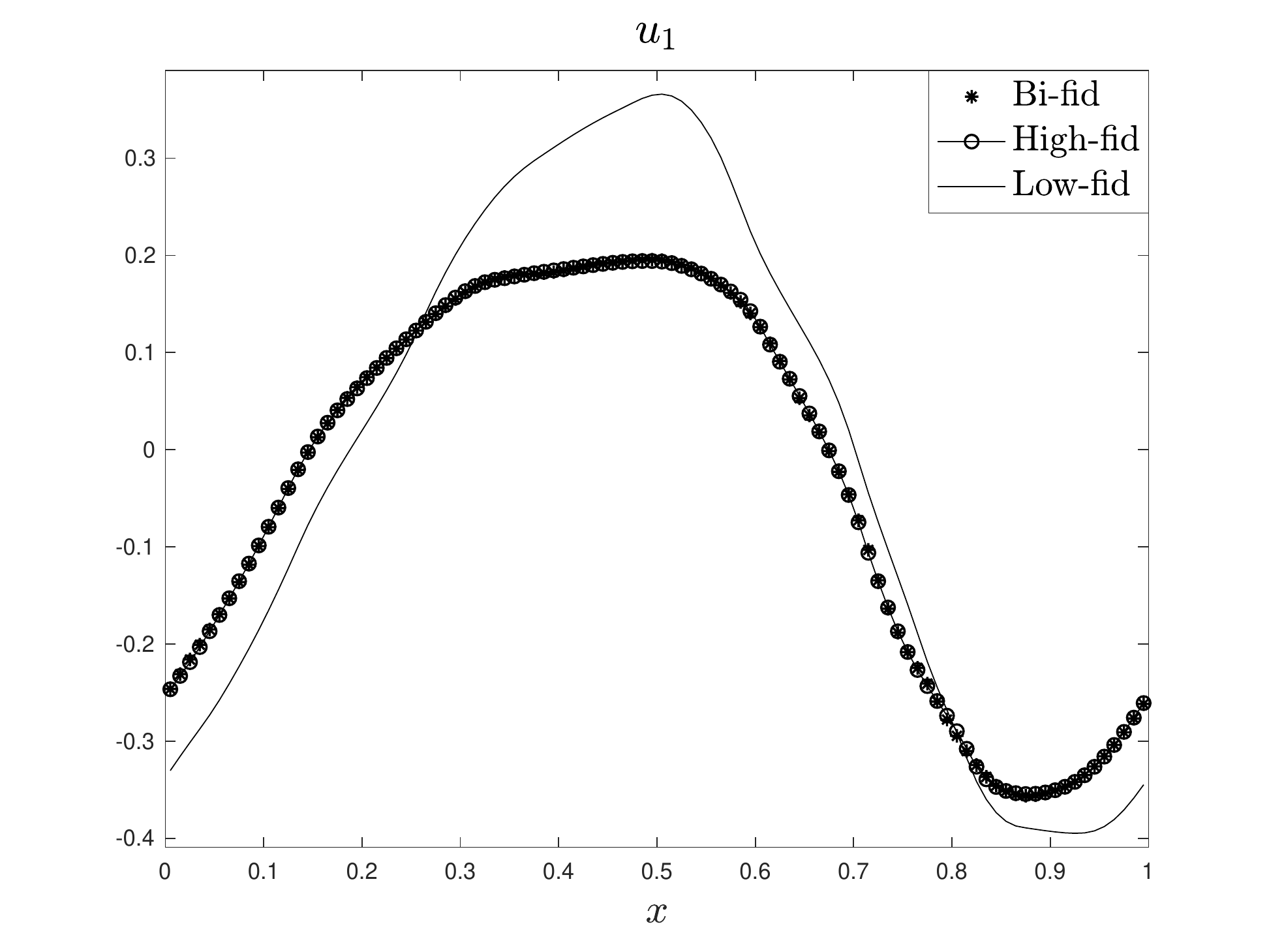}
\includegraphics[scale=0.08]{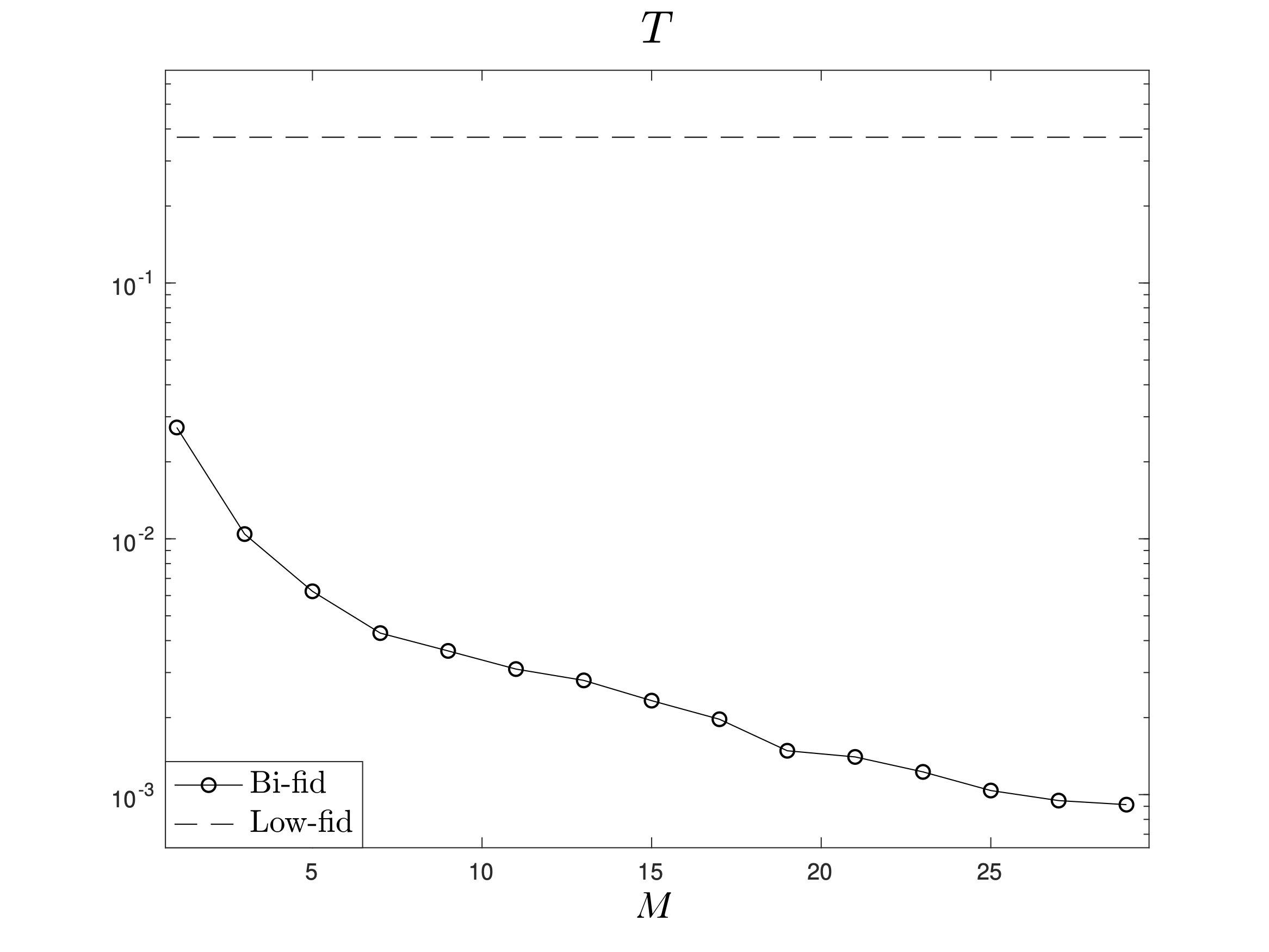}\hskip -.5cm
\includegraphics[scale=0.33]{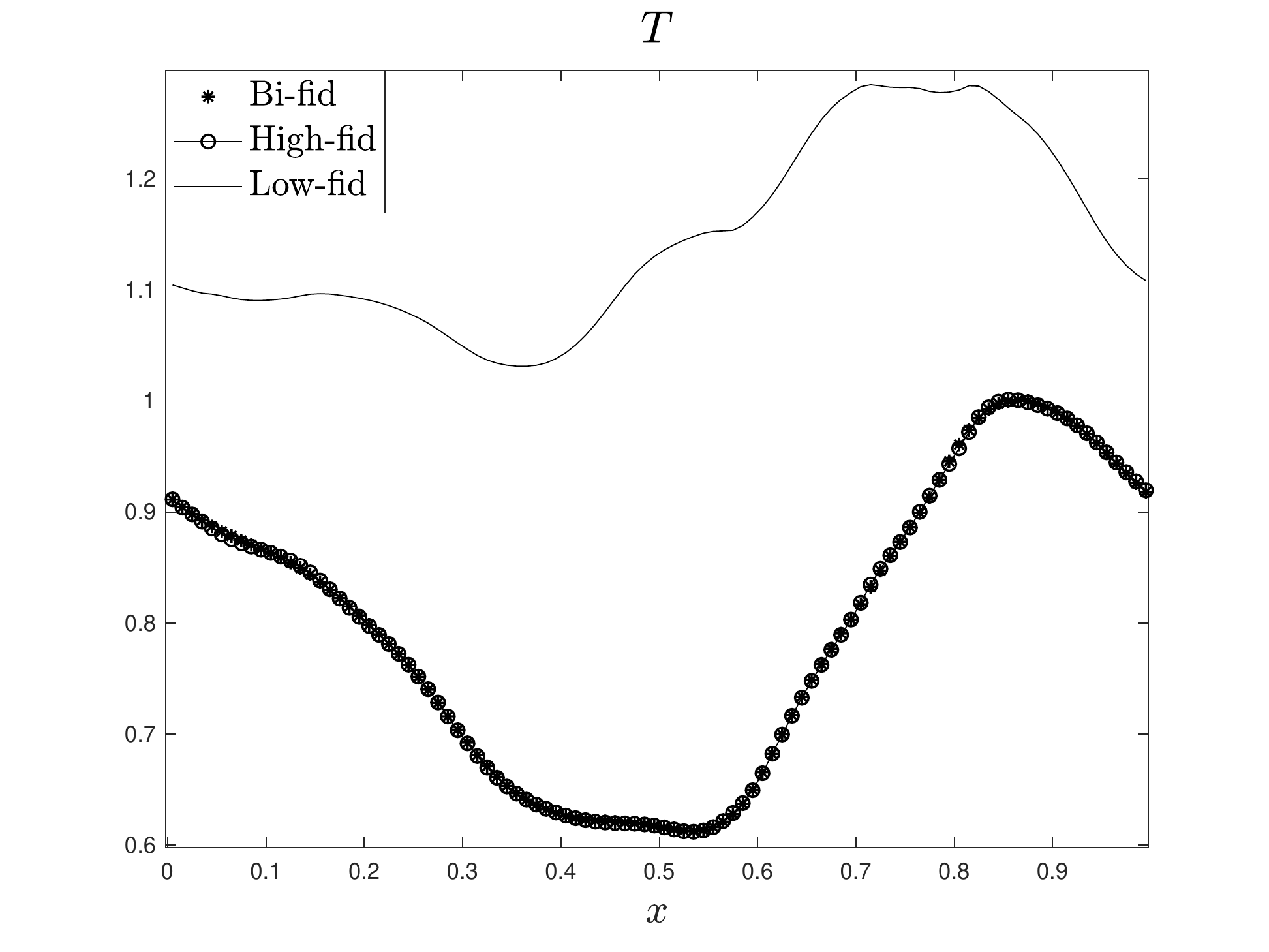}
\caption{{BFSC method for the Boltzmann equation.}  Results of the mixed regime test. (Left) The mean $L^2$ errors between high-fidelity and 
low- or bi-fidelity solutions with respect to the number of high-fidelity run; (Right) Comparison of the low-fidelity solution ($N_v^l=8$), high-fidelity solutions ($N_v^l=16$), and the corresponding bi-fidelity approximations $M=25$ for a fixed $z$. 
}
\label{Fig5-conv}
\end{figure}

\begin{figure}[!ht]
\includegraphics[width=0.5\linewidth]{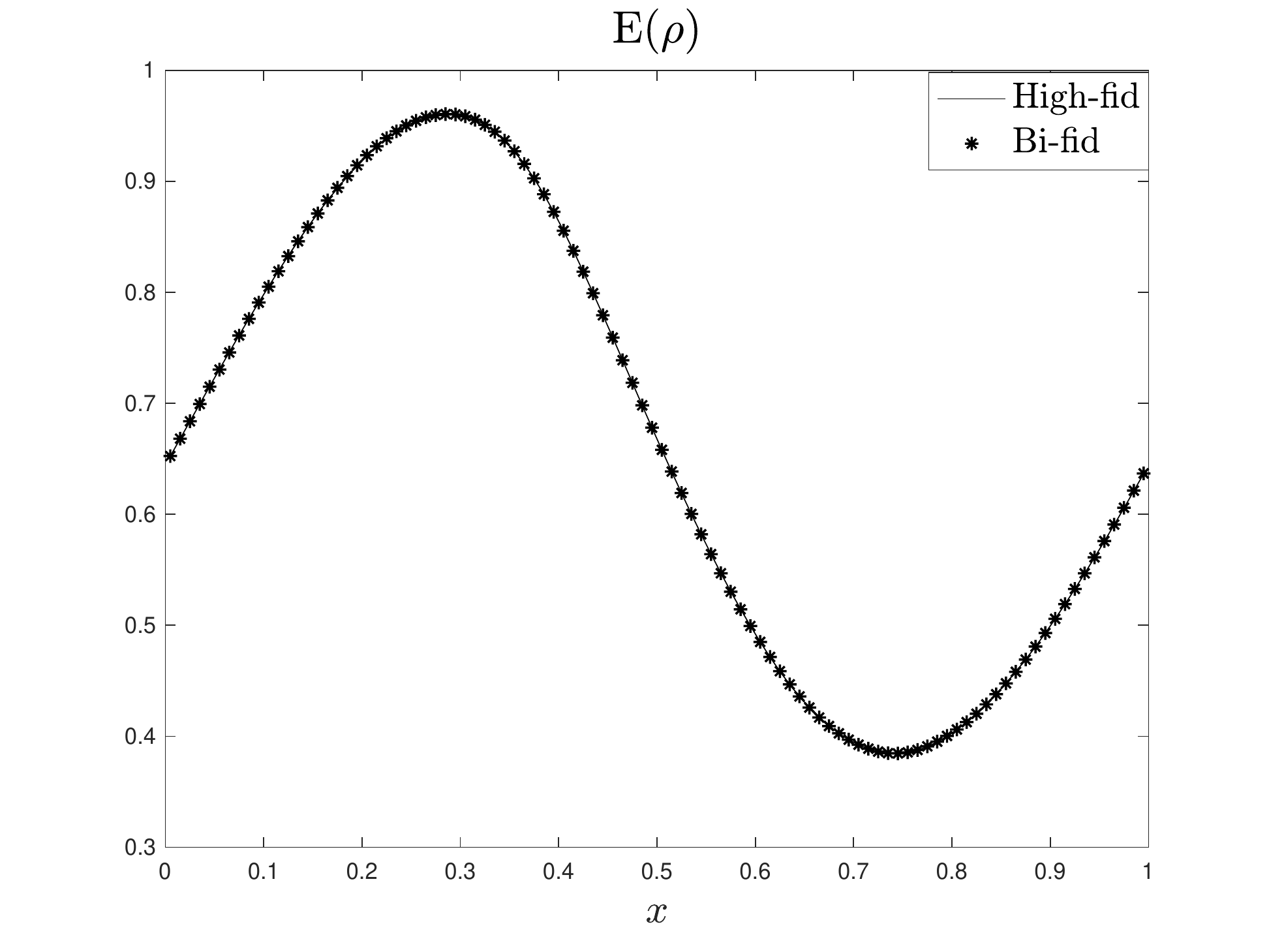}\hskip -.5cm 
\includegraphics[width=0.5\linewidth]{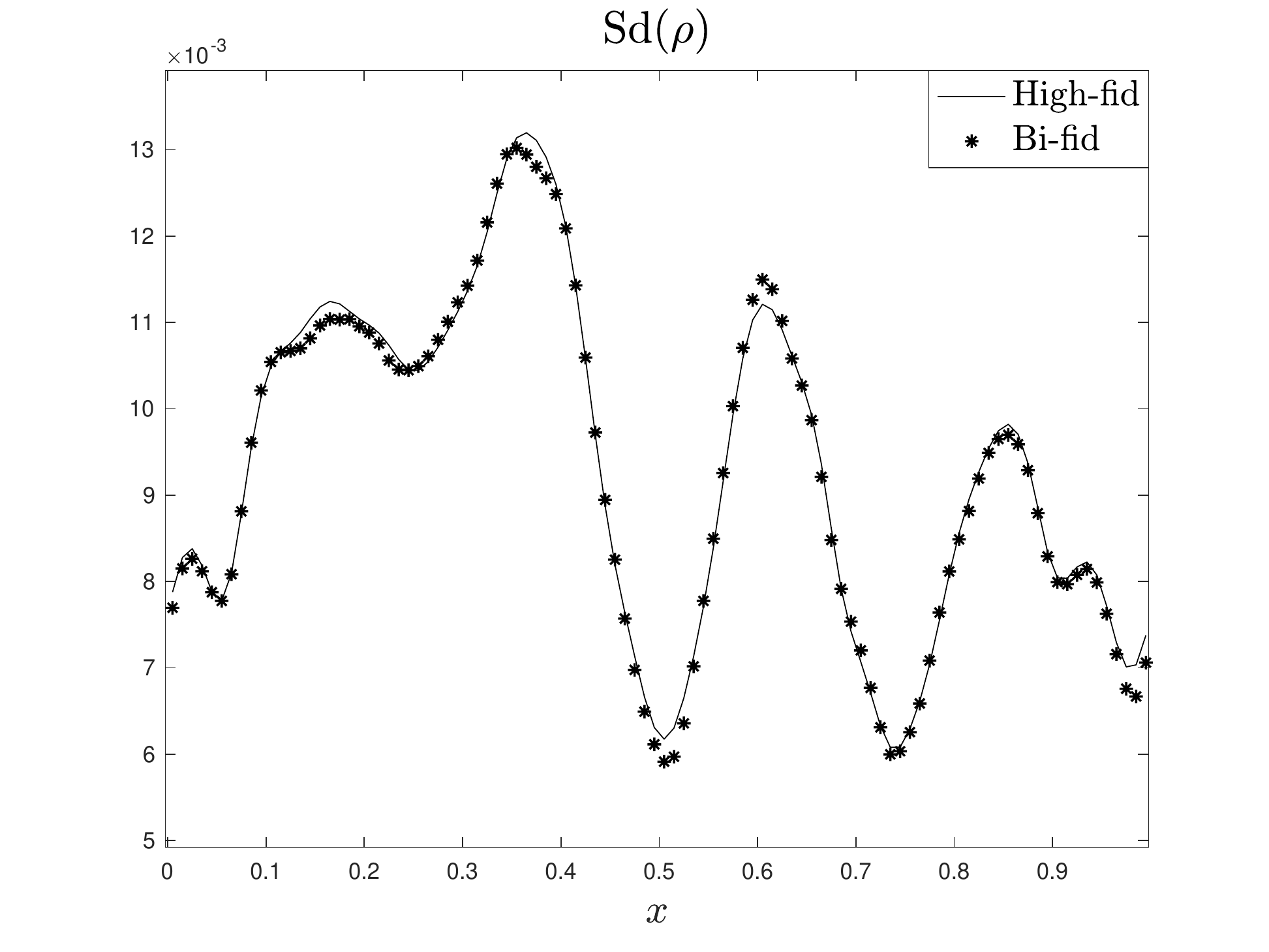}
\includegraphics[width=0.5\linewidth]{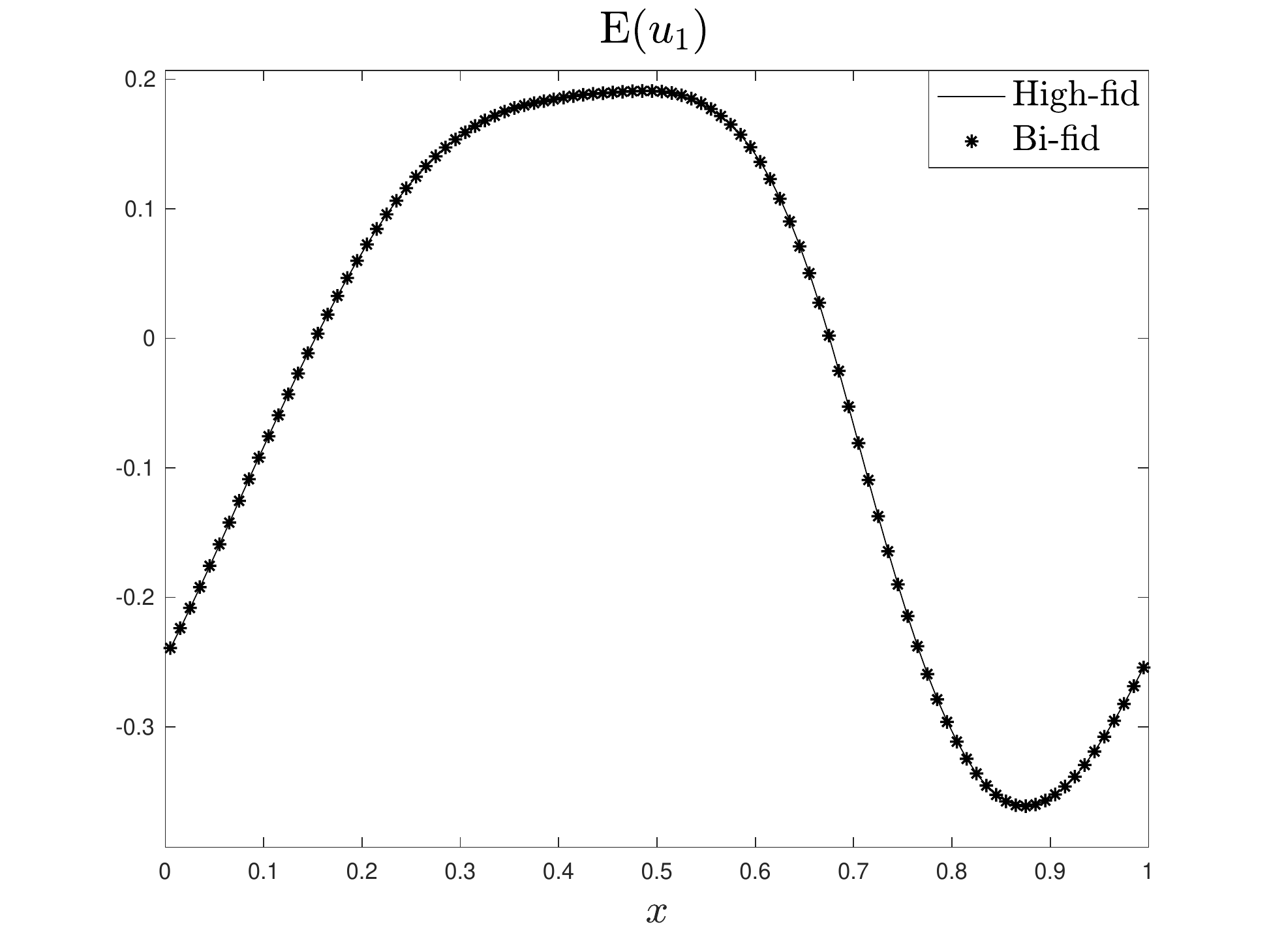}\hskip -.5cm
\includegraphics[width=0.5\linewidth]{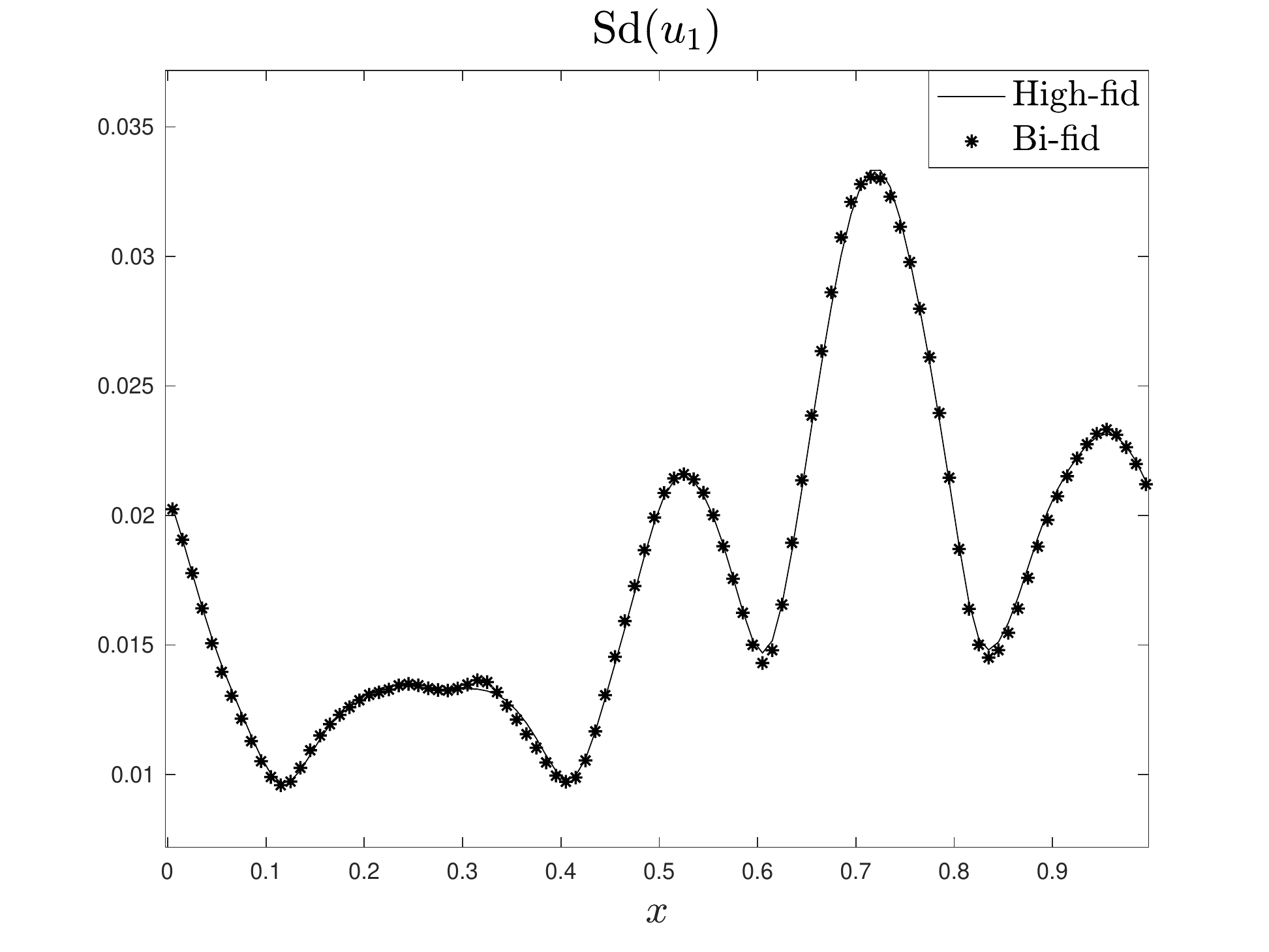}
\includegraphics[width=0.5\linewidth]{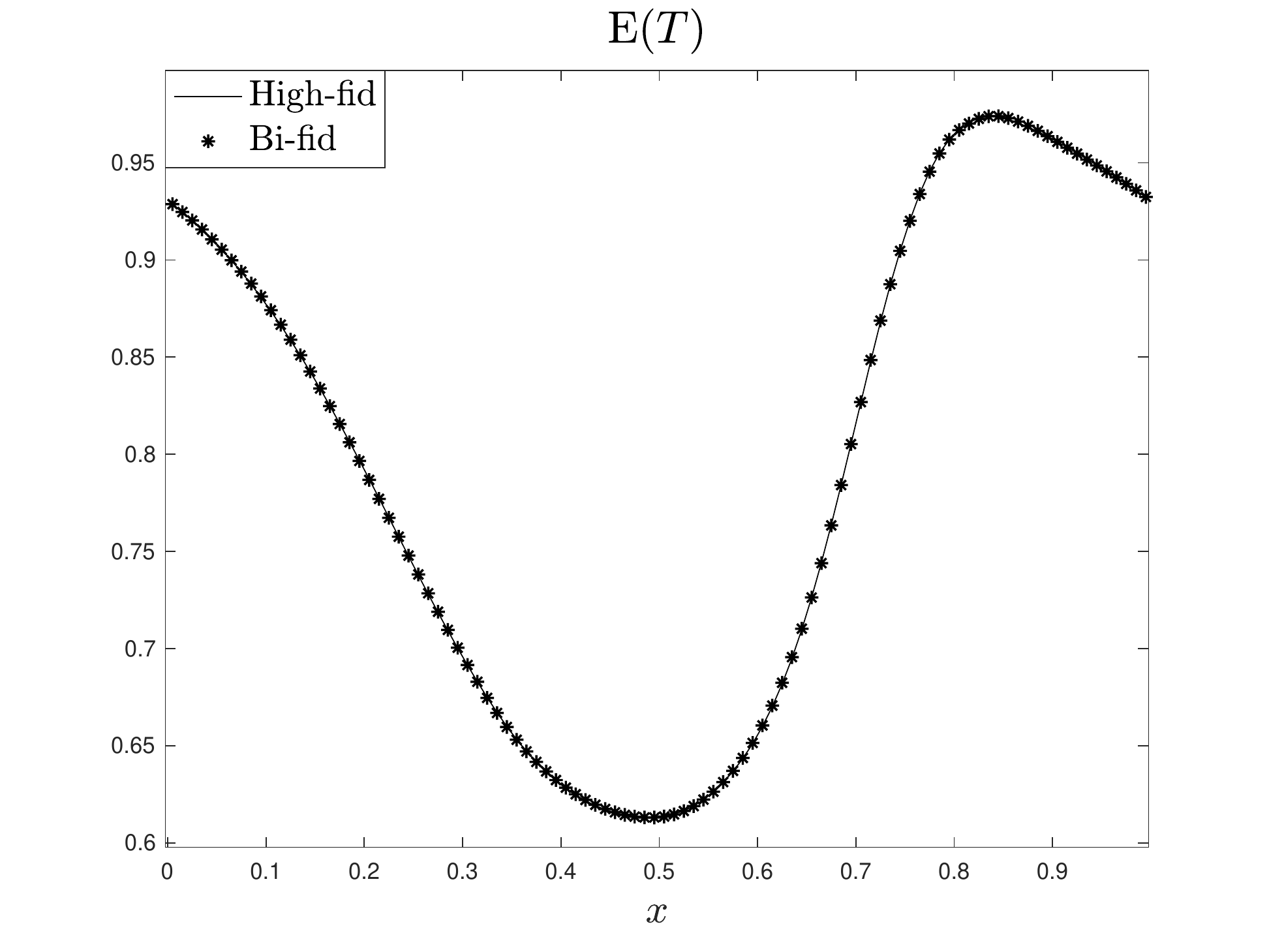}\hskip -.5cm
\includegraphics[width=0.5\linewidth]{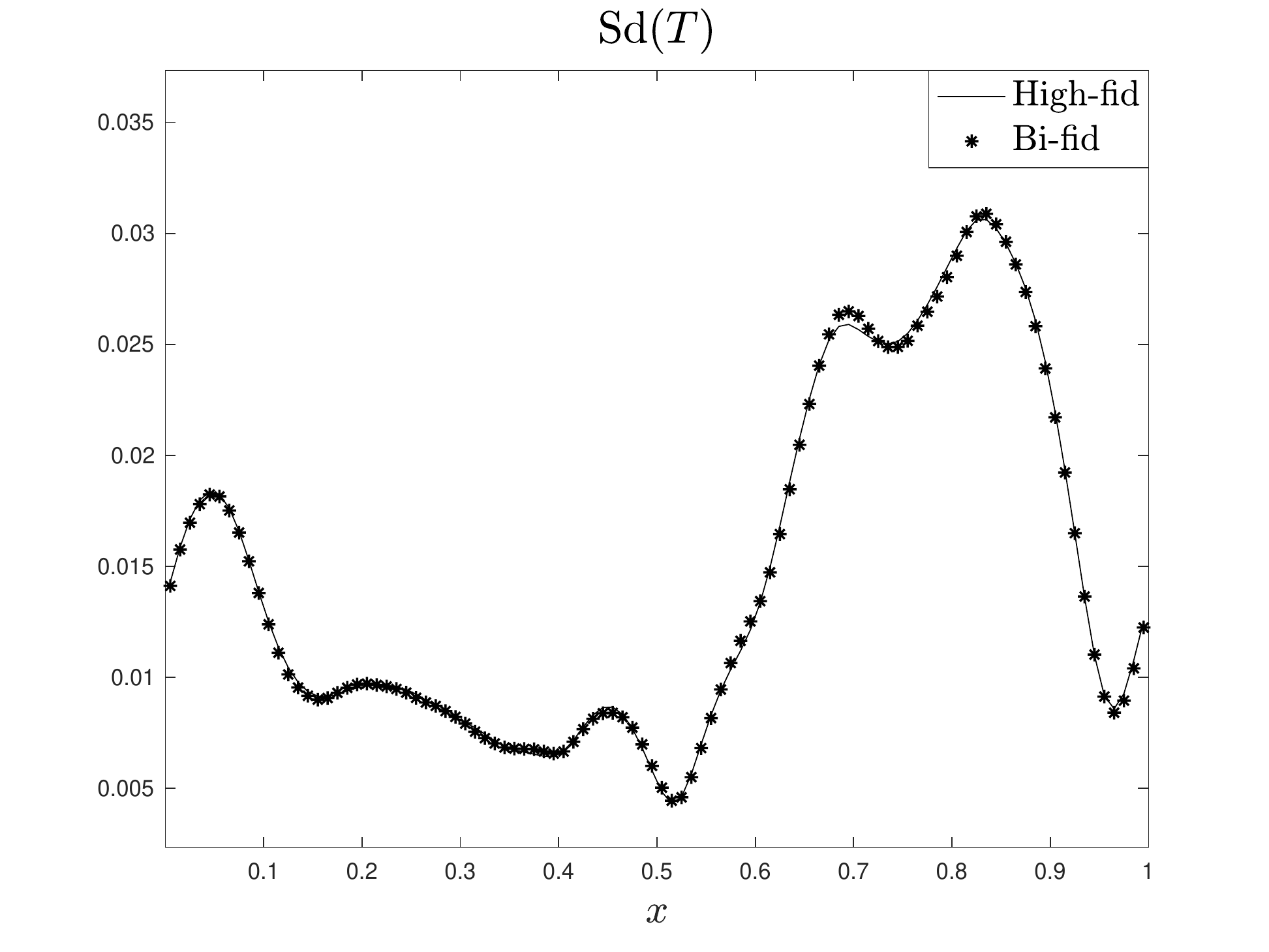}
\caption{{BFSC method for the Boltzmann equation.}  Results of the mixed regime test. The mean and standard deviation of $\rho$, $u_1$, $T$ of high-fidelity solutions and bi-fidelity solutions with $M=25$. 
}
\label{Fig5-mv}
\end{figure}

%-------------------------------------------------------
\subsubsection{Accuracy and convergence analysis}
\label{sec:3}
We discuss now some accuracy and convergence results for the bi-fidelity method applied to the Boltzmann equation. We first give a summary of the hypocoercivity framework studied in^^>\cite{LJ-UQ}, then we introduce the relation between the quantities of interests obtained from the Boltzmann and the compressible Euler system. The rough idea on which the analysis relies is that the bi-fidelity error can be to split it into two parts: the projection error and the remainder part. We mention that the error estimate of the bi-fidelity method for more general kinetic problems with multiple scales and uncertainties has been discussed in^^>\cite{GJL}. 

We first show that for each fixed $z\in \Omega$, the error between macroscopic quantities obtained from the Boltzmann equation and the compressible Euler system containing consistent initial data, is small and of order $\e$ as shown later in (\ref{h_Error}). Let us recall some useful notations for the norms as used in^^>\cite{LJ-UQ}. Let $f$ be the classical solution to the Boltzmann equation (\ref{eq:Boltzmann}). 
Consider a linearization around the global equilibrium together with a perturbation of $f$: 
$$ f = \mathcal M + \e \sqrt{\mathcal M}\, h, $$
with $$\mathcal M(v) = \frac{1}{(2 \pi)^{\frac{d_v}{2}}} e^{-\frac{|v|^2}{2}}. $$
In this setting, $h$ satisfies the perturbed equation 
\begin{equation}
\label{h-PE}
 \partial_t h + v\cdot\nabla_x h = \frac{1}{\e}\mathcal L(h) + \mathcal F(h, h),  
\end{equation} 
where the linearized operator $\mathcal L$ and the nonlinear operator $\mathcal F$ are given by \begin{align*}&\displaystyle 
\mathcal L(h) = \left(\sqrt{\mathcal M}\right)^{-1} \left[\mathcal Q(\sqrt{\mathcal M}h, \mathcal M)
+ \mathcal Q(\mathcal M, \sqrt{\mathcal M}h)\right], \\[2pt]
&\displaystyle \mathcal F(h, h) = 2 \left(\sqrt{\mathcal M}\right)^{-1}\mathcal Q(\sqrt{\mathcal M}h, \sqrt{\mathcal M}h). 
\end{align*}
Let denote $\partial_l^j := \partial/\partial v_j\, \partial/\partial x_l$ for multi-indices $j$ and $l$ and let introduce the following Sobolev norms: 
\begin{align}
\label{norm}
\begin{split}
\displaystyle \|h\|_{H_{x,v}^s}^2 & = \sum_{|j| + |l| \leq s} \| \partial_l^j h\|_{L_{x,v}^2}, \qquad
\|h\|_{H_{x,v}^{s,r}} = \sum_{|\nu|\leq r} \|\partial^{\nu} h\|_{H_{x,v}^s}^2, \\[4pt]
\displaystyle \|h\|_{H_{x,v}^s H_z^r} & = \int_{I_{z}} \|h\|_{H_{x,v}^{s,r}}\, \pi(z)d z, 
\qquad \|h\|_{H_{x,v}^{s,r} L_{z}^{\infty}} = \sup_{z\in I_{z}} \|h\|_{H_{x,v}^{s,r}}. 
\end{split}
\end{align}
We refer to^^>\cite[Theorem 2.5]{MB15} for details. Let $h_{\e}$ be the perturbed solution to the linearized equation (\ref{h-PE}). Suppose the initial data for (\ref{h-PE}) and (\ref{eq:Euler}) are consistent for each $z$. 
If the initial distribution $h_{in} \in \text{Null}(\mathcal L)$ and $h_{in}\in H_{x,v}^s$, then for each $z$, $\left(h_{\e} \right)_{\e>0}$ converges strongly to 
$$ h(t,x,v,z) = \left[\rho(t,x,z) + v\cdot u(t,x,z) + \frac{1}{2}(|v|^2 - d_v) T(t,x,z)\right]\mathcal M(v)$$
in $L_{[0,T]}^2 H_x^s L_v^2$ as the Knudsen number $\e\to 0$, where $\rho$, $u$, $T$ satisfy the Euler system (\ref{eq:Euler}). %We adapt the work in^^>\cite[Theorem 2.5]{MB15} to the acoustic scaling. 
Then, for all $z\in \Omega$, 
if $h_{in}$ belongs to $H_x^s L_v^2$, we have 
\begin{equation}\label{h_Error}
	\begin{split}
&\sup_{t\in[0,\infty]} \| h(t,z) - h_{\e}(t,z)\|_{L^2_{x,v}} \leq \\
&\sup_{t\in[0,\infty]} \| h(t,z) - h_{\e}(t,z)\|_{H_x^s L_v^2}\leq 
C \max\{ \e, \, \e V_T(\e)\}, 
\end{split}
\end{equation}
where $\forall T>0$, $V_T(\e)$ is defined as
$$V_T(\e)=\sup_{t\in [0, T]}\|h(t,z) - h_{\e}(t,z)\|_{L_x^{\infty}L_v^2} \to 0, \quad\text{ as   }\e\to 0. $$

%{\bf Error splitting. }
%Let now $\left\langle\,\cdot\,\right\rangle^H$ be an inner product space corresponding to the high-fidelity solution and $\left\|\,\cdot\,\right\|^H$ the corresponding induced norm, see^^>\cite{NGX14}. 
For each $z$, we can split the total error $\left\| u^H(z) - u^B(z)\right\|^H$ into two parts: 
\begin{equation}
\label{total-error}
\begin{split}
 &\left\| u^H(z) - u^B(z)\right\|^H \leq\\ &\left\| u^H(z) - P_{U^H(\gamma_M^L)}u^H(z) \right\|^H + \left\| P_{U^H(\gamma_M^L)}u^H(z) - u^B(z)\right\|^H,   
\end{split}
\end{equation}
where $u^H$ and $u^B$ denote the high-fidelity and bi-fidelity solutions. $\gamma_M^L = \{z_1^L,\hdots, z_M^L\}$. 
Also,^^>\cite[ the Lemma 4.3]{NGX14} shows that the estimate for the second term can be written as: 
\begin{equation}\label{TermII}
	\begin{split}
 &\left\| P_{U^H(\gamma_M^L)}u^H(z) - u^B(z)\right\|^H \leq\\& c \left\| P_{U^H(\gamma_M^L)}u^H(z)\right\|^H + 
\left\| \sqrt{{\bf G}^H}({\bf G}^L)^{-1} {\bf Q}\,{\bf f}^L\right\|, 
\end{split}
\end{equation}
where $c$ is a small constant such that $c=\e_1 + \e_2 + \e_1 \e_2$. 
Moreover ${\bf G}^L$ (or ${\bf G}^H$) is the Gramian matrix of $u^L(\gamma)$ (or $u^H(\gamma)$) given by \eqref{GM} and the vector ${\bf f}^L$ has entries
$$ f_i^L = \left\langle u^L({z}_i), u^L(z) \right\rangle^L, $$
with ${\bf Q}:={\bf I} - {\bf P}$, the orthogonal projection onto its kernel (with ${\bf P}$ the orthogonal projection matrix onto its range), see^^>\cite{NGX14} for details. Thus, the last term in \eqref{TermII} is related to the non-invertibility of high-fidelity Gramian matrix and it can be proven to be usually negligible. 
In addition, since $\left\| P_{U^H(\gamma_M^L)}u^H(z)\right\|^H \leq \|u^H(z)\|^H$, we can write 
\begin{equation}
\label{PU}
\left\| P_{U^H(\gamma_M^L)}u^H(z) - u^B(z)\right\|^H \leq c\, \|u^H(z)\|^H + 
\left\| \sqrt{{\bf G}^H}({\bf G}^L)^{-1} {\bf Q}\,{\bf f}^L\right\|. 
\end{equation} 

We study now the smoothness of the high-fidelity solution 
$u^H: \Omega \mapsto V^H$, where $z$ is a multivariate random parameter and $V^H$ is a Hilbert space with the usual inner product $\left\langle \,\cdot,\cdot \right\rangle^H$. Successively, we establish an estimate bound for the Kolmogorov width to show the convergence result of our bi-fidelity method. To that aim, let us recall the standard multivariate notation in^^>\cite{LZ}. We denote the countable set of finitely supported sequences of nonnegative integers by
$$ \mathcal F:= \left\{ \nu = (\nu_1, \nu_2, \cdots ): \nu_j \in \mathbb N, \, \text{ and } \nu_j \neq 0 \text {  for only a finite number of } j\right\}, $$
with $|\nu|:=\sum_{j \geq 1} |\nu_j|$. 
For $\nu \in \mathcal F$ supported in $\{1, \cdots, J \}$, one defines the partial derivative in $z$
\begin{equation}
\partial^{\nu} u = \frac{\partial^{|\nu|} u}{\partial^{\nu_1}z_1  \cdots \partial^{\nu_J}z_J},   
\end{equation} 
and the multi-factorial $\nu! := \prod_{j \geq 1}\nu_j !$, where $0! :=0$.  
If the initial distribution of the high-fidelity model satisfies 
\begin{equation}
\label{u-IC}
 \|h_{\e}^{in}(z)\|_{H_{x,v}^{1,r} L_z^{\infty}} \leq C_I,
\end{equation}
then for a fixed time $t>0$ and $|\nu|\leq r$, 
\begin{equation}\label{Anal}
\sup_{z\in \Omega}\|\partial^{\nu} u^H(t,z)\|^H \leq C^{\prime}\, e^{- \e \tau t} + \xi, 
\end{equation}
where $\xi$ depends on the order and discretization parameters $\Delta t$, $\Delta x$, $\Delta v$. 
Thus for all $z\in \Omega$, one gets 
$$ \|\partial^{\nu} u^H(t,z)\|^H \leq C^{\prime}\, e^{- \e \tau t} + \xi. $$
Here note that $C_I$, $C$, $C^{\prime}$ and $\tau$ are all positive generic constants independent of $\e$. We recall now the assumption on the random collision kernel: 
\begin{assumption}
\label{assp1}
Assume the collision kernel take the form
\begin{equation}\label{RB1}
\begin{split}
&B(|v-v^{\ast}|, \cos\theta, z) =\Phi(|v-v^{\ast}|) b(\cos\theta, z), \\ &\Phi(|v-v^{\ast}|) = C |v-v^{\ast}|^m, \, m \in [0,1],
\end{split}	
 \end{equation}
and $\left(\psi_j \right)_{j\geq 1}$ be an affine representer^^>\cite{Cohen15} of the cross section $b$, that is, 
\begin{equation}
\label{Assump-b}
  b(\eta, z) = \bar b(\eta) + \sum_{j \geq 1}z_j \psi_j(\eta), \quad z :=\left( z_j \right)_{j \geq 1}, \quad \eta = \cos\theta,    
\end{equation}
where the sequence $\left( \|\psi_j\|_{L^{\infty}(\eta)}\right)_{j\geq 1}\in \ell^p$ for $0<p<1$. 
In addition, one assumes that 
\begin{equation}\label{RB2}|b(\eta, z)| \leq C_0, \quad |\partial_{\eta}b(\eta, z)| \leq C_1, \quad
|\partial^{\nu}b(\eta, z)| \leq C_2, 
\end{equation}
for all $\eta\in[-1,1]$ and $|\nu|\leq r$ with $C$, $C_0$, $C_1$, $C_2$ are all positive constants. \end{assumption}

We now review the main result on convergence analysis in^^>\cite{LZ}:
\begin{theorem}
\label{MainThm}
If the assumptions for the random initial data, random collision kernel, namely (\ref{u-IC}) and {Assumption \ref{assp1}} are satisfied, for fixed time $t>0$ and fixed numerical discretization parameters $\Delta t$, $\Delta x$ and $\Delta v$, for all $z\in \Omega$, are chosen, then
\begin{equation}
\label{MainError}
\begin{split}
&\left\|u^H(t,z)- u^B(t,z)\right\|^H \leq\\& C_1\, \frac{e^{-\frac{\e \tau t}{2}}}{(N/2+1)^{q/2}} + C_2\, e^{-\e\tau t} + \chi
+ \left\| \sqrt{{\bf G}^H}({\bf G}^L)^{-1} 
{\bf Q}\,{\bf f}^L(z)\right\|,   
\end{split}
\end{equation} 
where $N$ is the size of the subspace $\gamma_M$ in Algorithm \ref{BiFi-pod}, with $q=\frac{1}{p}-1$ and $p$ dependent on the $\ell^p$-summability assumption of $(\psi_j)_{j \geq 1}$. Moreover,
$C_1$, $C_2$ and $\tau$ are all constants dependent on the initial data $u^{in}$ and {Assumption \ref{assp1}} about the collision kernel while $\delta_1$, $\delta_2$ are both sufficiently small. %Definitions of ${\bf G}^L$, ${\bf G}^H$, ${\bf Q}$ and $f^L$ are given below \eqref{TermII}, here $\chi$ is associated to the order and discretization mesh in the high-fidelity solver. 
\end{theorem}

For details and the proof refer to^^>\cite{LZ}. This Theorem indicates that the error between the bi-fidelity and high-fidelity solutions decays algebraically with respect to the number of high-fidelity runs. The convergence rate $q/2$ is independent of dimension of the random parameters and regularity of the initial data, and only relates to the $\ell^p$ summability of the affine representative $(\psi_j)_{j \geq 1}$ shown in Assumption \ref{assp1}. 

%---------------------------------------------------------------
\subsection{BFSC method in the diffusion limit}
\label{sec:diff2}
The linear transport equation in the diffusion scaling \eqref{pde_transport1d} introduced in Section 2 will be here employed in the framework of the bi-fidelity method. After the even and odd parity formulation of the transport equation \eqref{pde_transport1d_rj}, we have shown that the general linear transport equation degenerates to a diffusion equation in the limit $\varepsilon\to 0$. The idea we explore here is to use the so-called Goldstein-Taylor (GT) model--which is a discrete velocity approximation of the underlying kinetic equation with only two velocities--as the low-fidelity model. We show that the GT model share the same diffusive limit behavior of the original kinetic model \eqref{pde_transport1d}. However, the clear advantage of the GT model is that it is significantly less computationally expensive than the original kinetic equation. This permits to better explore the space spanned by the solution of the GT model in a random framework and successively to choose the best random points to be used in the high fidelity method. 
%We demonstrate in the following that the bi-fidelity approximations based on this low-fidelity model could produce reasonably satisfactory results across a large range of regimes, from the kinetic to diffusive regime. 

The one-dimensional Goldstein-Taylor (GT) model^^>\cite{Gol,Tay} with random inputs is given by
\begin{equation}
\label{GT}
\left\{
\begin{split}
& \partial_t u + \frac{1}{\epsilon}\partial_x u = \frac{\sigma(x,z)}{2\epsilon^2}(v-u), \\[4pt]
& \partial_t v - \frac{1}{\epsilon}\partial_x v = \frac{\sigma(x,z)}{2\epsilon^2}(u-v). 
\end{split}
\right.
\end{equation}
Here we assume a random scattering coefficient $\sigma(x,z)$ as for the original high fidelity model \eqref{pde_transport1d}. 
It is worth mentioning that the Goldstein-Taylor model \eqref{GT} can be regarded as a discrete-velocity kinetic counterpart of the linear transport equation where $u$ defines the density of particles traveling with velocity $1$, whereas $v$ that of particles traveling in the reverse direction with velocity $-1$. Besides, the GT model has significantly cheaper computational cost, yet shares the same limiting diffusion equations with the linear transport model as $\epsilon\to 0$.
We refer to^^>\cite{Lions1997} for rigorous results concerning the diffusion limit of two-velocity models and extensions to nonlinear diffusion coefficients.
This low-fidelity model is shown to be used in many interesting applications^^>\cite{GottliebX_CICP08}. We introduce now the macroscopic variables: the mass density $\rho$ and the flux $s$,
$$ \rho = u + v, \qquad s = \frac{u-v}{\epsilon}, $$
then the GT model \eqref{GT} is equivalent to the following system: 
\begin{equation}
\label{pde_transport1d_uv}
\left\{
\begin{split}
&\partial_t \rho +  \partial_x s = 0, \\[4pt]
&\partial_t s + \frac{1}{\epsilon^2}\partial_x\rho =
-\frac{\sigma(x,z)}{\epsilon^2}s. 
\end{split}
\right.
\end{equation}
The above system as said shares the same limit of the high fidelity model \eqref{pde_transport1d_rj}. In fact, in the diffusion limit $\epsilon\rightarrow 0$, the system \eqref{pde_transport1d_uv} can be approximated by the heat equation to the leading order, with random diffusion coefficient $\sigma(x,z)$, 
\begin{equation}
\label{pde_transport1d_sp_lowfi2_diff}
\left\{
\begin{split}
&s = -\frac{1}{\sigma(x,z)}\partial_x \rho, \\[4pt]
&\partial_t \rho = \partial_{x} 
\left[\frac{1}{\sigma(x,z)} \partial_x \rho \right].
\end{split}
\right.
\end{equation}
Comparing \eqref{transport1d_diff} and \eqref{pde_transport1d_sp_lowfi2_diff}, the two systems look similar except for the magnitude of the diffusion coefficient on the right hand-side which is different. However, if one defines the diffusion coefficient in \eqref{pde_transport1d_sp_lowfi2_diff} as $\sigma_{\text{GT}}$ and that in \eqref{transport1d_diff} as $\sigma_{\text{LTE}}$, then by assuming $\sigma_{\text{GT}}=\frac{1}{3}\sigma_{\text{LTE}}$ the two models share precisely the same diffusion limit. 

Motivated by the above observations, we employ the equivalent formulation of the GT equation, that is, system \eqref{pde_transport1d_uv} as our low-fidelity model.
Without loss of generality, the $d_z$-dimensional random variable $z=\{z_1, \cdots, z_d\}$ is assumed to follow the uniform distribution on $[-1,1]^{d_z}$, and the dimension of the random parameter is chosen as $d=5$. To compute the reference solutions for the mean and standard deviation of the high-fidelity quantities of interests, we use the high-order stochastic collocation method over 5-dimensional  sparse quadrature points with $5$-level Clenshaw-Curtis rules, i.e., evaluated on $2243$ quadrature points (cf.,\cite{XiuH_SISC05}).  
For the high-fidelity solver, at each given sample we employ the AP scheme^^>\cite{JPT2} developed for the deterministic linear transport equation \eqref{pde_transport1d} under the diffusive scaling. 
The standard 16-points Gauss-Legendre quadrature set is used for the velocity space to compute $\rho$ in \eqref{r_integral}.
For the low-fidelity (LF) solver, we use the deterministic AP method^^>\cite{JPT1} to solve the linear Goldstein-Taylor model \eqref{pde_transport1d_uv}. Periodic boundary conditions are considered. 

\begin{figure}[htb]
\centering
\includegraphics[scale=0.33]{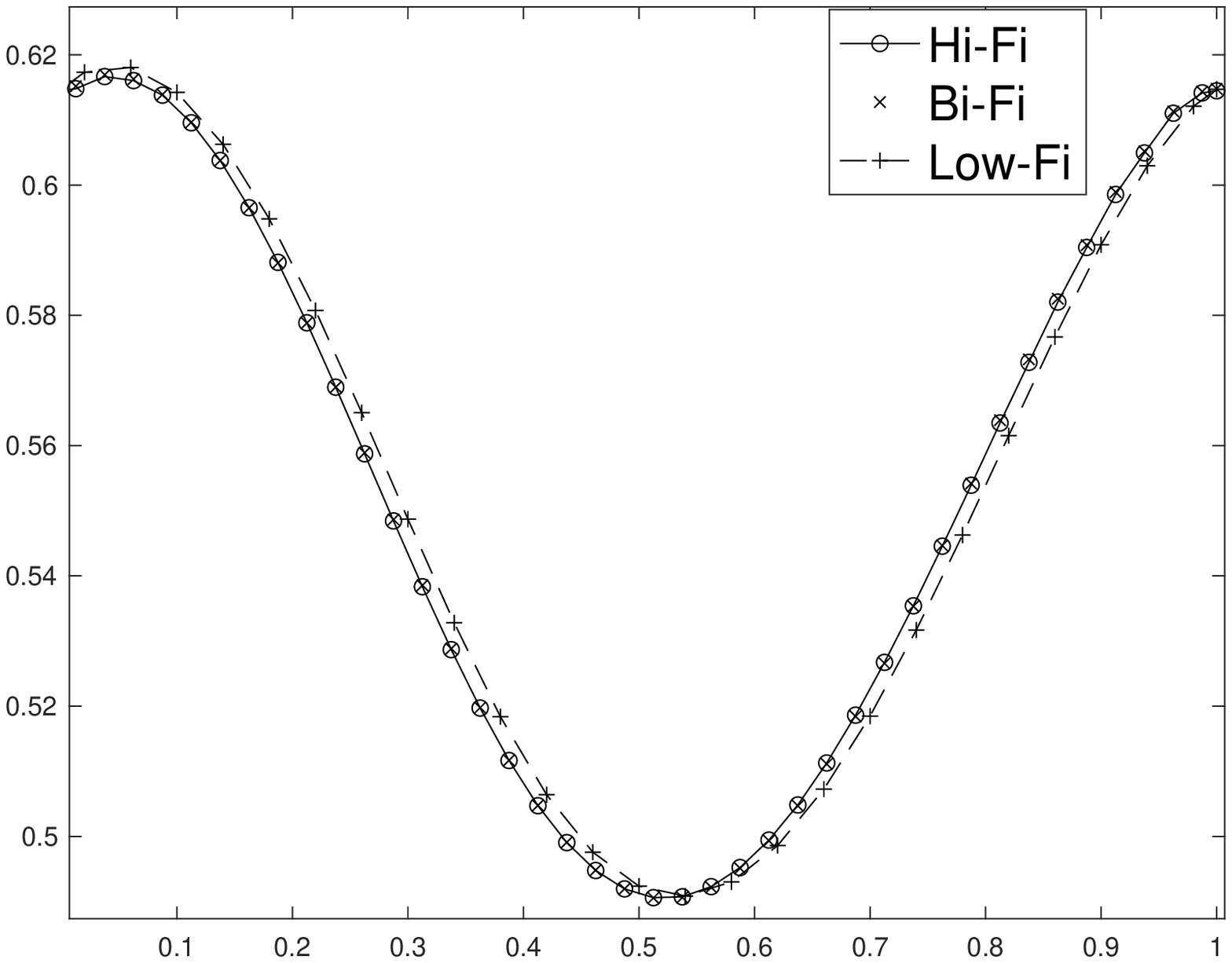}\hskip -.5cm
\includegraphics[scale=0.33]{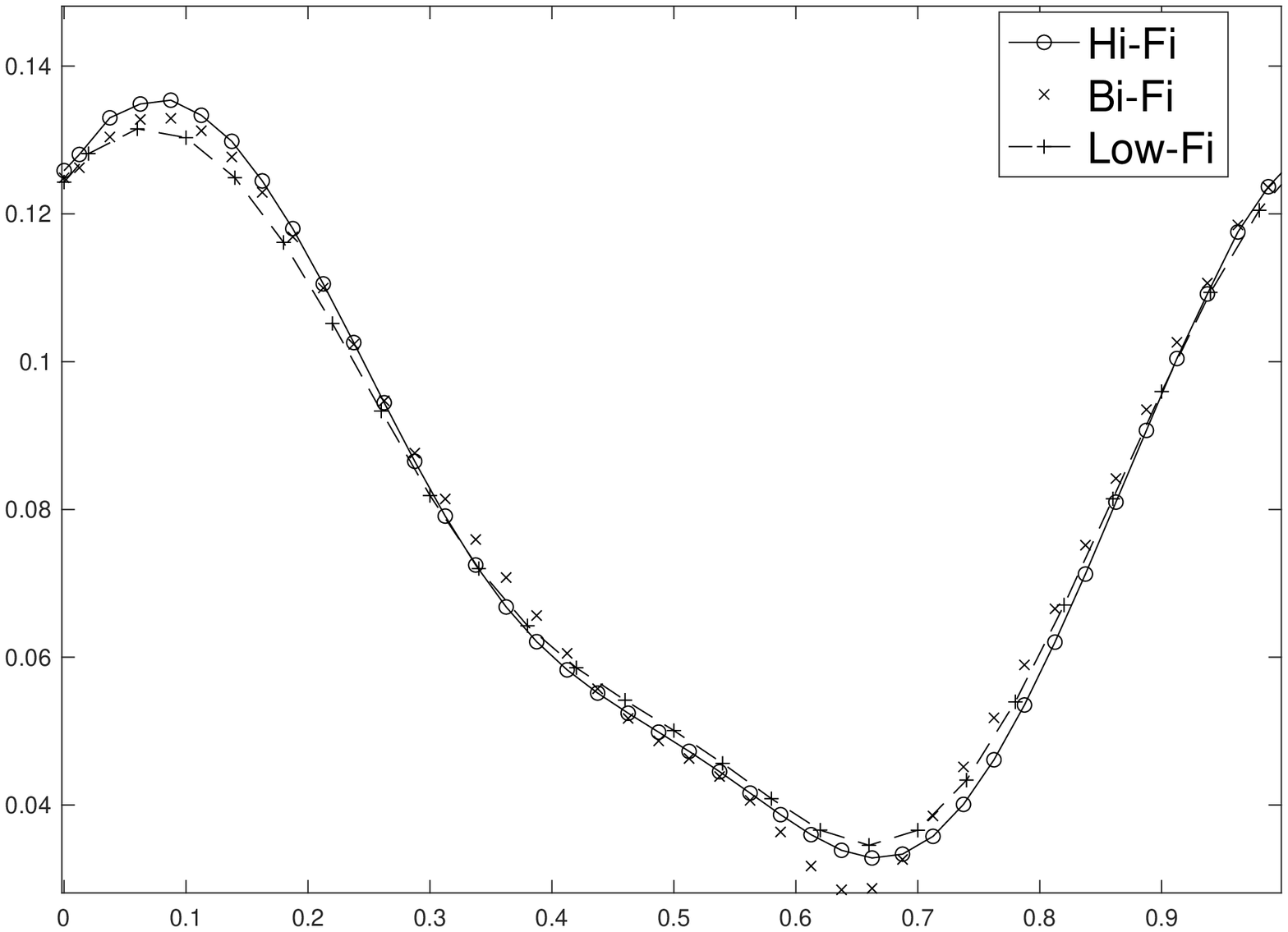}
\includegraphics[scale=0.33]{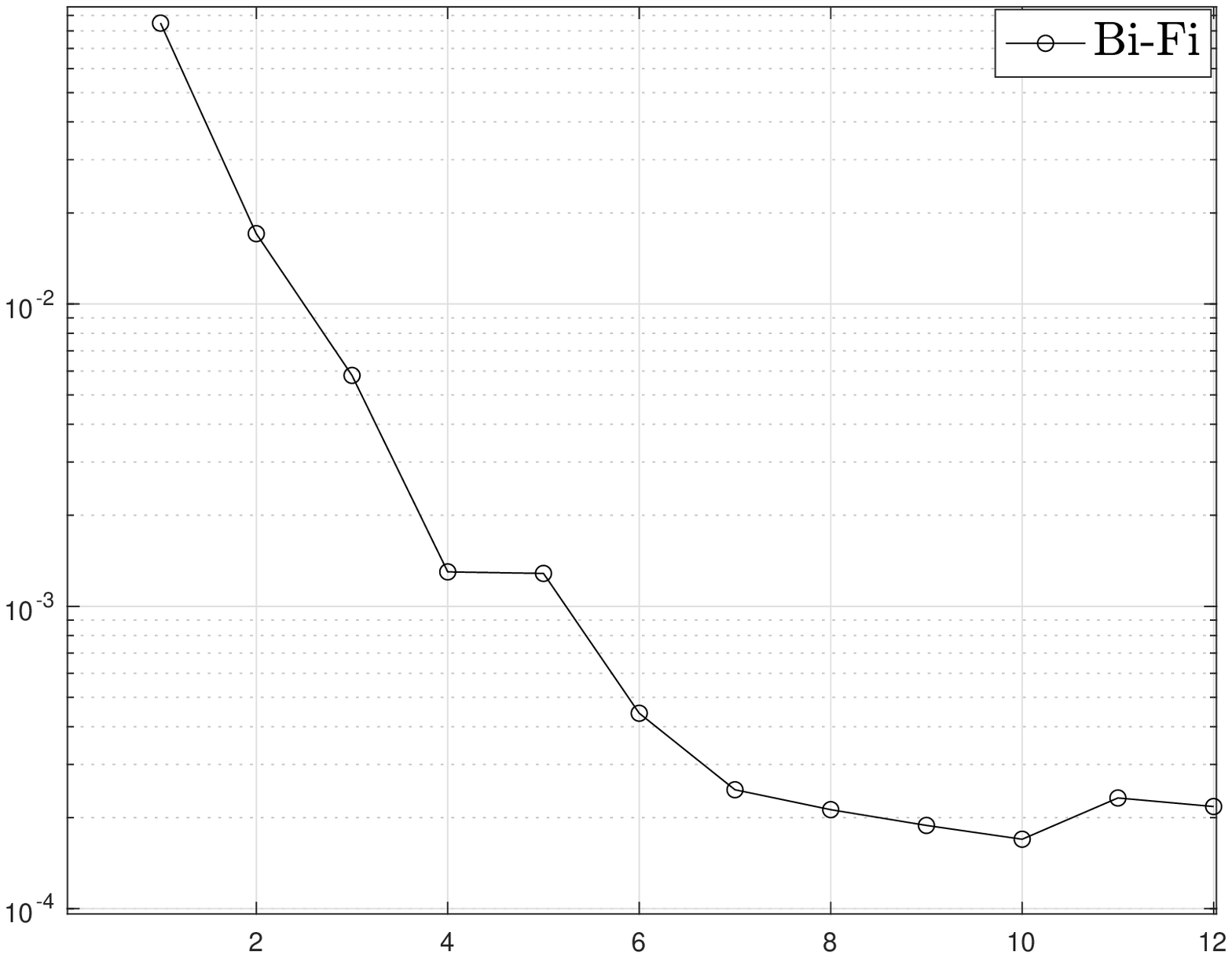}\hskip -.5cm
\includegraphics[scale=0.33]{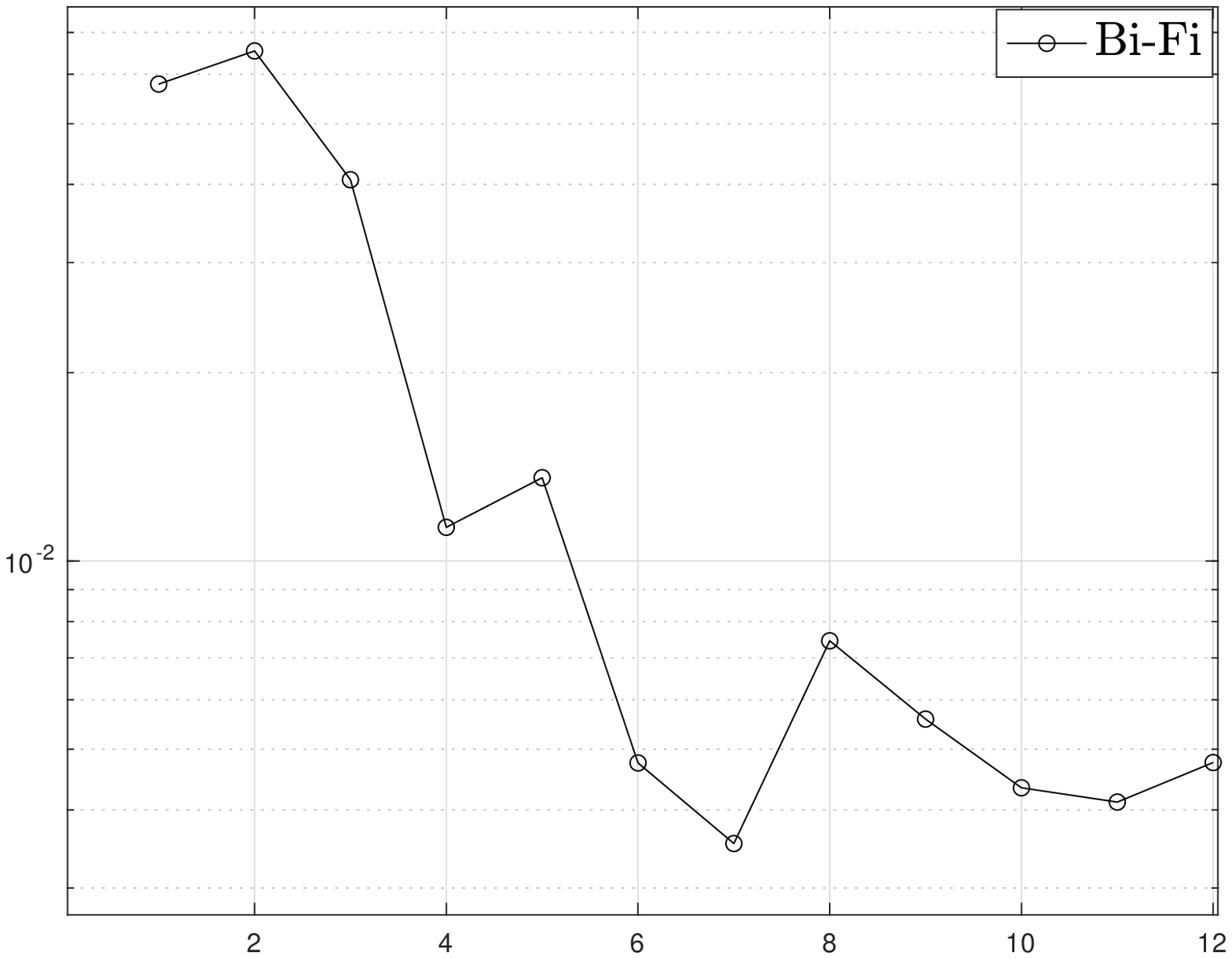}
\caption{{BFSC method for the linear transport.}  Test results of the uncertain cross-section and initial data in the diffusive scaling. The mean (left) and standard deviation (right) of the density $\rho$, obtained by $M=12$ high-fidelity runs and the sparse grid method with $2243$ quadrature points (crosses, first row). The corresponding errors are reported in the second row, $\epsilon=10^{-2}$.}
\label{Test4_Fig1}
\end{figure}

In the first test reported, we assume the random cross-section coefficient given by 
\begin{equation}
\label{Sigma}
\sigma (x,z) = 1 + 4 \sum_{i=1}^{d_z} \frac{1}{(i\pi)^2}\cos{(2\pi i x)} z_i,
\end{equation}
with $d=5$. Let also the uncertain initial data be
\begin{equation}
\label{Test2_IC}
f_0(x,v,z)=\rho_0\exp\left(-\left(\frac{v-0.5}{T_0}\right)^2\right)+ 
\rho_1 \exp\left(-\left(\frac{v+0.75}{T_1}\right)^2\right), 
\end{equation}
where 
\begin{align*}
& \rho_0(x,z) = 1 + 3 \sum_{k=1}^{d_z} \frac{\sin(2\pi k x)z_k}{(k\pi)^2} , 
\ T_0(x,z)=\frac{5+2\cos(2\pi x)}{20}\left(1+0.6 z_1 \right), \\[2pt]
& \rho_1(x,z) = 1 + 2 \sum_{k=1}^{d_z} \frac{\cos(2\pi k x)z_k}{(k\pi)^2} , 
\ T_1(x,z) = 0.5+0.2\cos(2\pi x)z_2. 
\end{align*}
We use $\Delta x=0.025$, $\Delta t=10^{-4}$ for the high-fidelity solver, and $\Delta x=0.04$, $\Delta t=2\times 10^{-4}$ for the low-fidelity solver. Numerical solutions and bi-fidelity errors at output time $T=0.02$ are shown in Figure~\ref{Test4_Fig1}. 
We observe that while the low-fidelity solution may miss details of the high-fidelity solution near the peaks, the mean and standard deviation of the bi-fidelity approximation agree well with the high-fidelity solution globally, when using $M=12$. A fast exponential decay of bi-fidelity errors is suggested from Figure ~\ref{Test4_EB}. With only $12$ simulations for the high-fidelity AP solver, the bi-fidelity mean can reach the accuracy level of less than $\mathcal{O}(10^{-4})$. In Figure~\ref{Test4_EB}, the error estimators bound well the true bi-fidelity errors, with the values of both metrics $R_e$ and $R_s$ around less than $10$.

\begin{figure}[tb]
\centering
\includegraphics[scale=0.31]{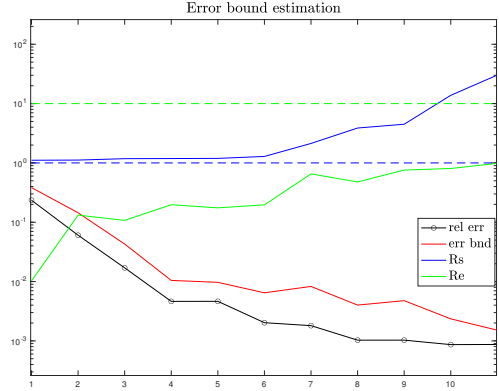}
\caption{{BFSC method for the linear transport.}  Error bound estimation in the diffusive scaling with uncertain cross-section and initial data. }
\label{Test4_EB}
\end{figure}

\begin{figure}[tb]
\centering
\includegraphics[scale=0.33]{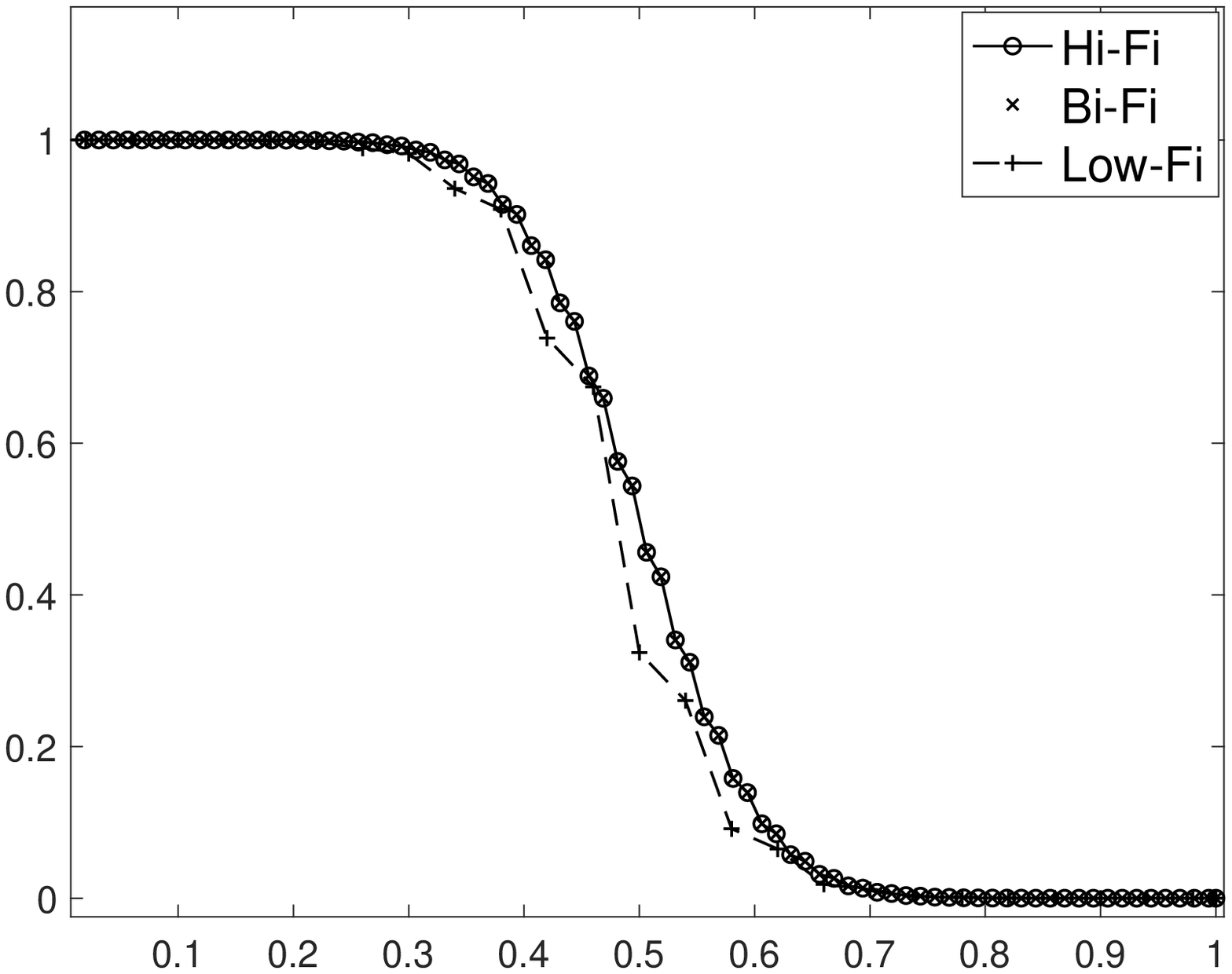}\hskip -.5cm 
\includegraphics[scale=0.33]{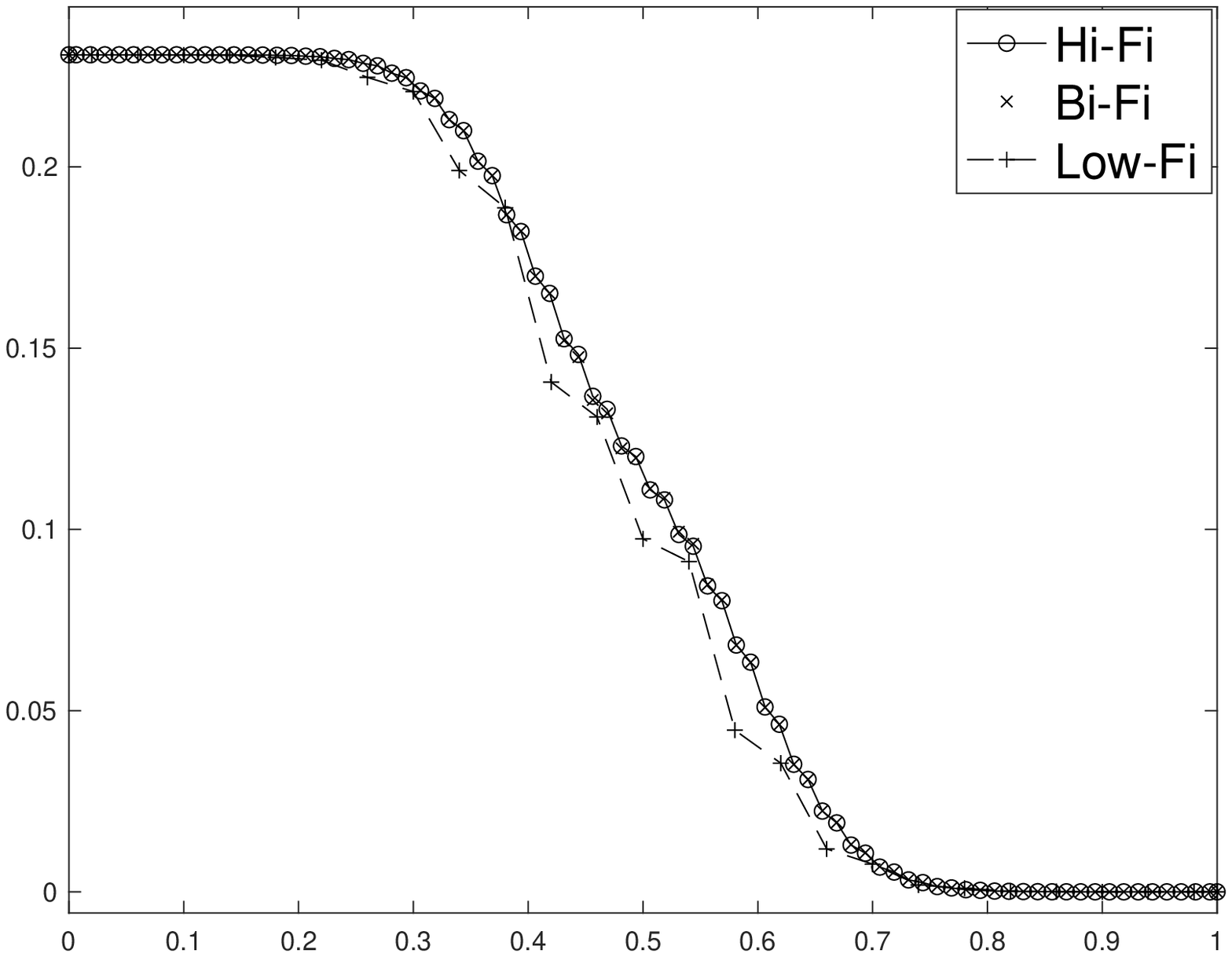}
\caption{{BFSC method for the linear transport.}  Results of the Riemann problem in the diffusive scaling. The mean (left) and standard deviation (right) of the density $\rho$, obtained by $M=12$ high-fidelity runs and the sparse grid method using $2243$ quadrature points (crosses). Here $\epsilon=10^{-8}$. } 
\label{Test2_Fig1}
\end{figure}

\begin{figure}[tb]
\centering
\includegraphics[scale=0.33]{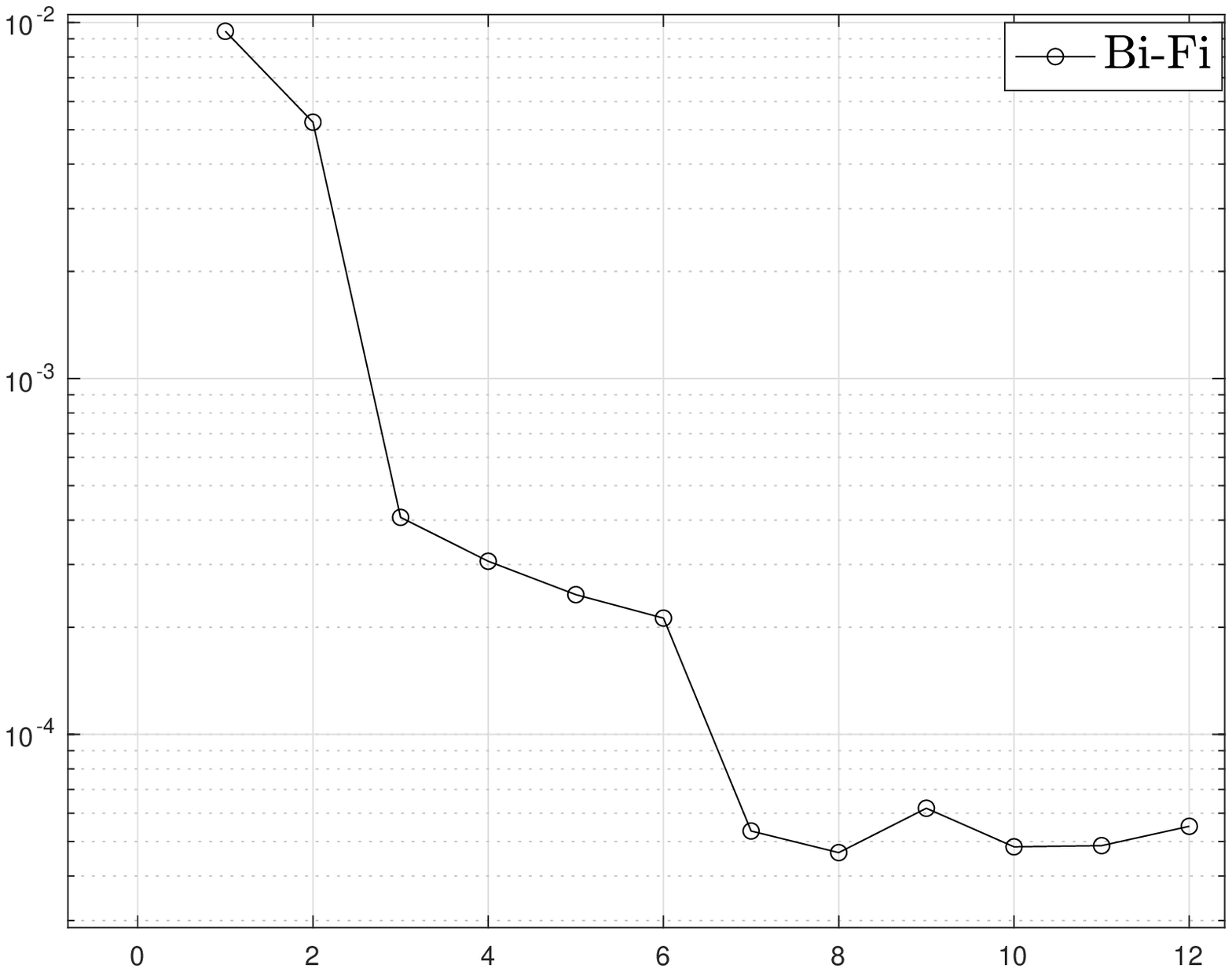}\hskip -.5cm
\includegraphics[scale=0.33]{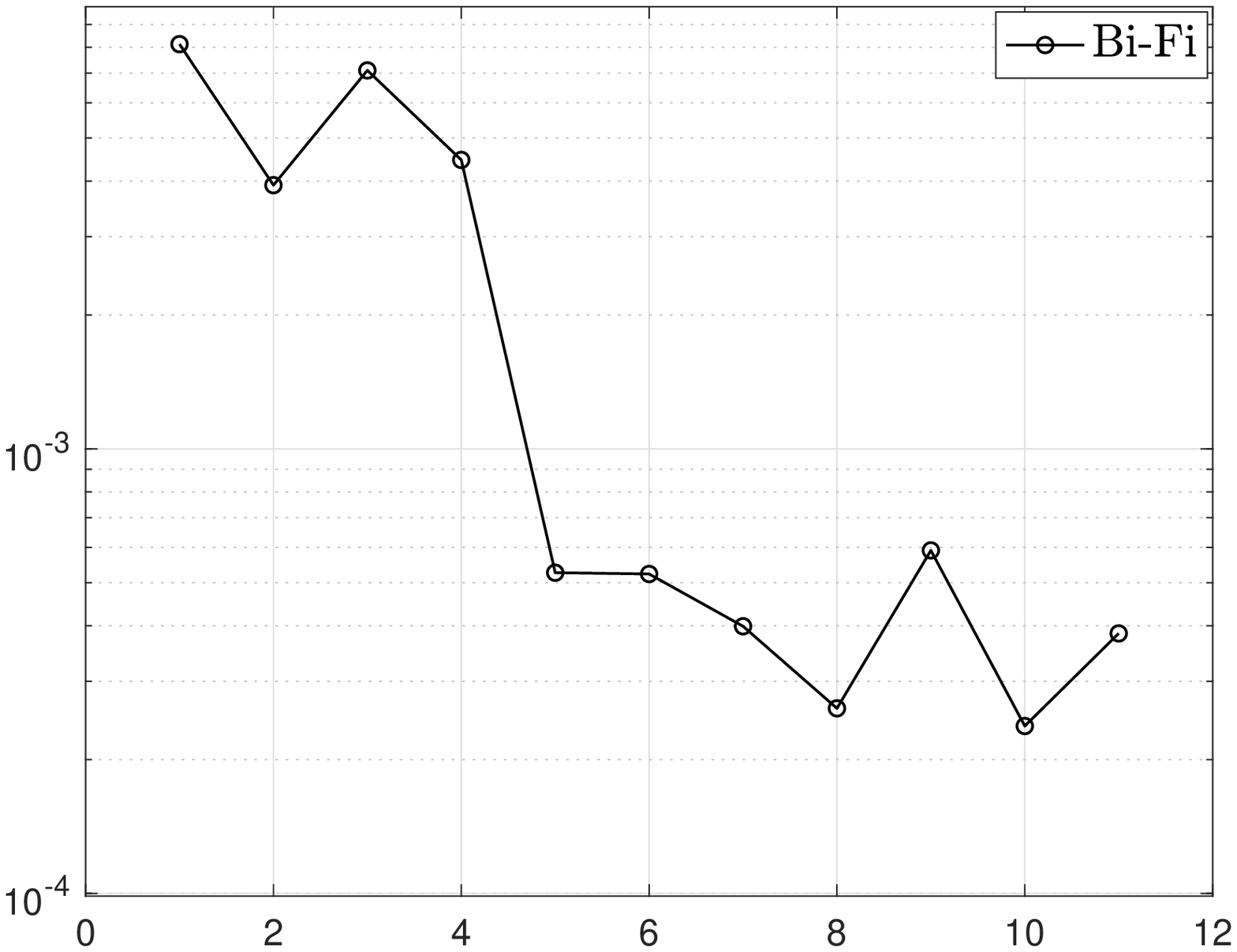}
\caption{{BFSC method for the linear transport.}  Results of the Riemann problem in the diffusive scaling. Errors of the bi-fidelity approximation mean (left) and standard deviation (right) of the density $\rho$ with respect to the number of high-fidelity runs. }
\label{Test2_Err}
\end{figure}

We now consider a Riemann problem. Let assume the same random coefficient \eqref{Sigma} with boundary conditions
\begin{eqnarray*}
	f(x=0,t,v,z)=1+0.4z_1, \quad\mbox{  if }v\ge 0 , \\[2pt]
	 f(x=1,t,v,z)=0, \quad\mbox{  if }v\le 0,
 \end{eqnarray*}
 with uncertain initial distribution given by 
 \begin{eqnarray*}
 	f(x,t=0,v,z)=1+0.4z_1, \qquad  0\leq x <0.5 , \\[2pt]
 	f(x,t=0,v,z)=0, \qquad  0.5\leq x \leq 1. 
 \end{eqnarray*}
In this example, we let the output time $T=0.01$ and $\epsilon=10^{-8}$. We set $\Delta t=2\times 10^{-4}$, $\Delta x=0.04$ in the low-fidelity solver, and $\Delta t=5\times 10^{-5}$, $\Delta x=0.0125$ in the high-fidelity solver. The Figure~\ref{Test2_Fig1} shows the numerical mean and standard deviation of the density $\rho$ for $\epsilon=10^{-8}$. While the low-fidelity solutions are not able to capture the detailed information around the transition region, the bi-fidelity approximation agrees with the high-fidelity $\rho$ at a quite satisfactory level. From Figure~\ref{Test2_Err}, it is obvious that the bi-fidelity errors decay fast with respect to the selected high-fidelity runs. With only $M=12$ high-fidelity simulation runs, the bi-fidelity errors for the mean and standard deviation of the density $\rho$ reach as small as $\mathcal{O}(10^{-4})$. Based on Figure~\ref{Test2_EB}, the error estimators bound well the true bi-fidelity errors, with both metrics $R_e$ and $R_s$ around or less than $10$. We also remark that further increasing the high-fidelity samples after $M=8$ will not help improving the quality of bi-fidelity approximation.

\begin{figure}[tb]
\centering
\includegraphics[scale=0.4]{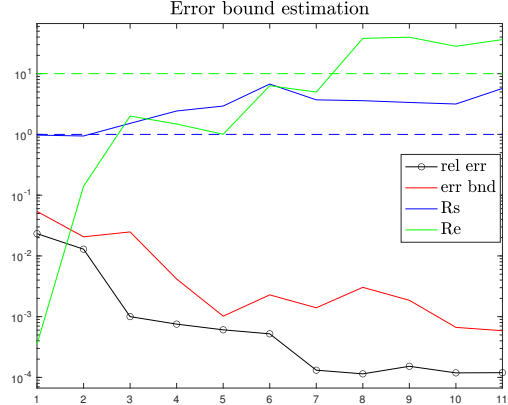}
\caption{{BFSC method for the linear transport.}  Error bound estimation for the Riemann problem in the diffusive scaling. }
\label{Test2_EB}
\end{figure}

%------------------------------------------------------
\subsection{BFSC method for kinetic epidemic models}
\label{sec:epidem}
We consider in this last example the quantification of uncertainty by using the bi-fidelity approach for the kinetic model of disease spread \eqref{eq:kineticc} introduced in Section \ref{sec:epidemic}. 
%We recall here for sake of clarity the high fidelity model which reads
%\begin{eqnarray}
%	\nonumber
%	\frac{\partial f_S}{\partial t} + \frac{\partial (v_S f_S)}{\partial x} &=& -F(f_S, I) +\frac1{\tau_S}\left(\frac{S}{2}-f_S\right)\\
%	\label{eq:kineticc2}
%	\frac{\partial f_I}{\partial t} + \frac{\partial (v_I f_I)}{\partial x} &=&  F(f_S, I)-\gamma f_I+\frac1{\tau_I}\left(\frac{I}{2}-f_I\right)\\
%	\nonumber
%	\frac{\partial f_R}{\partial t} + \frac{\partial (v_R f_R)}{\partial x} &=& \gamma f_I+\frac1{\tau_R}\left(\frac{R}{2}-f_R\right),
%\end{eqnarray}
%where $S$ are the susceptible, $I$ the infected and $R$ the recovered individuals, with respectively kinetic densities $f_S,f_I$ and $f_R$ and $v_S=\vs(x) v$, $v_I=\vi(x) v$, $v_R=\vr(x) v$ with $\vs,\vi,\vr \geq 0$. The quantity $\gamma=\gamma(x,{z})$ is the recovery rate of infected and  $F(\cdot,I)$ is the incidence rate.

The low-fidelity model, is based on considering individuals moving in two opposite directions (indicated by signs ``+'' and ``-''), with velocities $\pm \lambda_S$ for susceptible, $\pm \lambda_I$ for infected and $\pm \lambda_R$ for removed. This dynamics of the population through the two-velocity epidemic model^^>\cite{Bert} is given by 
\begin{eqnarray}
	\frac{\partial S^{\pm}}{\partial t} + \lambda_S \frac{\partial S^{\pm}}{\partial x} &=& - F(S^{\pm},I) \mp \frac{1}{2\tau_S}\left(S^+ - S^-\right)				
	\nonumber\\ 
	\frac{\partial I^{\pm}}{\partial t} + \lambda_I \frac{\partial I^{\pm}}{\partial x} &=&  F(S^{\pm},I) -\gamma I^{\pm} \mp \frac{1}{2\tau_I}\left(I^+ - I^-\right)								\label{eq.SIR_kinetic_diag}	\\ 
	\frac{\partial R^{\pm}}{\partial t} + \lambda_R \frac{\partial R^{\pm}}{\partial x} &=& \gamma I^{\pm} \mp \frac{1}{2\tau_R}\left(R^+ - R^-\right).									\nonumber
\end{eqnarray}
In the above system, individuals $S(x,t,{z})$, $I(x,t,{z})$ and $R(x,t,{z})$ are defined as
\begin{equation*}
 	S = S^+ + S^- , \quad 
 	I = I^+ + I^- , \quad 
 	R = R^+ + R^- .	
\end{equation*}
The transmission of infection is governed by the same incidence function as in the high-fidelity model, defined by \eqref{eq:incf}. The definition of the basic reproduction number $R_0$ is the same as introduced in \eqref{eq:R0}. The fluxes are now given by 
\begin{equation}
 	J_S = \lambda_S \left(S^+ - S^-\right) , \quad 
 	J_I = \lambda_I \left(I^+ - I^-\right) , \quad 
 	J_R = \lambda_R \left(R^+ - R^-\right) .	  	\label{eq.fluxes_kinetic_diag}
\end{equation}
Then we derive a hyperbolic model equivalent to \eqref{eq.SIR_kinetic_diag}, while presenting a macroscopic description of propagation of the epidemic at finite speeds. In this system the fluxes satisfy
\begin{eqnarray}
	\frac{\partial J_S}{\partial t} + \lambda_S^2 \frac{\partial S}{\partial x} &=& -F(J_S,I)  -\frac{J_S}{\tau_S} 
	\nonumber\\ 
	\frac{\partial J_I}{\partial t} + \lambda_I^2 \frac{\partial I}{\partial x} &=& \frac{\lambda_I}{\lambda_S}F(J_S,I)  -\gamma J_I -\frac{J_I}{\tau_I}	
	\label{eq.SIR_kinetic-fluxes}\\ 
	\frac{\partial J_R}{\partial t} + \lambda_R^2 \frac{\partial R}{\partial x} &=& \frac{\lambda_R}{\lambda_I}\gamma J_I -\frac{J_R}{\tau_R}\,.
	\nonumber
\end{eqnarray}
Let us now consider the behavior of the low-fidelity model in the diffusive regime. To this aim, we introduce the diffusion coefficients 
\begin{equation}
D_S=\lambda_S^2 \tau_S,\quad D_I=\lambda_I^2 \tau_I,\quad D_R=\lambda_R^2 \tau_R\,. 
\label{eq:diff2}
\end{equation}
Its diffusion limit can be formally recovered by letting the relaxation times $\tau_{S,I,R}\to 0$. Under this scaling, \eqref{eq.SIR_kinetic-fluxes} leads again to the parabolic reaction-diffusion system \eqref{eq:diff}. We remark that the motivation of choosing \eqref{eq.SIR_kinetic_diag} as our low-fidelity model in the bi-fidelity approximation is that it shares the same diffusion limit as the high-fidelity model. The only difference lies in the definition of two diffusion coefficients, as shown in \eqref{eq:diffcf} and \eqref{eq:diff2}.

\begin{figure}[tb]
\centering
\includegraphics[scale=0.33]{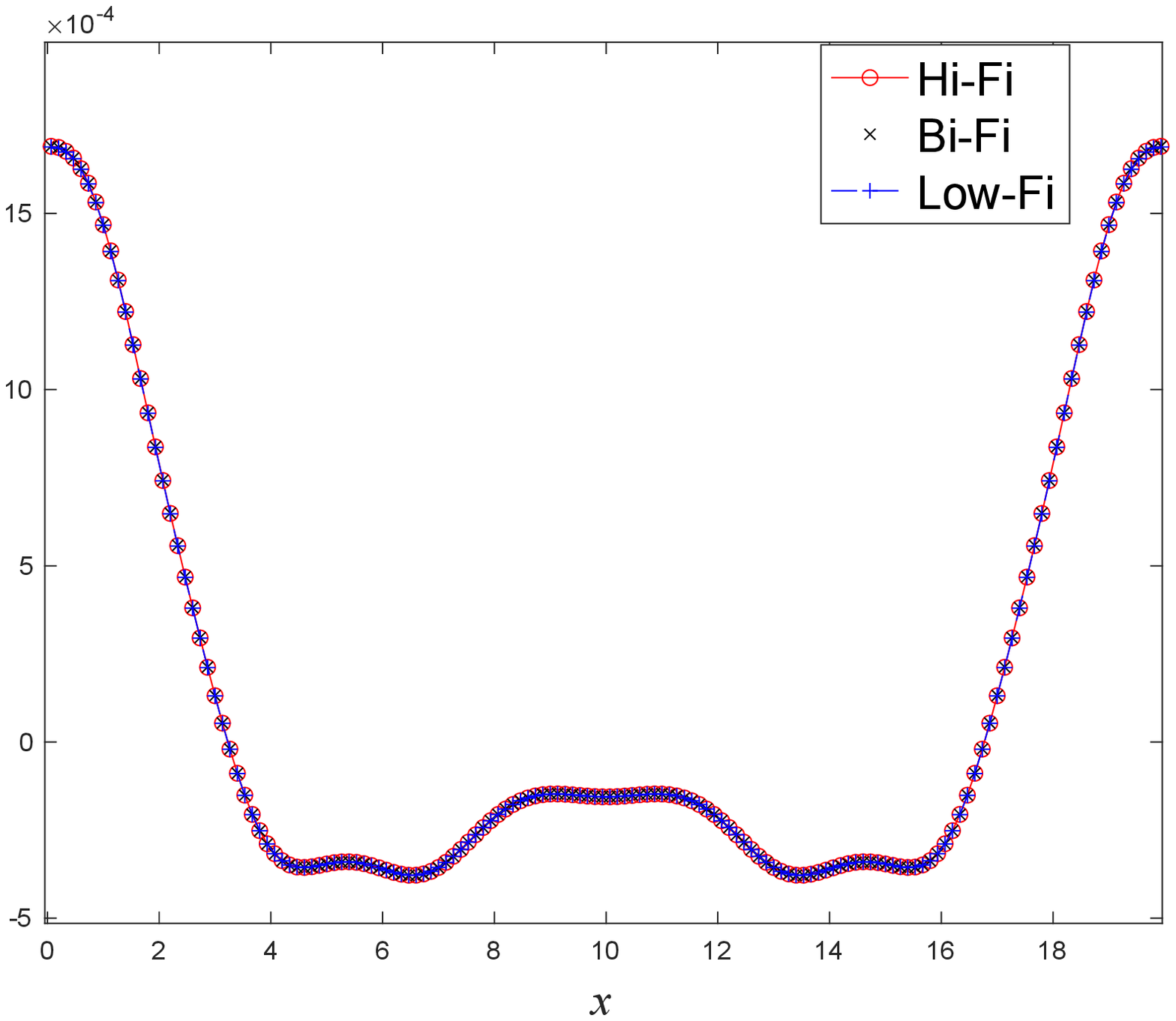}\hskip -.5cm
\includegraphics[scale=0.33]{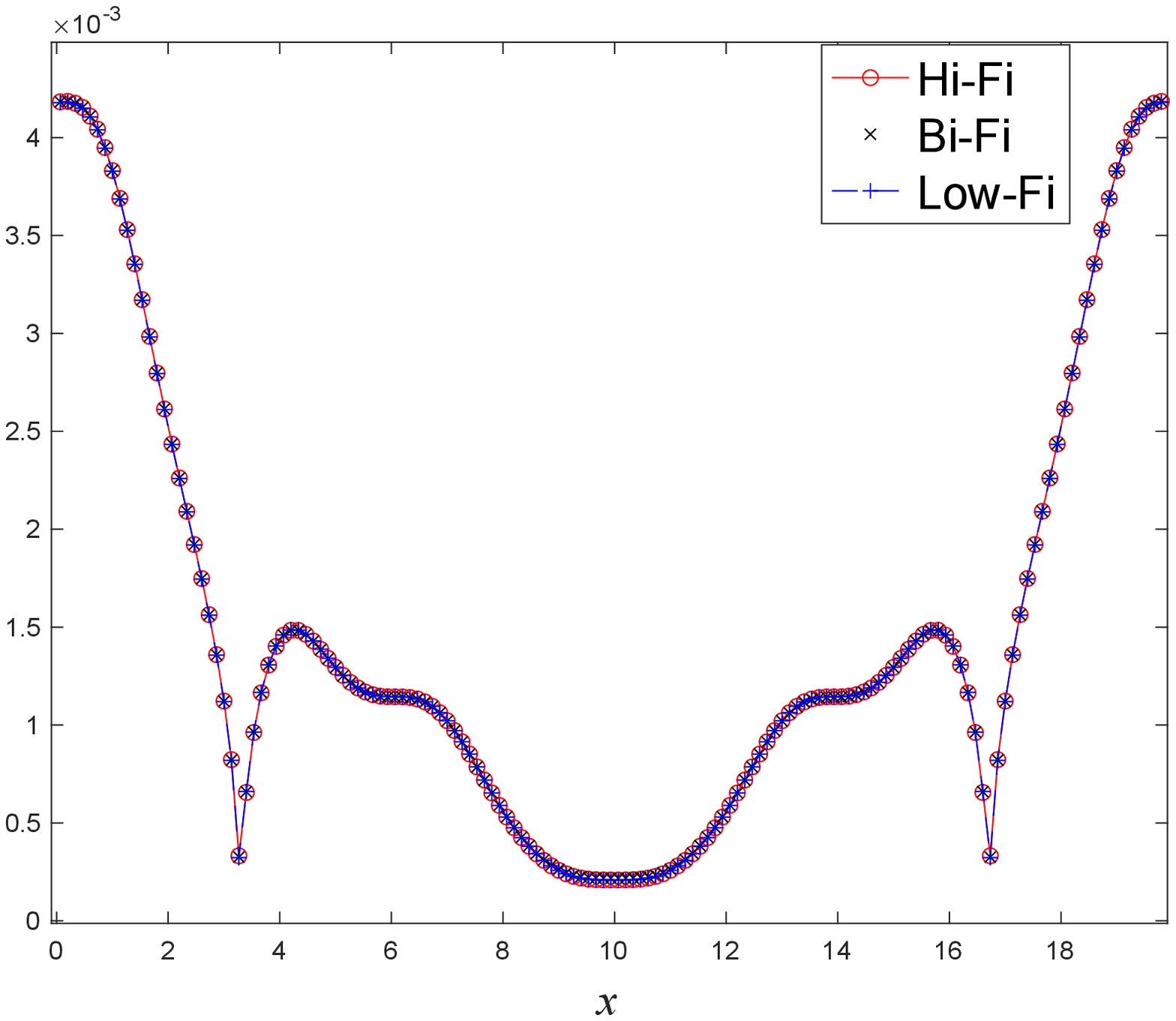}
\includegraphics[scale=0.08]{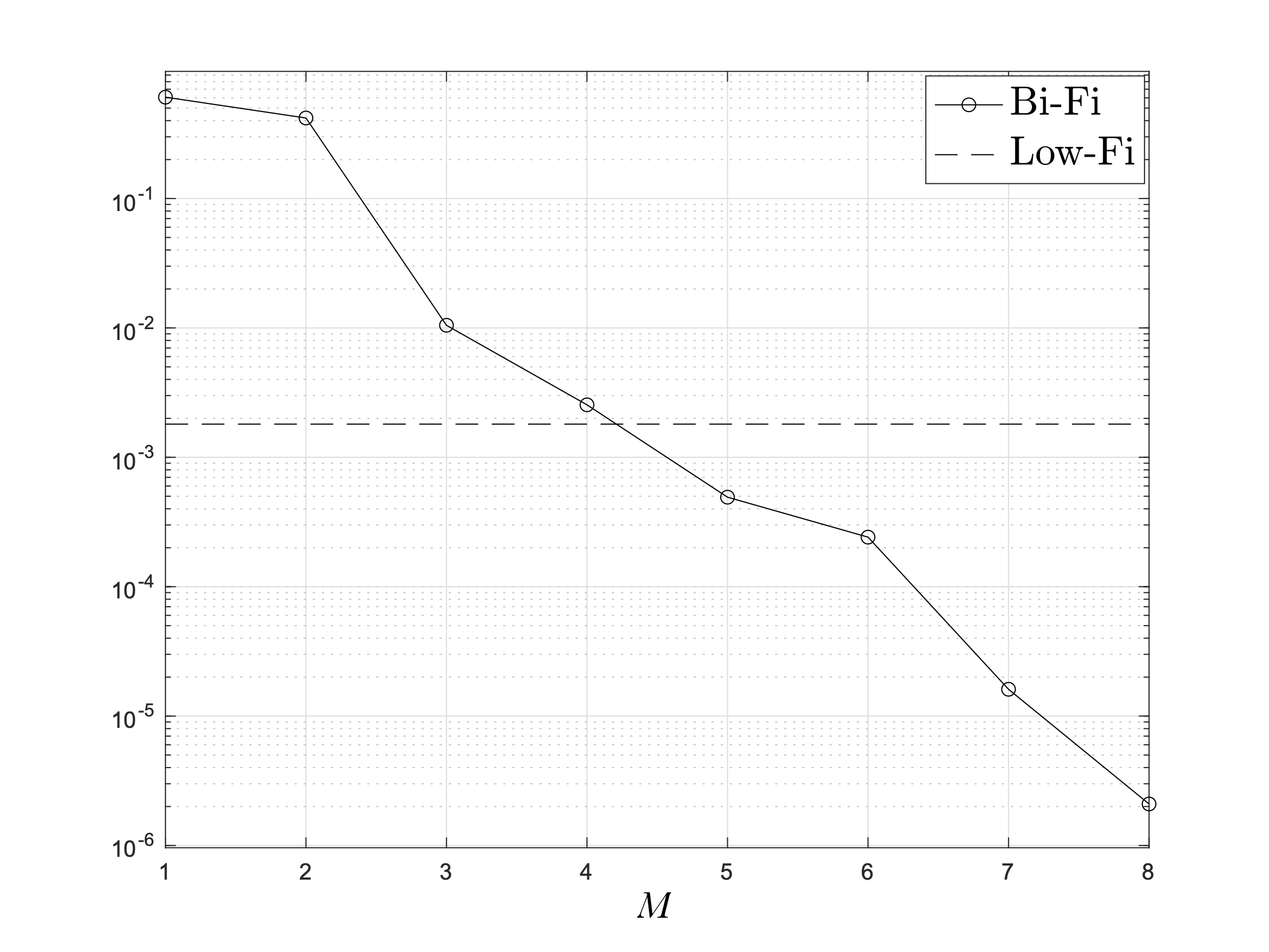}\hskip -.5cm
\includegraphics[scale=0.08]{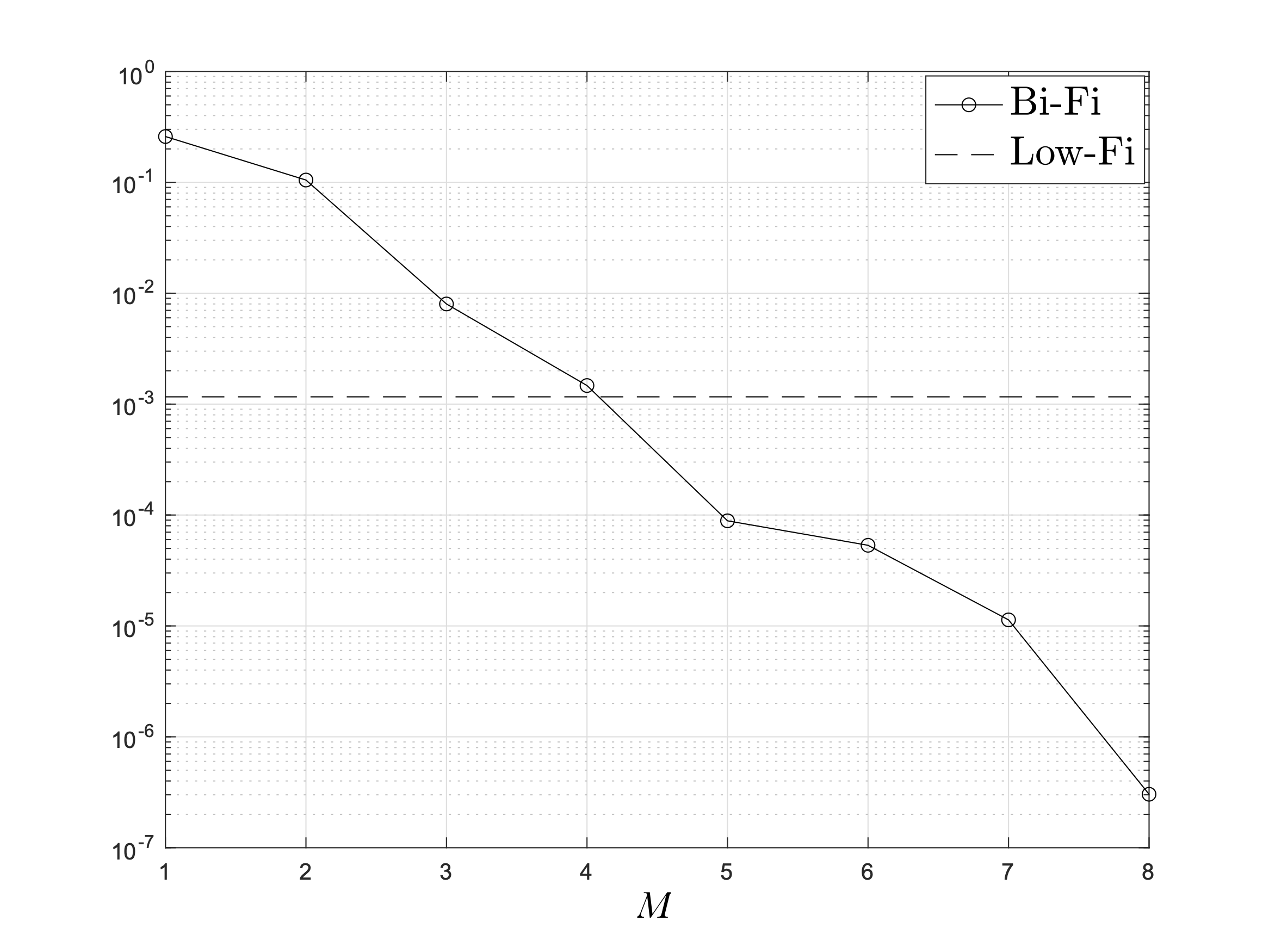}
\caption{{BFSC method for the kinetic SIR model.}  Results in the diffusive regime. First row: mean (left) and standard deviation (right) for the density $I$, by using $M=8$ high-fidelity runs. Second row: relative $L^2$ errors of the bi-fidelity approximation. }
\label{fig.test1a}
\end{figure}

In the following tests, we consider a 2-dimensional random vector ${\bf{z}}=(z_1, z_2)^T$, with independent random parameters $z_1$ and $z_2$ following a uniform distribution on $[-1,1]$. 
To compute the reference solutions, we adopt a $3$-rd level sparse grid quadrature based on Clenshaw-Curtis rules for the choice of the stochastic collocation nodes, with a total of $29$ points, for both the high-fidelity and low-fidelity models. Let assume the initial distributions of the high-fidelity kinetic SIR model \eqref{eq:kineticc} as follows:
\begin{equation}
f_i(x,v,0) = c\, i(x,0)\, e^{-\frac{v^2}{2}}, \qquad i \in \{S,I,R\}
\label{eq.ICf}
\end{equation}
where $c=\frac{1}{2} \sum_{i=1}^{N_v} w_i \, e^{-\frac{\zeta_i^2}{2}}$ is a re-normalization constant, with $N_v$ the number of Gauss-Legendre quadrature points used in velocity discretization. Let also the initial densities to be 
$$ S(x,0)= 1 - I(x,0), \qquad I(x,0) = 0.01 e^{-(x-10)^2}, \qquad R(x,0)=0, $$
with a physical domain $L=[0,20]$. Let finally the initial fluxes $J_S(x,0), J_I(x,0)$ and $J_R(x,0)$ be zeros and consider periodic boundary conditions. The same initial conditions for $S, I, R$ and $J_S, J_I, J_R$ are imposed in the low-fidelity SIR model \eqref{eq:density}-\eqref{eq.SIR_kinetic-fluxes}, to make the two models consistent. 
%In the following tests, we consider a $2$-dimensional random vector $z = (z_1, z_2)$, with i.i.d. random parameters $z_1$ and $z_2$ that follow a uniform distribution $z_j \sim U(-1,1)$, $j = 1, 2$.

Concerning spatially heterogeneous environments, we assume the contact rate^^>\cite{Wang2020,Bert} depends on space and contains uncertainties 
\begin{equation*}
\beta(x,{{z}}) = \beta^0({{z}}) \left(1+0.05\sin\left(\frac{13\pi x}{20}\right)\right),
\end{equation*}
where 
$$\beta^0({{z}}) = 11(1 + 0.6 z_1).$$ 
The uncertain recovery rate is given by 
$$\gamma({{z}}) = 10(1 + 0.4 z_2).$$
In the incidence function, we set $\kappa=0$ and $p=1$. For spatial and velocity discretizations we use $N_x=150$ in both the high-fidelity and low-fidelity models and $N_v=8$ for the high-fidelity model. 
In the first test, a parabolic configuration of speeds and relaxation parameters is considered. We let $\lambda_i^2=10^5$, $i \in \{S,I,R\}$, and $\tau_i=10^{-5}$ in the low-fidelity model and $\tau_i=3\times 10^{-5}$ in the high-fidelity model, to maintain consistency of the two simulations. 
%The time step size is set $\Delta t = 0.89 \times 10^{-2}$ in both models, and the low-fidelity solver is about 5 times faster than the high-fidelity solver. 
%%%
%%%
\begin{figure}[tb]
\centering
\includegraphics[scale=0.33]{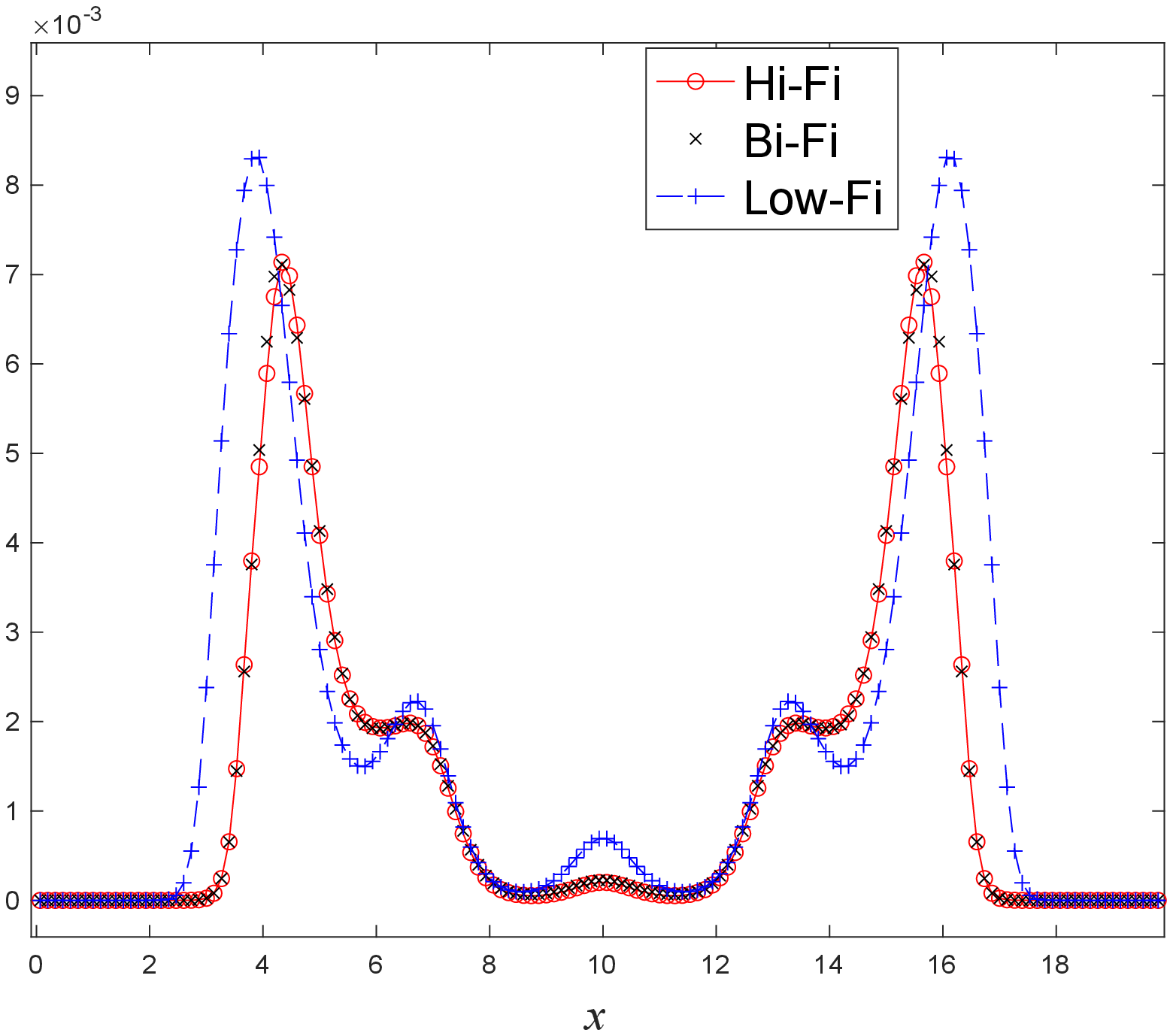}\hskip -.5cm
\includegraphics[scale=0.33]{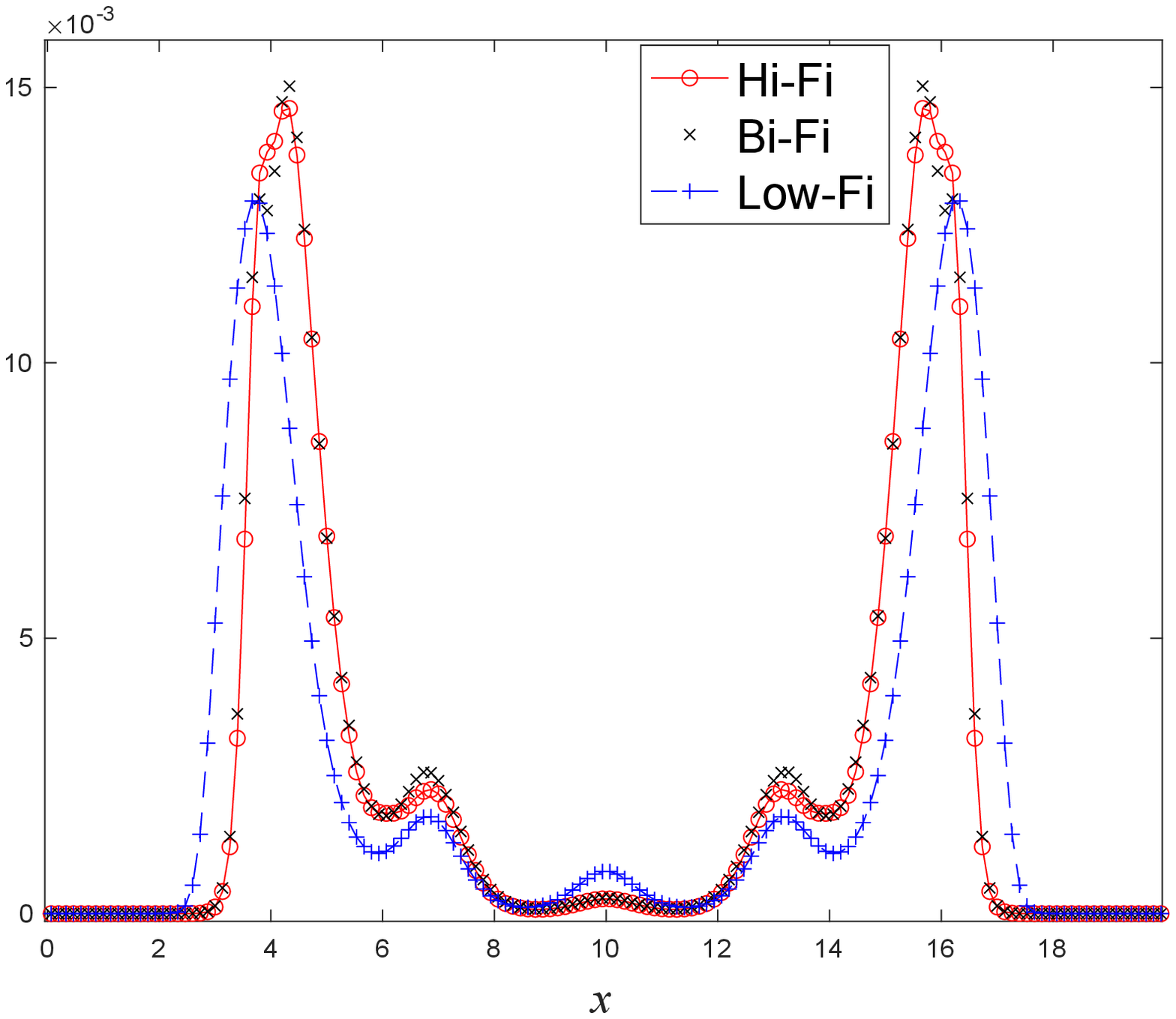}
\includegraphics[scale=0.08]{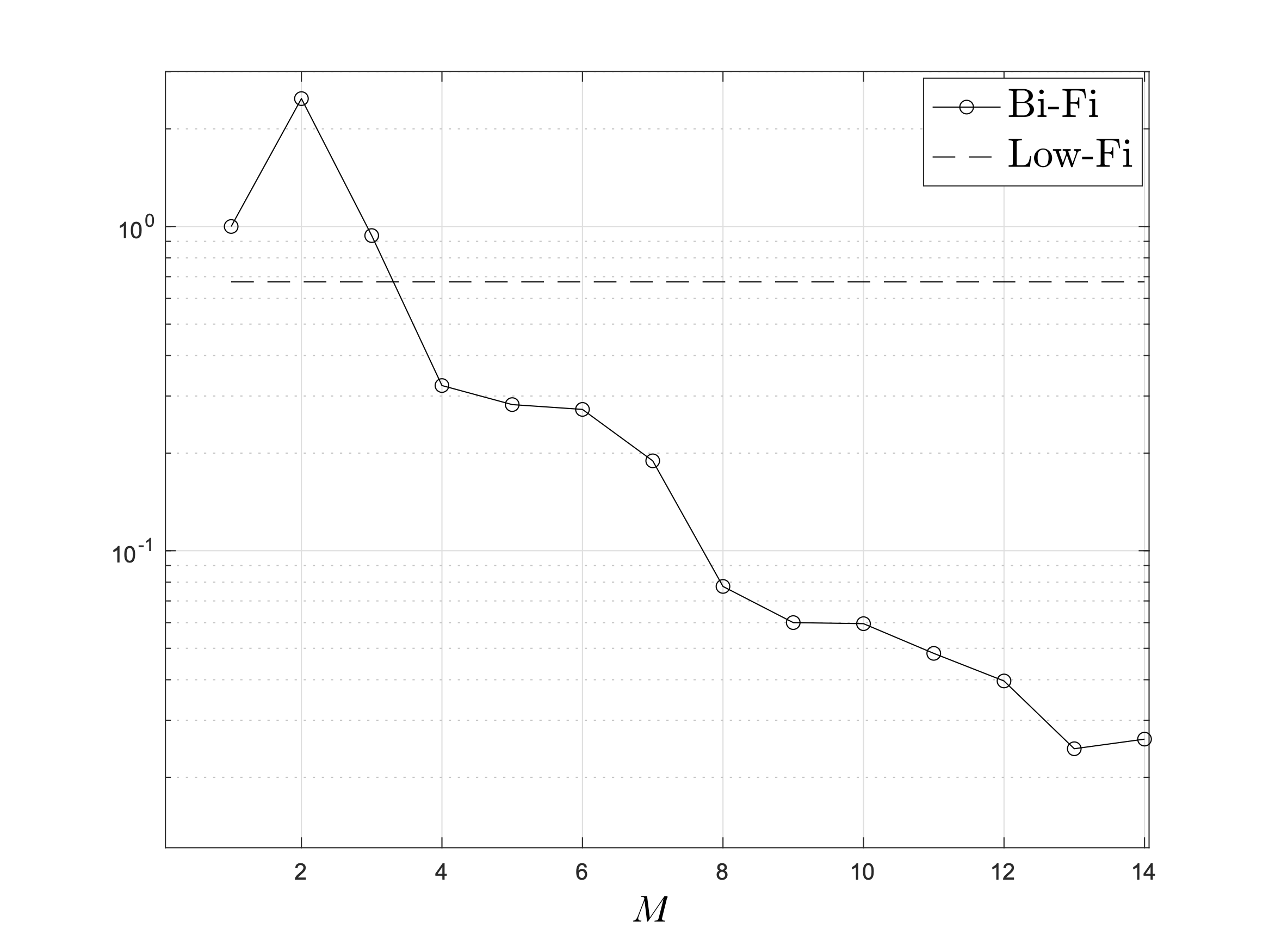}\hskip -.5cm
\includegraphics[scale=0.08]{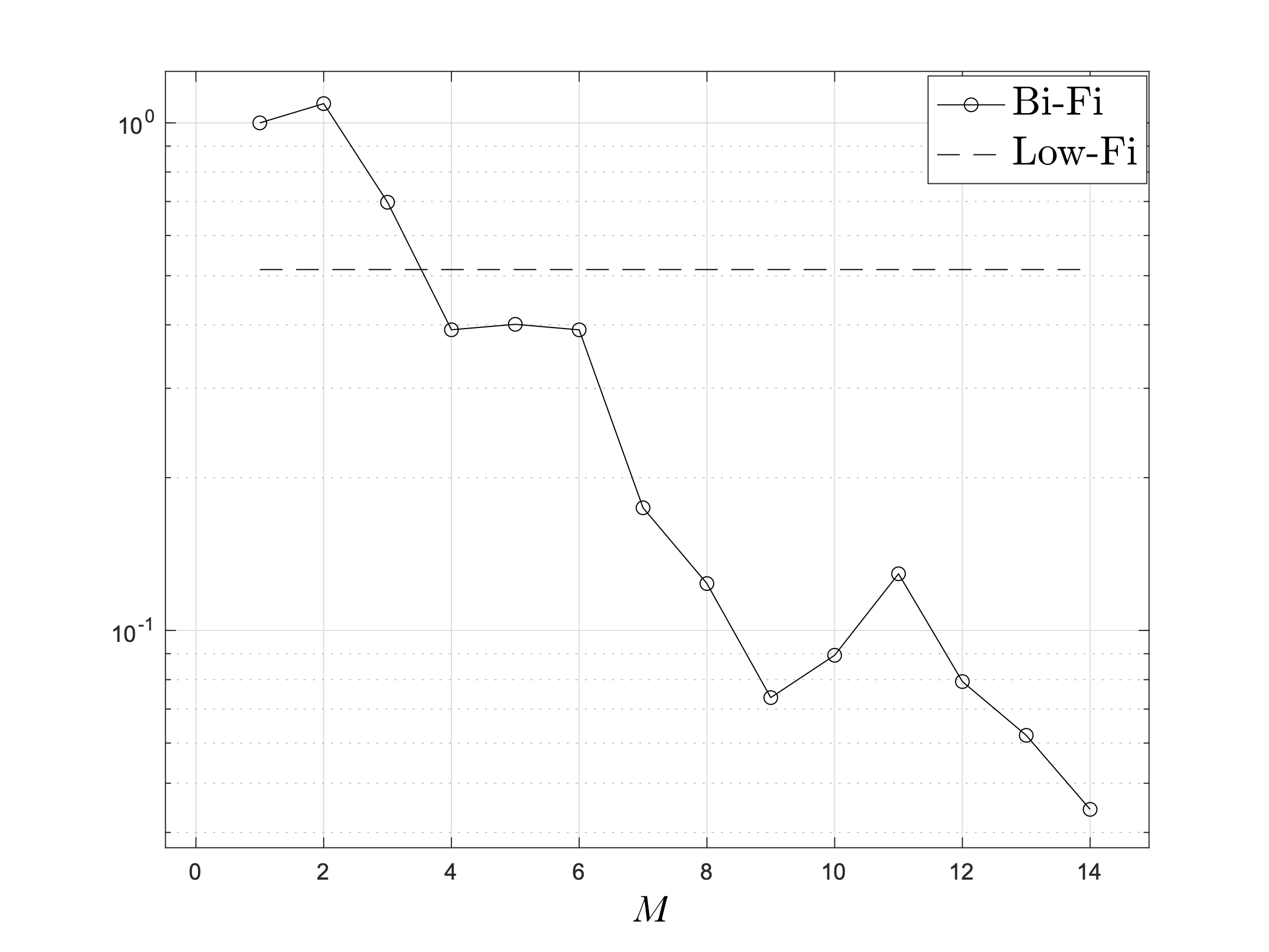}
\caption{{BFSC method for the kinetic SIR model.} Results in the hyperbolic regime. First row: mean (left) and standard deviation (right) for the density $I$, by using $M=8$ high-fidelity runs. Second row: relative $L^2$ errors of the bi-fidelity approximation. }
\label{fig.test1b}
\end{figure}
%%%

%%
%---------------------------------------------------------

In the first row of Figure \ref{fig.test1a}, the mean and standard deviation of the solution of compartment $I$ for the high-fidelity model, the low-fidelity model and the bi-fidelity approximation at time $t=5$ are shown. 
Since the high-fidelity and low-fidelity model share the same diffusive limit, a perfect agreement of the solutions are observed. In the second row, we plot the relative $L^2$ errors of the mean and standard deviation between the bi-fidelity and high-fidelity solutions at $t=5$, with respect to the number of selected important points of the bi-fidelity algorithm. A fast error decay is clearly observed. One can conclude that with only $8$ hi-fidelity sample points, bi-fidelity approximation is able to achieve a relative error of $\mathcal{O}(10^{-6})$ for both the mean and standard deviation. 
In the second test case, we consider the hyperbolic regime for $\lambda_i=1$, $i \in \{S,I,R\}$, with $\tau_i=1$ in the low-fidelity model and $\tau_i=3$ in the high-fidelity model. We plot the results in Figure \ref{fig.test1b}. Similar conclusions can be drawn as in the previous case.

\section{Conclusions}
In this survey we have discussed some recent developments in the field of multi-fidelity methods for kinetic equations with uncertain factors. To emphasize the wide range of applicability of the various techniques we adopted a model oriented presentation ranging from classical rarefied gas dynamics and transport theory to socio-economic modelling and epidemiology. Here we have focused essentially on two main classes of methods for uncertainty propagation: multi-fidelity methods based on control variates and bi-fidelity stochastic collocation methods. The main guideline in the development of the corresponding low-fidelity models is given by the classical legacy of fluid-dynamic and diffusive limits in kinetic theory. In particular, this allows appropriate error estimates to be placed alongside the various methods and to identify regimes in which the statistical error vanishes.

There are numerous open problems of great interest for future research. Among these we mention the use of low-fidelity models capable of covering a broader spectrum than classical fluid dynamics, such as the moment equations of extended thermodynamics in gasdynamics^^>\cite{manuel}. In addition to the problem of determining appropriate low-fidelity models in other areas of application where kinetic equations play a relevant role, e.g. plasma physics^^>\cite{Frank}, an important challenge is to move beyond methods that focus exclusively on models. Among these, the inclusion of information from experimental data and other informations sources, and their fusion with computational models is certainly one of the main challenges^^>\cite{multi,peherstorfersurvey}.

\subsection*{Acknowledgemnts} 
This work has been written within the
activities of GNCS and GNFM groups of INdAM (National Institute of
High Mathematics). G.D. and L.P. acknowledge the partial support of MIUR-PRIN Project 2017, No. 2017KKJP4X “Innovative numerical methods for evolutionary partial differential equations and applications”. L.L. holds the Direct Grant for Research supported by Chinese University of Hong Kong and Early Career Scheme 2021/22, No. 24301021, from Research Grants Council of Hong Kong. X. Zhu was supported by the Simons Foundation (504054).

\bibliographystyle{siam}
\bibliography{Survey_Ref.bib}

\end{document}